\documentclass[11pt,reqno,a4paper]{amsart}
\usepackage[margin=0.9in]{geometry}
\usepackage{amsmath, amsthm, amssymb}
\usepackage{enumitem}
\usepackage[colorlinks=true, pdfstartview=FitV, linkcolor=blue,citecolor=blue, urlcolor=blue]{hyperref}
\usepackage[usenames,dvipsnames]{xcolor}
\usepackage{subcaption}
\usepackage{mathtools}
\usepackage{bbm}
\usepackage{upgreek}
\usepackage{comment}
\usepackage{dsfont}
\usepackage{tikz}
\usepackage{xcolor}
\usepackage{dsfont}

\usepackage[numbers,sort&compress]{natbib}

 \usepackage[capitalise]{cleveref}
\usepackage{autonum}
\crefname{enumi}{}{}
\Crefname{enumi}{}{}

\crefformat{equation}{(#2#1#3)}
\crefrangeformat{equation}{(#3#1#4) to~(#5#2#6)}
\crefmultiformat{equation}{(#2#1#3)}%
{ and~(#2#1#3)}{, (#2#1#3)}{ and~(#2#1#3)}

\definecolor{labelkey}{rgb}{0,0,1}

\definecolor{redA}{RGB}{150,1,80}

\def\dd{{\rm d}}

\def\eps{\epsilon}
\def\e{{\rm e}} 
\def\de{{\partial}}

\def\RR {\mathbb{R}}
\def\R {\mathbb{R}} 

\def\NN {\mathbb{N}}

\def\Re{{\rm Re}}

\newcommand{\cA}{\mathcal{A}}

\newcommand{\cH}{\mathcal{H}}

\newcommand{\cJ}{\mathcal{J}}
\newcommand{\cK}{\mathcal{K}}
\newcommand{\cL}{\mathcal{L}}
\newcommand{\cM}{\mathcal{M}}

\newcommand{\cO}{\mathcal{O}}
\newcommand{\cP}{\mathcal{P}}

\newcommand{\cR}{\mathcal{R}}
\newcommand{\cS}{\mathcal{S}}
\newcommand{\cT}{\mathcal{T}}

\newcommand{\cX}{\mathcal{X}}
\newcommand{\cY}{\mathcal{Y}}
\newcommand{\cZ}{\mathcal{Z}}

\newcommand{\adv}{\mathsf{adv}}

\newcommand{\app}{\mathrm{app}}

\newcommand{\mrm}{\mathrm{m}}

\newcommand{\dhk}{{\tt{d}}}

\newtheorem{proposition}{Proposition}[section]
\newtheorem{theorem}[proposition]{Theorem}
\newtheorem{corollary}[proposition]{Corollary}
\newtheorem{lemma}[proposition]{Lemma}
\theoremstyle{definition}
\newtheorem{definition}[proposition]{Definition}
\newtheorem{remark}[proposition]{Remark}

\theoremstyle{definition}

\numberwithin{equation}{section}

\def\dhk{{\sf{d_{HK}}}}

\def\Z{\boldsymbol{Z}}

\def\cX{\mathcal{X}}
\def \X{\cX_{\eps,i}}
\def \cA{\mathcal{A}}
\def \I{\textnormal{I}_{\eps,i}}

\def \Xg{\boldsymbol{\cX}_{\eps}}
\def \Ii{\textnormal{I}_{\eps,i}}

\def \IIi{\textnormal{II}_{\eps,i}}
\def \IIIi{\textnormal{III}_{\eps,i}}

\newcommand{\zhk}[1]{z_{\rm #1}}

\usepackage{pict2e}

\DeclareRobustCommand{\blackinwhitediamond}{%
  \mathord{%
    \begingroup
    \setlength{\unitlength}{1.2ex}%
    \begin{picture}(1,1)
      \polygon(.5,0)(1,.5)(.5,1)(0,.5)
      \polygon*(.5,0.22)(.78,.5)(.5,.78)(.22,.5)
    \end{picture}%
    \endgroup
  }%
}
\title{Long-time dynamics of viscous Gröbli triangles}

\author{Michele Dolce}
\address{Institute of Mathematics, EPFL, Station 8, 1015 Lausanne, Switzerland}
\email{michele.dolce@epfl.ch}

\author{Antonio Tirotta}
\address{Institute of Mathematics, EPFL, Station 8, 1015 Lausanne, Switzerland}
\email{antonio.tirotta@epfl.ch}

\begin{document}

\maketitle

\begin{abstract}
In 1877 Gr\"obli discovered that three point vortices in the plane can move
self-similarly, forming a triangle that rotates while either expanding
indefinitely or collapsing in finite time. We study the viscous counterpart of
these configurations, namely the two-dimensional Navier--Stokes evolution
starting from three point vortices arranged in a Gr\"obli triangle. In the
expanding case, viscosity spreads the vortex cores at the same rate as the
triangle grows, so the vortices remain well separated relative to their size.
Building on this observation, we show that at high Reynolds number the solution
stays close to three Lamb--Oseen vortices centered near the vertices of the
inviscid triangle for times that grow exponentially with a power of the Reynolds
number. In the collapsing case, the same description holds up to times close to
the inviscid collision. As an application of the long-time result in the
expanding case, we obtain a vanishing viscosity realization of the vortex burst,
a nonuniqueness mechanism for point vortices in which a single vortex splits
into an expanding Gr\"obli triangle.\end{abstract}

\section{Introduction}

The understanding of the dynamics of highly concentrated vortices in incompressible fluids is a foundational theme in fluid dynamics \cite{meleshko2007bibliography}.  This has been an extremely active area of research since the pioneering studies of Helmholtz in the mid-19th century. He was the first
to idealize vortices as line filaments, and, after the subsequent refinements
of Kirchhoff, this led to the celebrated Helmholtz--Kirchhoff system describing the
(Hamiltonian) motion of point vortices (straight vortex filaments). These ideas also motivated
several investigations by J.J. Thomson\footnote{Thomson discovered the electron and, curiously enough, a plasma of electrons can realize experimentally several features of two-dimensional inviscid fluids \cite{Sch99}. For example, point vortex crystals appear consistently after a turbulent mixing phase.}, Kelvin and many others. Later,
Onsager used the point vortex model as one of the starting points for his
statistical theory of two-dimensional hydrodynamics \cite{onsager1949statistical}.  Since then, a vast literature has been devoted to the Helmholtz--Kirchhoff
system itself and we refer to Newton's book for a modern account on this subject \cite{NewtonBook}. In fact, point vortex dynamics can be regarded as a ``classical mathematics playground'' in Aref's words \cite{aref2007point}, in the sense that it retains several
features of the underlying fluid equations and exhibits remarkable connections
with other areas, e.g. celestial mechanics. In parallel, many results have justified
point vortex motion through desingularization procedures for inviscid fluids,
starting from the classical approaches of Marchioro and Pulvirenti \cite{MarPul84} and the variational methods \cite{Turkington_1985}. This  remains an extremely active area of research, e.g. \cite{CLZ21,DdPMP23,Hassainia25, donati2026confinementresultsnearpoint, hassainia2024desingularizationtimeperiodicvortexmotion}.

For two-dimensional viscous incompressible fluids, the analogous problem is to
understand to what extent the Navier--Stokes evolution of sharply concentrated
vortices is described by point vortex dynamics. The first results in this
direction go back to Marchioro \cite{Mar98}. A major breakthrough was then obtained by Gallay
\cite{Gallay2011}, who proved that the Navier--Stokes evolution emanating from a finite
collection of Dirac masses is, in the large Reynolds number regime, accurately
described by a superposition of Lamb--Oseen vortices whose centers follow a viscous regularization of the
Helmholtz--Kirchhoff system. Most importantly, this approximation holds on every fixed time
interval on which the inviscid point vortex motion is well defined. For desingularizations of more general concentrated initial data, Ceci and
Seis \cite{CeciSeis24} obtained concentration estimates in
Wasserstein distance, that are uniform in viscosity and valid on time intervals that may grow logarithmically
with the Reynolds number. More recently, building on the methods developed
by Gallay and \v{S}ver\'ak for one viscous vortex ring \cite{GaSring}, the
Lamb--Oseen approximation has been extended to much longer time scales in
special configurations. In the case of a translating vortex dipole
\cite{dolce2024long}, general pairs with non zero total circulation \cite{zhang2025long}, and rigidly rotating
vortex crystals \cite{dolce2026viscousvortexcrystals}, its validity was pushed
arbitrarily close to the diffusive time scale. In all these cases, however,
the rigidity of the underlying point vortex motion plays a fundamental role.

 The purpose of the present paper is to treat a different and dynamically richer situation: three vortices undergoing a self-similar expansion or collapse. This is a special case of ``the 3-vortex problem (that) is fundamental to much of the
development of the subject''-\cite{NewtonBook}, see also \cite{anurag2024new}. These solutions of the Helmholtz--Kirchhoff system were discovered by Gr\"obli\footnote{see \cite{aref1992grobli} for a nice historical account of Gr\"obli's biography.} in
his PhD thesis in 1877 and later rediscovered, in different forms, by other authors
\cite{aref1992grobli}. In the inviscid setting, a desingularization of the self-similar expanding  solutions was obtained in \cite{davila2024expanding}, while  \cite{zbarsky2021point,iftimie2018self} exploits the expanding mechanism to show that the mutual distances among concentrated vortices approximating the desired configuration grow
like $\sqrt{t}$, along with a suitable concentration estimate.  The viscous problem considered here is qualitatively different:
diffusion continuously changes the vortex cores, and the Navier--Stokes
evolution is uniquely determined by the atomic initial datum. Moreover, the
inviscid dynamics prescribes that the mutual distances between the vortices
grow (or collapse) like $\sqrt{t}$ (or $\sqrt{t_0-t}$), while viscosity spreads each vortex core at the rate
$\sqrt{\nu t}$. Hence, in the appropriate high Reynolds number regime, the
ratio between the viscous core size and the distance between the vortices
remains uniformly small for all times $t\geq 0$ (or up to times $t$ suitably close to $t_0^{-}$, depending on the viscosity). This structural feature is the main
reason why in the expanding case one can hope to propagate the point vortex approximation far beyond
the time scales reached in all previous viscous desingularization results. Moreover, if the dynamics is followed up to a sufficiently large time-scale, this has implications for a vanishing viscosity justification of the nonuniqueness mechanism in Helmholtz--Kirchhoff via the so called \emph{vortex burst} \cite{grotto2022burst}. 

\subsection{Mathematical setting}
We consider the two-dimensional Navier--Stokes equations in the whole plane
$\mathbb{R}^2$, written in vorticity form as
\begin{align}
\label{eq:NSintro}
    \begin{cases}
        \de_t \omega+\{\psi,\omega\}=\nu\Delta \omega,
        \qquad t>0,\ x\in \mathbb{R}^2,\\
        \Delta\psi=\omega.
    \end{cases}
\end{align}
Here $\nu>0$ is the kinematic viscosity,
\[
    \{f,g\}=\nabla^\perp f\cdot \nabla g
    =\de_1 f\,\de_2 g-\de_2 f\,\de_1 g
\]
is the standard Poisson bracket, $\nabla^\perp \psi=u$ the velocity field and the streamfunction is given by
\begin{align}
    \psi(x,t)=\frac{1}{2\pi}\int_{\mathbb{R}^2}
    \ln|x-y|\,\omega(y,t)\,\dd y.
\end{align}
We are interested in initial data consisting of a finite collection of point
vortices,
\begin{align}
\label{def:omin}
    \omega^{\rm in}(x)=\sum_{i=1}^N\Gamma_i\delta_{z_i^{\rm in}},
\end{align}
where $\Gamma_i\neq 0$ is the circulation of the vortex initially located at
$z_i^{\rm in}$. By the result of Gallagher and Gallay \cite{GaGa2005} the
Navier--Stokes system is globally well posed for such initial data. In the inviscid limit, the vortex centers are expected to evolve according to
the Helmholtz--Kirchhoff system
\begin{equation}
\label{eq:HK}
    z_i'(t)=\sum_{j\neq i}\frac{\Gamma_j}{2\pi}
    \frac{(z_i(t)-z_j(t))^\perp}{|z_i(t)-z_j(t)|^2},
    \qquad z_i(0)=z_i^{\rm in},
    \qquad i=1,\dots,N.
\end{equation}
This approximation has been justified, in different senses, both for inviscid
and viscous desingularizations as briefly mentioned previously. In this paper we focus on the case
$N=3$. Although the three-vortex problem is integrable, it already displays
several interesting phenomena, such as self-similar collapse and expansion \cite{NewtonBook}.

From now on we fix the origin to the total vortex center so that 
\begin{align}
\label{eq:center_of_vorticity}
     \sum_{i=1}^3\Gamma_i z_i^{\rm in}
         =0
\end{align}
and we identify $\mathbb R^2$ with $\mathbb C$, so that rotations are
written as multiplication by complex phases.
Then the Helmholtz--Kirchhoff trajectories we will be looking at take the form
\begin{equation}
\label{def:zjGrobli}
    z_i(t)=\dhk(t)e^{\mathrm{i}\theta(t)}{\sf n}_i,
    \qquad i=1,2,3, 
\end{equation}
where we set  \begin{equation}
\label{def:key_parameters_intro}
    d=\max_{i\neq j}|z_i^{\rm in}-z_j^{\rm in}|,
    \qquad
    \Gamma:=\max_{1\leq i\leq 3}|\Gamma_i|,
    \qquad
    \gamma_i:=\frac{\Gamma_i}{\Gamma}.
\end{equation}
We call $d$ the size of the initial triangle and $\Gamma$ the largest
circulation in modulus. The vectors $({\sf n}_1,{\sf n}_2,{\sf n}_3)$ where ${\sf n}_i:=z_i^{\rm in}/d$, $i=1,2,3$, form
the normalized initial triangle, so that
\begin{equation}
\label{def:normalized}
    \max_{i\neq j}|{\sf n}_i-{\sf n}_j|=1.
\end{equation}
Then, plugging the ansatz \cref{def:zjGrobli} in \cref{eq:HK} and exploiting the conserved quantities of this system, see Section \ref{sec:selfsimilar_HK}, the scaling and rotation factors are  explicitly given by
\begin{equation}
\label{def:dhk}
\dhk(t)=d\,\sqrt{1+\operatorname{sign}(\alpha)\frac{t}{T_{\rm adv}}}, \qquad \theta(t)
    =
    \frac{\beta}{\alpha}
    \log\left(1+ \operatorname{sign}(\alpha)\frac{t}{T_{\rm adv}}\right), \qquad T_{\rm adv}:=\frac{d^2}{|\alpha| \Gamma}
\end{equation}
where $\alpha$ and $\beta$ are dimensionless parameters depending only on the \emph{normalized} self similar configuration, whereas the \emph{advection time} $T_{\rm adv}$ is the characteristic time-scale of the point vortex motion.  When $\alpha>0$ we have a self-similar expanding configuration whereas for $\alpha<0$ we get a self-similar collapsing one. The case $\alpha\to 0$ formally corresponds to the 
relative equilibrium situation of an equilateral triangle, that is in fact covered by \cite{dolce2026viscousvortexcrystals}, meaning that we will always assume $\alpha\neq 0$. The explicit expressions for $\alpha,\beta$ and their derivations are recalled in
Subsection ~\ref{app:point_vortices}. These  solutions were constructed by Gr\"obli
in 1877, and the conditions to obtain the self-similar configurations are
\begin{align}
\label{eq:groebli_condition_intro}
\sum_{1\leq i<j\leq 3}\Gamma_i\Gamma_j=0\qquad \sum_{1\leq i<j\leq 3}    \Gamma_i\Gamma_j |z_i^{\rm in}-z_{j}^{\rm in}|^2
    =0 
,
\end{align}
together with a suitable choice of the initial positions described in Section \ref{sec:selfsimilar_HK}. The conditions above follow from the conservation of the Hamiltonian and of angular impulse for the Helmholtz--Kirchhoff system respectively. We will refer to these solutions as the expanding (collapsing) \emph{Gr\"obli triangle} when $\alpha>0$ ($\alpha<0$). A nice geometric description of the initial configurations leading to a self-similar motion can be directly inferred from the conserved quantities of the system, as shown for instance by Aref in \cite{ArefSelfSimilar}. Once
two vortices are fixed, the admissible positions of the third vortex lie on a
circle. Some arcs correspond to self-similar expansion, while the complementary
ones correspond to self-similar collapse. The endpoints of these arcs are
relative equilibria. In the rest of the paper we work in the
reference frame centered at the total center of vorticity, so that the Gr\"obli
triangle in \cref{def:zjGrobli} grows (collapses) directly around the origin, and we refer to 
\cref{fig:aref_circle_comparison} for a visualization of the setting.
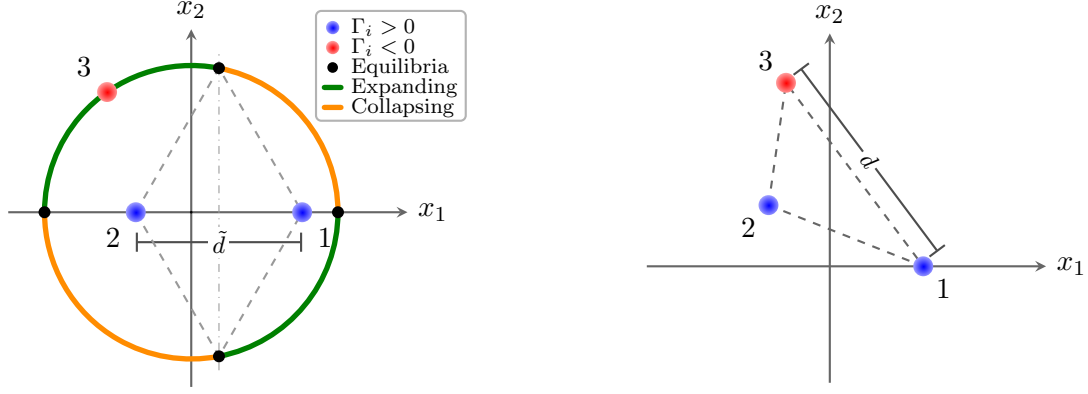
\begin{figure}[htbp]
\centering

\begin{minipage}{0.49\textwidth}
\centering
\begin{tikzpicture}[scale=2.2, >=stealth]

\definecolor{collorange}{RGB}{255,140,0}   

\tikzset{
    vortex pos/.style={circle, shading=radial, inner color=blue, outer color=blue!40, draw=none, minimum size=7.5pt, inner sep=0pt},
    vortex neg/.style={circle, shading=radial, inner color=red, outer color=red!40, draw=none, minimum size=7.5pt, inner sep=0pt},
    degen/.style={circle, fill=black, draw=none, minimum size=4.5pt, inner sep=0pt},
    arc coll/.style={line width=2pt, collorange, line cap=round},
    arc exp/.style={line width=2pt, green!50!black, line cap=round},
    gridline/.style={dashed, black!40, thick}
}

\pgfmathsetmacro{\GOne}{1}
\pgfmathsetmacro{\GTwo}{2}
\pgfmathsetmacro{\GThree}{-\GOne*\GTwo/(\GOne+\GTwo)}
\pgfmathsetmacro{\d}{1}

\pgfmathsetmacro{\xone}{-\GOne/(\GOne+\GTwo)*\d}
\pgfmathsetmacro{\xtwo}{ \GTwo/(\GOne+\GTwo)*\d}
\pgfmathsetmacro{\xmid}{0.5*(\xone+\xtwo)}

\pgfmathsetmacro{\R}{sqrt((\GOne+\GTwo+\GThree)/(\GOne+\GTwo))*\d}
\pgfmathsetmacro{\yeq}{sqrt(3)/2*\d}
\pgfmathsetmacro{\thetaE}{atan2(\yeq,\xmid)}

\coordinate (O)  at (0,0);
\coordinate (V1) at (\xone,0);
\coordinate (V2) at (\xtwo,0);
\coordinate (E1) at (\R,0);
\coordinate (E3) at (-\R,0);
\coordinate (E2) at (\xmid,\yeq);
\coordinate (E4) at (\xmid,-\yeq);

\draw[->, thick, black!60] (-1.1,0) -- (1.3,0) node[right, text=black] {$x_1$};
\draw[->, thick, black!60] (0,-1.05) -- (0,1.1) node[above, text=black] {$x_2$};

\fill (O) circle (0.008);
\draw[dash dot, black!40] (\xmid,-0.95) -- (\xmid,0.95); 

\draw[|-|, thick, black!70] (\xone, -0.18) -- (\xtwo, -0.18) node[midway, fill=white, inner sep=1.5pt, text=black] {\footnotesize $\tilde{d}$};

\draw[gridline] (V1) -- (E2) -- (V2);
\draw[gridline] (V1) -- (E4) -- (V2);

\draw[arc coll] (0:\R) arc (0:\thetaE:\R);
\draw[arc exp]  (\thetaE:\R) arc (\thetaE:180:\R);
\draw[arc coll] (180:\R) arc (180:{360-\thetaE}:\R);
\draw[arc exp]  ({360-\thetaE}:\R) arc ({360-\thetaE}:360:\R);

\node[degen] at (E1) {};
\node[degen] at (E2) {};
\node[degen] at (E3) {};
\node[degen] at (E4) {};

\node[vortex pos] at (V1) {}; \node[below left=2pt] at (V1) {$2$};
\node[vortex pos] at (V2) {}; \node[below right=2pt] at (V2) {$1$};

\pgfmathsetmacro{\thetaThree}{125}
\coordinate (V3) at ({\R*cos(\thetaThree)},{\R*sin(\thetaThree)});
\node[vortex neg] at (V3) {}; \node[above left=2pt] at (V3) {$3$};

\begin{scope}[shift={(0.85, 1.15)}]
    \draw[draw=black!30, fill=white, fill opacity=0.9, text opacity=1, rounded corners=2pt, thick] (-0.1, 0.06) rectangle (0.8, -0.6);
    \node[vortex pos, minimum size=5pt] at (0, -0.04) {}; \node[right] at (0.05, -0.04) {\scriptsize $\Gamma_i > 0$};
    \node[vortex neg, minimum size=5pt] at (0, -0.16) {}; \node[right] at (0.05, -0.16) {\scriptsize $\Gamma_i < 0$};
    \node[degen, minimum size=4pt]      at (0, -0.28) {}; \node[right] at (0.05, -0.28) {\scriptsize Equilibria};
    \draw[arc exp]  (-0.05, -0.40) -- (0.05, -0.40);      \node[right] at (0.05, -0.40) {\scriptsize Expanding};
    \draw[arc coll] (-0.05, -0.52) -- (0.05, -0.52);      \node[right] at (0.05, -0.52) {\scriptsize Collapsing};
\end{scope}
\end{tikzpicture}
\end{minipage}%
\hfill
\begin{minipage}{0.49\textwidth}
\centering
\begin{tikzpicture}[scale=2.2, >=stealth]

\definecolor{collorange}{RGB}{255,140,0}
\tikzset{
    vortex pos/.style={circle, shading=radial, inner color=blue, outer color=blue!40, draw=none, minimum size=7.5pt, inner sep=0pt},
    vortex neg/.style={circle, shading=radial, inner color=red, outer color=red!40, draw=none, minimum size=7.5pt, inner sep=0pt},
    gridline/.style={dashed, black!60, thick}
}

\pgfmathsetmacro{\GOne}{1}
\pgfmathsetmacro{\GTwo}{2}
\pgfmathsetmacro{\GThree}{-\GOne*\GTwo/(\GOne+\GTwo)}
\pgfmathsetmacro{\d}{1}

\pgfmathsetmacro{\xone}{-\GOne/(\GOne+\GTwo)*\d}
\pgfmathsetmacro{\xtwo}{ \GTwo/(\GOne+\GTwo)*\d}
\pgfmathsetmacro{\R}{sqrt((\GOne+\GTwo+\GThree)/(\GOne+\GTwo))*\d}
\pgfmathsetmacro{\thetaThree}{125}
\pgfmathsetmacro{\xthree}{\R*cos(\thetaThree)}
\pgfmathsetmacro{\ythree}{\R*sin(\thetaThree)}

\pgfmathsetmacro{\Gtot}{\GOne+\GTwo+\GThree}
\pgfmathsetmacro{\CVx}{(\GThree*\xthree)/\Gtot}
\pgfmathsetmacro{\CVy}{(\GThree*\ythree)/\Gtot}

\pgfmathsetmacro{\angleVtwo}{atan2(-\CVy, \xtwo-\CVx)}

\coordinate (CV)  at (0,0);

\begin{scope}[rotate=-\angleVtwo]
    \coordinate (V1s) at (\xone-\CVx, -\CVy);
    \coordinate (V2s) at (\xtwo-\CVx, -\CVy);
    \coordinate (V3s) at (\xthree-\CVx, \ythree-\CVy);
\end{scope}


\draw[->, thick, black!60] (-1.1,0) -- (1.3,0) node[right, text=black] {$x_1$};
\draw[->, thick, black!60] (0,-0.7) -- (0,1.4) node[above, text=black] {$x_2$};

\draw[gridline] (V1s) -- (V2s) -- (V3s) -- cycle;

\draw[|-|, thick, black!70] ([xshift=2.5pt, yshift=2.5pt]V2s) -- ([xshift=2.5pt, yshift=2.5pt]V3s) node[midway, fill=white, inner sep=1.5pt, text=black, sloped] {\footnotesize $d$};

\node[vortex pos] at (V1s) {}; \node[below left=1pt] at (V1s) {$2$};
\node[vortex pos] at (V2s) {}; \node[below right=1pt] at (V2s) {$1$};
\node[vortex neg] at (V3s) {}; \node[above left=1pt] at (V3s) {$3$};

\end{tikzpicture}
\end{minipage}

\caption{\textsc{Left}: Aref's circle characterizing the self-similar regimes of three point vortices for the condition $
\Gamma_1\Gamma_2+\Gamma_2\Gamma_3+\Gamma_3\Gamma_1=0.$ The origin is the center of vorticity of vortices $1$ and $2$, separated by distance $\tilde{d}$. The admissible locations of the third vortex lie on the circle
$
x_1^2+x_2^2
=
\frac{\Gamma_1+\Gamma_2+\Gamma_3}{\Gamma_1+\Gamma_2}\,\tilde{d}^2.$
Placing the third vortex on the green arcs produces self-similar expansion; the orange arcs produce self-similar collapse. The black dots represent relative equilibria. \textsc{Right:} The identical vortex configuration ($1, 2,$ and $3$) transformed into a reference frame centered on the total center of vorticity. In this frame, which is the one used in this paper, the self-similar motion of the dashed triangle occurs directly around the origin.}
\label{fig:aref_circle_comparison}
\end{figure}

To formulate the viscous approximation, let us first recall that a single
Dirac mass evolves under Navier--Stokes as a Lamb--Oseen vortex, that is
\begin{align}
\label{def:G}
    \omega^{\rm in}=\Gamma\delta_{z_0}
    \quad\Longrightarrow\quad
    \omega(x,t)=\frac{\Gamma}{\nu t}
    G\left(\frac{x-z_0}{\sqrt{\nu t}}\right),
    \qquad
    G(\xi):=\frac{1}{4\pi}e^{-|\xi|^2/4}.
\end{align}
It is therefore natural to expect that, at leading order, three interacting
point vortices give rise to a superposition of Lamb--Oseen vortices whose
centers follow the Helmholtz--Kirchhoff dynamics, namely
\begin{align}
\label{eq:omegaPVintro}
    \omega_{\rm PV}(x,t)
    =
    \frac{1}{\nu t}\sum_{i=1}^3
    \Gamma_i
    G\left(\frac{x-z_i(t)}{\sqrt{\nu t}}\right),
\end{align}
where $z_i(t)$ solves \cref{eq:HK}. To measure quantitatively the validity of an approximation, we introduce 
\begin{align}
\label{eq:QtyDef}
    \delta:=\frac{\nu}{\Gamma},
    \qquad
    T_{\rm diff}:=\frac{d^2}{\nu}.
\end{align}
 The parameter $\delta$ is the inverse circulation Reynolds number while
$T_{\rm diff}$ is the time at which viscosity spreads a vortex core over a
distance comparable with $d$. Most importantly, the parameter
\begin{equation}
    \label{def:epsintro}
    \eps(t):=\frac{\sqrt{\nu t}}{\dhk(t)} \quad \text{(\emph{effective aspect ratio})}
\end{equation}
measures the ratio between viscous spreading versus the expected inviscid dynamics. The standard aspect ratio used in \cite{Gallay2011,dolce2024long,dolce2026viscousvortexcrystals,zhang2025long} is instead fixed with the minimal distance reached by the Helmholtz-Kirchhoff dynamics. The key feature of the effective aspect ratio is that denoting
\begin{align}
\label{def:keytaueps}
\varsigma:=\operatorname{sign}(\alpha)\quad \Longrightarrow \quad \eps(t)^2=\frac{\delta}{|\alpha|}\frac{\frac{t}{T_{\rm adv}}}{1+\varsigma\frac{t}{T_{\rm adv}}}.
\end{align}
This is telling us that in the expanding case ($\varsigma=1$) the effective aspect ratio is proportional to $\sqrt{\delta}$ even at large times, whereas in the collapsing case ($\varsigma=-1$) it explodes as $t \to T_{\rm adv}$. 
\begin{remark}
    The standard aspect ratio used in relative equilibria as in \cite{dolce2024long,dolce2026viscousvortexcrystals,donati2025fast,zhang2025long} satisfies the following relation instead
    \begin{align}
    \label{def:standardeps}
        \varepsilon(t):=\frac{\sqrt{\nu t}}{d} \quad \Longrightarrow \quad \varepsilon(t)^2=\frac{\delta}{|\alpha|}\frac{t}{T_{\rm adv}}=\frac{t}{T_{\rm diff}}.
    \end{align}
    Hence, the parameter $\varepsilon$ remains a useful smallness parameter if and only if $t\ll T_{\rm diff}$, and $t\to T_{\rm diff}$ represents a technical but also physically relevant barrier \cite{dolce2026viscousvortexcrystals}.
\end{remark}
\subsection{Main result}
As a starting point, the estimates in
\cite{Gallay2011,donati2025fast} lead to the following informal statement
\begin{quote}
    On time intervals where the
Helmholtz--Kirchhoff dynamics remains well defined, i.e. collision-free, the Navier--Stokes
solution emanating from point vortices is well approximated by a
superposition of Lamb--Oseen vortices whose centers follow the point vortex
motion. More precisely, let $d_*=\min_{t,\, i\neq j}|(z_j-z_i)(t)|>0$, then
\begin{align}
\label{bd:naive}
    \frac{1}{\Gamma}
    \|(\omega-\omega_{\rm PV})(t)\|_{L^1(\mathbb{R}^2)}
    \lesssim
    \frac{\nu t}{d_*^2},
    \qquad
    0<t<T_{*,\rm adv}|\ln\delta|, \quad T_{*,\rm adv}:=\frac{d_*^2}{\max_{i}|\Gamma_i|}
\end{align}
\end{quote}
The general theory does not exploit the structural identity
\cref{def:keytaueps}. In particular, an aspect ratio based on a fixed minimal
separation cannot remain uniformly small along an expanding trajectory. The
effective aspect ratio is the mechanism that allows us to go beyond the
usual diffusive time scale obstruction in \cite{dolce2024long,dolce2026viscousvortexcrystals,zhang2025long} and, crucially for the burst limit, to obtain
estimates uniform in the initial size $d$.
Our first result is then summarized in the following simplified version.
\begin{theorem}
\label{thm:main_intro}
Let $\omega$ be the Navier--Stokes solution with initial datum \cref{def:omin} where $(\Gamma_i,z_i^{\rm in})$ generate an expanding $\varsigma=1$ (or collapsing $\varsigma=-1$)
Gr\"obli triangle $z_i(t)$ of \cref{eq:HK}. Let $d>0$ and $\{{\sf n}_i\}_{i=1}^3$ as in \cref{def:zjGrobli}. 
Then there exist sufficiently small constants $\delta_0,p_0,p_0'\in (0,1)$ and a large constant $C_0>1$, depending only on $\{{\gamma_i, \sf n}_i\}$ but independent of $d$ such that the following holds: assume that
$\delta\leq\delta_0$  and fix $p\in(0,p_0)$ (or $p'\in(0,p_0'))$. Then there exist trajectories $Z_j(t)$, defined through
the modulation conditions in \cref{eq:dtZi}, such that
\begin{align}
\label{eq:main_bound_intro}\frac1\Gamma
\left\|
\omega(t)
-
\frac1{\nu t}\sum_{i=1}^3
\Gamma_i
G\left(\frac{\cdot-Z_i(t)}{\sqrt{\nu t}}\right)
\right\|_{L^1}
&\leq
C_0
(\eps^2+\delta\sqrt{\eps})(t),
\\[2mm]
\max_{1\leq i\leq3}
\frac{|Z_i(t)-z_i(t)|}{\dhk(t)}
&\leq
C_0\sqrt{\delta}\eps(t)^2,
\end{align}
for all times
\begin{equation}
0<t<
\begin{cases}
T_{\rm adv}\exp(\delta^{-p}),
&\text{ if } \varsigma=1 \quad (\text{\rm expanding}),\\[1mm]
(1-\delta^{p'})T_{\rm adv},
&\text{ if } \varsigma=-1 \quad (\text{\rm collapsing}).
\end{cases}
\end{equation}
\end{theorem}
This result actually follows by the stronger Theorem \ref{thm:NonLinControl} proved in Section \ref{sec:NonlinearCorrection}, but this simplified version already entails the main points we want to highlight.
Note that the constants in the estimate are uniform
with respect to the initial size $d$ of the triangle. This uniformity is crucial
for the vortex burst (collapse) limit described next, but let us stress some important facts before moving on.
\begin{remark}
In contrast to the results for translating dipoles \cite{dolce2024long},
rotating pairs \cite{zhang2025long}, and crystals
\cite{dolce2026viscousvortexcrystals}, this theorem justifies the
underlying inviscid configuration, which expands and rotates
self-similarly, up to a time scale that grows as a stretched exponential in the inviscid limit $\delta\to 0$. In the collapsing regime, we instead get
$\delta^{p'}$-close to the natural limiting time $t=T_{\rm adv}$, at which the
inviscid collapse occurs.
This result is closer in spirit to
\cite{dolce2024long,dolce2026viscousvortexcrystals,zhang2025long,GaSring},
where the expected limiting time scale is
$t/T_{\rm adv}=\mathcal{O}(\delta^{-1})$. In those settings, one can
reach time scales $t/T_{\rm adv}=\delta^{-\sigma}$ for small $\sigma>0$
in the case of a single axisymmetric vortex ring \cite{GaSring}, and
for any $\sigma\in(0,1)$ for the configurations studied in
\cite{dolce2024long,dolce2026viscousvortexcrystals,zhang2025long}.
In our case, we do not quantify the small parameters $p,p'$,
although it seems possible to obtain reasonable lower bounds for the
admissible exponents by tracking the constants in our proof more
carefully. For the expanding case, our limitations are mostly technical,
although some parts of our proof clearly break down as $p\to 1$.
We stress that, over the time scales considered here, the expanding
case is the most technically challenging.

On the other hand, the result stated here for the collapsing case
follows essentially by a direct adaptation of the arguments developed
for the expanding case. As discussed in Remark~\ref{rem:nguyen_collapse},
T.~T.~Nguyen \cite{nguyen2026inviscid} very recently obtained the result
for any $p'\in(0,1)$, with bounds slightly different from those stated
in Theorem~\ref{thm:main_intro}.
The $p'=1$ is in fact a physically relevant regime, which is when the distances between the vortices would be
of order $d\sqrt{\delta}$, a length scale at which diffusive effects are expected
to take over. Hence, $p'=1$ is a natural barrier in the collapsing regime for the perturbative methods around the Lamb--Oseen vortices, and it is extremely interesting to understand the dynamics beyond this limit.
\end{remark}
\begin{remark}
\label{rem:GWR}
In the expanding case,  vortices cannot keep traveling away for all
times. By the results \cite{GalWay2005,GaGa2005}, for every finite Radon measure initial datum one has
\begin{align}
\label{eq:convergenceGW}
    \lim_{t\to\infty}
    \left\|
    \omega(t)
    -
    \frac{\Gamma_{\rm tot}}{\nu t}
    G\left(\frac{\cdot}{\sqrt{\nu t}}\right)
    \right\|_{L^1(\mathbb{R}^2)}
    =
    0,
    \qquad
    \Gamma_{\rm tot}
    :=
    \int_{\mathbb{R}^2}d(\omega^{\rm in}).
\end{align}
Thus every viscous expanding Gr\"obli triangle eventually relaxes towards a
single Lamb--Oseen vortex. Theorem~\ref{thm:main_intro} identifies a long
metastable regime before that relaxation, whose duration is at least of
order $T_{\rm adv}\exp(\delta^{-p})$.
Interestingly enough, a quantitative version of the limit in \cref{eq:convergenceGW} done by Gallay and Rodrigues \cite{gallay2008temps} 
gives a double exponential upper bound in the Reynolds number of the form $\exp(\exp(c\delta^{-2}))$ on the time needed to approach the Lamb--Oseen profile. It remains an interesting open question whether the double exponential bound is optimal or can be improved.  
\end{remark}

\begin{remark}
\label{rem:expanding_relaxation}
A possible mechanism for the eventual relaxation of the expanding
Gr\"obli triangle is a slow change in its shape, caused by the cumulative
effect of viscous corrections to the Helmholtz--Kirchhoff positions.
This could bring the vortices close to a relative equilibrium, after
which vortex merging could drive relaxation towards a single
Lamb--Oseen vortex. Indeed, the first correction to the dynamics of
the rescaled positions arising from the approximate solution is of
order $\eps^4/t\approx\delta^2/t$ at long times. Such a correction could
accumulate into an $\mathcal{O}(1)$ change in the rescaled configuration
towards a relative equilibrium, suggesting a time scale of order
$T_{\rm adv}\exp(c\delta^{-2})$ to reach it.
Subsequent vortex merging could then yield a relaxation time of the
same exponential order, with one fewer exponential than the general
upper bound recalled in Remark~\ref{rem:GWR}.
Understanding the evolution of the expanding Gr\"obli triangle beyond
the time scale considered here, including how and when it eventually
stops expanding, remains an interesting open problem.
\end{remark}

\begin{remark}
\label{rem:GallaySverakrebound}
The inviscid dynamics near collapse is sensitive to perturbations.
In their recent work \cite{gallay2026threevortexsystemhopffibration},
Gallay and {\v S}ver{\'a}k show that, for suitable collapsing
configurations, small perturbations of the initial positions can replace
a self-similar collision by a near-collision followed by a \emph{rebound}
and expansion of the vortices. Capturing this dynamical behavior in the
viscous problem might require new ideas, such as quantitative control
of less well-prepared initial data, rather than initial data consisting
exactly of point vortices (see also \cite{donati2025fast}).
\end{remark}

\begin{remark}
\label{rem:nguyen_collapse}
The very recent result of T.~T.~Nguyen \cite{nguyen2026inviscid}
establishes the viscous realization of the collapsing Gr\"obli triangle
up to times $t=(1-\delta^{p'})T_{\rm adv}$ for every fixed
$p'\in(0,1)$. This result appeared after the completion of our manuscript,
and our collapsing result was obtained independently.
An important feature of the argument in \cite{nguyen2026inviscid} is the normalization of the
nonlinear remainder by $\eps(t)^2$, rather than by $\delta$ as in \cref{eq:NonlinearAnsatz}, combined with the use of logarithmic
time $s=-\log(1-t/T_{\rm adv})$ naturally adapted to the collapse. Our choice instead works well for the expanding problem, where $\eps(t)^2$ and $\delta$ are comparable after the initial
advection time, and allows us to treat some collapsing regime with only minor modifications of the arguments. Near collapse, however, $\eps(t)^2/\delta$ grows, so normalizing by $\eps^2$ permits weaker absolute bounds on the remainder. Together with the change of time, this compensates for the growth of the error of the approximate solution near collapse.
We retain our weaker result for the collapsing case for completeness,
since it follows directly from the arguments developed for the
expanding case. For the sharper result on collapsing Gr\"obli triangles, we refer to \cite{nguyen2026inviscid}.
\end{remark}

\subsection{The vortex burst limit}
The classical Helmholtz--Kirchhoff dynamics ceases to be defined when the three
vortices collide. In a generalized point-vortex formulation,
however, a single vortex of circulation
$\Gamma_{\rm tot}=\sum_i\Gamma_i$ can be continued after the singular time by
an admissible expanding Gr\"obli triangle, and  such continuations are inherently nonunique. Grotto and Pappalettera \cite{grotto2022burst} used this mechanism to
construct weak solutions of the Euler equations, in the sense of Schochet, that are nonunique. Following their terminology, we then
call the splitting of one Dirac mass into an expanding Gr\"obli triangle a
\emph{vortex burst}.

We then aim at making sense of such a branch as a vanishing viscosity limit. To do so, we can
start from three vortices whose initial separation $d_\nu$ tends to zero as $\nu \to0$ and
follow them until they reach a distance independent of $\nu$. The
logarithmic time scales available from  the general desingularization estimates in \cite{CeciSeis24}
correspond only to regimes such as
\begin{equation}
    d_\nu\to0 \quad \text{ and }\quad \frac{d_\nu^2}{\Gamma}\Big|\log\Big(\frac{\nu}{\Gamma}\Big)\Big|\to+\infty\qquad \text{ as }\nu \to 0.
\end{equation}
This allows logarithmically shrinking separations, meaning that one can capture 
only a  weak burst regime. On the other hand, as a direct application of our Theorem \ref{thm:main_intro}, we are able to handle the stronger regime 
\begin{equation}
\label{eq:condexp_intro}
    d_\nu\rightarrow 0,
    \qquad
    \frac{d_\nu^2}{\Gamma}
    \exp\Big(\Big(\frac{\Gamma}{\nu}\Big)^p\Big)
    \rightarrow+\infty, \qquad \text{ as }\nu \to 0,
\end{equation}
which includes distances that are degenerating at any polynomial power in the inverse circulation Reynolds number. Note that in this case the advection time becomes 
\begin{align}
\label{def:Tadvnu}
    T_{\rm adv,\nu}:=\frac{d_\nu^2}{\alpha \Gamma},
\end{align}
meaning that as $\nu \to 0$ there is a singularity in the phase $\theta(t)$ in \cref{def:dhk}. However, upon suitably correcting these singularities by choosing 
\begin{equation}
\label{def:betanuintro}
    \beta_\nu
    =
    \beta_*+\frac{\beta}{\alpha}\log \Big(\frac{d_\nu^2}{\alpha\Gamma t_\star}\Big)
    \qquad
    \mod 2\pi
\end{equation}
where $t_\star>0$ is a prescribed reference time. We can consider the initial data
\begin{equation}
\label{eq:omnuin_intro}
    \omega_\nu^{\rm in}
    =
    \sum_{i=1}^3
    \Gamma_i\delta_{d_\nu e^{\mathrm{i}\beta_\nu}{\sf n}_i}.
\end{equation}
Then
\begin{equation}
    \omega_\nu^{\rm in}
    \xrightharpoonup[\nu\to0]{*}
    \Gamma_{\rm tot}\delta_0,
    \qquad
    \Gamma_{\rm tot}:=\sum_{i=1}^3\Gamma_i,
\end{equation}
as finite Radon measures.
For $t>0$, set
\begin{equation}
\label{def:burstbranch_intro}
    z_{0,i}(t)
    :=
    \sqrt{\alpha\Gamma t}\,
    \exp\left(
        \mathrm{i}\beta_*
        +
        \mathrm{i}\frac{\beta}{\alpha}
        \log\left(\frac{t}{t_\star}\right)
    \right)
    {\sf n}_i,
    \qquad i=1,2,3,
\end{equation}
with the convention that $z_{0,i}(0):=0$. The burst branch emanating from \cref{eq:omnuin_intro} as $\nu\to0$ is expected to be
\[
    \omega_0^{\rm b}(t)
    :=
    \sum_{i=1}^3\Gamma_i\delta_{z_{0,i}(t)},
    \qquad t\geq0.
\]
We then have the following.
\begin{theorem}
\label{theo:burst_intro}
Let $\omega_\nu$ be the unique Navier--Stokes solution with initial datum
\cref{eq:omnuin_intro}, where $d_\nu$ satisfies
\cref{eq:condexp_intro} and the phases $\beta_\nu$ are chosen according to
\cref{def:betanuintro}. Then, for every fixed $T>0$,
\[
    \omega_\nu
    \xrightharpoonup[\nu\to0]{*}
    \omega_0^{\rm b}
    \qquad
    \text{in }
    C\bigl([0,T];\mathcal M(\mathbb R^2)\bigr).
\]
\end{theorem}
Thus, through different well-prepared families
converging to the same measure $\Gamma_{\rm tot}\delta_0$ in the limit $\nu\to 0$, one can select any
of the expanding branches in Aref's arcs in Fig. \ref{fig:aref_circle_comparison}, together with an arbitrary limiting
rotation. In this precise sense, Theorem \ref{theo:burst_intro} gives a viscous
realization of the vortex burst mechanism, and we refer to Fig. \ref{fig:burst} for a schematic representation of it.
\begin{remark}
Although Theorem \ref{theo:burst_intro} is formulated as a vanishing viscosity limit
on a fixed  time interval, it can also be interpreted as a consequence of the long-time dynamics of
the normalized configuration. Indeed, for every fixed $t>0$ we know that
\[
    \frac{t}{T_{\rm adv,\nu}}
    =
    \frac{\alpha\Gamma t}{d_\nu^2}
    \longrightarrow+\infty
    \qquad\text{as }\nu\to0.
\]
This connects the result with the long-time inviscid analysis in \cite{iftimie2003large}. They proved,
for compactly supported $L^p$ Euler vorticity with $p>2$, that every
self-similar rescaling on a spatial scale $t^\kappa$ with $\kappa>1/2$ sees only the total
circulation concentrated at the origin. The square root scale $\kappa=1/2$ is therefore
the borderline case not resolved by their result. The limiting branch in
Theorem \ref{theo:burst_intro} separates precisely on this scale and produces a
nontrivial atomic configuration.
\end{remark}
\begin{remark}
In \cite{albritton2025vanishingviscositynonuniquesolutions},
the authors study vanishing viscosity selection
for the forced two-dimensional Euler equations from Navier--Stokes with the same force, with vorticity in
$L^p$, $2<p<\infty$.
Their work builds on Vishik's nonuniqueness scenario
\cite{vishik2018nonuniquenessI,vishik2018nonuniquenessII}, see also \cite{albritton2024vishik,castro2025vishik,dolce2025selfsimilar}, where an unstable radial self-similar vortex is the building block for the nonuniqueness mechanism. In 
\cite{albritton2025vanishingviscositynonuniquesolutions}, they   identify a critical size $\nu^{\kappa}$
for perturbations of the initial data, measured in a suitably rescaled
Sobolev norm, for which  $o(\nu^{\kappa_{\rm c}})$ perturbations select
the radial solution, whereas along a sequence at the critical size the obtain the nonradial solutions with the same limiting initial data and force. 
In our setting, a single Dirac mass is compatible with both
a stationary point vortex and a self-similar expanding Gr\"obli triangle in the generalized point-vortex formulation of \cite{grotto2022burst}. Thus, within the setting described above, our long-time estimates allow us to follow the expanding configuration until its separation is independent of the viscosity, thereby realizing an expanding Gr\"obli triangle as a vanishing viscosity limit.
\end{remark}
\subsection{Main ideas}
To describe the main points of the proof and the mechanisms behind the two
results above, we first focus on the long-time approximation of the
expanding Gr\"obli triangle (with a brief explanation of why the collapsing case is very similar within our time-scales). Then we explain how this estimate gives
a viscous realization of vortex bursts.

\subsubsection{Long-time approximation of the expanding Gr\"obli triangle}

The proof of Theorem \ref{thm:main_intro} follows the general strategy developed in
\cite{GaSring,GaSarnold} and refined and used in \cite{dolce2024long,dolce2026viscousvortexcrystals,zhang2025long,donati2025fast}. The main difference is that the point vortex motion considered here
is dynamically richer. The Gr\"obli triangle expands and rotates, and this creates new
difficulties that are absent in all the relative equilibria configurations studied previously \cite{dolce2024long,dolce2026viscousvortexcrystals,zhang2025long}.
One crucial point is on how we propagate
rather precise information on the correction to the Helmholtz--Kirchhoff
trajectories, where we reveal  how the stability properties of the  linearized Helmholtz-Kirchhoff system around a given solution enter in the analysis of the viscous problem.

As usual, we start by working in self-similar variables around the vortex centers
$Z_i(t)$ and write
\begin{align}
    \omega(x,t)
    =
    \sum_{i=1}^3
    \frac{\Gamma_i}{\nu t}
    \left(
        \Omega_{{\rm app},i}+\delta w_i
    \right)
    \left(
        \frac{x-Z_i(t)}{\sqrt{\nu t}},t
    \right),
    \qquad
    Z_i(t)=z_i(t)+\rho_i(t).
\end{align}
Here $z_i(t)$ denotes the Helmholtz--Kirchhoff trajectory, while
$\rho_i(t)$ are modulation parameters chosen so that the perturbations $w_i$
have vanishing first order moments. The approximate profile, constructed in Section \ref{sec:ApproxSol}, is expanded as
\begin{align}
    \Omega_{{\rm app},i}(\xi,t)
    =
    G(\xi)+\sum_{n=2}^4\epsilon(t)^n\Omega_{n,i}(\xi,t),
\end{align}
where the profiles $\Omega_{n,i}$ are constructed as in the general
expansion in \cite{Gallay2011}, but here we need a more precise identification of the dependence on the effective aspect ratio
$\epsilon$ and $\delta$. The key point is to exploit that the effective
aspect ratio satisfies \cref{def:keytaueps}.
As a consequence, the residual produced by the fourth-order approximate
solution is uniformly small in time in the expanding case in the time-scale under consideration. Indeed, denoting with $ \mathsf{NS}$ the full Navier--Stokes equations in self-similar
variables, the approximate solution is constructed to satisfy
\begin{align}
    \delta^{-1}\mathsf{NS}(\Omega_{\rm app})
    =
    \mathcal{O}\left(\delta^{-2}\epsilon(t)^5+\eps(t)\right)\lesssim  \begin{cases}
       \sqrt{\delta} &\quad \text{ in the expanding case},\\
        \delta^{\frac{1-5p'}{2}}&\quad \text{ in the collapsing case}.
    \end{cases}
\end{align}
Thus, this error remains small \emph{for all $t\geq0$} in the expanding case, whereas in the collapsing case it is enough to assume that $p'$ is sufficiently small\footnote{The gain in the approach in \cite{nguyen2026inviscid} can be seen directly for the size of $\delta^{-1}{\sf NS}(\Omega_{\rm app})$. In the collapsing case, setting
 $s=-\log(1-t/T_{\rm adv})$ gives
$\delta t\partial_t=|\alpha|\eps^2\partial_s$. Combining this change of time with the normalization
$\Omega_i-\Omega_{\rm app,i}=\eps^2 w_i$ replaces the estimate
$\delta^{-2}\eps^5$ by
$\eps$. Thus, an approximate solution up to order $\eps^4$
contributes a forcing of order $\eps$, which remains small for every
fixed $p'<1$.}. This is a first structural advantage of the expanding Gr\"obli
geometry. In fact, in the previous long-time constructions for relative equilibria \cite{dolce2024long,dolce2026viscousvortexcrystals,zhang2025long}, the relevant aspect ratio grows like
$\sqrt{\nu t}/d$, and approaching the diffusive time scale requires an
approximate solution of arbitrarily high order instead of the fixed order $4$. Here, instead, the self-similar
expansion keeps the aspect ratio uniformly small, so that the approximate
solution does not deteriorate at large times.

But these good properties of the approximate solution are not enough to overcome  more serious nonlinear difficulties. The perturbations $w_j$ solve an equation
of the schematic form
\begin{align}
    t\partial_t w_i
    +
    \delta^{-1}\mathsf{L}_{\Omega_{\rm app}}w_i
    =
    \delta^{-1}
    \left(
        \mathsf{NS}(\Omega_{\rm app})
        +
        \mathtt{T}_{\rho_i}(\Omega_{\rm app})
        +
        \mathtt{L}_{\rho_i}(w_i)
    \right)
    +
    \mathcal{N}(w_i).
\end{align}
The operator on the left-hand side is the linearization of the Navier--Stokes
dynamics around the approximate solution while $\mathcal{N}(w_i)$ denotes the
nonlinear terms. The new terms $\mathtt{T}_{\rho_j}$ and
$\mathtt{L}_{\rho_j}$ arise from the fact that the true centers of vorticity are the
modulated positions $Z_i=z_i+\rho_i$, rather than the pure
Helmholtz--Kirchhoff positions $z_i$. To control $w_i$, we have to face two main difficulties.

\smallskip
$\blackinwhitediamond$ \emph{First difficulty.} Linearizing only around the Lamb--Oseen vortex, one obtains
\begin{align}
    \delta^{-1}\mathsf{L}_{\Omega_{\rm app}}
    =
    \delta^{-1}\mathsf{L}_{G}
    +
    \delta^{-1}\epsilon(t)^2\mathsf{L}_2
    +
    \cdots .
\end{align}
The stability properties of the Lamb--Oseen vortex are by now well understood
\cite{GalWay2002,GaSarnold,Weipseudo}, and they yield coercive estimates in suitable weighted spaces after
projecting out the translational modes, corresponding to imposing zero first order moments. However, the term
\begin{equation}
    \delta^{-1}\epsilon(t)^2\mathsf{L}_2
\end{equation}
cannot be treated perturbatively, since $\delta^{-1}\epsilon^2$ is not small on the time scales considered here. Thus the stability of the
Lamb--Oseen vortex alone is not sufficient. The way around this difficulty is to include the leading elliptic deformation
of the core, encoded in $\Omega_2$, as part of the reference profile. More
precisely, we write informally
\begin{align}
    \mathsf{L}_{\Omega_{\rm app}}
    =
    \mathsf{L}_{G+\epsilon^2\Omega_2}
    +
    \epsilon(t)^3\mathsf{L}_3
    +
    \cdots .
\end{align}
We then build the weighted energy estimates around
$G+\epsilon(t)^2\Omega_2$, rather than around $G$ alone. This is done using
an Arnold type functional framework inspired by \cite{GaSarnold}. The basic idea is to construct a
functional space in which the inviscid part of the linearized operator is almost
skew-adjoint, while the viscous part yields the required coercivity. In this
way the changes caused by $\Omega_2$ are absorbed into the main
operator, and the remaining linear errors are perturbative, since for instance in the expanding case
    $\delta^{-1}\epsilon(t)^3
    \lesssim
    \sqrt{\delta}$.
However, the construction of the desired  functional framework requires, at least
approximately, a functional relationship between the corrected leading order
vorticity $G+\epsilon(t)^2\Omega_2$ and the total streamfunction. In the relative equilibria
setting treated in \cite{dolce2024long,dolce2026viscousvortexcrystals,zhang2025long}, this relationship can be constructed by exploiting the strong symmetries inherent to the inviscid approximation. In the present setting though, those symmetries
are absent. Nevertheless, we show that the functional relationship can be determined explicitly up to
order two and we believe that this is a rather general feature. Indeed, this builds on some nice cancellations that only rely on the structure of the operators and the equations rather than the particular configuration at hand. After these structural properties have been established, the leading order model is
\begin{equation}
\label{eq:wintro2}
    t\partial_t w_i
    +
    \delta^{-1}\mathsf{L}_{G+\epsilon^2\Omega_2}w_i
    =
    \delta^{-1}
    \left(
        \mathtt{T}_{\rho_i}(\Omega_{\rm app})
        +
        \mathtt{L}_{\rho_i}(w_i)
    \right)
    +
    \mathcal{O}\left(\delta^{-1}\epsilon(t)^3+\eps(t)\right),
\end{equation}
up to nonlinear terms which are treated by a bootstrap argument. 
\begin{remark}
    This informal reduction is also telling us that the correct approximation of the viscous blobs has an elliptical shape due to the deformation encoded in $\Omega_2$, which is a general fact \cite{Gallay2011}. However, in our case said deformation is of order $\mathcal{O}(\delta)$  for \emph{all}  sufficiently large times in the expanding case, whereas it grows as $t\to T_{\rm adv}$ in the collapsing one. Interestingly enough, this tells us that  in the viscous setting the three vortices must always remember their interactions even when they are very far apart. This is also reminiscent of what happens in the infinite time inviscid desingularization of certain  Gr\"obli triangles done in \cite{davila2024expanding}. 
\end{remark}

$\blackinwhitediamond$ \emph{Second difficulty.}
The remaining problem is to control the modulation parameters $\rho_i(t)$.
The forcing terms $\mathtt{T}_{\rho_i}(\Omega_{\rm app})$ generated by the
moving frame are not automatically small, because the leading profile is still
the Lamb--Oseen vortex. Moreover, it is not enough to know that the corrected
triangle remains ``geometrically" close to the Gr\"obli triangle. More precisely, we need to
propagate quite precise bounds on
\begin{align}
    \frac{|\rho_i(t)|}{\dhk(t)}
    \qquad
    \text{and}
    \qquad
    \frac{|t\partial_t\rho_i(t)|}{\dhk(t)}.
\end{align}
The first quantity controls the deformation of the triangle, while the second
is needed to rule out fast oscillations of the modulated centers. Such oscillations need
not be visible from the shape of the triangle alone, but they enter directly
in the equations because we work in a moving frame. This is where the precise linearized dynamics of the expanding
Helmholtz--Kirchhoff system enters\footnote{Note that in \cite{dolce2026viscousvortexcrystals,zhang2025long}, a full nonlinear modulation of the vortex center presents analogous difficulties. However, in \cite{zhang2025long} those are bypassed by using \emph{pseudo-momenta} whereas in the more general setting in \cite{dolce2026viscousvortexcrystals} symmetries are carefully exploited to obtain quantitative smallness on the moments and translation errors.}. A purely geometric control of the side
lengths, as in the classical analysis of three point vortices \cite{MarPul94}, is
too crude for controlling $t\partial_t\rho_j$. Instead, we complement the study of the linearization around   Gr\"obli triangles
obtained by Aref \cite{ArefSelfSimilar} as follows. Schematically we have that $\boldsymbol{\rho} = (\rho_1, \rho_2, \rho_3)$ satisfies a system of the form 
\begin{align}
    \de_t \boldsymbol{\rho}=\mathfrak{L}(t,\boldsymbol{\rho})+\cO(t^{-1}\dhk \eps^4)+\text{nonlinear errors}
\end{align}
where $\mathfrak{L}$ is the Helmholtz-Kirchhoff system linearized at the Gr\"obli triangle. Following \cite{ArefSelfSimilar}, it is convenient to reformulate the problem in self-similar variables adapted to the motion that makes $\mathfrak{L}$ time-independent. As we show in Section \ref{sec:selfsimilar_HK}, the problem is reduced in characterizing the spectral picture of a $3\times 3$ matrix. In particular, denoting with $\mathcal U_{\mathfrak L}(t,s)$ the propagator, we are able to deduce that
\begin{equation}
    \|\mathcal U_{\mathfrak L}(t,s)\|
    \lesssim \left(1+\Big|\log\frac{(1+\varsigma\frac{t}{T_{\rm adv}})}{(1+\varsigma\frac{s}{T_{\rm adv}})}\Big|\right)
    \left(
    \frac{\dhk(t)}{\dhk(s)}
    \right)^{\varsigma}.
\end{equation}
Up to logarithmic factors, this bound is compatible with the growth/collapse of the underlying configuration. Combining this information with the bootstrap bounds on $w_i$, we can then prove 
\begin{align}
    \frac{|\rho_i(t)|}{d_{\rm HK}(t)}
    +
    \frac{|t\partial_t\rho_i(t)|}{d_{\rm HK}(t)}
    \lesssim
    \delta^{q_\varsigma}\epsilon(t)^2
    \begin{cases}
        \left(\log\left( e + \frac{t}{T_{\rm adv}}\right)\right)^2&\quad \text{ in the expanding case},\\
        (1-\frac{t}{T_{\rm adv}})^{-1}&\quad \text{ in the collapsing case},
    \end{cases}
\end{align}
where $q_{\varsigma}>0$ can be made close to $1$ when $p,p'$ are small.
While the algebraic loss is naturally encoded in the collapsing case, the logarithmic loss in the expanding case is related to the integration in time of the errors together with the bound on $\mathcal{U}_{\mathfrak{L}}$. This is the main obstruction to pushing the approximation
beyond the stretched--exponential time scale in Theorem \ref{thm:main_intro}. In any
case, one cannot expect the expanding approximation to hold for all times despite the good properties of the approximation, because the relaxation described in Remark \ref{rem:GWR} will eventually start.
\subsubsection{The vanishing viscosity vortex burst} Let us briefly sketch here the main points behind Theorem \ref{theo:burst_intro}. First of all, we observe that with the initial data \cref{eq:omnuin_intro}
then the corresponding point vortex trajectories are
\begin{equation}
    z_{i,\nu}(t)
    =
    \sqrt{d_\nu^2+\alpha\Gamma t}\,
    \exp\left(
        \mathrm{i}\beta_*
        +
        \mathrm{i}\frac{\beta}{\alpha}
        \log\left(
            \frac{t+T_{\rm adv,\nu}}{t_\star}
        \right)
    \right)
    {\sf n}_i
\end{equation}
where $T_{\rm adv,\nu}$ is given in \cref{def:Tadvnu} and we exploited the definition of $\beta_\nu$.  For every fixed $t>0$, this trajectory converges to
\cref{def:burstbranch_intro}. 

Once the target is understood, the quantitative burst convergence follows by
combining this explicit limit with the $L^1$ profile estimate, the modulation
bound, and the corresponding Kantorovich--Rubinstein estimates, following the
concentration framework of \cite{CeciSeis24}.

\subsection{Notations and conventions}
\label{sec:notation}
We will perform operations on vectors in $\RR^2$ as if they were complex numbers. Then, we define the mean and $k$-th moment vectors as 
\begin{equation}
    \label{def:mk}
    {\rm M}(f)=\int_{\mathbb{R}^2}f(\xi)\dd \xi, \qquad {\rm m}^k(f)\coloneqq \int_{\RR^2}\xi^k f(\xi)\dd \xi, \qquad \text{for } k\in \NN.
\end{equation}
We will use bold symbols to refer to vectors encoding quantities related to the three vortices. For example, we will use $\Z$ \ to collect all positions. Thus, 
\begin{equation}
    \Z := \Z(t) = \begin{pmatrix}
        Z_1(t) \\ Z_2(t) \\ Z_3(t)
    \end{pmatrix}.
\end{equation}
Then, given $X_i$, $i=1,2,3$  a n$L^2(\R^2)$-based Hilbert spaces, we set 
\begin{equation}
    \boldsymbol X := X_1 \times X_2 \times X_3 .
\end{equation}
Let $f_i\in X_i$ and $L^i:\boldsymbol{X}\to X_i$ be a given linear operator. We write
\begin{equation}
    \boldsymbol f := (f_1,f_2,f_3) \in \boldsymbol X, 
    \qquad
    \boldsymbol{L}(\boldsymbol{f}) = (L^1(\boldsymbol{f}),L^2(\boldsymbol{f}),L^3(\boldsymbol{f})) \in \boldsymbol X,
\end{equation}
where clearly $\boldsymbol{L}:\boldsymbol{X}\to \boldsymbol{X}$ is a linear operator as well. Then, we denote
\begin{equation}
    \left\langle \boldsymbol f,\boldsymbol g \right\rangle_{\boldsymbol X}
    :=
    \left\langle f_1,g_1 \right\rangle_{X_1}
    +
    \left\langle f_2,g_2 \right\rangle_{X_2}
    +
    \left\langle f_3,g_3 \right\rangle_{X_3}.
\end{equation}
When $X_1=X_2=X_3=L^2(\R^2)$, we abuse in notation and omit the subscript above.
We avoid writing in bold standard differential operators and moments that are applied componentwise, and whose vector notation is clear from context. For instance, we write
\begin{align}
    \nabla \boldsymbol f
    :=
    \left(
        \nabla f_1,
        \nabla f_2,
        \nabla f_3
    \right), \quad  \mathcal L \boldsymbol f
    :=
    \left(
        \mathcal L f_1,
        \mathcal L f_2,
        \mathcal L f_3
    \right),\\
\mrm^1:\boldsymbol{X}\to (\mathbb{R}^2)^3, \quad \mrm^1(\boldsymbol{f})=\big(\mrm^1(f_1),\mrm^1(f_2),\mrm^1(f_3)\big).
\end{align}

\section{Self-similar Helmholtz--Kirchhoff dynamics}
\label{sec:selfsimilar_HK}

In this section we recall the properties of self-similar three-vortex
motions that are used in the paper.  We first recall the derivation of the trajectories and
the geometric characterization of the admissible configurations.  We then
reduce the linearized dynamics to an autonomous system and record the
corresponding bounds.  We refer to Aref's paper
\cite{ArefSelfSimilar} for some algebraically involved computations fully detailed there.

\subsection{Self-similar trajectories}
\label{app:point_vortices}

Recall the self-similar ansatz
\begin{equation}
\label{eq:self_similar_ansatz_app}
    z_j(t)=\dhk(t)e^{\mathrm{i}\theta(t)}{\sf n}_j,
    \qquad j=1,2,3,
\end{equation}
where the normalized configuration $({\sf n}_1,{\sf n}_2,{\sf n}_3)$
satisfies \cref{def:normalized} and has center of vorticity at the origin.
By the conservation of the Hamiltonian associated to the Helmholtz--Kirchhoff system and the quadratic impulse, one has that a  self-similar motion satisfies
\begin{equation}
\label{eq:groebli_circulation_condition_app}
    \sum_{1\leq i<j\leq3}\Gamma_i\Gamma_j=0,
    \qquad
    \sum_{1\leq i<j\leq3}
    \Gamma_i\Gamma_j|z_i^{\rm in}-z_j^{\rm in}|^2=0.
\end{equation}
For three vortices these conditions are also sufficient.  Indeed, they imply
\begin{align}
    \Gamma_1\bigl(|z_1-z_2|^2-|z_1-z_3|^2\bigr)
    &={}
    \Gamma_2\bigl(|z_2-z_3|^2-|z_1-z_2|^2\bigr) \notag\\
    &={}
    \Gamma_3\bigl(|z_1-z_3|^2-|z_2-z_3|^2\bigr).
\end{align}
We can then derive the equations for the squared mutual distances $|z_i-z_j|^2-|z_i-z_k|^2$ from the Helmholtz-Kirchhoff system and, using the relations above one gets 
\begin{equation}
    \frac{\dd}{\dd t}\log|z_1-z_2|^2
    =
    \frac{\dd}{\dd t}\log|z_2-z_3|^2
    =
    \frac{\dd}{\dd t}\log|z_3-z_1|^2.
\end{equation}
Hence the initial triangle has sides changing at the exact same logarithmic rate, meaning that the triangle preserves its shape. To determine the dilation and rotation rates one can directly insert \cref{eq:self_similar_ansatz_app} in the Helmholtz--Kirchhoff system. In particular, define the dimensionless (and constant in time)
 velocity 
\begin{equation}
\label{eq:dimensionless_HK_velocity_app}
    \mathcal V_j(\boldsymbol{{\sf n}})
    :=
    \sum_{m\neq j}\frac{\gamma_m}{2\pi}
    \frac{({\sf n}_j-{\sf n}_m)^\perp}
    {|{\sf n}_j-{\sf n}_m|^2},
    \qquad j=1,2,3.
\end{equation}
Recalling that $v^\perp= {\rm i}v$ with the complex number notations, one finds that $\mathcal{V}_j(\boldsymbol{\mathsf{n}})$ must be a real linear combination of ${\sf n}_j$ and ${\sf n}_j^\perp$. Namely, there exists
 $\alpha,\beta\in\mathbb R$ such that
\begin{equation}
\label{eq:self_similar_condition_alpha_beta_app}
    \mathcal V_j({\boldsymbol{\sf n}})
    =
    \frac{\alpha}{2}{\sf n}_j+\beta{\sf n}_j^\perp,
    \qquad j=1,2,3.
\end{equation}

This identity can also be understood geometrically.
For a noncollinear triangle, preservation of shape implies
that the instantaneous motion consists of a common dilation,
a common rotation, and possibly a translation.
Conservation of the center of vorticity, together with
$\sum_j\gamma_j\neq0$, eliminates the translation.
Thus the dilation and rotation coefficients are the same
for all three vortices.
Therefore, a triangle preserves its shape if and only if the same infinitesimal movement of each $\mathsf{n}_j$ is either in the direction of ${\sf n}_j$ (stretching) or ${\sf n}_j^\perp$ (rotation). Then, for any $i\neq j$, the constants $\alpha,\beta$  are given by
\begin{equation}
\label{eq:alpha_pair_formula_app}
    \alpha
    =2\frac{
    \bigl(\mathcal V_i({\boldsymbol{\sf n}})-\mathcal V_j({\boldsymbol{\sf n}})\bigr)
    \cdot({\sf n}_i-{\sf n}_j)}
    {|{\sf n}_i-{\sf n}_j|^2},
    \qquad
    \beta
    =\frac{
    \bigl(\mathcal V_i({\boldsymbol{\sf n}})-\mathcal V_j({\boldsymbol{\sf n}})\bigr)
    \cdot({\sf n}_i-{\sf n}_j)^\perp}
    {|{\sf n}_i-{\sf n}_j|^2}.
\end{equation}
Comparing \cref{eq:self_similar_ansatz_app,eq:self_similar_condition_alpha_beta_app}
in the Helmholtz--Kirchhoff system yields
\begin{equation}
    \de_t\dhk=\frac{\alpha\Gamma}{2\dhk},
    \qquad
    \de_t\theta=\frac{\beta\Gamma}{\dhk^2}.
\end{equation}
Thus, setting $\varsigma=\operatorname{sign}(\alpha)$ and recalling that
$T_{\rm adv}=d^2/(|\alpha|\Gamma)$, we recover
\begin{equation}
\label{eq:self_similar_trajectory_app}
    \dhk(t)
    =d\sqrt{1+\varsigma\frac{t}{T_{\rm adv}}},
    \qquad
    \theta(t)
    =\frac{\beta}{\alpha}
    \log\left(1+\varsigma\frac{t}{T_{\rm adv}}\right).
\end{equation}

We finally recall the geometric description of the configurations satisfying
\cref{eq:groebli_circulation_condition_app} given by Aref in \cite{ArefSelfSimilar}.  After relabeling, choose the pair $1,2$ of vortices of the same sign, place
their center of vorticity at the origin,
rotate the picture and divide all lengths by
$\ell:=|z_1^{\rm in}-z_2^{\rm in}|$ so that their positions are
\begin{equation}
    \Big(\frac{\gamma_2}{\gamma_1+\gamma_2},0\Big),
    \qquad
    \Big(-\frac{\gamma_1}{\gamma_1+\gamma_2},0\Big),
\end{equation}
and we refer to Fig. \ref{fig:aref_circle_comparison} for a visualization.
Writing the position of the third vortex as $x_1+\mathrm{i}x_2$, the first
condition in \cref{eq:groebli_circulation_condition_app} gives
\begin{equation}
    \gamma_3=-\frac{\gamma_1\gamma_2}{\gamma_1+\gamma_2}.
\end{equation}
The second condition then reduces to the circle
\begin{equation}
\label{eq:aref_circle_app}
    x_1^2+x_2^2
    =
    1-\frac{\gamma_1\gamma_2}{(\gamma_1+\gamma_2)^2}
    =
    \frac{\gamma_1+\gamma_2+\gamma_3}{\gamma_1+\gamma_2}.
\end{equation}
Let $L$ denote the maximal separation in these coordinates:
\[
L:=
\max\left\{
1,\,
\sqrt{\left(x_1-\frac{\gamma_2}{\gamma_1+\gamma_2}\right)^2+x_2^2},
\,
\sqrt{\left(x_1+\frac{\gamma_1}{\gamma_1+\gamma_2}\right)^2+x_2^2}
\right\}
=\frac{d}{\ell}.
\]
In these coordinates, accounting for this change of normalization, we note that \cref{eq:alpha_pair_formula_app} becomes
\begin{equation}
\label{eq:alpha_aref_app}
    \alpha
    =
    L^2\frac{\gamma_3x_2}{\pi}
    \left[
        \frac{1}{\left(x_1-\frac{\gamma_2}{\gamma_1+\gamma_2}\right)^2+x_2^2}
        -
        \frac{1}{\left(x_1+\frac{\gamma_1}{\gamma_1+\gamma_2}\right)^2+x_2^2}
    \right].
\end{equation}
Consequently, $\alpha=0$ at the two intersections of the circle with $x_2=0$
and at the two intersections with
$x_1=(\gamma_2-\gamma_1)/(2(\gamma_1+\gamma_2))$.  These are, respectively,
the two collinear and the two equilateral relative equilibria.  The sign of
$\alpha$ is constant on each of the four remaining arcs and alternates from
one arc to the next, giving the expanding and collapsing configurations.

\subsection{Linearization in self-similar variables}Rewriting the Helmholtz--Kirchhoff system in complex notation, one has
\begin{equation}
\label{eq:HKcomplex_app}
    \frac{\dd\overline{z_j}}{\dd t}
    =
    \frac{1}{2\pi\mathrm{i}}
    \sum_{m\neq j}\frac{\Gamma_m}{z_j-z_m},
    \qquad j=1,2,3.
\end{equation}
Let
\begin{equation}
\label{eq:self_similar_time_app}
    \tau
    :=
    \varsigma T_{\rm adv}
    \log\left(1+\varsigma\frac{t}{T_{\rm adv}}\right),
    \qquad
    \upomega
    :=
    \frac{\beta\Gamma}{d^2}
    =
    \frac{\beta}{|\alpha|T_{\rm adv}}.
\end{equation}
Then $\dd\tau/\dd t=(1+\varsigma t/T_{\rm adv})^{-1}$,
$\theta(t)=\upomega\tau$, and
\begin{equation}
    z_j(t)-z_m(t)
    =
    \sqrt{1+\varsigma\frac{t}{T_{\rm adv}}}
    e^{\mathrm{i}\upomega\tau}
    (z_j^{\rm in}-z_m^{\rm in}).
\end{equation}
From now on, we use the notation
\begin{equation}
    Z_{ij}(t) := Z_i(t)-Z_j(t),
\end{equation}
and define $z_{ij}(t)$ and $\rho_{ij}(t)$ analogously.
Writing $Z_j=z_j+\rho_j$ and using
\begin{equation}
    \frac{1}{\overline Z_{ij}}
    -
    \frac{1}{\overline z_{ij}}
    =
    -\frac{\overline\rho_{ij}}{\overline z_{ij}^{2}}
    +
    \frac{\overline\rho_{ij}^{2}}
    {
    \overline z_{ij}^{2}
    \overline Z_{ij}
    } ,
    \label{eq:ExactHKExpansionSec2}
\end{equation}
the linearized problem of
\cref{eq:HKcomplex_app} around the self-similar Gr\"obli triangle is
\begin{equation}
\label{eq:linearized_HK_app}
    \de_t\rho_j
    =
    -\frac{\mathrm{i}}{2\pi}
    \sum_{m\neq j}\Gamma_m
    \frac{\overline{\rho_j}-\overline{\rho_m}}
    {(\overline{z_j}-\overline{z_m})^2}.
\end{equation}
Define the matrix $M$ by
\begin{equation}
\label{eq:def_M_app}
    (Mh)_j
    :=
    \frac{1}{2\pi}
    \sum_{m\neq j}\Gamma_m
    \frac{h_j-h_m}{(z_j^{\rm in}-z_m^{\rm in})^2},
    \qquad h\in\mathbb C^3.
\end{equation}
Upon setting
\begin{equation}
\label{eq:rotating_variable_app}
    g(\tau)
    :=
    \begin{pmatrix}
        e^{-\mathrm{i}\upomega\tau}\boldsymbol\rho(t(\tau))\\[2mm]
        e^{\mathrm{i}\upomega\tau}\overline{\boldsymbol\rho(t(\tau))}
    \end{pmatrix},
\end{equation}
equation \cref{eq:linearized_HK_app} becomes
\begin{equation}
\label{eq:autonomous_HK_app}
    \de_\tau g=\mathfrak B g,
    \qquad
    \mathfrak B
    :=
    \begin{pmatrix}
        -\mathrm{i}\upomega I_3&-\mathrm{i}\overline M\\
        \mathrm{i}M&\mathrm{i}\upomega I_3
    \end{pmatrix},
\end{equation}
where the matrix above is constant in time. Notice that the same autonomous reduction holds for both values of
$\varsigma$ since the sign only enters the relation between $t$ and $\tau$. Then, we have the following properties for the matrix $\mathfrak{B}$.
\begin{lemma}
\label{lem:SpectralPropertiesB}
    The characteristic polynomial of $\mathfrak B$ is
    \begin{equation}
    \label{eq:characteristic_B_app}
        \det(\lambda I_6-\mathfrak B)= (\lambda^2+\upomega^2)\left(\lambda^2-\left( \frac{1}{2 T_{\rm adv}} \right)^2\right)^2
    \end{equation}
    and the eigenvalues are
    \begin{equation}
    \label{eq:spectrum_B_app}
        \sigma(\mathfrak B)=
        \left \{\pm\mathrm{i}\upomega,\,\pm \frac{1}{2 T_{\rm adv}}\right \}.
    \end{equation}
    Consequently, $\pm 1/(2T_{\rm adv})$ have algebraic multiplicity $2$ and therefore for every $\tau\geq0$,
    \begin{equation}
    \label{eq:semigroup_B_app}
        \left\|e^{\tau\mathfrak B}\right\|
        \lesssim
        \left(1+\frac{\tau}{ T_{\rm adv}}\right)e^{\frac{\tau}{2 T_{\rm adv}}}.
    \end{equation}
\end{lemma}
\begin{remark}
The linearly growing term in \cref{eq:semigroup_B_app} is related to the possibility of Jordan blocks of size two. In fact, the matrix is diagonalizable only for four specific choices of the initial triangle at fixed circulations, which are characterized in Appendix \ref{app:explicit}. The bound \cref{eq:semigroup_B_app} is the only consequence of the Jordan structure, while the exact scaling of the unstable eigenvalues $1/(2T_{\rm adv})$ is crucial in the subsequent analysis. Indeed, since we are working in the logarithmic time-scale $\tau$, polynomial losses cause log-type losses whereas the scaling of the unstable eigenvalues dictate the polynomial powers of time that one loses in the estimates.
\end{remark}
\begin{proof}
    Using the Schur complement for the block matrix $\mathfrak{B}$ gives that
    \begin{equation}
        \det(\lambda I_6-\mathfrak B)= \det((\lambda^2 + \upomega^2)I_3 - M\overline{M}). 
        \label{eq:PolyB}
    \end{equation}    
    Thanks to \cite[Sec.~III]{ArefSelfSimilar}, we have that the characteristic polynomial of the matrix $\overline{M}M$ is 
    \begin{equation}
        \det(\lambda I_3 - \overline{M}M) = \lambda \left( \lambda - \left( \upomega^2 + \frac{1}{4 T_{\rm adv}^2} \right )\right)^2.
    \end{equation}
    Here the eigenvalue $\lambda_1 = 0$ is simple with eigenvector $v_1 = (1,1,1)^T$ and $\lambda_2 = \upomega^2 + \frac{1}{4 T_{\rm adv}^2}$ has  algebraic multiplicity two with eigenvector $v_2 = (z_{1}^{\rm in}, z_{2}^{\rm in}, z_{3}^{\rm in})^T$. Indeed, $v_1 $ is related to the translation invariance in $\RR^2$, which is reflected in the fact that all rows of $M$ sum to zero. Therefore 
    $Mv_1=0$. On the other hand, the self-similar equation at $t=0$ gives
    \begin{equation}
    \label{eq:central_configuration_identity_app}
        Mv_2
        =
        \left(\upomega+\frac{\mathrm{i}\varsigma}{2T_{\rm adv}}\right)
        \overline{v_2}, \quad \Longrightarrow \quad
        \overline M Mv_2
        =
        \left(\upomega^2+\frac{1}{4T_{\rm adv}^2}\right)v_2.
    \end{equation}
    Then, following the computations in  \cite[Sec.~III.A]{ArefSelfSimilar}, one has that    \begin{equation}
    \label{traceMM}
        \operatorname{tr}(\overline M M)
        =
        2\left(\upomega^2+\frac{1}{4T_{\rm adv}^2}\right).
    \end{equation}
    To get this identity, one has to crucially exploit  the cancellations given by the constraints in \cref{eq:groebli_circulation_condition_app}.
    Knowing the other two eigenvalues and the trace of the matrix $\overline{M}M$ fixes the third eigenvalue. As a consequence, the matrix $\overline{M}M$ admits a Jordan decomposition of the form
    \begin{equation}
        \overline{M}M \sim \begin{pmatrix}
            0 & 0 & 0 \\ 0 & \lambda_2  & \kappa \\ 0 & 0 & \lambda_2
        \end{pmatrix}
    \end{equation}
    for some $\kappa\in \mathbb{C}$ depending only on the initial Gr\"obli triangle. We give the explicit expression in Appendix \ref{app:explicit}. Going back to \cref{eq:PolyB}, we have
    \begin{equation}
        \det(\lambda I_6-\mathfrak B) = (\lambda^2 + \upomega^2) \left(\lambda^2 + \upomega^2  - \upomega^2 - \frac{1}{4T_{\rm adv}^2}\right)^2 = (\lambda^2 + \upomega^2) \left(\lambda^2 - \left(\frac{1}{2T_{\rm adv}}\right)^2\right)^2.
    \end{equation}
    This proves \cref{eq:characteristic_B_app} and \cref{eq:spectrum_B_app}.  If $\kappa\neq0$, then $\mathfrak B$ has one Jordan block of size two at each of the eigenvalues $\pm1/(2T_{\rm adv})$. If $\kappa=0$, then $\mathfrak B$ is diagonalizable.  In either case, the estimate \cref{eq:semigroup_B_app} follows.
\end{proof}
\begin{remark}
Denoting
    $$v_3:=(\zeta_1,\zeta_2,\zeta_3)^T, \qquad \zeta_1=\frac{\overline{z_2-z_3}}{\Gamma_1}, \quad \zeta_2=\frac{\overline{z_3-z_1}}{\Gamma_2}, \quad \zeta_3=\frac{\overline{z_1-z_2}}{\Gamma_3}, $$
where $z_i$ are the positions evaluated at $t=0$, in \cite{ArefSelfSimilar}  Aref checked that $v_3$ is also an eigenvector of $\overline{M}M$ associated to the eigenvalue $\lambda_2$ in the proof of Lemma \ref{lem:SpectralPropertiesB}. He then concludes that $v_3$ is not an affine transformation of $v_1,v_2$, and this would readily  imply that $\mathfrak{B}$ is always diagonalizable. Unfortunately, a direct computation exploiting the second condition in \cref{eq:groebli_circulation_condition_app} reveals that $\det(v_1,v_2,v_3)=0$, and in fact $v_3$ is a dilation and rotation of $v_2$. His conclusions about  instability of the collapsing case remain valid. The stability of the expanding configuration instead hides a logarithmic growth that is related to the accumulation of errors when we compare two nearby expanding Gr\"obli triangles. It is however worth noting that the key observation is the algebraically involved calculation of \cref{traceMM}, which is the building block of  Lemma \ref{lem:SpectralPropertiesB}. Hence, we have also verified the identity \cref{traceMM} with a SageMath 10.7 script provided in the code attached to the arXiv version of this paper. 
\label{rem:3arefvectors}
\end{remark}
Consider now the equation
\begin{equation}
    \begin{cases}
        \partial_t f = \mathfrak{L}(t) f + F, \\
        f(0) = 0,
    \end{cases}
    \label{eq:LeqSec2}
\end{equation}
where $\mathfrak{L}$ is the operator associated with \cref{eq:linearized_HK_app}.
Then, using the spectral characterization in Lemma \ref{lem:SpectralPropertiesB}, we have that
\begin{corollary}
\label{lem:transfer_matrix}
Let $\varsigma\in\{1,-1\}$ and $\mathcal U_{\mathfrak L}(t,s)$ denote the solution operator associated with \cref{eq:linearized_HK_app}.  For
$0\leq s\leq t$ in the interval of existence of the self-similar trajectory,
\begin{equation}
\label{eq:transfer_matrix_unified_app}
    \bigl\|\mathcal U_{\mathfrak L}(t,s)\bigr\|
    \lesssim \left( 1 + \left| \log \frac{\dhk(t)}{\dhk(s)} \right|\right)
    \left(\frac{\dhk(t)}{\dhk(s)}\right)^{\varsigma}
\end{equation}
Consequently, if $f$ solves \cref{eq:LeqSec2}, then
\begin{equation}
    |f(t)|
    \lesssim
    \int_0^t \left( 1 + \left| \log \frac{\dhk(t)}{\dhk(s)} \right|\right)
    \left(\frac{\dhk(t)}{\dhk(s)}\right)^{\varsigma}
    |F(s)|\,\dd s.
    \label{eq:BoundForcingHKLin}
\end{equation}
\end{corollary}

\begin{proof}
By Lemma \ref{lem:SpectralPropertiesB}, since there may be a Jordan block of size two 
\begin{equation}
    \left\|e^{(\tau(t)-\tau(s))\mathfrak B}\right\|
    \lesssim\left( 1 + \frac{\tau(t) - \tau(s)}{ T_{\rm adv}} \right)
    \exp\left(\frac{\tau(t)-\tau(s)}{2T_{\rm adv}}\right),
\end{equation}
where the phase factors in \cref{eq:rotating_variable_app} are unitary. Since
\begin{equation}
    \frac{\tau(t)-\tau(s)}{ T_{\rm adv}} = 2\left| \log \frac{\dhk(t)}{\dhk(s)} \right|,
    \qquad
    \exp\left(\frac{\tau(t)-\tau(s)}{2T_{\rm adv}}\right)
    =
    \left(
        \frac{1+\varsigma t/T_{\rm adv}}
             {1+\varsigma s/T_{\rm adv}}
    \right)^{\varsigma/2}
    =
    \left(\frac{\dhk(t)}{\dhk(s)}\right)^{\varsigma},
\end{equation}
the first estimate \cref{eq:transfer_matrix_unified_app} follows.  The second one \cref{eq:BoundForcingHKLin} follows  from Duhamel's formula.
\end{proof}

\section{The self-similar framework}
In this section we introduce the self-similar framework that naturally desingularizes the point vortex motion for the Navier--Stokes equations.    Consider the rescaled vorticity ansatz
    \begin{equation}
        \omega(x,t) = \sum_{i=1}^3\frac{\Gamma_i}{\nu t} \Omega_i\left( \frac{x-Z_i(t)}{\sqrt{\nu t}} ,t\right) 
        \label{eq:vorticityAnsatz}
\end{equation}
where we fix $\mathbf{\Gamma},\boldsymbol{Z}(0)$ so that we have a Gr\"obli triangle. More precisely, set $Z_i(0)=z_i^{\rm in}$, assume that the vortex center is fixed at the origin as in \cref{eq:center_of_vorticity} and that the constraints in \cref{eq:groebli_condition_intro} hold true. Then, we define
    \begin{equation}
       \gamma_i=\frac{\Gamma_i}{\Gamma},\qquad \xi_i := \frac{x-Z_i(t)}{\sqrt{\nu t}}, \quad i=1,2,3
    \end{equation}
    and, in view of \cref{eq:vorticityAnsatz}, we  split the stream function as
    \begin{equation}
        \psi(x,t) = \sum_{i=1}^3\Gamma_i \Psi_i\left(\xi_i ,t\right), \qquad \Psi_i = \Delta^{-1} \Omega_i, \quad i=1,2,3.
    \end{equation}
    Note that for all $i,j=1, 2, 3$,
    we have
    \begin{equation}
        \xi_i = \xi_j + \frac{Z_j - Z_i}{\sqrt{\nu t}}.
        \label{eq:transformXi}
    \end{equation}
    Hence, thanks to \cref{eq:transformXi}, once we fix  $$\xi := \xi_i,$$ the other two positions $\xi_j, \xi_k$ can be recovered from \cref{eq:transformXi}. We  introduce the following notation:
\begin{itemize}
    \item we denote the rescaled diffusion operator as
    \begin{equation}
    \label{def:cL}
    \cL:=\Delta_\xi+\frac12 \xi \cdot \nabla_\xi +1;
\end{equation}

\item we denote the translation operators as
\begin{equation}
(\cT_{\eps,\boldsymbol{Z}}^{i,j}f)(\xi)=(\cT_{\eps(t),\boldsymbol{Z}(t),\dhk}^{i,j})(f)(\xi) \coloneqq   f \left( \xi + \frac{Z_i(t) -Z_j(t)}{\eps \dhk(t)} \right).
    \label{eq:Tij}
\end{equation}
\end{itemize}
We then have the following.
\begin{lemma}
    Let $\omega$ be as in \cref{eq:vorticityAnsatz}. Then, $\omega$ solves \cref{eq:NSintro} with initial data \cref{def:omin} if 
        \begin{equation}
            t \de_t \Omega_i = \mathcal{L} \Omega_i - \frac{1}{\delta} \left\{ \Phi_{\epsilon, \Z,i}\, , \Omega_i\right\}, \quad \Omega_i(0)=G,
        \label{eq:OmegaIEq}
        \end{equation}
        where the interaction-streamfunction is defined as
        \begin{equation}
           \label{def:PhiepsZi} \Phi_{\epsilon, \Z,i}={\Phi}_{\eps,i}(\mathbf{\Psi},\boldsymbol{Z}):= \gamma_i\Psi_i+\sum_{j\neq i}\gamma_j\mathcal{T}_{\eps,\boldsymbol{Z}}^{i,j}\Psi_j - \frac{ \delta t }{\eps \dhk} \partial_t Z_i \cdot \xi^\perp
        \end{equation}
\end{lemma}
\begin{proof}
    The proof is a direct computation as detailed in  \cite{Gallay2011} and therefore is omitted here.
\end{proof}

\subsection{Choice of the positions}
Following \cite{GaSring,dolce2024long}, it is convenient to fix the position vector $\boldsymbol{Z}(t)$ so that we follow the vortex center of each vortex $\Omega_i$. Namely, recalling \cref{def:mk}, we require that
\begin{equation}
    m^{1}_j(\Omega_i)(t)=\int_{\mathbb{R}^2}\xi_j\Omega_i(\xi,t) \dd \xi=0  \quad \text{for all } i = 1,2,3 \quad j=1,2 \quad t\geq0.
\end{equation}
At leading order where all vortices are assumed to evolve exactly as Lamb-Oseen vortex,  this choice easily recovers the Helmholtz-Kirchhoff dynamics as shown in \cite{dolce2024long}. From a  technical point of view, the zero first-order moment constraint is a natural modulation accounting for the translation symmetries of $\mathbb{R}^2$. Then, we have the following.
\begin{lemma}
    \label{lem:PositionChoice}
    Let $i=1,2,3$ and $\Omega_i$ be the solutions to \cref{eq:OmegaIEq}. If for all $t>0$,  $Z_i(t)$ is chosen  so that
    \begin{equation}
        t\de_t Z_i(t) = \frac{\dhk\eps(t)}{\delta} \sum_{j\neq i}\gamma_j\int_{\RR^2}\left( \big(\nabla^{\perp}  \cT_{\eps,\Z}^{i,j} \Psi_j \big)
         \Omega_i\right)(\xi,t) \, \dd\xi , \quad Z_i(0)=z_i^{\rm in}
        \label{eq:dtZi}
    \end{equation}
then, with the notation \cref{def:mk}, the following moment identities hold true
    \begin{equation}
        {\rm M}(\Omega_i(t)) = 1, \qquad {\rm m}^1(\Omega_i(t)) = (0,0), \qquad \text{ for all } t\geq0.
    \end{equation}
\end{lemma}
\begin{proof}
    As in \cite{dolce2024long,dolce2026viscousvortexcrystals}, the proof readily follows by testing \cref{eq:OmegaIEq} against $\xi_1$ and $\xi_2$ and, upon integrating by parts, the choice \cref{eq:dtZi}  implies that the first order moments satisfy
    \begin{equation}
        t \de_t {\rm m}^1(\Omega_i) = -\frac{1}{2} {\rm m}^1(\Omega_i).
    \end{equation}
    Since $\de_t (t |{\rm m}^1(\Omega_i)|^2) = 0$, and as $t \to 0$, $ t |{\rm m}^1(\Omega_i)|^2 \to 0$, we have ${\rm m}^1(\Omega_i(t))= (0,0).$
\end{proof}

\subsection{Basic functional setting}
\label{sec:Func}
We now introduce the basic functional setting required to construct the approximate solution and control its nonlinear evolution, where we mostly recall results from \cite{GalWay2005,Gallay2011,dolce2024long,dolce2026viscousvortexcrystals}. The only difference is a new lemma related to a parabolic equation involving the diffusion  operator. 

We first recall the definition of the spaces 
\begin{align}
\label{def:cY}
  \cY \,&=\, \biggl\{f\in L^2(\RR^2)\,:\, \int_{\RR^2} |f(\xi)|^2
  \,\e^{|\xi|^2/4}\,\dd \xi < \infty\biggr\}\,,\\
  \label{def:calZ}
  \cZ \,&=\, \bigl\{f:\RR^2\to \RR \,:\, \xi \mapsto \e^{|\xi|^2/4}f(\xi) \in
  \cS_*(\RR^2)\bigr\}\,,
\end{align}
where $\cS_*(\RR^2)$ denotes the space of smooth functions with at most
polynomial growth at infinity. More precisely, \[ f\in \cS_*(\RR^2) \quad\Longleftrightarrow\quad \forall\,\alpha\in\mathbb N^2\ \exists\, C_\alpha>0,\ N_\alpha\in\mathbb N \ \text{such that}\ |\partial_\xi^\alpha f(\xi)| \leq C_\alpha(1+|\xi|)^{N_\alpha}, \quad \xi\in\RR^2 . \] 
The space $\cY$ is a Hilbert space equipped with the inner product 
\begin{equation}
  \langle{f,g}\rangle_{\cY} \,=\, \int_{\RR^2}f(\xi)\,g(\xi)\,\e^{|\xi|^2/4}\,\dd \xi\,, \qquad
  \forall\,f,g \in \cY,
\end{equation}
 and we have the orthogonal decomposition 
\begin{equation}\label{Ydecomp}
  \cY \,= \,\bigoplus_{n=0}^{\infty}\,\cY_n\,,
\end{equation}
where $\cY_n = \bigl\{f\in \cY\,:\, f(\xi) = a(r)\cos(n\theta) + b(r)\sin(n\theta) \text{ with }  a,b:
\RR_+\to\RR\bigr\}$. We also denote with $\cP_n$ the $L^2$-projection operator onto the $n$-th angular Fourier mode.

Then, we recall the $\cO_{\cS_*}$ and $\cO_\cZ$ notation introduced in \cite[Definition 3.4]{dolce2024long}.
\begin{definition}\label{def:topo}
Let $M \in \NN$. 

\smallskip\noindent 1)
Let $g_\epsilon \in \cS_*(\RR^2)$ depend on a small parameter $\epsilon > 0$. We say that $g_\epsilon \,=\, \cO_{\cS_*}\bigl(\epsilon^M\bigr)$ if, for all $\alpha \in \NN^2 $ and all $m\in \{0,\dots, M\}$ there exists $C>0$ and $\widetilde{N}\in \NN$ such that
\[
  |\partial^{\alpha}_\xi\de_\epsilon^mg_\epsilon(\xi)| \le C(1+|\xi|)^{\widetilde{N}}\epsilon^{M-m}
  \quad \forall\,\xi \in \RR^2\,.
\]

\noindent 2) Similarly, if $f_\epsilon \in \cZ$, we write $f_\epsilon = \cO_{\cZ}
\bigl(\epsilon^M\bigr)$ if $\e^{|\xi|^2/4}f_\epsilon = \cO_{\cS_*}\bigl(\epsilon^M\bigr)$. 
\end{definition}

We now state the relevant properties of the main linear operators involved in this paper.  \subsubsection*{{\rm i)} \textbf{The diffusion operator}} The diffusion operator $\cL$ defined by \cref{def:cL} is
a linear operator in $\cY$ with (maximal) domain
\begin{equation}\label{def:cLdom}
  D(\cL) \,=\, \bigl\{f\in \cY \, : \,\Delta f\in \cY, \, 
  \xi\cdot \nabla f \in \cY\bigr\}\,.
\end{equation}
It is well known, see \cite[Appendix~A]{GalWay2002}, that $\cL$ is {\em self-adjoint} in
$\cY$ with compact resolvent and purely discrete spectrum.   From \cite[Lemma 3.5]{dolce2024long} we get that:
\begin{lemma}\label{lem:Linvert}
For any $\kappa > 0$ and any $f \in \cZ$ one has $(\kappa - \cL)^{-1}f
\in \cZ$.   
\end{lemma}
In the construction of the approximate solution, since the effective aspect ratio $\eps(t)$ includes the time-dependent $\dhk(t)$ in its definition, we encounter the following  parabolic equation to solve
    \begin{align}
    \label{eq:LinOm2}
        \begin{dcases}
            \displaystyle\tau\partial_\tau f+\frac{q}{1+\varsigma\tau}f-\mathcal{L}f=H(\tau,\xi), \qquad \tau>0, \, \xi \in \mathbb{R}^2, \,\varsigma \in \{ -1,1\},\, q\geq1, \\ 
     f|_{\tau=0}=(q-\mathcal L)^{-1}H(0,\cdot).
        \end{dcases}
    \end{align} 
    The case $\varsigma=1$ corresponds to the expanding configuration whereas $\varsigma=-1$ to the collapsing one. We need to provide a control in time of the solution $f$ in the space $\mathcal{Z}$.
    Let us define the following time intervals 
    \begin{equation}
        I_{\varsigma} := \begin{cases}
            [0,\infty) & \text{if } \varsigma = 1,\\
            [0,1) & \text{if } \varsigma = -1.
        \end{cases}
    \end{equation}
\begin{lemma}
\label{lem:lin}
    Let $f$ be the regular solution to \cref{eq:LinOm2}.   Assume that $H(\tau, \xi)$ is smooth up to $\tau = 0$, $H(\tau, \cdot) \in \cZ$ and  $M(H(\tau, \cdot)) = 0$ for every $\tau \in I_{\varsigma}$.
    Moreover, suppose that there exists a family of suitable positive, bounded functions $A_j(\tau)$ such that , for every multi-index
    $\alpha$ and every $j\in\mathbb N_0$,
    \begin{equation}
        \left|(\tau\partial_\tau)^j\partial_\xi^\alpha
        \left(e^{|\xi|^2/4}H(\tau,\xi)\right)
        \right|
        \leq
        C_{\alpha,j}(1+|\xi|)^{N_{\alpha,j}}A_j(\tau)
        \end{equation}
with constants independent of $\tau$ and $\xi$.
    Then
    \begin{enumerate}[label = \roman *)]
        \item If $\varsigma = 1$, $f(\tau,\cdot)\in \mathcal{Z}$ for every $\tau\geq 0$. More precisely, for every $\beta \in \mathbb{N}$ and for every multi-index $\alpha$ there exist  constants $C_{\alpha, \beta,q},N_{\alpha,\beta}>0$, depending only on $q$ and bounds on the forcing, such that the family $\{f(\tau,\cdot)\}_{\tau\geq0}\in \mathcal{Z}$ satisfies 
    \begin{align}
    \label{bd:lemmapara}
    \big|(\tau\partial_\tau)^{\beta}\partial_\xi^\alpha\big(\e^{|\xi|^2/4}f(\tau,\xi)\big)\big|\leq C_{\alpha,\beta,q}(1+|\xi|)^{N_{\alpha,\beta}} \sum_{j = 0}^\beta \sup_{0 \leq r \leq \tau} A_j(r),
    \end{align}
    for all $\tau\geq 0, \, \xi \in \mathbb{R}^2$.
    \item If $\varsigma = -1$, $f(\tau,\cdot)\in \mathcal{Z}$ for every $\tau\in [0,1)$ in the sense specified in item i), but with the bound
        \begin{align}
    \big|(\tau\partial_\tau)^{\beta}\partial_\xi^\alpha\big(\e^{|\xi|^2/4}f(\tau,\xi)\big)\big|\leq \frac{C_{\alpha,\beta,q}}{(1-\tau)^\beta}(1+|\xi|)^{N_{\alpha,\beta}} \sum_{j = 0}^\beta \sup_{0 \leq r \leq \tau} (1-r)^{j+1}A_j(r),
    \end{align}
     for all $\tau \in [0,1), \, \xi \in \mathbb{R}^2$.
    \end{enumerate}
\end{lemma}

The proof of the lemma above is presented in Appendix \ref{app:LemmaLin}. 
\begin{remark}
Note that if $M(H(\tau,\cdot)\neq 0$ then it is not hard to see that the solution to \cref{eq:LinOm2} may actually grow like $\ln(1+\tau)$ when $\varsigma=1$. This is a consequence of the fact that the Lamb-Oseen vortex is in the kernel of $\mathcal{L}$. Indeed, taking the time independent forcing $H(\tau, \xi)=G(\xi)$ we achieve the sharp growth. This growth would be essentially harmless within our time-scale, but since the approximate solution naturally satisfies the zero average constraint it is nicer to have a better uniform control on it.
\end{remark}
\subsubsection*{{\rm ii)} 
\textbf{The Lamb-Oseen operator}}
The operator arising from the linearization of the $2D$ Euler equation at the Lamb-Oseen vortex is denoted by  $\Lambda:D(\Lambda)\to \cY$ and defined as 
\begin{equation}\label{def:Lambda2}
  \Lambda f \,=\, \bigl\{\Psi_G\,,\,f\bigr\} + \bigl\{\Delta^{-1}f\,,\,G\}\,,
  \qquad f \in D(\Lambda)\,, 
\end{equation}
where $G$ is defined in \cref{def:G} and $\Psi_G=\Delta^{-1}G$. The operator $\Lambda$ is considered in its maximal domain
$$  D(\Lambda) \,=\, \bigl\{f\in \cY \, : \, \{\Psi_G, f\} \in \cY\bigr\}\,$$ and it commutes with the  direct sum decomposition
\cref{Ydecomp}. Moreover, if $f \in \cZ$ then $\Lambda f \in \cZ$. 
It is  well known that there exists a functional relationship between the streamfunction and vorticity of the Lamb-Oseen vortex, given by
\begin{align}
\label{def:functionalLO}
    \Psi_G+F(G)=0, \qquad F(s)=\frac{1}{4\pi}\Big(\gamma_E-{\rm Ein}(\log\big(\frac{1}{4\pi s}\big))\Big) , \qquad 0< s\leq \frac{1}{4\pi},
\end{align}
where $\gamma_E$ is the Euler-Mascheroni constant and ${\rm Ein}(x)=\int_0^xs^{-1}(1-\e^{-s}) \dd s$ is the exponential integral function. Thus, introducing 
\begin{align}
\label{def:AH}
    A(\xi):=F'(G)(\xi)=\frac{4}{|\xi|^2}\big(\e^{|\xi|^2/4}-1\big), \qquad \mathsf{H}:=A+\Delta^{-1},
\end{align}
where $\mathsf{H}$ is self-adjoint in $L^2(\mathbb{R}^2)$,
we can rewrite $\Lambda$ in the Hamiltonian form 
\begin{align}
\label{eq:LambdaA}
    \Lambda f=\{\mathsf{H}f,G\}.
\end{align}
We then recall the following known properties of $\Lambda$, whose proof can be found for instance in \cite{Gallay2011,dolce2026viscousvortexcrystals}.
\begin{proposition}\label{prop:Lambda}
The following properties hold true:
\begin{enumerate}[label=\arabic*)]
    \item The operator $\Lambda$ is skew-adjoint in the Hilbert space $\cY$ with
kernel 
\begin{equation}\label{kerLam}
  \mathrm{Ker}(\Lambda) \,=\, \cY_0 \,\oplus\, \bigl\{\beta_1\partial_1 G +
  \beta_2\partial_2 G  \,:\, \beta_1, \beta_2 \in \RR\bigr\}\,.
\end{equation}
\item If $h \in \mathrm{Ker}(\Lambda)^\perp \cap \cZ$ or $h\in \cZ$ with $\cP_0(h)=\mrm^1_1(h)=\mrm^1_2(h)=0$ the equation $\Lambda f = h$
has a unique solution $f \in \mathrm{Ker}(\Lambda)^\perp \cap \cZ$.
\label{it:solvability}
\item If $\mathcal{P}_0\varphi=0$ then the unique solution with $\mathcal{P}_0 f =0$ satisfies
\begin{align}
  \Lambda f=\{\varphi,G\}\quad \iff \quad  \mathsf{H}f=\varphi.
\end{align}
\label{it:equivalence}
\item $\mathsf{H}(\partial_i G)=0$ for $i=1,2$.
\label{it:kernel}
\end{enumerate}
\end{proposition}
The property \ref{it:solvability} is needed to properly invert the operator $\Lambda$ under suitable solvability conditions, where we can already see the relevance of having zero first order moments. Instead, the simple equivalence \ref{it:equivalence} was observed in \cite{dolce2026viscousvortexcrystals} and it will be useful for us in the construction of the approximate functional relationship. Finally, \ref{it:kernel} follows by a direct computation and is a key structural property related to translation invariance in $\mathbb{R}^2$. In particular, property \ref{it:kernel} is quite useful to eliminate some potentially dangerous terms in the nonlinear argument.

\subsubsection*{{\rm iii)} 
\textbf{The translation operator}}
In the equation \cref{eq:OmegaIEq}, the interactions among the vortices involve the translation operator $\mathcal{T}_{\eps,\boldsymbol{Z}}^{i,j}$ in \cref{eq:Tij}. In \cite[Lemma~3.3, ~3.9, ~3.11, ~A.2]{dolce2026viscousvortexcrystals}, the authors derived all the properties that we need for operators of the form
\begin{align}
\label{def:Tepsgen}
    \mathcal{T}_{\varepsilon,Z}(f)(\xi)=f(\xi+\varepsilon^{-1}Z),
\end{align}
such as Taylor expansions in $\mathcal{S}_*$ or the expansions of integral quantities as the ones involved in \cref{eq:dtZi}. Setting $\varepsilon\to (\eps \dhk)(t)$ and $Z\to Z_i-Z_j$, it is straightforward to see that $\mathcal{T}^{i,j}_{\eps,\boldsymbol{Z}}$ is exactly a translation operator as in \cref{def:Tepsgen}. Thus, we have the following.
\begin{lemma}
    Let $f \in \cZ$ and  $0<z_{min} \leq |Z|$. Let $\varphi = \Delta^{-1} f$ and $ \cT_{\epsilon,Z}$ be defined as in \cref{def:Tepsgen}. 
    Then,
    \begin{enumerate}[label= \roman *)]
        \item for all $M \in \mathbb{N}$ one has
    \begin{equation}
        (\cT_{\varepsilon,Z} \varphi)(\xi) = \frac{M(f)}{2\pi} \log \left( \frac{|Z|}{\varepsilon }\right) + \sum_{n=1}^{M} \sum_{k=0}^n \varepsilon^n {c}_{n,k} \Re\left( \frac{\xi^{n-k}}{Z^n} m^k(f)\right) + \cO_{\cS_*}\left(\left(\varepsilon /|Z|\right)^{M+1}\right), 
    \end{equation}
    where
    \begin{equation}
        {c}_{0,k} = (-1)^{k-1}, \quad \quad {c}_{n,k} := \begin{pmatrix}
            n \\k
        \end{pmatrix} (-1)^{n+k-1} \frac{1}{2\pi n} \quad \text{ for } n \in \mathbb{N} \ \backslash \ {0}, \ k \in \mathbb{N}.
    \end{equation}
    As a consequence, for any $h \in \cZ$, the following integral can be expanded as 
\begin{equation}
        \int_{\RR^2} (\nabla^{\perp}\cT_{\varepsilon,Z}\varphi)h d\xi= \sum_{n=1}^{M} \varepsilon^{n}  \frac{Z^\perp}{|Z|^2} \frac{1}{\overline{Z^{n-1}}} \sum_{k=0}^{n-1} {\widetilde{c}}_{n,k} \overline{m^k(f)m^{n-k-1}(h)} + \cO\left(\left(\varepsilon /|Z|\right)^{M+1}\right),
        \label{eq:IntGradTeExp}
    \end{equation}
    with ${\widetilde{c}}_{n,k} = (n-k) c_{n,k}$.
    \item Let $\mathcal{T}^{i,j}_{\eps,\boldsymbol{Z}}$ be defined as in \cref{eq:Tij}. Then, the same expansions of the item above holds true upon the changes $\varepsilon\to(\eps \dhk)(t)$ and $Z\to Z_i(t)-Z_j(t)$. Moreover
    \begin{equation}
\label{eq:Tij*}
(\cT^{i,j}_{\eps,\boldsymbol{Z}} )^*= \cT^{i,j}_{-\eps,\boldsymbol{Z}} =\cT^{j,i}_{\eps,\boldsymbol{Z}}.        
    \end{equation}
    where $(\cT^{i,j}_{\eps,\boldsymbol{Z}} )^*$ is the $L^2(\mathbb{R}^2)$-adjoint of $\cT^{i,j}_{\eps,\boldsymbol{Z}} $.
    \end{enumerate}
    \label{lem:TeProperties}
\end{lemma}

\section{Approximate solution and functional relationship}
\label{sec:ApproxSol}
In this section, we build an approximate solution to \cref{eq:OmegaIEq}
via an asymptotic expansion in the parameter $\eps$. The approximate solution is essentially the one constructed by Gallay in
\cite{Gallay2011}, but here we need more precise information. For instance, we rewrite the approximate solution according to the Euler and Navier--Stokes splitting used in \cite{GaSring,dolce2024long,zhang2025long,dolce2026viscousvortexcrystals}, which allows us to better
isolate the dependence on the two small parameters $\eps$ and $\delta$.
Furthermore, a crucial ingredient is the construction of a second order approximate functional relationship
between the approximate vorticity and the corresponding stream function, as in
\cite{GaSring,dolce2024long,dolce2026viscousvortexcrystals,zhang2025long}. In our setting, however, the existence of said functional relation cannot rely on symmetries as in all previous constructions. Nonetheless, we can directly define it rather explicitly. This heavily relies on the general structural properties of the first correction to the ansatz that the system evolves as a collection of interacting Lamb-Oseen vortices.
\subsection{Approximate solution}
We recall that the rescaled
vorticity profile satisfies
\begin{equation}
    t \partial_t \Omega_i
    =
    \mathcal L \Omega_i
    -
    \frac{1}{\delta}
    \left\{
        \Phi_{\eps,\boldsymbol{Z},i},\Omega_{i}
    \right\},
    \label{eq:OmegaIEqApproxRecall}
\end{equation}
where $\Phi_{\eps,\boldsymbol{Z},i}$ is defined in \cref{def:PhiepsZi}.
We look for a fourth order approximate solution of the form 
\begin{align}
    &\Omega_{\app,i}(\xi,t)=\sum_{n=0}^4\eps^n(t)\Omega_{n,i}(\xi,t), \qquad \Psi_{\app,i}(\xi,t)=\sum_{n=0}^4\eps^n(t)\Psi_{n,i}(\xi,t)
    \label{def:AppSol}
\end{align}
where $\Psi_{n,i}(\xi,t) = \Delta^{-1} \Omega_{n,i}(\xi,t)$.
The profiles $(\Omega_{n,i})_{n\geq0}$ have to be determined in order to improve the accuracy of the approximate solution when plugged in \cref{eq:OmegaIEq}.
The zero-th order profile is dictated by the initial datum. More precisely, by
the uniqueness of the two-dimensional Navier--Stokes equation with a measure as
initial datum \cite{GaGa2005}, we set
\begin{equation}
    \Omega_{0,i}=G,
    \qquad
    \Psi_{0,i}=\Psi_G.
    \label{eq:ZeroProfileApprox}
\end{equation}
In order to build an approximate solution, at each step we should also define an approximate position for the centers of vortices. However, we anticipate that for the time-scales under consideration it is not necessary to update it at each iteration as done in \cite{dolce2024long} and therefore to construct the approximate solution we fix once and for all in this section 
\begin{equation}
    \boldsymbol{Z}(t) = \boldsymbol{z}(t)
\end{equation}
where $\zhk{i}$ solves \cref{eq:HK}. 
Given the approximate profiles, we define the corresponding approximate interaction-streamfunction by
\begin{equation}
    \label{def:PhiepsZiSec4} \Phi_{{\rm app},i} :={\Phi}_{\eps, i}(\mathbf{\Psi}_{\rm app},\boldsymbol{z}):= \gamma_i\Psi_{{\rm app},i}+\sum_{j\neq i}\gamma_j\mathcal{T}_{\eps,\boldsymbol{z}}^{i,j}\Psi_{\rm app, j} - \frac{ \delta t }{\eps \dhk} \partial_t z_i \cdot \xi^\perp.
\end{equation}
Then, the error associated to the approximate solution is defined as
\begin{equation}
    \cR_{i} := \delta(t\de_t - \cL)(\Omega_{{\rm app},i}) + \{ \Phi_{{\rm app},i}, \Omega_{{\rm app},i} \}.
    \label{def:Reminder}
\end{equation}
\begin{remark}
    Corrections to the Helmholtz-Kirchhoff system \cref{eq:HK} are only relevant for approximate solutions from the order $\eps^5$, as first highlighted in \cite{dolce2024long}. For the expanding configuration, in view of the key property $\eps\lesssim \sqrt{\delta}$ for all $t\geq 0$ we can  safely neglect errors of order $\mathcal{O}(\delta^{-2}\eps^5)\lesssim \mathcal{O}(\eps)$ for all $t\geq 0$, which is the main reason why it is not necessary to look for a higher-order expansion. Similarly, for the collapsing case our time scales ensure that analogous errors are suitably small. 
\end{remark}
Following \cite{GaSring,dolce2024long},
 it is convenient to expand $\Omega_{{\rm app},i}$ also in $\delta$ as to handle separately viscous effects. Thus, we expand each approximate profile at n-th order with
 the Eulerian and Navier-Stokes components as
 \begin{equation}
     \Omega_{n,i}= \Omega_{n,i}^E + \delta \Omega_{n,i}^{NS}, \qquad \Psi_{n,i}= \Psi_{n,i}^E + \delta \Psi_{n,i}^{NS}.
 \end{equation}
 Accordingly, with the obvious identification, we split the approximate solution as
 \begin{equation}
     \Omega_{{\rm app},i}= \Omega_{{\rm app},i}^{E} + \delta \Omega_{{\rm app},i}^{NS}, \qquad \Psi_{{\rm app},i}= \Psi_{{\rm app},i}^{E} + \delta \Psi_{{\rm app},i}^{NS}.
 \end{equation}
We are now ready to define our approximate solution, for which some computationally heavier details that follow by adapting \cite{Gallay2011} are provided in the Appendix \ref{app:ApproxSolProof}.
\begin{remark}
        It would be appropriate to include also a dependence on the parameter $\varsigma\in \{1,-1\}$ in the approximate solution to distinguish the expanding and collapsing cases respectively. However, to avoid making the notation even heavier, we explicitly specify  the differences only when needed.
\end{remark}
\subsubsection{The zeroth-order approximate solution.}
We first rewrite the error generated by the zeroth-order approximate solution. This corresponds to the ansatz 
$  \Omega_{\app,i}=G$ and accordingly we define 
    \begin{equation}
    \cR_{0,i} := \delta(t\de_t - \cL)(G) + \{ \Phi_{\eps,i}(\Psi_G,\boldsymbol{z}), G\}, \qquad \text{ for } i=1,2,3.
    \label{def:Ri0}
\end{equation}
We then rewrite the zeroth--order remainder in \cite{Gallay2011} as follows, where we sketch the proof with our notation  for convenience of the reader.
\begin{lemma}
    \label{lem:R0}
    Let $\cR_{0,i}$ be defined as in \cref{def:Ri0}. Then
    \begin{equation}
        \cR_{0,i}(\xi,t)
        =
        \sum_{n=2}^4 \eps(t)^n \cH_{n,i}(\xi,t)
        +
        \cO_{\cZ}(\eps^5),
        \label{eq:Res0th}
    \end{equation}
    where
    \begin{equation}
        \cH_{n,i}(\xi,t)
        :=
        \dhk(t)^n c_{n,0}
        \sum_{j\neq i}
        \gamma_j
        \left\{
            \Re\left(
                \frac{\xi^n}{(z_{i} - z_{j})^n}
            \right),
            G(\xi)
        \right\}.
        \label{eq:defHin}
    \end{equation}
    In particular, $
        \cR_{0,i}=\cO_{\cZ}(\eps^2)$ and, for every $n=2,3,4$,
    \begin{equation}
        M(\cH_{n,i})=0,
        \qquad
        m^1(\cH_{n,i})=(0,0).
    \end{equation}
\end{lemma}

\begin{proof}
    By definition of the residual and the fact that $(t\partial_t-\mathcal{L})(G)=\{\Psi_G,G\}=0$  we have
    \begin{equation}
        \cR_{0,i}
        =
        \left\{
            \sum_{j\neq i}
            \gamma_j\mathcal T_{\eps,\boldsymbol z}^{i,j}\Psi_G
            -
            \frac{\delta t}{\eps\dhk}\partial_t z_i\cdot \xi^\perp,
            G
        \right\}.
        \label{eq:R0Reduced}
    \end{equation}
    We now use the expansion of the translation operators from
    Lemma \ref{lem:TeProperties}. The constant term gives no contribution because
    its Poisson bracket with $G$ vanishes. Moreover, the first-order term is
    exactly cancelled by the Helmholtz--Kirchhoff transport term. 
     Hence the expansion starts at order two and Lemma \ref{lem:TeProperties} gives
    \begin{equation}
        \cR_{0,i}
        =
        \sum_{n=2}^4
        (\eps \dhk)^n c_{n,0}
        \sum_{j\neq i}
        \gamma_j
        \left\{
            \Re\left(
                \frac{\xi^n}{(z_{i} - z_{j})^n}
            \right),
            G
        \right\}
        +
        \cO_{\cZ}(\eps^5),
    \end{equation}
    which is \cref{eq:Res0th}. Since $|z_i(t)-z_j(t)|\approx \dhk(t)$, we have
    \begin{equation}
        \cH_{n,i}=\cO_{\cZ}(1),
        \qquad
        n=2,3,4,
    \end{equation}
    which implies that $
        \cR_{0,i}=\cO_{\cZ}(\eps^2).$
    It remains to prove the moment conditions. The average is straightforward because $\mathcal{H}_{n,i}$ is defined through a Poisson bracket. For the first-order moments, we simply exploit the fact that $G$ is radial and $n\geq 2$ to deduce that for any $a\in \mathbb{C}$ we have  $\mathcal{P}_n(\{\Re(a\xi^n),G\}) = \{\Re(a\xi^n),G\}$ for any $n \geq 1$. Thus, by orthogonality of different angular modes we get $
        m^1(\cH_{n,i})=(0,0)$ for $
        n=2,3,4$  as desired. 
\end{proof}

\subsubsection{Definition of the fourth-order approximate solution}

We now construct the approximate solution up to order $M=4$. 
By Lemma \ref{lem:R0} we know that no first-order correction is needed.
We introduce the decomposition
\begin{equation}
    \Omega_{n,i}(\xi,t)
    =
    (\Omega_{n,i}^{E}
    +
    \Omega_{n,i}^{E_{\rm r}}
    +
    \delta \Omega_{n,i}^{NS})(\xi,t),
\label{eq:ProfileSplittingApprox}
\end{equation}
and the streamfunction is decomposed  accordingly. Here the superscript $E$ denotes the Eulerian non-radial contribution, $E_{\rm r}$
the radial Eulerian correction, and $NS$ the Navier--Stokes correction. 
\begin{definition}
    \label{def:ApproxProfilesOrder4}
    When $\varsigma=1$ ($\varsigma=-1$), the profiles of the fourth-order approximate solution are defined for all $t\geq 0$ ($t\in [0,T_{\rm adv})$) as
    follows:
    \begin{itemize}[itemsep=4pt]
        \item \textbf{ Order $2$}: we set \begin{equation}
        \Omega_{2,i}^E = - \frac{1}{\gamma_i} \Lambda^{-1} \cH_{2,i}, \qquad \Omega_{2,i}^{NS} = - \frac{1}{\gamma_i}\Lambda^{-1} \left( t \partial_t +\frac{1}{1 + \varsigma \frac{t}{T_{\rm adv}}}- \cL \right) \Omega_{2,i}^E, \qquad \Omega_{2,i}^{E_\mathrm r} = 0.
    \end{equation}
    \item \textbf{Order $3$}: we set
    \begin{equation}
        \Omega_{3,i}^E = - \frac{1}{\gamma_i} \Lambda^{-1} \cH_{3,i}, \qquad \Omega_{3,i}^{NS} = - \frac{1}{\gamma_i} \Lambda^{-1} \left( t \partial_t +\frac{3}{2}\frac{1}{1 + \varsigma \frac{t}{T_{\rm adv}} }- \cL \right) \Omega_{3,i}^E, \qquad \Omega_{3,i}^{E_r} = 0.
    \end{equation}
    \item \textbf{Order $4$}: we set
    \begin{equation}
        \Omega_{4,i}^E = - \frac{1}{\gamma_i}\Lambda^{-1} \widetilde{\cH}_{4,i}^{E},
    \end{equation}
    \begin{equation}
        \Omega_{4,i}^{NS} = -\frac{1}{\gamma_i}\Lambda^{-1} \left((1-\cP_0) \widetilde{\cH}_{4,i}^{NS} + \left(t\partial_t + 2 \frac{1}{1 + \varsigma \frac{t}{T_{\rm adv}} } - \cL \right)^{} \Omega_{4,i}^E \right),
    \end{equation}
    and $\Omega_{4,i}^{E_r}$ is the solution to 
    \begin{equation}
        \left(t\partial_t + 2 \frac{1}{1 + \varsigma \frac{t}{T_{\rm adv}} } - \cL \right)^{}\Omega_{4,i}^{E_r} = - \cP_0(\widetilde{\cH}_{4,i}^{NS}), \qquad \Omega_{4,i}^{E_r}(\cdot,0)=-(2-\mathcal{L})^{-1}\mathcal P_0 {\widetilde{\mathcal{H}}_{4,i}^{NS}(0)},
    \end{equation}
    where
     $\widetilde{\cH}_{4,i}^{E} = \cO_{\cZ}(1)$, $\widetilde{\cH}_{4,i}^{NS} = \cO_{\cZ}(1) \ (\widetilde{\cH}_{4,i}^{NS} = \cO_{\cZ} ((1-t/T_{\rm adv})^{-1}))$ are defined respectively in  \cref{def:TildeH4EAppendix,def:TildeH4NSAppendix}.
    \end{itemize}
    The associated stream functions are defined by
    \begin{equation}
        \Psi_{n,i}^{\mathsf{q}}:=\Delta^{-1}\Omega_{n,i}^{\mathsf{q}},
        \qquad
        n=2,3,4,
        \qquad
        \mathsf{q}\in\{E,E_{\rm r},NS\}.
    \end{equation}
    \end{definition}

\begin{remark}
    Note that the parabolic equation that defines $\Omega_{4,i}^{E_r}$ has the same structure as the one in \cref{eq:LinOm2} up to the rescaling in time $\tau=t/T_{\rm adv}$. Moreover, the fact that $\Omega^{E_r}_{2,i}=\Omega^{E_r}_{3,i}=0$ is a consequence of a nice cancellation exploiting the Hamiltonian structure of the operator $\Lambda$, see Appendix \ref{app:ApproxSolProof} for details.
\end{remark}
Thanks to the above definitions, we have the following result, which can be directly extracted from \cite{Gallay2011} as well. For completeness, we present the proof in  Appendix \ref{app:ApproxSolProof}.
\begin{proposition}
    Let the profiles $\Omega_{n,i}^{\mathsf{q}}$, with $n=2,3,4$ and
    $\mathsf{q}\in\{E,E_{\rm r},NS\}$, be defined by Definition \ref{def:ApproxProfilesOrder4}. If $\varsigma=1$ ($\varsigma=-1$), then for every $i=1,2,3$ and $t\geq 0$ ($t\in [0,T_{\rm adv})$) one has:
    \begin{enumerate}[label = \roman *), itemsep=4pt]
        \item the profiles satisfy the moment conditions
    \begin{equation}
        M(\Omega_{n,i}^{\mathsf{q}})=0,
        \qquad
        m^1(\Omega_{n,i}^{\mathsf{q}})=(0,0),
        \qquad
        n=2,3,4,
                \qquad
        \mathsf{q}\in\{E,E_{\rm r},NS\},
        \label{eq:ApproxProfilesMoments}
    \end{equation}
and are therefore well-defined;
        \item for $m=0,1$ and $\mathsf{E}=\{E,E_r\}$, the vorticity profile corrections satisfy
\begin{align}
       \label{eq:ApproxProfilesBoundE}&   (t\de_t)^m\Omega_{n,i}^{\mathsf{E}}(\cdot,t)
    =\begin{cases}\cO_{\cZ}(1), &\quad \text{ if } \varsigma=1,\\
    \cO_{\cZ}\big((1-\frac{t}{T_{\rm adv}})^{-m}\big), &\quad \text{ if } \varsigma=-1,
\end{cases} \\
&(t\de_t)^m\Omega_{n,i}^{NS}(\cdot,t)=\begin{cases}\cO_{\cZ}(1), &\quad \text{ if } \varsigma=1,\\
    \cO_{\cZ}\big((1-\frac{t}{T_{\rm adv}})^{-(1+m)}\big), &\quad \text{ if } \varsigma=-1.
\end{cases}
        \label{eq:ApproxProfilesBound}
\end{align}
    \item the residual associated with $\Omega_{\app,i}$ defined in \cref{def:Reminder} satisfies 
    \begin{equation}
        \cR_{i}
        =\begin{cases}
        \cO_{\cZ}
        \left(
            \eps^5+\delta^2\eps^2
        \right), \quad &\text{ if } \varsigma=1,\\
        \cO_{\cZ}
        \left(
            \eps^5+\delta^2\eps^2(1-\frac{t}{T_{\rm adv}})^{-2}
        \right) \quad &\text{ if } \varsigma=-1.
        \end{cases}
            \label{eq:ResidualOrder4Bound}
    \end{equation}
    \item In the collapsing case $\varsigma=-1$, let $m=0,1$, $t\in [0,(1-\delta^{p'})T_{\rm adv})$ and $p'\in (0,\frac{1}{5})$.Then
    \begin{align}
        \Omega_{\rm app, i}^{NS}=\cO_{\cZ}( \eps) , \qquad\mathcal{R}_i=\cO_{\cZ}(\eps^5+\delta^2\eps).
    \end{align}
    \end{enumerate}
    \label{prop:ApproximateSol}
\end{proposition}
\begin{remark}
In the collapsing case $\varsigma=-1$, note that the first Navier-Stokes correction at $t=(1-\delta^{p'})T_{\rm adv}$ has size
\[
    \delta\eps(t)^2
    \left(1-\frac{t}{T_{\rm adv}}\right)^{-1}
    =
    \cO\left(\delta^{2(1-p')}\right).
\]
The formal  threshold here is therefore $p'=1$. This corresponds to vortices at distance $d\sqrt{\delta}$, which is exactly where one should not be able to ignore viscous effects. 
\end{remark}
\subsection{Functional relationship}
\label{sec:FunctionalRelationship}
We now construct an approximate functional relationship between the Eulerian part of the approximate vorticity and its stream function. This step is crucial in order to apply Arnold's energy method in the nonlinear argument. In contrast with the symmetric configurations considered in \cite{dolce2024long, dolce2026viscousvortexcrystals,zhang2025long}, here finding this functional relation  is more delicate. However, we only need to construct it explicitly up to order $\mathcal{O}(\eps^2)$ (in a sense specified below).

We first introduce the second-order
Eulerian approximation
\begin{equation}
    \Omega_{{\rm app},i}^{(2),E} := G +
    \eps^2 \Omega_{2,i}^{E},\qquad \Psi_{{\rm app},i}^{(2),E} :=\Psi_G+\eps^2 \Psi_{2,i}^{E},
\label{eq:OmegaPsiApp2E}
\end{equation}
where the streamfunction is defined according to Definition \ref{def:ApproxProfilesOrder4}.
The associated interaction-streamfunction is defined by
    \begin{equation}
        \Phi_{\app,i}^{(2),E}
        :=
        \gamma_i\Psi_{\app,i}^{(2),E}
        +
        \sum_{\ell\neq i}
        \gamma_\ell
        \cT_{\eps,\boldsymbol z}^{i,\ell}
        \Psi_{\ell,\app}^{(2),E}
        -
        \frac{\delta t}{\eps\dhk}
        \partial_t z_i
        \cdot \xi^\perp
        -
        \sum_{\ell\neq i}
        \frac{\gamma_\ell}{2\pi}
        \log\left(
        \frac{|z_i-z_\ell|}{\eps\dhk}
        \right),
    \label{eq:PhiApp2EFunctional}
    \end{equation}
where the last terms are additive constants inserted to remove the zeroth-order terms in the expansions of the translated stream functions.
Recall that the Lamb--Oseen vortex satisfies the exact functional relationship in \cref{def:functionalLO}.
As  in \cite{dolce2024long}, we introduce the following definition.
\begin{definition}
    We say that a smooth function $F:(0, + \infty) \to \R$ belongs to the class $\mathsf{K}$ if $F(G) \in \cS_*(\R^2)$.
\end{definition}
We will use the following Lemma, whose proof can be found in \cite{dolce2024long}.
\begin{lemma}
\label{lem:KclassFunctional}
Let \(H\in\cS_*(\R^2)\). There exists \(F\in\mathsf K\) such that
\(F(G)=H\) if and only if \(H\) is radial. Moreover, if \(F\in\mathsf K\),
then \(F^{(m)}(G)G^m\in\cS_*(\R^2)\) for every \(m\in\mathbb N\).
\end{lemma}
For a smooth function $f = f(\eps, s,t)$ we denote the Taylor polynomial up to order $M$ by
\begin{equation}
    \Pi_M f= \sum_{k=0}^M \frac{\eps^k}{k!} \frac{\textnormal{d}^k f }{\textnormal{d}\eps^k} \Bigg |_{\eps = 0}.
\end{equation}
Note that the Taylor polynomial will always be computed at the Gaussian $G$. This is important because $\Omega_{{\rm app},i}^{(2),E}$ is not granted to be always positive on the whole domain as required by the definition of the class $\mathsf{K}$. To measure the accuracy of the functional relationship, it is convenient to define 
    \begin{equation}
        \Theta_i := \Phi_{{\rm app},i}^{(2),E} + F_{i}(\Omega_{{\rm app},i}^{(2),E}),
        \label{eq:Thetai}
    \end{equation}
in the region where $\Omega_{{\rm app},i}^{(2),E}>0$.
This is the quantity that we need to be small on a suitable inner region.
Finally, we have
\begin{proposition}
    \label{prop:functionalRel}
    Let $\Omega_{{\rm app},i}^{(2),E}$, $\Phi_{{\rm app},i}^{(2),E}$ and $\Theta_i$ be defined as in \cref{eq:OmegaPsiApp2E,eq:Thetai}. If $\varsigma=1$ ($\varsigma=-1$), then for $t\geq0$ ($t\in[0,T_{\rm adv})$) and for every $i=1,2,3$ setting 
    \begin{equation}
    \label{def:formFi}
        F_i(s) := \gamma_i F(s),
    \end{equation}
    it holds true that
    \begin{equation}
        \Pi_2 (\Theta_i) = 0.
        \label{eq:FunctionalRelation}
    \end{equation}
      Moreover, there exist $\sigma \in (0,1)$, $N \in \mathbb{N}$ and $C>0$, independent of $\eps$ such that for all $|\xi|\leq 2\eps^{-\sigma}$ one has
    \begin{align}
        \left | \nabla\Theta_i(\xi,t)\right| &\leq C \eps^{3}\left(1+|\xi|^N\right) 
        \label{eq:NablaThetai}\\
        \label{bd:tdetF2}\left|t\de_t\left(F_i'\left( \Omega_{{\rm app},i}^{(2),E}\right)\right)\right|&\leq C(1+|\xi|)^NG(\xi)^{-1}\begin{cases}
            \eps^2 \quad &\text{ if } \varsigma=1,\\
            \eps^2(1-\frac{t}{T_{\rm adv}})^{-1}\quad &\text{ if } \varsigma=-1.
        \end{cases}
    \end{align}
\end{proposition}

\begin{remark}
    Actually, any radial correction to the second order profiles would have still allowed us to provide an approximate functional relationship at second order introducing a time-dependent correction to $F_i$.
    Such a time-dependent functional relationship would then satisfy Proposition \ref{prop:functionalRel} uniformly in time, upon suitable hypotheses on the radial correction.
\end{remark}

\begin{proof}
We first expand $\Phi_{\app,i}^{(2),E}$. By the expansion of the translation
operators in Lemma \ref{lem:TeProperties}, applied to $G$, and by the cancellation
of the order $\eps$ term with the Helmholtz--Kirchhoff transport term, we get
\begin{equation}
\Phi_{\app,i}^{(2),E}
=
\gamma_i\Psi_G
+
\eps^2
\left[
\gamma_i\Psi_{2,i}^{E}
+
h_i
\right]
+
\cO_{\cS_*}(\eps^3),
\label{eq:PhiApp2EExpansion}
\end{equation}
where
\begin{equation}
h_i(\xi,t)
:=
\dhk(t)^2 c_{2,0}
\sum_{\ell\neq i}
\gamma_\ell
\Re\left(
\frac{\xi^2}{(z_i(t)-z_\ell(t))^2}
\right).
\label{eq:defhiFunctional}
\end{equation}
By Lemma \ref{lem:TeProperties} the translated stream functions $\eps^2 \cT_{\eps,\boldsymbol{z}}^{i,\ell}\Psi_{\ell,2}^{E}$ with $\ell\neq i$ in the definition of $\Phi_{\app,i}^{(2),E}$ do not contribute at
order $\eps^2$ since the corresponding vorticity profiles have zero mass
and zero first order moments. By the definition of $\cH_{2,i}$ in \cref{eq:defHin}, we have that 
\begin{equation}
\cH_{2,i}=\{h_i,G\}.
\label{eq:H2hiFunctional}
\end{equation}
Recall the second order Eulerian vorticity profile expansion in \cref{eq:OmegaPsiApp2E}.
Since we look for a functional relationship of the form \cref{def:formFi}, imposing that the second order polynomial of $\Theta_i$ at $G$  vanishes is equivalent to
\begin{equation}
    \gamma_i \Psi_G + \eps^2 \left(
    \gamma_i\Psi_{2,i}^{E} + h_i \right) + \gamma_i F(G) + \eps^2 \gamma_i F'(G)\Omega_{2,i}^{E}= 0.
    \label{eq:FullTylor2Initial}
\end{equation}
Appealing to \cref{def:AH,def:functionalLO}, since $F'(G)(\xi)=A(|\xi|)$ with $A$ in \cref{def:AH}, we end up with
\begin{equation}
     \left(
    \gamma_i\Psi_{2,i}^{E} + h_i  \right)  + \gamma_i A\Omega_{2,i}^{E}= 0.
    \label{eq:FullTylor2}
\end{equation}
By Definition \ref{def:ApproxProfilesOrder4}, recall that\begin{equation}
\gamma_i\Lambda\Omega_{2,i}^{E}
+
\cH_{2,i}
=
0.
\end{equation}
Using $\Lambda f = \{{\sf H} f, G \}$ with \cref{eq:H2hiFunctional}, noting that $\mathcal{P}_0 h_i = 0$ and recalling that both $\Omega_{2,i}^{E}$ and $h_i$
have only angular mode $2$, by Proposition \ref{prop:Lambda} we have
\begin{equation}
\label{eq:iffFunctional}
    \gamma_i\Lambda\Omega_{2,i}^{E}
+
\cH_{2,i}
=
0\quad \iff \quad \gamma_i \mathsf{H} \Omega_{2,i}^E = - h_i  \quad \iff  \quad \gamma_i \Psi_{2,i}^E + \gamma_i A\Omega_{2,i}^E + h_i = 0,
\end{equation}
which shows \cref{eq:FunctionalRelation}.
We now justify the estimate
\cref{eq:NablaThetai}. Since
$\Omega_{2,i}^{E}=\cO_\cZ(1)$, there exist
$C>0$ and $N_0\in\mathbb N$ such that
\begin{equation}
\left|
\frac{
\Omega_{2,i}^{E}
}{G}
\right|
\leq
C(1+|\xi|)^{N_0}.
\label{eq:ProfileRatioFunctional}
\end{equation}
Thus, if $|\xi|\leq 2\eps^{-\sigma}$, then
\begin{equation}
\eps^2
\left|
\frac{
\Omega_{2,i}^{E}
}{G}
\right|
\leq
C\eps^{2-\sigma N_0}.
\end{equation}
Choosing $\sigma>0$ sufficiently small, for instance
$
\sigma<\frac{2}{N_0},
$
we obtain
$\Omega_{\app,i}^{(2),E}>0$ in the region
$|\xi|\leq 2\eps^{-\sigma}$, when $\eps$ is sufficiently small. In this region
Taylor's formula is legitimate, and the remainder in the expansion of
$\Theta_i$, together with its first $\xi$-derivative, is bounded by
$
C\eps^3(1+|\xi|^N).
$
This proves \cref{eq:NablaThetai}.
 It remains to estimate the time derivative $F_i'$ along $\Omega_{\rm{app},i}^{(2),E}$.
Let us introduce 
\begin{equation}
    l_{\varsigma}(t) := \begin{cases}
           1 \quad &\text{ if } \varsigma=1,\\
             \left(1-\frac{t}{T_{\rm adv}}\right)^{-1}\quad &\text{ if } \varsigma=-1.
        \end{cases}
\end{equation}
To prove \eqref{bd:tdetF2}, we start by computing
\begin{align}
    \label{def:tftFi2'}
        t\de_t\left(F_i'\left( \Omega_{{\rm app},i}^{(2),E}\right)\right) & = \gamma_i F''\left(\Omega_{{\rm app},i}^{(2),E}\right) t \de_t \Omega_{{\rm app},i}^{(2),E}.  
    \end{align}
First of all, note that
\begin{equation}
    |t \partial_t(\eps^2)| \lesssim  \eps(t)^2 l_\varsigma(t),
    \label{eq:tdteps2}
\end{equation}
and, by Proposition \ref{prop:ApproximateSol} $(ii)$,
we have
\begin{equation}
    |t \de_t \Omega_{{\rm app},i}^{(2),E}| \leq C \eps^2 l_\varsigma (1+|\xi|)^N G(\xi).
    \label{eq:tdtOmega2E}
\end{equation}
In the region
$|\xi|\leq2\eps^{-\sigma}$ and for $\eps$ sufficiently small,
\begin{equation}
    \frac12G(\xi)
    \leq
    \Omega_{\app,i}^{(2),E}(\xi,t)
    \leq
    \frac32G(\xi).
    \label{eq:XiComparableG}
\end{equation}
Arguing as in \cite[Proposition~4.4]{dolce2024long}, the bound
\cref{eq:XiComparableG} and Lemma \ref{lem:KclassFunctional} imply,
possibly increasing $N$, that
\begin{equation}     \left|F_{}''(\Omega_{\app,i}^{(2),E})\right|
    \leq
    C(1+|\xi|)^NG(\xi)^{-2}.
    \label{eq:FSecondXiBound}
\end{equation}
Combining \cref{def:tftFi2'}, \cref{eq:tdtOmega2E} and  \cref{eq:FSecondXiBound}, we get 
\begin{equation}
    \left|
        t\partial_t
        \left(
            F_i'
            \left(\Omega_{\app,i}^{(2),E}
            \right)
        \right)
    \right|
    \leq
    C\eps^2 l_\varsigma(t)
    (1+|\xi|)^NG(\xi)^{-1}
    \label{eq:dtFiPrimeFinal}
\end{equation}
which is \cref{bd:tdetF2}.
\end{proof}
\begin{remark}
    A rough estimate, carrying on the power of $|\xi|$ lost in the bounds above, suggests that $\sigma< 1/4$. 
\end{remark}

\section{The nonlinear correction}
\label{sec:NonlinearCorrection}

In this section we prove that the approximate solution constructed in
\cref{sec:ApproxSol} remains close to the exact solution of
\cref{eq:OmegaIEq}. We recall that, for each permutation
$(i,j,k)$ of $(1,2,3)$, the  vorticity  profile $\Omega_i$ satisfies
\begin{equation}
    t\partial_t \Omega_i
    - \cL \Omega_i
    + \frac1\delta
    \left\{\Phi_{\epsilon,\Z,i},\Omega_i\right\}
    =0,
    \label{eq:OmegaINonlinearRecall}
\end{equation}
where we recall that  $\Phi_{\eps,\boldsymbol{Z},i}$ is defined in \cref{def:PhiepsZi}.
We write the solution as a perturbation of the fourth order approximate
solution,
\begin{equation}
    \Omega_i=\Omega_{\app,i}+\delta w_i,
    \qquad
    \Psi_i=\Psi_{\app,i}+\delta\varphi_i, \qquad Z_i(t)=z_i(t)+\rho_i(t),
        \label{eq:NonlinearAnsatz}
\end{equation}
where $\varphi_i=\Delta^{-1}w_i$, $Z_i(t)$ is chosen according to Lemma \ref{lem:PositionChoice} which, thanks to the properties of the approximate solution, guarantees that 
\begin{align}
\label{eq:zeromomNL0}M(w_i(t))=m^1_1(w_i(t))=m^1_2(w_i(t))=0, \qquad \text{ for all } t\geq0 \ (t\in[0,T_{\rm adv}]),\quad i=1,2,3.
\end{align}
The goal is to control the correction
\[
    \boldsymbol w=(w_1,w_2,w_3), \qquad \boldsymbol{ \varphi}=(\varphi_1,\varphi_2,\varphi_3),\qquad \boldsymbol{\rho}=(\rho_1,\rho_2,\rho_3)
\]
on the time interval 
\begin{equation}
    0<\frac{t}{T_{\rm adv}}<\begin{dcases}
        \exp\left(\frac{1}{\delta^p}\right) \quad & \text{ when } \varsigma=1 \text{ and } p\in (0,1) \text{ (expanding)},\\
        (1-\delta^{p'}) \quad & \text{ when } \varsigma=-1 \text{ and } p' \in (0,\tfrac15) 
        \text{ (collapsing)}.
    \end{dcases}
    \label{eq:NonlinearTimeScale}
\end{equation} Before introducing the detailed technical setup, let us summarize the main result of this section here.
\begin{theorem}
    Let $\varsigma=1$ ($\varsigma=-1$) and $w_i$ and $\rho_i$ be the respectively solution of \cref{eq:PDEwi} and \cref{eq:PDErhoi} with zero initial data. There exists $C>1$ and sufficiently small constants $c_0,\delta_0, p_0\in (0,1/2)$ ($p_0'\in (0,1/5)$) such that for any $\Gamma, \ \nu, \ d$ the following holds true: for all $\delta=\nu/\Gamma \leq \delta_0$,  let $T_{\rm adv},\ \eps(t)$ be defined as in   \cref{eq:QtyDef}. Then for all $p\in (0,p_0)$ ($p'\in (0,p_0')$) and $t$ as in \cref{eq:NonlinearTimeScale} one has
    \begin{equation}
        \| \boldsymbol{w}(t) \|_{\Xg} \leq \begin{dcases}
            C \eps(t), &\quad \text{ if } \varsigma=1\\
            C\eps(t)^{1-q(p')} &\quad \text{ if } \varsigma=-1
        \end{dcases} \qquad \frac{|\boldsymbol{\rho}(t)|}{\dhk(t)}\leq \begin{dcases}
            C\delta^{1-2p}\eps(t)^2, &\quad \text{ if } \varsigma=1\\
            C\delta^{1-\tilde{q}(p')}\eps(t)^2 &\quad \text{ if } \varsigma=-1
        \end{dcases}, 
    \end{equation}
    where $q(p'),\tilde{q}(p')\in (0,1)$ are sufficiently small constants depending on $p'$ and $\mathcal{Y}\subset\mathcal{X}_{\eps, i}\subset L^1(e^{|\xi|^{c_0}}d\xi)$  is defined below in \cref{def:Xieps}.
    \label{thm:NonLinControl}
\end{theorem}
The proof of the theorem above builds upon the Arnold-type functional framework used in \cite{GaSring,dolce2024long,dolce2026viscousvortexcrystals,zhang2025long,donati2025fast,nguyen2026inviscid}. Here we have to face new challenges related to the delicate control of the self-similar motion of the positions and their stability properties. In the previous works \cite{GaSring,dolce2024long,dolce2026viscousvortexcrystals,zhang2025long} symmetries were crucially exploited to control the nonlinear correction $\boldsymbol{\rho}$. 

To explain the main ideas and difficulties, we have to first carefully rewrite the equations and isolate the leading order contributions.

\subsection{Derivation of the equations}
Recall the vector notation introduced in \cref{sec:notation}. 
To derive the equations for $\boldsymbol{w},$ we recall the streamfunction introduced in \cref{def:PhiepsZi} and we write that
\begin{align}
    \Phi_{\eps,\boldsymbol{Z},i}&=\Phi_{\epsilon,i}(\boldsymbol{\Psi}_{\rm app}+\delta \boldsymbol{\varphi},\boldsymbol{z}+\boldsymbol{\rho})\\
    &=\underbrace{\Phi_{\epsilon,i}(\boldsymbol{\Psi}_{\rm app},\boldsymbol{z})}_{\Phi_{{\rm app},i}}
    +\delta\Big(
        \gamma_i\varphi_i+
        \sum_{j\neq i}\gamma_j\cT_{\eps,\boldsymbol{z}}^{i,j}\varphi_j\Big)
    -\frac{\delta t}{\eps\dhk}\partial_t\rho_i\cdot\xi^\perp
    \notag\\
    &\quad
    +\sum_{j\neq i}\gamma_j
\big(\cT_{\eps,\boldsymbol{z}+\boldsymbol{\rho}}^{i,j}-\cT_{\eps,\boldsymbol{z}}^{i,j}\big)
    (\Psi_{\app,j}+\delta \varphi_j).
    \label{eq:PhiNLDecomposition}
\end{align}
Then, we isolate the Eulerian approximate interaction-streamfunction up to order $2$ modulo constants and we write, 
\begin{align}
 \Phi_{{\rm app},i}
    &=
    \Phi^{(2),E}_{\app,i}
    +\delta\Phi^{(2),NS}_{\app,i}
    +\eps^3\widetilde\Phi_{\app,i},
    \label{eq:PhiAppSplit}
\end{align}
where
\begin{align}
   \label{def:tildePhiapp} \widetilde\Phi_{\app,i}
    &:=
    \gamma_i(\Psi_{3,i}+\eps\Psi_{4,i})
    +\sum_{j\neq i}\gamma_j\cT_{\eps,\boldsymbol{z}}^{i,j}
        (\Psi_{3,j}+\eps\Psi_{4,j})=\mathcal{O}_{\cS_*}(1),
\end{align}
where the last identity readily follows by Definition \ref{def:ApproxProfilesOrder4} and Proposition \ref{prop:ApproximateSol}.
Hence, substituting \cref{eq:NonlinearAnsatz} and
\cref{eq:PhiNLDecomposition} into \cref{eq:OmegaINonlinearRecall}, recalling the definition of the remainder $\mathcal{R}_i$ in \cref{def:Reminder}, we obtain
\begin{align}
    &t\partial_t w_i-\cL w_i
    +\frac1\delta
    \Big(
        \Lambda^i_{\app}w_i
        +\sum_{j\neq i}\Upsilon^{i,j}_{\app}w_j
    \Big)
    =
    -\frac1{\delta^2}
    \big(\cR_i+\mathsf{T}_i(\boldsymbol{\rho})\big)
    +\mathsf{L}^i(\boldsymbol{w},\boldsymbol{\rho})
    +\mathsf{N}^i(\boldsymbol{w},\boldsymbol{\rho}).
    \label{eq:PDEwi}\\
\label{eq:PDErhoi}&t\de_t\rho_i=\mathsf{P}_{\eps}^i(\boldsymbol{\Psi}_{\rm app}+\delta\boldsymbol{\varphi},\boldsymbol{\Omega}_{\rm app}+\delta\boldsymbol{w},\boldsymbol{z}+\boldsymbol{\rho})-t\de_tz_i\\
    &w_i|_{t=0}=0, \qquad \rho_i|_{t=0}=(0,0),
\end{align}
where we start with zero initial data for $w_i$ thanks to the uniqueness result in \cite{GaGa2005} and we define:
\begin{itemize}
    \item the linear operators due to the linearization at the Eulerian approximate solution:
    \begin{equation}
        \Lambda^i_{\app}f
        :=
        \left\{\Phi^{(2),E}_{\app,i},f\right\}
        +
        \left\{\gamma_i\Delta^{-1}f,\Omega_{\app,i}^{(2),E}\right\},
        \label{eq:LambdaIAppCorrected}
    \end{equation}
    and, for $j\neq i$,
    \begin{equation}
        \Upsilon^{i,j}_{\app}f
        :=
        \gamma_j
        \left\{
            \cT_{\eps,\boldsymbol{z}}^{i,j}\Delta^{-1}f,
            \Omega_{\app,i}^{(2),E}
        \right\}.
        \label{eq:UpsilonIJApp}
    \end{equation}

    \item the \emph{translation} forcing terms due to the correction of the position:
    \begin{align}
    \label{def:translation_forcing}
        \mathsf T_i(\boldsymbol{\rho})
        &:= -
         \frac{\delta t}{\eps\dhk}\left\{
        \partial_t\rho_i\cdot\xi^\perp,
            \Omega_{\app,i}
        \right\}+\sum_{j\neq i}
        \gamma_j
        \left\{
\big(\cT_{\eps,\boldsymbol{z}+\boldsymbol{\rho}}^{i,j}
            -\cT_{\eps,\boldsymbol{z}}^{i,j}\big)
            \Delta^{-1}\Omega_{\app,j},
            \Omega_{\app,i}
        \right\}        \label{eq:TijkDefinitionCorrected}
    \end{align}

    \item the linear error terms for $\boldsymbol w$:
    \begin{align}
        \mathsf{L}^i(\boldsymbol{w},\boldsymbol{\rho})
        &:=
        \mathsf L_{1}^i(\boldsymbol{w})+\mathsf L_{2}^i(\boldsymbol{w})+\mathsf L_{3}^i(\boldsymbol{w},\boldsymbol{\rho})+\mathsf L_{4}^i(\boldsymbol{w},\boldsymbol{\rho}),
        \label{eq:LijkDefinition}
    \end{align}
    where
    \begin{align}
        &\mathsf L_{1}^i(\boldsymbol{w})
        :=
        \left\{
            w_i,\Phi^{(2),NS}_{\app,i}
            +\delta^{-1}\eps^3\widetilde\Phi_{\app,i}
        \right\},
        \label{eq:L1Definition}
        \\
        &\mathsf L_{2}^i(\boldsymbol{w})
        :=
        \Big\{\Omega_{\app,i}^{(2),NS}
            +\delta^{-1}\eps^3(\Omega_{3,i}+\eps\Omega_{4,i}),
            \gamma_i\Delta^{-1}w_i
            +\sum_{j\neq i}\gamma_j\cT_{\eps,\boldsymbol{z}}^{i,j}\Delta^{-1}w_j
        \Big\},
        \label{eq:L2Definition}
        \\
        &\mathsf L_{3}^i(\boldsymbol{w},\boldsymbol{\rho})
        :=
        \delta^{-1}\sum_{j\neq i}\gamma_j
        \left\{
\big(\cT_{\eps,\boldsymbol{z}}^{i,j}-\cT_{\eps,\boldsymbol{z}+\boldsymbol{\rho}}^{i,j}\big)\Delta^{-1}w_j,
            \Omega_{\app,i}
        \right\},
         \label{eq:L3Definition}
        \\
        &\mathsf L_{4}^i(\boldsymbol{w},\boldsymbol{\rho})
        :=
        \delta^{-1}\sum_{j\neq i}\gamma_j
        \left\{
\big(\cT_{\eps,\boldsymbol{z}}^{i,j}-\cT_{\eps,\boldsymbol{z}+\boldsymbol{\rho}}^{i,j}\big)\Delta^{-1}\Omega_{\app,j},
            w_i
        \right\};
        \label{eq:L4Definition}
    \end{align}

        \item the nonlinear error terms:
    \begin{align}
        \mathsf{N}^i(\boldsymbol{w}, \boldsymbol{\rho})
        &:=
\mathsf{N}_{1}^i(\boldsymbol{w},\boldsymbol{\rho})+\mathsf{N}_{2}^i(\boldsymbol{w},\boldsymbol{\rho}),
        \label{eq:NijkDefinition}
    \end{align}
    where
    \begin{align}
        &\mathsf N_{1}^i(\boldsymbol{w},\boldsymbol{\rho})
         :=
        \Big\{
            w_i,\gamma_i\varphi_i
            +\sum_{j\neq i}\gamma_j\cT_{\eps,\boldsymbol{z}+\boldsymbol{\rho}}^{i,j}\varphi_j
        \Big\},
        \label{eq:N1Definition}
        \\
        &\mathsf N_{2}^i(\boldsymbol{w},\boldsymbol{\rho})
         :=
        \frac{ t}{\eps\dhk}
        \left\{
        \partial_t\rho_i\cdot\xi^\perp,
            w_i
        \right\},
        \label{eq:N2Definition}
    \end{align}
    \item the position operator 
    \begin{equation}
\mathsf{P}_{\eps}^i(\boldsymbol{f},\boldsymbol{g},\boldsymbol{q}):=\frac{\dhk\eps}{\delta}\sum_{j\neq i}\gamma_j\big\langle \nabla^\perp \big(\mathcal{T}_{\eps,\boldsymbol{q}}^{i,j}\, f_j\big),g_i\big\rangle
        \label{eq:PpositionOp}
    \end{equation}
\end{itemize}
It is useful to rewrite \cref{eq:PDEwi,eq:PDErhoi} in vector form, for which we denote
\begin{equation}
    \mathcal J_{\app}
    :=
    \begin{pmatrix}
        \Lambda^1_{\app} & \Upsilon^{12}_{\app} & \Upsilon^{13}_{\app} \\
        \Upsilon^{21}_{\app} & \Lambda^2_{\app} & \Upsilon^{23}_{\app} \\
        \Upsilon^{31}_{\app} & \Upsilon^{32}_{\app} & \Lambda^3_{\app}
    \end{pmatrix}.
    \label{eq:JappDefinition}
\end{equation}
Then, with obvious vectorial identifications for the other terms, we rewrite \cref{eq:PDEwi,eq:PDErhoi} as
\begin{align}
\label{eq:wvec}&t\partial_t\boldsymbol w
    -\cL \boldsymbol w
    +\frac1\delta\mathcal J_{\app}\boldsymbol w
    =
    -\frac1{\delta^2}
    \big(
        \boldsymbol{\cR}
        +\boldsymbol{\mathsf T}(\boldsymbol{\rho})
    \big)
    +
    \boldsymbol{\mathsf L}(\boldsymbol w,\boldsymbol{\rho})
    +
    \boldsymbol{\mathsf N}(\boldsymbol w,\boldsymbol{\rho}),
    \\
\label{eq:rhovec}&t\partial_t\boldsymbol{\rho}=\boldsymbol{\mathsf{P}}_{\eps}(\boldsymbol{\Psi}_{\rm app}+\delta\boldsymbol{\varphi},\boldsymbol{\Omega}_{\rm app}+\delta\boldsymbol{w},\boldsymbol{z}+\boldsymbol{\rho})-t\de_t\boldsymbol{z},\\
&\boldsymbol{w}|_{t=0}=0, \quad \boldsymbol{\rho}|_{t=0}=0.
\end{align}
The zero first order moment condition in \cref{eq:zeromomNL0} (ensured by the choice of $\rho$) reads as
\begin{equation}
\label{eq:momwvec}
M(\boldsymbol{w}(t))=\mrm^1(\boldsymbol{w}(t))=0, \quad \text{ for all }t\geq0.
\end{equation}
\begin{remark}
    From now on, we focus our attention to the expanding case $\varsigma=1$ and in Section \ref{sec:collapsing} we explain which changes need to be done to cover the collapsing case as well. In fact, the main difference lies on different bounds when integrating in time, which in the expanding case lead to a logarithmic loss $\log(1+\frac{t}{T_{\rm adv}})(1+\log(1+\frac{t}{T_{\rm adv}}))$ whereas in the collapsing one to an algebraic loss $(1-\frac{t}{T_{\rm adv}})^{-1}$. This changes slightly the bootstrap assumptions but essentially all estimates can be translated very easily upon changing the logarithmic loss with the algebraic one.
\end{remark}
To prove Theorem \ref{thm:NonLinControl} in the expanding case $\varsigma=1$,  we rely on a standard bootstrap argument: first of all, by a standard short-time estimate\footnote{In the short-time regime one has $\delta^{-1}\eps^2\ll 1$ and the  arguments are much easier. In fact, one can use the detailed short-time proof in \cite{donati2025fast}, which handles the case of an externally forced point vortex in which the space $\mathcal{X}_{\eps}$ can be replaced with $\mathcal{Y}$. For our problem, in the short-time regime one can safely consider the action of the other vortices as an external forcing and therefore the desired short-time estimates even holds in the stronger space $\mathcal{Y}\subset \mathcal{X}_{\eps}$.} we know that there exists constants $\tilde{C}>0$, $\tilde{c}\in (0,1)$ such that
\begin{align}
     \| \boldsymbol{w} \|_{\Xg} \leq  \widetilde{C} \eps(t) \qquad \text{ for  all} \ t \in [0,\tilde{c}\,T_{\rm adv}].
\end{align}

Hence, by continuity, we can assume that there exists  some time $T\geq \tilde{c}T_{\rm adv}$ such that
\begin{equation}
    \| \boldsymbol{w} \|_{\Xg} \leq 4 \widetilde{C} \eps(t) \qquad \text{ for  all} \ t \in [0,T), \quad T \leq T_{\rm adv} \exp(\delta^{-p}),
    \tag{B}
    \label{eq:BootstrapHP}
\end{equation}
Our goal is to improve the factor $4$ in \cref{eq:BootstrapHP}. 
\subsection{Main  difficulties} \label{sec:maindiff}
 The first key ingredient for the Arnold functional framework is to construct a suitable functional space where $\mathcal{J}_{\rm app}$ is approximately skew-adjoint. In our case,  we can allow for errors up to $\delta^{-1}\eps^3$. In particular, at least in the region where we have the functional relationship between $\boldsymbol{\Omega}^{(2),E}_{\rm app}$ and $\boldsymbol{\Phi}^{(2),E}_{\rm app}$, defining the self-adjoint operator
\begin{equation}
    \cM_{\rm app} = \begin{pmatrix}
    \vspace{4pt}
        \gamma_1(F_1'(\Omega_{{\rm app}, 1}^{(2),E}) + \gamma_1 \Delta^{-1})  & \gamma_1\gamma_2 \cT_{\eps,\boldsymbol{z}}^{1,2} \Delta^{-1} & \gamma_1\gamma_3 \cT_{\eps,\boldsymbol{z}}^{1,3} \Delta^{-1} \\
            \vspace{4pt}
\gamma_2\gamma_1 \cT_{\eps,\boldsymbol{z}}^{2,1} \Delta^{-1}& \gamma_2(F_2'(\Omega_{{\rm app}, 2}^{(2),E}) + \gamma_2 \Delta^{-1}) & \gamma_2\gamma_3 \cT_{\eps,\boldsymbol{z}}^{2,3} \Delta^{-1}\\
    \vspace{4pt}
\gamma_3\gamma_1 \cT_{\eps,\boldsymbol{z}}^{3,1} \Delta^{-1}& \gamma_3\gamma_2 \cT_{\eps,\boldsymbol{z}}^{3,2} \Delta^{-1} & \gamma_3(F'_3(\Omega_{{\rm app}, 3}^{(2),E}) + \gamma_3 \Delta^{-1})
    \end{pmatrix}
    \label{eq:Mapp}
\end{equation}
 one can informally expect that        \begin{equation}
        \delta^{-1}\langle \cJ_{\rm app} (\boldsymbol{f}), \cM_{\rm app}(\boldsymbol{f}) \rangle \approx \cO(\delta^{-1}\epsilon^3)\lesssim \sqrt{\delta}.
    \end{equation}
    Moreover,  by suitably modifying the operator $\mathcal{M}_{\rm app}$ in the region where the functional relationship cannot be justified, we can construct  a good coercive energy functional thanks to the first-order moment constraints \cref{eq:momwvec}. Overall, the control of the left-hand side of \cref{eq:wvec} is analogous to \cite{dolce2024long,dolce2026viscousvortexcrystals,zhang2025long} and  therefore we do not comment  more on this part.
\subsubsection{\textbf{Correction to the position in absence of symmetries}}
    The first new key difficulty is instead  related to the nonlinear correction of the position $\boldsymbol{\rho}.$ Indeed, the analogues of the translation forcing terms in \cref{def:translation_forcing} were controlled by crucially exploiting the symmetries in the previous problems \cite{dolce2024long,GaSring,dolce2026viscousvortexcrystals,zhang2025long}. Moreover, in the cases studied in \cite{dolce2026viscousvortexcrystals,zhang2025long}, if one chooses the nonlinear correction $\boldsymbol{\rho}$ to fix the moment condition \cref{eq:momwvec} it will encounter technical difficulties that are analogous to what we face here. Namely, the presence of errors involving $\mathcal{T}^{i,j}_{\eps,\boldsymbol{z}+\boldsymbol{\rho}}-\mathcal{T}^{i,j}_{\eps,\boldsymbol{z}}$, whose smallness is directly related to the smallness of $\boldsymbol{\rho}$.  In \cite{dolce2026viscousvortexcrystals,zhang2025long} the authors implement different strategies (both of them heavily relying on the underlying symmetries) that allow them to actually use a very accurate approximate position instead of a fully nonlinear one, meaning that in both cases one can set $\boldsymbol{\rho}=0$ and control with different arguments the first order moments. 
    
    On the other hand, here we introduce a more direct study of the modulation parameter $\boldsymbol{\rho}$ solving \cref{eq:rhovec} and we highlight  how the stability properties of the Helmholtz-Kirchhoff dynamics enter in these viscous problems. Indeed, expanding in $\eps$ and isolating the leading order linear terms as $\eps,\delta\to0$, from  \cref{eq:rhovec} we deduce schematically that
    \begin{align}
\de_t\boldsymbol{\rho}=\mathfrak{L}(t)\boldsymbol{\rho}+\mathcal{O}(t^{-1}\dhk\eps^4)+\text{lower order errors,} 
    \end{align}
    where $\mathfrak{L}(t)$ is the  linearization of Helmholtz-Kirchhoff \cref{eq:HK} at the Gr\"obli triangle \cref{def:zjGrobli}. Denoting with $\mathcal{U}_{\mathfrak{L}}(t,s)$ the solution operator associated to $\mathfrak{L}$, since $\boldsymbol{\rho}(0)=0$, we can expect that
    \[
    \boldsymbol{\rho}(t)=\int_0^t\mathcal{U}_{\mathfrak{L}}(t,s)\mathcal{O}(s^{-1}(\dhk \eps^4)(s))ds+\text{lower order terms}.
    \]
   Then, Aref characterized  the properties of $\mathcal{U}_{\mathfrak{L}}(t,s)$ in \cite{ArefSelfSimilar}. As we recall in Section \ref{sec:selfsimilar_HK}, the self-similar structure of the Gr\"obli triangle allow for a change of variables in which $\mathfrak{L}(t)$ is mapped into a \emph{time-independent} operator possessing \emph{unstable} eigenvalues. This spectral instability is reflected in the bound
    \begin{align}
        \|\mathcal{U}_{\mathfrak{L}}(t,s)\|\lesssim \left(1+\log \left(\frac{\dhk(t)}{\dhk(s)}  \right)\right)\frac{\dhk(t)}{\dhk(s)}.
    \end{align}
    Hence, recalling that $\eps^2\lesssim \delta$ and ignoring the lower order terms, we expect that at leading order 
    \begin{align}
        \frac{|\boldsymbol{\rho}(t)|}{\dhk(t)} \lesssim \log \left(e+\frac{t}{T_{\rm adv}} \right)\int_0^t \frac{\eps^4(s)}{s}ds 
        \lesssim  \delta \eps^2(t) \left(\log\left(e+\frac{t}{T_{\rm adv}}\right)\right)^2.
    \end{align} 
    Thus, at best we can hope to propagate
    \begin{align}
    \label{bd:expectedrho}
     \frac{|\boldsymbol{\rho}(t)|}{\dhk(t)}+\frac{|t\de_t\boldsymbol{\rho}(t)|}{\dhk(t)}\lesssim \delta\eps^2  \left(\log\left(e+\frac{t}{T_{\rm adv}}\right)\right)^2.
\end{align}
    With the bounds \cref{bd:expectedrho} at hand, it remains to understand how we one can improve the bootstrap bounds \cref{eq:BootstrapHP}. Here, the size of the remainder $$\delta^{-2}\boldsymbol{\mathcal{R}}=\mathcal{O}(\delta^{-2}\eps^5+\eps^2)=\mathcal{O}(\eps)$$ 
    is clearly consistent with the desired bound. However, looking at the translation errors applied to the approximate solution in $\boldsymbol{\mathsf{T}}(\boldsymbol{\rho})$, see \cref{def:translation_forcing}, we need to face terms like
    \begin{align}
    \label{eq:mainideas1non}
     &  \delta^{-2}\Big(\delta \frac{|t\de_t\boldsymbol{\rho}|}{\eps \dhk(t)}\Big)=\mathcal{O}\big(\eps(\log(e+t/T_{\rm adv}))^2\big),\\
\label{eq:mainideas2non}&\delta^{-2}{\nabla}(\mathcal{T}^{i,j}_{\eps,\boldsymbol{z}+\boldsymbol{\rho}}-\mathcal{T}^{i,j}_{\eps,\boldsymbol{z}})\Delta^{-1}G=\mathcal{O}\Big(\frac{\delta^{-2}\eps |\rho|}{\dhk}\Big)=\mathcal{O}\big(\eps(\log(e+t/T_{\rm adv}))^2\big),
    \end{align}
 meaning that if we treat these terms naively we encounter a logarithmic loss that can destroy the bootstrap argument. Fortunately,  we observe that the terms containing these losses are linear in $\xi$, and therefore the translation errors in reality looks like
 \begin{align}
     \delta^{-2}\boldsymbol{\mathsf{T}}(\boldsymbol{\rho})\approx C\big(\eps(\log(e+t/T_{\rm adv}))^2\big) \nabla G+ \mathcal{O}\big(\eps^2(\log(e+t/T_{\rm adv}))^2\big).
 \end{align}
 But derivatives of the Lamb-Oseen vortex are in the kernel of $\mathsf{H}$ (recall that $\Lambda=\{\mathsf{H}\cdot,G\}$) and a direct inspection reveals that $\mathsf{H}$ is the $\mathcal{O}(1)$ operator appearing in $\mathcal{M}_{\rm app}$. Thus, within the Arnold's functional framework we can basically ignore derivatives of the Lamb-Oseen profile, and this is a consequence of the translation symmetries in $\mathbb{R}^2$ that are naturally built in the problem. 

For the terms in $\boldsymbol{\mathsf{L}}(\boldsymbol{w})$ we can either exploit the fact that we are taking Poisson brackets with terms of order $\mathcal{O}(\eps^2+\delta^{-1}\eps^3)$ or that we can gain from the smallness of $\boldsymbol{\rho}$ when the operator $\mathcal{T}^{i,j}_{\eps, \boldsymbol{z}}-\mathcal{T}^{i,j}_{\eps, \boldsymbol{z}+\boldsymbol{\rho}}$ is involved. The nonlinear terms are instead standard and easier to handle.
\begin{remark}
        The bound in \cref{bd:expectedrho} for $\boldsymbol{\rho}$ can also be deduced via the simpler geometric type arguments involving the control on the ratios between the sides of the Gr\"obli triangle, see \cite{MarPul94} for instance. However, to control $t\de_t\boldsymbol{\rho}$ one has to anyway extract the linearized operator $\mathfrak{L}$ to exploit cancellations with $t\de_t\boldsymbol{z}$ in \cref{eq:rhovec}. But once the operator $\mathfrak{L}$ has been isolated, we prefer to rely on the previously studied properties of the solution operator. We believe that this argument is more robust and flexible. Moreover, it is worth stressing that one can improve the powers of $\eps$ in \cref{bd:expectedrho} by defining a more accurate approximate position (whereas $\boldsymbol{z}$ is  only taking care of the Lamb-Oseen vortex). This would slightly simplify the treatment of the all the terms where we have to use the smallness of $\boldsymbol{\rho}$ (as $\boldsymbol{\mathsf{T}}(\boldsymbol{\rho})$), but it would slightly complicate the construction of the approximate solution. If one wants to push to larger time-scales to improve the range of $p$, then it is necessary to take a more accurate approximate solution and modulate accordingly the approximate position as well.  For our purposes, we prefer to keep clean the approximate solution and exploit the nice properties of $\mathsf{H}$ instead of correcting the Helmholtz-Kirchhoff position.
\end{remark}

\subsubsection{\textbf{Coercivity in intermediate regions}}
Finally, there is another subtle technical point to handle in the expanding case. What we have discussed before works perfectly in the interior region $|\xi|\leq \eps^{-\sigma_1}$, which is where we can rigorously justify the functional relationship and introduce a weighted $L^2$-space through $F'_i$. In the region $|\xi|\geq \eps^{-\sigma_1}$, the usual strategy put forward by Gallay in \cite{Gallay2011}, and used in \cite{dolce2024long,GaSring,dolce2026viscousvortexcrystals,zhang2025long}, is to first fix a \emph{spatial-plateau}  by using a simple weight like 
\[
\exp(\eps(t)^{-2\sigma_1}/4),  \text{ in the intermediate region } \eps(t)^{-\sigma_1}\lesssim |\xi|\lesssim \eps(t)^{-\sigma_2}\text{ with } 0<\sigma_1<1<\sigma_2.
\]
For $|\xi|\gtrsim \eps^{-\sigma_2}$ a weight of the type $\exp(|\xi|^{\widetilde{\sigma}})$ is more than enough to close the estimates, because if $\sigma_2$ is sufficiently large then one gains the desired smallness from the decay of the velocity and vorticity of the approximate solution.
The plateau in the intermediate region is  useful because, with the standard aspect ratio $\varepsilon(t)=\sqrt{\nu t}/d$, the weight {decreases} monotonically in time and when we apply $-t \de_t$ to the weight we gain a good contribution that is fundamental in proving coercivity properties in this intermediate region. On the other hand, our effective aspect ratio satisfies 
\begin{align}
t\de_t\eps(t)=\frac12\frac{1}{1+\frac{t}{T_{\rm adv}}}\eps(t),
\end{align}
meaning that we lose good properties when $t\gg T_{\rm adv}$. This is because the time-dependent weight defined with the $\eps(t)$ is reaching a plateau also in time, so we are not considering  strong enough norms for very large times.\footnote{Note that in the collapsing case $t\de_t\eps$ explodes as $t\to T_{\rm adv}$ and this is analogous to what happens in previous papers for $t\to \infty$.}
To overcome this obstacle in the expanding case, we add a large logarithmic correction to the weight, namely we will use
\[
\exp\big(\eps(t)^{-2\sigma_1}/4-\sigma_3^2\log(1+\tfrac{t}{T_{\rm adv}})\big)  \text{ in the intermediate region}.
\]
This is well-defined within our time-scale  $t\leq T_{\rm adv} \exp(\delta^{-p})$ if we take $p<\sigma_1$. Now, this weight is useful because from the $-t \de_t $ of this weight we gain the constant $\sigma_3^2,$ which can be chosen large if $\delta$ is sufficiently small. Instead, in \cite{dolce2024long,dolce2026viscousvortexcrystals} one gains the large constant $\eps^{-2\sigma_1}$, which is however not strictly necessary to get the desired coercivity properties (any large constant does the job).

\begin{remark}
    This modified plateau gives the most restrictive technical condition on $p$. Indeed, in the proof we need $\sigma_1$ sufficiently small and we do not quantify this  parameter precisely, even though a more careful track of constants should give a nice lower bound on $\sigma_1$ (related for instance to the $\sigma$ appearing in the region of validity of the functional relationship in Proposition \ref{prop:functionalRel}). However, reaching any $p<1$ might be related to obtaining $\sigma_1$ close to $1$, which is instead a more structural obstacle of the weights used in the proofs of all the results in \cite{dolce2024long,dolce2026viscousvortexcrystals,zhang2025long,GaSring}. Interestingly enough, in the simpler setting of an externally forced viscous point vortex, Gallay and Donati in \cite{donati2025fast} can actually handle an interior region defined all the way up to $|\xi|= \eps^{-1}/C$ with $C\gg 1$, which is probably the sharpest region where one can hope to justify the functional relationship (namely  $\sigma=1$ in Proposition \ref{prop:functionalRel}, in the sense that we restrict to $|\xi|\leq \eps^{-1}/C$). 

    There is of course the other possibility of setting a spatial dependence of the weight in the intermediate region as well. However, the inner region is well adapted to the streamlines geometry of the deformed vortex, meaning that one should be careful in gluing continuously the complicated $F'(\Omega_{{\rm app},i}^{(2),E},t)$ with a simpler radial object for instance. This is a strategy that, combined with the idea of the non-radial weight in \cite{Gallay2011}, can help in bypassing the need of a functional relationship all the way up $\eps^{-1}$ and still cover a nice Gaussian-type localization of the vortices. However, the end result will not be qualitatively different from the statement we can obtain with the simpler spatial-plateau strategy.
\end{remark}
Our task is to now make rigorous the bounds built on the intuitions discussed in this section.
\subsection{The functional setting} Let us introduce the functional setting needed to prove Theorem \ref{thm:NonLinControl}, which is built upon the tools used in \cite{dolce2024long, GaSring,Gallay2011,dolce2026viscousvortexcrystals,zhang2025long}. We recall that we consider the expanding case $\varsigma=1$, but all the results adapt easily to $\varsigma=-1$ where in fact some technicality is easier.
\subsubsection{\textbf{Weights and functional spaces}}
\label{sec:weights}
First, recall the constant $\sigma\in(0,1)$ appearing in Proposition \ref{prop:functionalRel}, related to the region of validity of the functional relationship. We introduce the following parameters to be fixed in the proof
\begin{equation}
    \sigma_1 \in (0,\sigma), \quad \sigma_2>\frac32, \quad \sigma_3>1, \quad {\widetilde{\sigma}}=\frac{\sigma_1}{\sigma_2}.
\end{equation}
Define the time-dependent height
\begin{align}
\label{def:hepst}
    \mathsf{h}_{\eps}(t ):=\frac14\eps(t)^{-2\sigma_1}-\sigma_3^2\log(1+\tfrac{t}{T_{\rm adv}})=\frac14\Big(\frac{\delta}{|\alpha|} \frac{\tfrac{t}{T_{\rm adv}}}{1+\tfrac{t}{T_{\rm adv}}}\Big)^{-\sigma_1}-\sigma_3^2\log(1+\tfrac{t}{T_{\rm adv}})
\end{align}
and observe that, if $\delta$ is sufficiently small we have
\begin{align}
\label{bd:hepst}\frac18\eps(t)^{-2\sigma_1}\leq \mathsf{h}_\eps(t)\leq \frac 14 \eps(t)^{-2\sigma_1}  \text{ for all } t\in [0 ,T_{\rm adv}\exp(\delta^{-p})) \text{ if } p<\sigma_1.
\end{align}
Hence, from now on we assume that $p_0,\delta_0$ are sufficiently small so the condition above holds true in our time-scale, meaning that we can safely  define the three regions
\begin{align}
   \vspace{4pt} \Ii &= \big\{ \xi \in \R^2 : |\xi| \leq 2 \eps^{-\sigma_1},  \  \gamma_i F_i'\big(\Omega_{{\rm app},i}^{(2),E},t\big) < \gamma_i^2\exp(\mathsf{h}_\eps(t)) \big\}, \qquad &\text{(Inner)} \\
   \vspace{4pt} 
    \IIi & = \big\{ \xi \in \R^2 : \xi \not\in \Ii, |\xi| \leq \epsilon^{-\sigma_2} \big(4\eps^{2\sigma_1}\mathsf{h}_\eps\big)^{1/(2{\widetilde{\sigma}})}\big\}, \qquad &\text{(Intermediate)}\\
    \vspace{4pt} 
    \IIIi & = \big\{ \xi \in \R^2 : |\xi| >\epsilon^{-\sigma_2}\big(4\eps^{2\sigma_1}\mathsf{h}_\eps\big)^{1/(2\widetilde{\sigma})} \big\} \qquad &\text{(Outer)}
\end{align}
We work in the weighted $L^2$-based spaces
\begin{equation}
\label{def:Xieps}
    \cX_{\epsilon,i} := \left \{ f \in L^2(\R^2) : \| f\|_{\cX_{\epsilon,i}} = \sqrt{\langle f, A_{\epsilon,i} f \rangle} < \infty \right \},
\end{equation}
where the weight is defined as
\begin{equation}
\label{def:Aieps}
          \vspace{4pt}
  A_{\epsilon,i}=A_{\eps, i}(\xi,t) := \begin{cases}
\gamma_iF_i'(\Omega_{{\rm app},i}^{(2),E},t) & \text{in I}_{\epsilon,i},
        \vspace{4pt}
\\
        \vspace{4pt}
    \gamma_i^2\exp\big(\mathsf{h}_{\eps}(t)\big) & \text{in II}_{\epsilon,i},\\ 
        \gamma_i^2\exp(|\xi|^{2 {\widetilde{\sigma}}}/4) & \text{in III}_{\epsilon,i}.
    \end{cases}
\end{equation}
\begin{figure}[tbp]
    \centering    \includegraphics[width=0.5\textwidth]{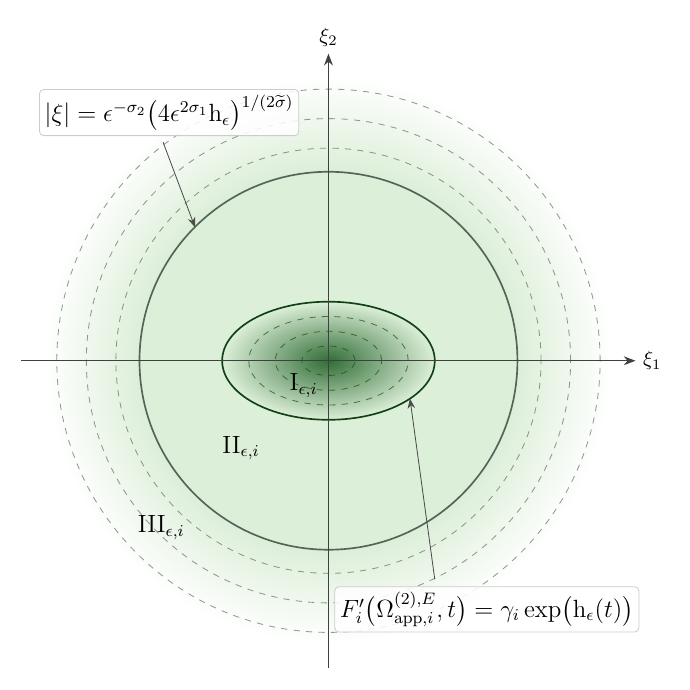} 
    \caption{A schematic picture of the regions introduced above. The intensity of the color represents the size of $A_{\eps, i}^{-1}$ defined in \cref{def:Aieps}, which is constant on the level sets defining the interior region and decreases like a Gaussian towards the boundary of the interior region. Then it is constant (in space) on the intermediate region and enjoys the simpler radial stretched--exponential decay on the outer region.}
    \label{fig:regions}
\end{figure}

The constants $\gamma_i$ in the weight are needed to ensure that the weight in $\mathrm{I}_{\eps, i}$ is always positive since $F_i'(\Omega_{{\rm app},i}^{(2),E})\approx \gamma_iF'(\Omega_0)$ in the region of interest. Following the conventions in Section \ref{sec:notation} we have
\begin{equation}
    \|\boldsymbol f\|_{\Xg}^2
    :=
    \sum_{i=1}^3
    \|f_i\|_{\cX_{\eps, i}}^2.
\end{equation}
As $\epsilon \to 0$, we informally have $A_{\epsilon,i}(\xi) \to \gamma_i^2F'_{}(G) = \gamma_i^2 A(\xi) $, with $A$ in \cref{def:AH}, and therefore we denote
\begin{equation}
    \cX_{0,i} := L^2(A(\xi) d\xi).
\end{equation}

\begin{remark}\label{rem:continuity_weight}
Note that the weight $A_{\eps, i}$ is continuous and piecewise smooth by definition, meaning that we can safely integrate by parts to get  standard identities like 
\begin{align}
\label{eq:intbypart}
    \langle \{f,g\},A_{\eps, i}g\rangle=-\frac12\langle \{f,A_{\eps, i}\},g^2\rangle, \qquad \langle \Delta f, A_{\eps, i}g\rangle=-\langle \nabla f,\nabla(A_{\eps, i}g)\rangle.
\end{align}
Without continuity, we would need to take care of jump-type terms in the first identity, whereas  the second identity also requires the piecewise smoothness of $A_{\eps, i}$ and good properties of $\partial \mathrm{I}_{\eps, i}$. For the latter point,  since $F'_i(\Omega_{{\rm app},i}^{(2),E})$ is a perturbation of $F_{}'(G)$ in the region of interest (thanks to Proposition \ref{prop:ApproximateSol}), as in \cite[Proposition 4.4]{dolce2024long} it holds that $\mathrm{I}_{\eps, i}$ is diffeomorphic to an open disk and 
\begin{align}
    \{|\xi|\leq \eps^{-\sigma_1}/16\}\subset\mathrm{I}_{\eps, i}\subset\big\{|\xi|\leq 16\eps^{-\sigma_1}\big\},
\end{align}
for $\delta$ sufficiently small. In view of these facts, in the sequel we will perform the integrations by parts as in \cref{eq:intbypart} without explicitly mentioning the properties of $A_{\eps, i}$ that makes them true.
\end{remark}
\begin{remark}
    We have defined the intermediate and outer regions in a slightly different way compared to \cite{dolce2024long,dolce2026viscousvortexcrystals}. However,  the property \cref{bd:hepst} guarantees that  our regions are always comparable to the ones appearing in \cite{dolce2024long,dolce2026viscousvortexcrystals}. Hence,  we will quote results in these papers without a proof when the only differences are due to the irrelevant constants in the definitions of the regions.
\end{remark}
We then state the following properties of $\mathcal{X}_{\eps, i}$, whose proofs are the same as   in \cite{dolce2024long}.
\begin{proposition}
\label{prop:weights}
    For any $\epsilon>0$ sufficiently small, the following holds true:
    \begin{enumerate}[label=\roman*),itemsep=4pt]
        \item $\cY \hookrightarrow \cX_{\eps, i} \hookrightarrow L^p$ for any $p \in [1,2]$.
        \item $\gamma_i^2\exp(|\xi|^{2{\widetilde{\sigma}}}/4) \lesssim A_{\eps, i}(\xi) \lesssim \gamma_i^2A(\xi)$ for all $\xi \in \R^2$.
        \item $\| A_{\eps, i} - \gamma_i^2A \|_{C^1(\I)} \lesssim \epsilon^{\eta_2} A$ for any $\eta_2 < 2$.
        \item $|m^k(f)| \lesssim \| f\|_{\X}$ for any $f \in \X$.
    \end{enumerate}
    \label{prop:PropertiesSpaces}
\end{proposition}
We also recall the following functional inequalities proved in \cite[Lemma 4.5 - 4.6]{dolce2024long} (some of which already proved in \cite{GalWay2002}).
\begin{lemma}
Let $f,\nabla f \in \X$ and let $\phi=\Delta^{-1}f$.
Then there exists a constant $C>0$ such that the following estimates hold: 
\begin{equation}
\|(1+|\cdot|)^{-1}\phi\|_{L^q}
+
\|\nabla \phi\|_{L^q}
\le
C \|f\|_{\X}, \qquad \text{ for all }q\in(2,\infty)
\end{equation}
and
\begin{equation}
\|(1+|\cdot|)\nabla\phi\|_{L^\infty}
\le
C \big( \|\nabla f\|_{\X}^{1/2}
+
\|f\|_{\X}^{1/2} \big)
\|f\|_{\X}^{1/2}.
\label{eq:InfNormNablaPhi}
\end{equation}
If in addition
$M[f]=m^1_1[f]=m^1_2[f]=0$,
then 
\begin{equation}
\|(1+|\cdot|)^2 \phi\|_{L^\infty}
+
\|(1+|\cdot|)^2 \nabla\phi\|_{L^q}
\le
C \|f\|_{\X},  \qquad \text{ for all }q\in(2,\infty),
\end{equation}
and
\begin{equation}
\|(1+|\cdot|)^3 \nabla\phi\|_{L^\infty}
\le
C \big( \|\nabla f\|_{\X}^{1/2}
+
\|f\|_{\X}^{1/2} \big)
\|f\|_{\X}^{1/2}.
\end{equation}
\label{lem:SobolevEstimates}
\end{lemma}
As in \cite{dolce2024long}, a direct corollary of the result above is the following.
\begin{corollary}
    \label{cor:EstTZapp}
    Let $f \in \X$ be such that $\textnormal{M}[f] = \textnormal{m}^1_1[f] = \textnormal{m}^1_2[f] = 0$ and denote $\phi = \Delta^{-1} f$. Then, for all $q \in (2, \infty)$, we have
    \begin{equation}
        \| \mathbbm{1}_{\I} \cT_{\eps, \boldsymbol{z}}^{i,j} \phi \|_{L^\infty} \leq C \epsilon^2 \| f\|_{\X}, \qquad  \| \mathbbm{1}_{\I}(1+|\cdot|)^{-1} \nabla\cT_{\eps, \boldsymbol{z}}^{i,j} \phi \|_{L^q} \leq C \epsilon^3 \| f\|_{\X}.
    \end{equation}
\end{corollary}
    \begin{proof}
        Setting $\eta := \xi +\frac{z_{i}(t) - z_{j}(t)}{\eps \dhk}$,  since $\xi \in \Ii$ we have $|\xi| \leq 2 \eps^{-\sigma_1 } \ll \eps^{-1}$. Thus, for $\eps$ small enough, since $|z_i-z_j|\approx \dhk$, it holds true $c_1 \eps^{-1} \leq |\eta| \leq c_2 \eps^{-1}$ for some $c_1,c_2>0.$ Thus, the proof follows by arguing exactly as in the proof of \cite[Corollary 4.7]{dolce2024long}.
    \end{proof}    

\subsubsection{\textbf{The main operator}}
Now, following \cite{dolce2026viscousvortexcrystals}, we make rigorous the intuition built in Section \ref{sec:maindiff} and we define the operator
\begin{equation}
\label{def:Meps}
    \cM_{\eps} = \begin{pmatrix}
    \vspace{4pt}
        \left( A_{\eps,1} + \gamma_1^2 \Delta^{-1}\right) & \gamma_1 \gamma_2 \cT_{\eps,\boldsymbol{z}}^{1,2} \Delta^{-1} & \gamma_1 \gamma_3 \cT_{\eps,\boldsymbol{z}}^{1,3} \Delta^{-1} \\ \vspace{4pt}\gamma_2\gamma_1 \cT_{\eps,\boldsymbol{z}}^{2,1} \Delta^{-1}& \left( A_{\eps,2} + \gamma_2^2 \Delta^{-1}\right) & \gamma_2 \gamma_3 \cT_{\eps,\boldsymbol{z}}^{2,3} \Delta^{-1}\\
        \vspace{4pt}\gamma_3 \gamma_1 \cT_{\eps,\boldsymbol{z}}^{3,1} \Delta^{-1}& \gamma_3 \gamma_2 \cT_{\eps,\boldsymbol{z}}^{3,2} \Delta^{-1} &  \left( A_{\eps,3} + \gamma_3^2 \Delta^{-1}\right)
    \end{pmatrix}.
\end{equation}
Notice that on the inner region $\mathbf{I}_\eps$, $\mathcal{M}_{app} = \mathcal{M}_\eps$. It is also convenient to split further $\cM_\eps$ as follows.
\begin{equation}
    \cM_{\eps} =  \cA_\eps  + \cK_\eps,
    \label{eq:Meps}
\end{equation}
where
\begin{equation}
    \cA_\eps = \begin{pmatrix}
       \left( A_{\eps,1} + \gamma_1^2 \Delta^{-1} \right) & 0 & 0 \\ 0 &  \left( A_{\eps,2} + \gamma_2^2 \Delta^{-1}\right) & 0\\ 0& 0 &  \left( A_{\eps,3} + \gamma_3^2 \Delta^{-1}\right) 
    \end{pmatrix}
\end{equation}
and $\mathcal{K}_\eps$ contains only the off-diagonal elements of $\mathcal{M}_\eps$.
Then, as in \cite[Lemma 5.7]{dolce2026viscousvortexcrystals}, we have the following.
\begin{lemma}
    \label{lem:estimatesMepsKeps}
    The operators $\cM_\eps, \cK_\eps: \Xg \to \Xg$ are self-adjoint with respect to the standard $L^2$-inner product. Moreover, under the hypothesis of Proposition \ref{prop:PropertiesSpaces} and Lemma \ref{lem:SobolevEstimates}, for any $\boldsymbol{g} \in \Xg$ and any $\boldsymbol{f} \in \Xg $ such that $\textnormal{M}(\boldsymbol{f}) = \textnormal{m}_1^1(\boldsymbol{f}) = \textnormal{m}_2^1(\boldsymbol{f}) = 0$, there exists a constant $C>0$ such that
    \begin{align}
        &| \langle \boldsymbol{g}, \cM_{\eps}(\boldsymbol{f}) \rangle| +| \langle \boldsymbol{g}, \cA_{\eps}(\boldsymbol{f}) \rangle| \leq C \| \boldsymbol{f}\|_{\Xg}\| \boldsymbol{g}\|_{\Xg}, \\
        & | \langle \boldsymbol{g}, \cK_{\eps}(\boldsymbol{f}) \rangle| + \eps^{-1}| \langle \boldsymbol{g}, \nabla\cK_{\eps}(\boldsymbol{f}) \rangle| \leq C \eps^2  \| \boldsymbol{f}\|_{\Xg}\| \boldsymbol{g}\|_{\Xg}.
    \end{align}
\end{lemma}
\begin{proof}
    The self-adjointness follows from the properties of the translated
    operators stated in Lemma \ref{lem:TeProperties} and a direct calculation.
    The proof of the two inequalities is similar to \cite[Lemma 5.7]{dolce2026viscousvortexcrystals}. In particular the first one can be proven in the exact same way, so we omit the proof, while we keep the proof of the second one for convenience of the reader.
    For simplicity, in proving the bounds, we consider just one of the terms in $\mathcal{K}_\eps$ since the others are treated similarly.
    By H\"older's inequality, for a fixed $q>2$ we get that
    \begin{align}
        \left|\langle g_i, \cT_{\eps,\boldsymbol{z}}^{i,j} \Delta^{-1} f_j\rangle \right| \lesssim \, \| \mathbbm{1}_{\Ii} g_i\|_{L^1} \| \mathbbm{1}_{\Ii}\cT_{\eps,\boldsymbol{z}}^{i,j} \Delta^{-1} f_j\|_{L^\infty} 
        + \| \mathbbm{1}_{\Ii^c} A_{\eps, i}^{\frac{1}{2}} g_i\| \| \mathbbm{1}_{\Ii^c}A_{\eps, i}^{-\frac{1}{2}}\cT_{\eps,\boldsymbol{z}}^{i,j} \Delta^{-1} f_j\|.
    \end{align}
    Using Lemma \ref{lem:SobolevEstimates}, Corollary \ref{cor:EstTZapp} and the decay properties of $A_{\eps, i}$ we have
    \begin{equation}
        \left| \langle g_i, \cT_{\eps,\boldsymbol{z}}^{i,j} \Delta^{-1} f_j\rangle \right| \lesssim \eps^2 \| \boldsymbol{f}\|_{\Xg}\| \boldsymbol{g}\|_{\Xg}.
        \label{eq:gTijf}
    \end{equation}
    We are left to bound 
    $|\langle \boldsymbol{g}, \nabla\cK_{\eps}(\boldsymbol{f}) \rangle|$. Again, it is enough to consider
    one off-diagonal contribution. Splitting inner and outer region, for $1/p+1/q=1$ with $q>2$ we have
    \begin{align}
        |\langle g_i, \nabla\cT_{\eps,\boldsymbol{z}}^{i,j} \Delta^{-1}f_j \rangle| \lesssim \, &\| \mathbbm{1}_{\Ii}\langle \cdot \rangle g_i\|_{L^p} \| \mathbbm{1}_{\Ii}\langle \cdot \rangle ^{-1} \nabla \cT_{\eps,\boldsymbol{z}}^{i,j} \Delta^{-1} f_j\|_{L^q} \\
        &+ \| \mathbbm{1}_{\Ii^c} A_{\eps, i}^{\frac{1}{2}} g_i\| \| \mathbbm{1}_{\Ii^c}A_{\eps, i}^{-\frac{1}{2}}\nabla\cT_{\eps,\boldsymbol{z}}^{i,j} \Delta^{-1} f_j\|
    \end{align}
    Using again Lemma \ref{lem:SobolevEstimates} and Corollary \ref{cor:EstTZapp} we get
        \begin{align}
        \label{bd:TeijNON}
        |\langle g_i, \nabla\cT_{\eps,\boldsymbol{z}}^{i,j} \Delta^{-1}f_j \rangle| \lesssim \eps^3 \| \boldsymbol{f}\|_{\Xg}\| \boldsymbol{g}\|_{\Xg}
    \end{align}
the desired bound and therefore the proof is concluded.
\end{proof}
\subsubsection{\textbf{Energy and dissipation functionals}}
With $\mathcal{M}_\eps$ defined as in \cref{def:Meps}, we are ready to introduce the energy functional which is related to the Arnold quadratic form, namely
\begin{equation}
    E_\eps(\boldsymbol{w}) = \frac{1}{2} \langle \boldsymbol{w}, \mathcal{M}_\eps(\boldsymbol{w}) \rangle \qquad \boldsymbol{w} \in \Xg.
    \label{eq:EepsW}
\end{equation}
We also introduce the diffusive quadratic form
\begin{equation}
    D_\eps(\boldsymbol{w}) := -\frac{1}{2} 
    \left \langle
    t\de_t\mathcal{A}_\eps\boldsymbol{w}, \boldsymbol{w} \right \rangle - \langle \boldsymbol{\mathcal{L}}(\boldsymbol{w}), \cM_\eps(\boldsymbol{w}) \rangle
    \label{eq:defDeps}
\end{equation}
and the radially symmetric weight
\begin{equation}
    \mathrm{r}_\eps(\xi) = \begin{cases}
        |\xi| & \textnormal{if } |\xi| < \eps^{-\sigma_1}/16, \\ 
        \sigma_3 & \textnormal{if } \eps^{-\sigma_1}/16\leq |\xi| < C_1 \eps^{-\sigma_2}, \\ 
        | \xi|^{\widetilde{\sigma}} & \textnormal{if } |\xi| \geq C_1\eps^{-\sigma_2},
    \end{cases}
    \label{def:WeightPho}
\end{equation}
where $C_1>0$ is a constant implicitly defined in the proof (related to the bounds on $\mathsf{h}_\eps$ in \cref{bd:hepst}).
Then, in view of Lemma \ref{lem:estimatesMepsKeps}, we can argue as in \cite[Proposition 5.11,Proposition 5.13]{dolce2026viscousvortexcrystals} to use the properties proved in \cite{dolce2024long} and obtain the following.
\begin{proposition}
\label{prop:EnDiss}
Let $\varsigma=1$. If $ \delta_0,\sigma_1> 0$ are sufficiently small and $\sigma_3$ sufficiently large, there exists $\kappa_E,\kappa_D, \ C>0$ independent of $\eps$, such that the following holds true:  let $\boldsymbol{w} \in \Xg$ with $\mathrm{r}_\eps \boldsymbol{w} \in \Xg$, $\nabla w_i \in \X$ for $i=1,2,3$ and assume that $M(\boldsymbol{w}) =m_1^1(\boldsymbol{w}) = m_2^1(\boldsymbol{w}) = 0 $. Then 
    \begin{align}
    \label{prop:EepsCoercivity}    &E_{\eps}(\boldsymbol{w})\geq \kappa_E\|\boldsymbol{w}\|_{\boldsymbol{\mathcal{X}}_\eps}^2, \\
    \label{prop:DepsCoercivity}    &D_\eps(\boldsymbol{w}) \geq \kappa_D \left( \| \nabla \boldsymbol{w}\|_{\Xg} ^2+ \| \mathrm{r}_\eps\boldsymbol{w} \|_{\Xg}^2 + \| \boldsymbol{w}\|_{\Xg} ^2\right).
    \end{align}
\end{proposition}
\begin{proof}
The proof follows by the same line of reasoning used in \cite[Proposition 5.11,5.13]{dolce2026viscousvortexcrystals}, which is itself based on reducing the proof to the coercivity bounds proved in detail in \cite{dolce2024long}. The only relevant difference is on how we gain coercivity in the intermediate region. Before detailing this important point, let us  briefly  sketch the key points behind the desired estimates. First use the splitting \cref{eq:Meps} and bound the pieces with $\mathcal{K}_\eps$ thanks to Lemma \ref{lem:estimatesMepsKeps} (for the diffusion functional integrate by part once the terms involving $\mathcal{L}$). 
Then,  to handle the terms involving $\mathcal{A}_\eps$, we observe that we have the same structure as in \cite{dolce2024long}. In particular,  the most technically challenging part  in \cite[Proposition 4.9, 4.13]{dolce2024long} is to properly extract the desired lower bounds from the leading order pieces at $\eps=0$. For the energy, it is enough to split 
\begin{align}
w_i=\mathbbm{1}_{\mathrm{I}_{\eps, i}}w_i+(1-\mathbbm{1}_{\mathrm{I}_{\eps, i}})w_i.
\end{align}
Then, for the piece supported on $\mathrm{I}_{\eps, i} $, one can use property \textit{iii)} in Proposition \ref{prop:weights} and exploit the coercivity properties with $\eps\to 0$ in the space $\mathcal{X}_{0,i}$, which were proved in \cite{GaSarnold}. The other piece is much easier to handle in view of the strong decay properties of $A_{\eps, i}$ outside the inner region.

On the other hand, for the diffusion functional we can follow the Step 1 in the proof of \cite[Proposition 4.13]{dolce2024long} (combined with Lemma \ref{lem:estimatesMepsKeps}) to get that 
\begin{align}
\label{bd:Deps}
    D_{\eps}(\boldsymbol{w})\geq \sum_{i=1}^3\left(-\frac12\langle\mathbbm{1}_{\mathrm{II}_{\eps, i}} t \de_t A_{\eps, i}w_i,w_i\rangle+\mathcal{D}_{\mathcal{L},\eps}[w_i]\right)-C_2\eps\big(\|\boldsymbol{w}\|^2_{\boldsymbol{\mathcal X}_\eps}+\|\nabla\boldsymbol{w}\|^2_{\boldsymbol{\mathcal X}_\eps}\big),
\end{align}
where 
\begin{align}
\label{eq:DLeps}
\mathcal{D}_{\mathcal{L},\eps}[w_i]=\|\nabla w_i\|_{\mathcal{X}_{\eps, i}}^2+\frac14\langle w_i,w_i(\xi\cdot\nabla A_{\eps, i})\rangle+\langle w_i,\nabla w_i\cdot \nabla A_{\eps, i}\rangle-\|w_i\|^2_{L^2}-\frac12 \|w_i\|^2_{\mathcal{X}_{\eps, i}}.
\end{align}
Please note that in the interior region we also rely on the property \cref{bd:tdetF2} to absorb in the error term  contribution from $t\de_t F'_i$ (this holds also in the $\varsigma=-1$ case).

Then, here we can proceed with the  smoother splitting of $w_i$ done in Step 2 of \cite[Proposition 4.13]{dolce2024long}, which  gives 
\[
w_i=w_{{\rm int},i}+w_{{\rm out},i}, \quad \operatorname{supp}(w_{{\rm int},i})\subset \{|\xi|\leq \eps^{-\sigma_1}/16\}, \quad \operatorname{supp}(w_{{\rm out},i})\subset\{|\xi|\geq\eps^{-\sigma_1}/32\}.
\]
For the inner part the whole Step 3 of \cite[Proposition 4.13]{dolce2024long} goes through, where the idea is to exploit the nice coercivity properties available with the $A_{0,i}$ weight proved in \cite{GaSarnold}. For the outer part, the key point  is  to gain from the inner product terms in \cref{eq:DLeps} to absorb the last two terms in \cref{eq:DLeps}. This can be done by either exploiting the properties of the weight $A_{0,i}$ in $\{|\xi|\geq \eps^{-\sigma_1}/32\}\cap \mathrm{I}_{\eps, i}$ or by using the nice exponential radial growth in the outer region. However, in region $\mathrm{II}_{\eps, i}$ one cannot gain anything from the two inner products since the weight is constant. Therefore, one has to absorb differently the contributions from the negative terms in \cref{eq:DLeps}. For the intermediate region $\mathrm{II}_{\eps, i}$, from the definition of $\mathsf{h}_\eps$ 
in \cref{def:hepst} we compute that
\begin{align}
-\frac12\mathbbm{1}_{\mathrm{II}_{\eps, i}} t\de_t A_{\eps, i}=\Big(\frac{\sigma_1}{8}\frac{1}{1+\tfrac{t}{T_{\rm adv}}}\eps(t)^{-2\sigma_1}+\frac{\sigma_3^2}{2}\frac{\tfrac{t}{T_{\rm adv}}}{1+\tfrac{t}{T_{\rm adv}}}\Big)\mathbbm{1}_{\mathrm{II}_{\eps, i}} A_{\eps, i}\geq \frac{\sigma_3^2}{4}\mathbbm{1}_{\mathrm{II}_{\eps, i}} A_{\eps, i},
\end{align}
where the last inequality holds for $\delta$ sufficiently small. Hence  we get 
\begin{align}
    -\frac12\langle\mathbbm{1}_{\mathrm{II}_{\eps, i}} t \de_t A_{\eps, i}w_i,w_i\rangle\geq \frac{\sigma_3^2}{4}\|\mathbbm{1}_{\mathrm{II}_{\eps, i}}w_i\|^2_{\mathcal{X}_{\eps, i}}=\frac14\|\mathbbm{1}_{\mathrm{II}_{\eps, i}}\mathsf{r}_{\eps}w_i\|^2_{\mathcal{X}_{\eps, i}}.
\end{align}
Note that in the setting of \cite{dolce2024long} one gains a good large term of order $\varepsilon(t)^{-2\sigma_1}$ (with $\varepsilon(t)^2=\nu t/d^2$) instead of our fixed constant $\sigma_3^2/4$. However, choosing $\sigma_3$ large enough we can conclude exactly with the same arguments. Indeed, we observe that 
\begin{align}
\|\mathbbm{1}_{\mathrm{II}_{\eps, i}}w_i\|_{\mathcal{X}_{\eps, i}}^2= \frac{1}{\sigma_3^2}\|\mathbbm{1}_{\mathrm{II}_{\eps, i}}\mathsf{r}_{\eps}w_i\|_{\mathcal{X}_{\eps, i}}^2.
\end{align}
For the $L^2$ norm of the outer piece we instead get that
\begin{align}
    \|\mathbbm{1}_{\{|\xi|\geq \eps^{-\sigma_1}/16\}}w_i\|_{L^2}\leq \mathcal{O}(\eps^\infty)\|\mathsf{r}_{\eps}w_i\|_{\mathcal{X}_{\eps, i}}.
\end{align}
Thus, overall, we can follow  all the steps of the proof  in \cite{dolce2024long} and conclude  
\begin{align}
    D_{\eps}(\boldsymbol{w})\geq (\tilde{\kappa}-C\eps^{2\sigma_1})\left( \| \nabla \boldsymbol{w}\|_{\Xg}^2 + \| \mathrm{r}_\eps\boldsymbol{w} \|_{\Xg}^2 + \| \boldsymbol{w}\|_{\Xg}^2 \right)-\frac{1}{\sigma_3^2}\| \boldsymbol{w}\|_{\Xg}^2.
\end{align}
Choosing $\sigma_3$ large enough and $\delta$ sufficiently small, the proof of \cref{prop:DepsCoercivity} is over. Note that in the whole proof in \cite{dolce2024long} one only requires $\sigma_1$ sufficiently small, whereas $\sigma_2$ large is never used to gain any smallness, which is the reason why we do not need any condition on $\sigma_2$ in this result.
\end{proof}
\subsection{Bound on the corrections of the positions}
We now control the correction of the exact vortex centers with respect to the Helmholtz-Kirchhoff trajectories. 
We use the position operator $\mathsf{P}_{\eps}^i$ introduced in \cref{eq:PpositionOp} that allows to rewrite the equation for the vortex centers \cref{eq:dtZi} as
\begin{equation}
t\partial_t Z_i
=
\mathsf P_\eps^i
\left(
\boldsymbol\Psi,
\boldsymbol\Omega,
\boldsymbol Z
\right).
\label{eq:PositionEquationBeps}
\end{equation}
Accounting for the mild instability of the Helmholtz-Kirchhoff system, we are able to prove the following result that holds for slightly more general perturbations of the Lamb-Oseen vortex.
\begin{proposition}
\label{prop:ControlZZhk}
    Let $\varsigma=1$ and $\Omega_i = G+ \delta \widetilde{\Omega}_i$, where $G$ is defined in \cref{def:G} and $\widetilde{\Omega}_i \in L^\infty_t(\X)$. Assume that  $\textnormal{M}(\widetilde{\Omega}_i) = \textnormal{m}^1_1(\widetilde{\Omega}_i) = \textnormal{m}^1_2(\widetilde{\Omega}_i) =0$ and $\sup_{0 \leq t \leq T}\|\widetilde{\Omega}_i\|_{\cX_{\eps,i}}\leq \widetilde{\rm C}$ where $\widetilde{ \rm C}$ is a fixed constant independent from $\delta$. If $\delta$ is sufficiently small, then there exists a constant $C>0$ independent of $\delta$ such that
    \begin{align}
    &\frac{|t \partial_t \rho_i |(t)}{\dhk(t)} \leq C \delta \epsilon^2(t)\,  \left(\log\Big(e+\frac{t}{T_{\rm adv}}\Big)\right)^2 \sup_{0 \leq t \leq T} \|\widetilde{\boldsymbol{\Omega}}\|_{\boldsymbol{\cX}_\eps} + \cO( \eps^\infty),\\
        \label{eq:ControlZZhk}    \vspace{2pt}
        &\frac{|\rho_i|(t)}{\dhk(t)}\leq C \delta^2 \,  \left(\log\Big(e+\frac{t}{T_{\rm adv}}\Big)\right)^2 \sup_{0 \leq t \leq T} \|\widetilde{\boldsymbol{\Omega}}\|_{\boldsymbol{\cX}_\eps} + \cO( \eps^\infty),
    \end{align}
    for $0\leq t \leq T$ with $T\leq T_{\rm adv} \exp(\delta^{-p})$ and $0 < p < 1/2$.
\end{proposition}
\begin{remark}
    In the expanding case, using the fact that $\delta^{-1}\eps^2 $ is bounded, we can write our solution as
    \begin{align}
        \Omega_i  = \Omega_{{\rm app},i} + \delta w_i &= G+ \delta  (w_i + \delta^{-1} \eps^2 (\Omega_{2,i} + \eps  \Omega_{3,i} + \eps^2  \Omega_{4,i})) \\ & 
        =: G + \delta \, \widetilde{\Omega}_i,
    \end{align}
    which has the same form as the one in Proposition \ref{prop:ControlZZhk}. We comment in Section \ref{sec:collapsing} about how to adapt these bounds to the collapsing case $\varsigma=-1$.
\end{remark}
\begin{proof}
It is convenient to denote 
\begin{align}
\label{notation_fij}
    {\sf f}_{ij}:={\sf f}_i-{\sf f}_j \quad \text{for } {\sf f}\in \{Z,z,\rho\}
\end{align}
    Since $Z_i = z_i + \rho_i$, we have that
    \begin{equation}
        t \partial_t \rho_i = t \partial_t Z_i - t \partial_t z_i.
    \end{equation}
    We aim at proving via a bootstrap argument the estimate
    \begin{equation}
        \frac{|\rho_i|(t)}{\dhk(t)} \leq C_0  \left(\delta^2  \left(\log\Big(e+\frac{t}{T_{\rm adv}}\Big)\right)^2\sup_{0 \leq t \leq T} \|\widetilde{\boldsymbol{\Omega}}\|_{\boldsymbol{\cX}_\eps} + \eps(t)^{N_0} \right)
        \tag{B$_\rho$}
        \label{eq:BootstrapHyp},
    \end{equation}
    where $N_0>10$ is an arbitrary fixed number (used to replace more precisely the $\cO(\eps^\infty)$).
    Let \(t_0\) be the supremum of the times \(T>0\) such that
    \cref{eq:BootstrapHyp} holds for every \(t\in[0,T]\).
    Since \(Z_i(0)=z_i(0)\), by standard arguments  we can assume that \(t_0>0\).
    
    Note that, thanks to the bootstrap hypothesis, taking $\delta$ small enough we know that for all \(i\neq j\) and
all \(t\in[0,t_0]\)
    \begin{equation}
        \frac c2 \dhk(t) \leq|Z_{ij}(t)| \leq 2 \dhk(t),
        \label{eq:DistanceComparable}
    \end{equation}
    where $c \in (0,1)$ is a positive constant that depends only on the initial normalized triangle.
    Using \cref{eq:PositionEquationBeps},
    we rewrite
        
    \begin{equation}
        t\partial_t\rho_i
        =
        \left[
        \mathsf{P}_{\eps}^i
        \left(
        \boldsymbol\Psi_G,
        \boldsymbol G,
        \boldsymbol z+\boldsymbol\rho
        \right)
        -
        \mathsf{P}_{\eps}^i
        \left(
        \boldsymbol\Psi_G,
        \boldsymbol G,
        \boldsymbol z
        \right)
        \right]
        +
        \mathcal I_i,
        \label{eq:rhoEquationSplit}
        \end{equation}
        where the error term is defined by
        \begin{equation}
            \mathcal I_i
            := \mathcal I_{1,i}+ \mathcal I_{2,i}
            \label{eq:IiDefinition}
        \end{equation}
        with
        \begin{equation}
        \mathcal I_{1,i} := \mathsf{P}_{\eps}^i
        \left(
        \boldsymbol\Psi,
        \boldsymbol\Omega,
        \boldsymbol z+\boldsymbol\rho
        \right)
        -
        \mathsf{P}_{\eps}^i
        \left(
        \boldsymbol\Psi_G,
        \boldsymbol G,
        \boldsymbol z+\boldsymbol\rho
        \right),
        \qquad 
        \mathcal I_{i ,2} :=
       \mathsf{P}_{\eps}^i
        \left(
        \boldsymbol\Psi_G,
        \boldsymbol G,
        \boldsymbol z
        \right)
        -
        t\partial_t z_i .
        \end{equation}
    For the $\mathcal{I}_{2,i}$ term, combining the definition of the position operator in \cref{eq:PpositionOp} with the expansions of the translation operators in Lemma \ref{lem:TeProperties} evaluated at the Lamb-Oseen vortex, we get
    \begin{equation}
        | \mathcal{I}_{2,i}| \leq \dhk \cO(\eps^\infty).
        \label{eq:Ii2}
    \end{equation}
On the other hand, for $\mathcal{I}_{1,i}$ we aim at applying Lemma \ref{lem:estimatesMepsKeps}, which is effectively replacing the Taylor expansion of integral quantities in Lemma \ref{lem:TeProperties}.
    But first, since $\mathsf{P}_{\eps}^i$ is bilinear in the first two components once the position is fixed, expanding the vorticity profile and the corresponding streamfunctions, we note that 
    \begin{equation}
    \label{def:splitIi1}
        \mathcal{I}_{1,i} = \delta \left[\mathsf{P}_{\eps}^i
        \left(
        \boldsymbol{\widetilde{\Psi}},
        \boldsymbol G,
        \boldsymbol z+\boldsymbol\rho
        \right)+\mathsf{P}_{\eps}^i
        \left(
        \boldsymbol\Psi_G,
        \boldsymbol{\widetilde{\Omega}},
        \boldsymbol z+\boldsymbol\rho
        \right)  + \delta \mathsf{P}_{\eps}^i
        \left(
        \boldsymbol{\widetilde{\Psi}},
        \boldsymbol{\widetilde{\Omega}},
        \boldsymbol z+\boldsymbol\rho
        \right) \right].
    \end{equation}
    Then, considering for instance the first term, by the definition \cref{eq:PpositionOp} and the zero first order moment conditions for $\widetilde{\Omega},$ we can argue as done to prove \cref{bd:TeijNON} in Lemma \ref{lem:estimatesMepsKeps} and get  that 
    \begin{align}
       \delta\left|\mathsf{P}_{\eps}^i
        \left(
    \boldsymbol{\widetilde{\Psi}},
        \boldsymbol G,
        \boldsymbol z+\boldsymbol\rho
        \right)\right|\leq \dhk\eps\sum_{j\neq i}\left|\big\langle \nabla^\perp \big(\mathcal{T}_{\eps,\boldsymbol{\boldsymbol{z}+ \boldsymbol{\rho}}}^{i,j}\, \widetilde{\Psi}_j\big),G\big\rangle\right|\lesssim \dhk \eps^4\sup_{0 \leq t \leq T}{ \|\boldsymbol{\widetilde{\Omega}}\|_{\cX_\eps}}.
    \end{align}
    For the last term in \cref{def:splitIi1} we can argue analogously whereas for the term in the middle we could even exploit Lemma \ref{lem:TeProperties}. Overall, we conclude that
    \begin{equation}
        |\mathcal I_{1,i}| \leq C_I \dhk \eps^4 \sup_{0 \leq t \leq T}{ \|\boldsymbol{\widetilde{\Omega}}\|_{\cX_\eps}}, 
        \label{eq:Ii1}
    \end{equation}
    where we note that $C_I>0$ does not depend on the $C_0$ in \cref{eq:BootstrapHyp} (we only need \cref{eq:DistanceComparable}).
    Therefore, combining \cref{eq:Ii2} and \cref{eq:Ii1}, we have
    \begin{equation}
        |\mathcal{I}_i| \leq C_I  \dhk \eps^4 \sup_{0 \leq t \leq T} \| \boldsymbol{\widetilde{\Omega}}\|_{\boldsymbol{\cX}_\eps}+  \dhk \cO(\eps^\infty) .
        \label{eq:sizeIterm}
    \end{equation}
    Going back to \cref{eq:rhoEquationSplit} we now linearize $
\mathsf{P}_{\eps}^i
\left(
\boldsymbol\Psi_G,
\boldsymbol G,
\boldsymbol z+\boldsymbol\rho
\right)
-
\mathsf{P}_{\eps}^i
\left(
\boldsymbol\Psi_G,
\boldsymbol G,
\boldsymbol z
\right)$ around $\boldsymbol{z}$.
Indeed, by the expansion of the translation operators in Lemma \ref{lem:TeProperties}, we have 
\begin{equation}
    \mathsf{P}_{\eps}^i
    \left(
    \boldsymbol\Psi_G,
    \boldsymbol G,
    \boldsymbol z+\boldsymbol\rho
    \right)
    -
    \mathsf{P}_{\eps}^i
    \left(
    \boldsymbol\Psi_G,
    \boldsymbol G,
    \boldsymbol z
    \right)
    =
    \frac{(\eps \dhk)^2}{\delta}
    \frac{\rm i}{2\pi}
    \sum_{j\neq i}
    \gamma_j
    \left(
    \frac{1}{\overline Z_{ij}}
    -
    \frac{1}{\overline z_{ij}}
    \right)
    +
    \mathcal O(\dhk\eps^\infty).
\label{eq:BepsPureDifference}
\end{equation}
Since $ Z_{ij}=z_{ij}+\rho_{ij}$, 
using the identity
\begin{equation}
    \frac{1}{\overline Z_{ij}}
    -
    \frac{1}{\overline z_{ij}}
    =
    -\frac{\overline\rho_{ij}}{\overline z_{ij}^{2}}
    +
    \frac{\overline\rho_{ij}^{2}}
    {
    \overline z_{ij}^{2}
    \overline Z_{ij}
    } ,
    \label{eq:ExactHKExpansion}
\end{equation}
we can rewrite \cref{eq:rhoEquationSplit} as follows,
\begin{equation}
    t\partial_t\rho_i
    =
    \frac{(\eps \dhk)^2}{\delta}
    \left[
    -\frac{\rm i}{2\pi}
    \sum_{j\neq i}
    \gamma_j
    \frac{\overline\rho_{ij}}{\overline z_{ij}^{2}}
    +
    \mathcal N_i(\boldsymbol\rho)
    \right]
    +
    \mathcal I_i + \dhk\cO( \eps^\infty),
    \label{eq:rhoLinearizedNonlinearError}
\end{equation}
where
\begin{equation}
    \mathcal N_i(\boldsymbol\rho)
    :=
    \frac{\rm i}{2\pi}
    \sum_{j\neq i}
    \gamma_j
    \frac{\overline\rho_{ij}^{2}}
    {
    \overline z_{ij}^{2}
    \overline Z_{ij}
    } .
    \label{eq:NiDefinition}
\end{equation}
Combining the bootstrap assumption \cref{eq:BootstrapHyp} with \cref{eq:DistanceComparable}, we obtain
\begin{equation}
    \left|\mathcal N_i(\boldsymbol\rho)\right|
    \lesssim 
    \sum_{j\neq i}
    \frac{|\rho_{ij}|^2}
    {|z_{ij}|^2 |Z_{ij}|}
    \leq C_0^2 
    \frac{\delta^4}{\dhk}
    \left(\log\Big(e+\frac{t}{T_{\rm adv}}\Big)\right)^4
    \sup_{0 \leq t \leq T}
    \|\widetilde{\boldsymbol\Omega}\|_{\boldsymbol \cX_\eps}^2 +  \frac{1}{\dhk}\cO(\eps^\infty).
    \label{eq:boundN}
\end{equation}
Recalling that $\Gamma=\frac{(\dhk\epsilon)^2}{\delta t}$ and $\gamma_i=\Gamma_i/\Gamma$,
we can rewrite \cref{eq:rhoLinearizedNonlinearError} as
\begin{equation}
    \partial_t\rho_i
    =
    -\frac{\rm i}{2\pi}
    \sum_{j\neq i}
    \Gamma_j
    \frac{\overline\rho_{ij}}{\overline z_{ij}^{2}}
    +
    \Gamma\mathcal N_i(\boldsymbol\rho)
    +
    \frac{1}{t}\big(\underbrace{\mathcal{I}_i+\left(\dhk+\frac{\Gamma t}{\dhk}\right)\cO(\eps^\infty)}_{:=\widetilde{\mathcal I}_i}\big), \qquad
    \label{eq:rhoLinearized}
\end{equation}
The first term on the right hand side is precisely the linearized Helmholtz--Kirchhoff system around
$\boldsymbol z(t)$, as
detailed in \cref{sec:selfsimilar_HK}.
Introducing the enlarged variables $\widetilde{\boldsymbol\rho}:=\left(
        \boldsymbol\rho,
        \overline{\boldsymbol\rho}
    \right)$,
the previous system \cref{eq:rhoLinearized} can be written in the form
\begin{equation}
    \partial_t \widetilde{\boldsymbol\rho}
    =
    \mathfrak{L}(t)\widetilde{\boldsymbol\rho}
    +
    \boldsymbol{\widetilde{\mathcal N}}(t),
    \label{eq:rhoExpandedSystem}
\end{equation}
where \(\mathfrak{L}(t)\) is the linearized Helmholtz--Kirchhoff operator and
\begin{equation}
    \boldsymbol{\widetilde{\mathcal N}}(t)
    :=
    \left(
        \Gamma\boldsymbol{\mathcal N}(\boldsymbol\rho)
        +
        \frac{1}{t}\widetilde{\boldsymbol{\mathcal I}},
        \,
        \Gamma\overline{\boldsymbol{\mathcal N}(\boldsymbol\rho)}
        +
        \frac{1}{t}\overline{\widetilde{\boldsymbol{\mathcal I}}}
    \right).
\end{equation}
Since \(\boldsymbol\rho(0)=0\), Duhamel's formula gives
\begin{equation}
    \widetilde{\boldsymbol\rho}(t)
    =
    \int_0^t
    \mathcal{U}_\mathfrak{L}(t,s)
    \widetilde{\boldsymbol{\mathcal N}}(s)
    \,ds,
    \label{eq:DuhamelRhoExpanded}
\end{equation}
where \(\mathcal{U}_\mathfrak{L}(t,s)\) denotes the solution operator associated with the
homogeneous system $\partial_t f=\mathfrak{L}(t)f$.
As shown in \cref{sec:selfsimilar_HK}, under the change of variables
\begin{equation}
\label{def:sstimetau}
    \tau
    :=
    T_{\rm adv}
    \log\left(1+\frac{t}{T_{\rm adv}}\right),
    \qquad
    g(\tau)
    :=
    \begin{pmatrix}
        e^{-\mathrm{i}\upomega\tau}\boldsymbol f(t(\tau))\\[2mm]
        e^{\mathrm{i}\upomega\tau}\overline{\boldsymbol f(t(\tau))}
    \end{pmatrix},
\end{equation}
for a suitable parameter \(\upomega \in \mathbb R\), the homogeneous system is transformed into
\begin{equation}
    \partial_\tau g = \mathfrak{B}g,
\end{equation}
 where $\mathfrak{B}$ is \emph{time independent}.
Applying Corollary \ref{lem:transfer_matrix} to \cref{eq:rhoExpandedSystem} gives
\begin{equation}
    \left|\mathcal U_\mathfrak{L}(t,s)\right|
    \lesssim \left( 1+ \log \frac{\dhk(t)}{\dhk(s)}\right)
    \frac{\dhk(t)}{\dhk(s)} .
    \label{eq:TransferMatrixBoundRho}
\end{equation}
Taking the first three components in \cref{eq:DuhamelRhoExpanded}, we then get
\begin{equation}
    |\boldsymbol\rho(t)|
    \lesssim
    \int_0^t \left( 1+ \log \frac{\dhk(t)}{\dhk(s)}\right)
    \frac{\dhk(t)}{\dhk(s)}
    \left(
        \Gamma
        \left|\boldsymbol{\mathcal N}(\boldsymbol\rho)(s)\right|
        +
        \frac{1}{s}
        |\widetilde{\boldsymbol{\mathcal I}}(s)|
    \right)
    \,ds .
    \label{eq:rhoDuhamelBound}
\end{equation}
We are left to estimate the two contributions on the right-hand
side of \cref{eq:rhoDuhamelBound}, for which we need the explicit formula for $\dhk$  defined in \cref{def:dhk}.
We first estimate the nonlinear Helmholtz--Kirchhoff remainder. By
\cref{eq:boundN}, we have
\begin{equation}
    \left|\mathcal N_i(\boldsymbol\rho)(s)\right|
    \lesssim C_0^2
    \frac{\delta^4}{\dhk(s)}
    \left(\log\Big(e+\frac{s}{T_{\rm adv}}\Big)\right)^4
    \sup_{0 \leq t \leq T}
    \|\widetilde{\boldsymbol\Omega}\|_{\boldsymbol \cX_\epsilon}^2 .
\end{equation}
Therefore,
\begin{align}
    \int_0^t \left( 1+ \right .& \left . \log \frac{\dhk(t)}{\dhk(s)}\right)   \frac{\dhk(t)}{\dhk(s)}
    \Gamma
    \left|\mathcal N_l(\boldsymbol\rho)(s)\right|
    \,ds
     \\ &  \lesssim C_0^2
    \delta^4
    \dhk(t)
    \sup_{0 \leq t \leq T}
    \|\widetilde{\boldsymbol\Omega}\|_{\boldsymbol \cX_\epsilon}^2
    \left(\log\Big(e+\frac{t}{T_{\rm adv}}\Big)\right)^3 \int_0^t
    \frac{\Gamma}{\dhk(s)^2}
    \log^2\left(e+\frac{s}{T_{\rm adv}}\right)
    \,ds \\ & 
    \lesssim C_0^2
    \delta^4
    \dhk(t)
    \left(\log\Big(e+\frac{t}{T_{\rm adv}}\Big)\right)^6
    \sup_{0 \leq t \leq T}
    \|\widetilde{\boldsymbol\Omega}\|_{\boldsymbol \cX_\epsilon}^2,
    \label{eq:DuhamelNonlinearRho}
\end{align}
where we used
\begin{equation}
    \frac{\Gamma}{\dhk(s)^2}
    =
    \frac{\epsilon(s)^2}{\delta s} \lesssim 
    \frac{1}{s+ T_{\rm adv}}.
\end{equation}
We now estimate the perturbative error $\mathcal I_i $. By \cref{eq:sizeIterm}, keeping the constant $C_I$ explicit, 
 we have
\begin{align}
    \int_0^t \left( 1+ \right .&  \left . \log \frac{\dhk(t)}{\dhk(s)}\right)
    \frac{\dhk(t)}{\dhk(s)}
    \frac{1}{s}
    |\widetilde{\mathcal I}_i(s)|
    \,ds
    \\ & \leq C C_I
    \dhk(t) \left(\log\Big(e+\frac{t}{T_{\rm adv}}\Big)\right)\sup_{0 \leq t \leq T} \|\widetilde{\boldsymbol\Omega}\|_{\boldsymbol \cX_\epsilon}
    \int_0^t
    \frac{\epsilon(s)^4}{s}
    \,ds + \dhk(t) \int_0^t \frac{\cO(\eps^{\infty}(s))}{s} ds
    \\
    &\leq CC_I
    \delta^2
    \dhk(t)
    \left(\log\Big(e+\frac{t}{T_{\rm adv}}\Big)\right)^2
        \sup_{0 \leq t \leq T}
        \|\widetilde{\boldsymbol\Omega}\|_{\boldsymbol \cX_\epsilon} + C \dhk(t) \eps^{\infty}(t),
        \label{eq:DuhamelErrorRho}
\end{align}
where we used again $\eps(t)^2 = \frac{\delta}{|\alpha|} \frac{t}{t+T_{\rm adv}}$ to get that
\begin{equation}
    \frac{\epsilon(t)^4}{t}
    \lesssim
    \frac{\delta^2}{t+T_{\rm adv}}.
\end{equation}
Combining \cref{eq:DuhamelNonlinearRho} and
\cref{eq:DuhamelErrorRho}, on our time interval, we thus obtain
\begin{align}
    |\rho_i(t)|
    & \leq \widetilde{C} \left[ 
    \delta^2
    \dhk(t)\left(\log\Big(e+\frac{t}{T_{\rm adv}}\Big)\right)^2
    \sup_{0 \leq t \leq T} \|\widetilde{\boldsymbol\Omega}\|_{\boldsymbol \cX_\epsilon} + \dhk\eps(t)^{N} \right. \\ & \left.\quad+ C_0^2
    \delta^4
    \dhk(t)
    \left(\log\Big(e+\frac{t}{T_{\rm adv}}\Big)\right)^6
    \sup_{0 \leq t \leq T} \|\widetilde{\boldsymbol\Omega}\|_{\boldsymbol \cX_\epsilon}^2 \right],
    \label{eq:rhoFinalBound}
\end{align}
where the constant $\widetilde{C}$ does not depend on $C_0$ and $N\geq N_0$ is an arbitrary large number replacing the $\cO(\eps^\infty)$ errors.
Taking then $C_0$ sufficiently large and $\delta$ sufficiently small, we improve the bootstrap estimate and this proves \cref{eq:ControlZZhk}.

Finally, we estimate \(t\partial_t\rho_i\). Returning to \cref{eq:rhoEquationSplit},
we first bound the linear contribution. Since $|z_{ij}(t)|\sim \dhk(t)$,
\cref{eq:rhoFinalBound} gives
\begin{align}
    \frac{(\dhk(t)\epsilon(t))^2}{\delta}
    \sum_{j\neq i}
    \frac{|\rho_{ij}(t)|}{|z_{ij}(t)|^2}
    &\lesssim
    \delta\epsilon(t)^2
    \dhk(t)
    \left(\log\Big(e+\frac{t}{T_{\rm adv}}\Big)\right)^2
    \sup_{0 \leq t \leq T}
    \|\widetilde{\boldsymbol\Omega}\|_{\boldsymbol \cX_\epsilon}.
\end{align}
The nonlinear contribution is of higher order. Indeed, by \cref{eq:boundN},
\begin{align}
    \frac{(\dhk(t)\epsilon(t))^2}{\delta}
    \left|\mathcal N_i(\boldsymbol\rho)(t)\right|
    &\lesssim
    \delta^3
    \epsilon(t)^2
    \dhk(t)
    \left(\log\Big(e+\frac{t}{T_{\rm adv}}\Big)\right)^4
    \sup_{0 \leq t \leq T}
    \|\widetilde{\boldsymbol\Omega}\|_{\boldsymbol{ \mathcal{X}}_\epsilon}^2  + \dhk(t) \cO(\eps(t)^\infty).
\end{align}
Moreover, by \cref{eq:sizeIterm} and $\eps^2 = \frac{\delta}{|\alpha|} \frac{s}{s+T_{\rm adv}}$,
\begin{equation}
    |\mathcal I_i(t)|
    \lesssim
    \delta\epsilon(t)^2
    \dhk(t) \left(\log\Big(e+\frac{t}{T_{\rm adv}}\Big)\right)^2
    \sup_{0 \leq t \leq T}
    \|\widetilde{\boldsymbol\Omega}\|_{\boldsymbol{ \mathcal{X}}_\epsilon} + \dhk(t)\cO(\eps(t)^\infty).
\end{equation}
Therefore,
\begin{equation}
    |t\partial_t\rho_l(t)|
    \lesssim
    \delta\epsilon(t)^2
    \dhk(t)
    \left(\log\Big(e+\frac{t}{T_{\rm adv}}\Big)\right)^2
    \sup_{0 \leq t \leq T}
    \|\widetilde{\boldsymbol\Omega}\|_{\boldsymbol{ \mathcal{X}}_\epsilon} +  \dhk(t)\cO(\eps(t)^\infty).
\end{equation}
that concludes the proof.
\end{proof}

\subsection{Bound on the nonlinear profile.}
\label{sec:NLbounds}
It remains to control the nonlinear correction $\boldsymbol{w}$, for which we need to improve the bootstrap assumption \cref{eq:BootstrapHP} by proving that, under the hypothesis \cref{eq:BootstrapHP}, we can actually show that
\begin{equation}
\label{bd:goalw}
    \| \boldsymbol{w}(t) \|_{\Xg} \leq 2 \widetilde{C} \eps(t) \qquad \text{ for  all} \ t \in [0,T), \quad T \leq T_{\rm adv}\exp(\delta^{-p}).
\end{equation}
The proof of this bound occupies the rest of the section.
\subsubsection{\textbf{Energy identity}} To prove \cref{bd:goalw}, we first compute the time-derivative of the energy.
\begin{lemma}
    Let $\varsigma=1$ and $\boldsymbol{w}$ be the solution to \cref{eq:wvec} and $E_\eps(\boldsymbol{w})$ be defined as in \cref{eq:EepsW}. Then 
    \begin{equation}
    \label{eq:tdtE}
        t \partial_t E_\eps(\boldsymbol{w}) + D_\eps(\boldsymbol{w}) = \textnormal{F} + \textnormal{A} + \textnormal{L} + \textnormal{N}
    \end{equation}
    where $D_\eps(\boldsymbol{w})$ is defined in \cref{eq:defDeps} and we define:
    \begin{itemize}
        \item the forcing term
        \begin{equation}
            \textnormal{F}:= -\frac{1}{\delta^2} \big(\langle \boldsymbol{\cR}, \mathcal{M}_\eps(\boldsymbol{w}) \rangle +  \langle \boldsymbol{\mathsf{T}}(\boldsymbol{\rho}), \mathcal{M}_\eps(\boldsymbol{w}) \rangle\big);
                \label{eq:forcing}
        \end{equation}
        \item the advection term
            \begin{equation}
             \textnormal{A}:=  -\frac{1}{\delta}\langle \cJ_{\mathrm{app}} (\boldsymbol{w}), \mathcal{M}_\eps(\boldsymbol{w}) \rangle;
             \label{eq:advection}
        \end{equation}
        \item the linear error term
        \begin{equation}
             \textnormal{L}:=  \langle \boldsymbol{\mathsf{L}}(\boldsymbol{w},\boldsymbol{\rho}), \mathcal{M}_\eps(\boldsymbol{w}) \rangle;
             \label{eq:linear}
        \end{equation}
        \item the commutator and non-linear error terms
        \begin{equation}
             \textnormal{N}:=  \langle \boldsymbol{\mathsf{N}}(\boldsymbol{w},\boldsymbol{\rho}), \mathcal{M}_\eps(\boldsymbol{w})\rangle + \frac{1}{2}\langle \boldsymbol{w}, [t\partial_t,\cK_\eps](\boldsymbol{w}) \rangle.
             \label{eq:NLterms}
        \end{equation}
    \end{itemize}
    \label{lem:energyTimeDerivative}
\end{lemma}
\begin{proof}
Computing the time derivative of $E_\eps$, we have
\begin{align}
    t\partial_t E_\eps(\boldsymbol{w}) & = \frac{1}{2} 
    \langle  t\partial_t \boldsymbol{w},  \mathcal{M}_\eps(\boldsymbol{w})\rangle + \frac{1}{2} 
    \langle   \boldsymbol{w},  \mathcal{M}_\eps(t\partial_t\boldsymbol{w})\rangle + \frac{1}{2} 
    \langle   \boldsymbol{w}, [t\partial_t,\mathcal{M}_\eps] (\boldsymbol{w})\rangle \\ 
    & = \langle  t\partial_t \boldsymbol{w},  \mathcal{M}_\eps(\boldsymbol{w})\rangle + \frac{1}{2} 
    \langle   \boldsymbol{w}, [t\partial_t,\mathcal{M}_\eps] (\boldsymbol{w})\rangle,
    \label{eq:tdtEeps}
\end{align}
where we used that $\cM_{\eps}$ is self-adjoint. 
Moreover,
\begin{equation}
    [t\partial_t,\mathcal{M}_\eps] = \begin{pmatrix}
        t\partial_tA_{\eps,1} & 0 & 0 \\ 0 & t\partial_tA_{\eps,2}  & 0\\ 0& 0 & t\partial_tA_{\eps,3}  
    \end{pmatrix} + [t\partial_t,\cK_\eps]. 
\end{equation}
The first term is part of the definition of $D_\eps(\boldsymbol{w})$, while the second one is absorbed into the definition of the term $\textnormal{N}$. All the other terms are obtained by using \cref{eq:wvec} in the first term in \cref{eq:tdtEeps}.
\end{proof}
We now have to estimate all the terms in the right-hand side of \cref{eq:tdtE}. We shall use the following convention. Given any term $\mathrm{Q}\in \{\mathrm{F},\mathrm{A},\mathrm{L},\mathrm{N}\}$, we will split it in different terms of the form
\begin{align}
    \mathrm{Q}=\sum_{n=1}^M\mathrm{Q}_n= \sum_{n=1}^M\sum_{i=1}^3 \Big( \mathrm Q^{\mathrm h}_{n;i}+ Q^{\mathrm e}_{n;i}+ \sum_{j\neq i} \Big( \mathrm Q^{\mathrm h}_{n;i,j} + \mathrm Q^{\mathrm e}_{n;i,j} + \sum_{k\neq i} \mathrm Q^{\mathrm e}_{n;i,j,k} \Big) \Big). 
\end{align}
Here the indices after the semicolon identify the vortex and its interactions, while the superscripts $\mathrm h$ and $\mathrm e$ stand respectively for hard and easy terms. We will only detail the hard terms and we will explain why the easy terms are in fact lower order.
\subsubsection{\textbf{Bound of the forcing terms}}
We start with the bounds for the forcing terms in \cref{eq:forcing}. 
We prove the following.

\begin{lemma}
\label{lem:forcing}
    Fix $p \in [0,1/4)$. Then,  for $\sigma_1 >0$ small enough and $\sigma_2>1$, under the bootstrap hypothesis \cref{eq:BootstrapHP} there exists a constant $C>0$, independent of $\eps,\delta$ such that 
    \begin{equation}
        |\textnormal F(t)| \leq C \eps^2 + \frac{1}{10}  D_\eps(\boldsymbol{w})(t), \qquad \text{ for all } t\in[0,T), \quad T\leq T_{\rm adv}\exp(\delta^{-p}).
    \end{equation}
\end{lemma}
\begin{proof}
We split 
    \begin{equation}
        \textnormal{F}=:-(\textnormal{F}_1 + \textnormal{F}_2)
    \end{equation}
    where
    \begin{equation}
            \textnormal{F}_1:= \frac{1}{\delta^2} \langle \boldsymbol{\cR}, \mathcal{M}_\eps(\boldsymbol{w}) \rangle\qquad \textnormal{F}_2 :=\frac{1}{\delta^2} \langle \boldsymbol{\mathsf{T}}(\boldsymbol{\rho}), \mathcal{M}_\eps(\boldsymbol{w}) \rangle
                \label{eq:forcing2}.
    \end{equation}
We now proceed estimating each one of these terms.

\smallskip
\paragraph{$\blackinwhitediamond$ \textit{Bound on $\textnormal{F}_1$}}
Using Lemma \ref{lem:estimatesMepsKeps} and the estimates on the remainder of the approximate solution in Proposition \ref{prop:ApproximateSol}, we have
    \begin{equation}
        \delta^{-2}| \langle \boldsymbol{\cR}, \mathcal{M}_\eps(\boldsymbol{w}) \rangle| \lesssim \delta^{-2} \| \boldsymbol{\cR}\|_{\Xg} \| \boldsymbol{w}\|_{\Xg} \lesssim \delta^{-2} (\eps^5+\delta^2 \eps^2)  \| \boldsymbol{w}\|_{\Xg}\lesssim \eps \| \boldsymbol{w}\|_{\Xg},
    \end{equation}
    where in the last inequality we used that $\eps^2 \lesssim \delta$.
    Then, by Proposition \ref{prop:EnDiss} and Young's inequality, we get
    \begin{equation}
        |\textnormal{F}_1| \lesssim  \eps \, \sqrt{D_\eps(\boldsymbol{w})} \leq C \eps^2 + \frac{1}{100}D_\eps(\boldsymbol{w}),
    \end{equation}
which is consistent with the desired bound.

\smallskip
\paragraph{$\blackinwhitediamond$ \textit{Bound on $\textnormal{F}_2$}}    
    We consider now the other term 
    \begin{equation}
        \mathrm{F}_2=\delta^{-2} \sum_{i=1}^3\langle \mathsf{T}_i(\boldsymbol{\rho}), \mathcal{M}_\eps(\boldsymbol{w})_i \rangle 
    \end{equation}
    First of all, we recall that 
    \begin{align}
\mathcal{M}_{\eps}(\boldsymbol{w})_i=(A_{\eps, i}+\gamma_i^2\Delta^{-1})w_i+\gamma_i\sum_{j\neq i}\gamma_j\mathcal{T}^{i,j}_{\eps, \boldsymbol{z}}\Delta^{-1}w_j.
    \end{align}
    Thus, we can split
    \begin{align}
        \mathrm{F}_2=\sum_{i=1}^3 \Big(\mathrm{F}_{2;i}^{\mathrm h}+\sum_{j\neq i} \Big(\mathrm{F}^{\mathrm h}_{2;i,j}+\mathrm{F}^{\mathrm e}_{2;i,j}+\sum_{k\neq i}\mathrm{F}_{2;i,j,k}^{\mathrm e}\Big)\Big)
        \end{align}
    where, recalling  \cref{def:translation_forcing}, we define 
    \begin{align}
&\mathrm{F}_{2;i}^{\mathrm h} :=\delta^{-1}\frac{1}{\eps \dhk} \left \langle \left\{  t\partial_t \rho_{i} \cdot \xi^\perp , \Omega_{{\rm app},i} \right\},(A_{\eps, i}+\gamma_i^2\Delta^{-1})w_i\right \rangle,\\
&\mathrm{F}_{2;i,j}^{\mathrm h} := \delta^{-2}
        \gamma_j\left \langle
        \left\{
\big(\cT_{\eps,\boldsymbol{z}+\boldsymbol{\rho}}^{i,j}
            -\cT_{\eps,\boldsymbol{z}}^{i,j}\big)
            \Delta^{-1}\Omega_{\app,j},
            \Omega_{\app,i}
        \right\},    (A_{\eps, i}+\gamma_i^2\Delta^{-1})w_i\right \rangle,\\
        &\mathrm{F}_{2;i,j}^{\mathrm e} :=\delta^{-1}\frac{\gamma_i\gamma_j}{\eps \dhk} \left \langle \left\{  t\partial_t \rho_{i} \cdot \xi^\perp , \Omega_{{\rm app},i} \right\},\mathcal{T}^{i,j}_{\eps,\boldsymbol{z}}\Delta^{-1}w_j\right \rangle,\\
        &\mathrm{F}_{2;i,j,k}^{\mathrm e} := \delta^{-2}
        \gamma_i\gamma_j\gamma_k\left \langle
        \left\{
\big(\cT_{\eps,\boldsymbol{z}+\boldsymbol{\rho}}^{i,j}
            -\cT_{\eps,\boldsymbol{z}}^{i,j}\big)
            \Delta^{-1}\Omega_{\app,j},
            \Omega_{\app,i}
        \right\},    \mathcal{T}^{i,k}_{\eps,\boldsymbol{z}}\Delta^{-1}w_k\right \rangle
    \end{align}
We now estimate each term separately, but we first observe that a naive application of Lemma \ref{lem:estimatesMepsKeps} and Proposition \ref{prop:ControlZZhk}, would yield a bound of the form
    \begin{equation}
        |\mathrm{F}_{2;i}^{\mathrm h} | \leq \frac{1}{100} D_\eps(\boldsymbol{w}) + C \eps^2   \left(\log\left(e+\frac{t}{T_{\rm adv}}\right)\right)^4 
    \end{equation}
    which would not allow us to close the bootstrap argument. Therefore, a more accurate analysis is needed. Consider $\mathrm{F}_{2;i}^{\mathrm h}$, which we rewrite as 
    \begin{equation}
        \mathrm{F}_{2;i}^{\mathrm h} = \frac{1}{\delta}\frac{ 1}{\eps \dhk} \left \langle t \partial_t \rho_{i} \cdot \nabla \Omega_{{\rm app},i} ,(A_{\eps, i} + \gamma_i^2 \Delta^{-1}) w_i\right \rangle.
    \end{equation}
    Since the approximate solution is of the form $\Omega_{{\rm app},i} = G + \eps^2 \widetilde{\Omega}_i$, with $\widetilde{\Omega}_i \in \cX_{\eps, i}$ where $M(\widetilde{\Omega}_i) = m_{1}^1(\widetilde{\Omega}_i) =m_{2}^1(\widetilde{\Omega}_i) = 0$, we have 
    \begin{equation}
        \left|\mathrm{F}_{2;i}^{\mathrm h}\right| \lesssim \left| \frac{1}{\delta}\frac{1}{\eps \dhk} \left \langle t \partial_t \rho_{i} \cdot \nabla G , (A_{\eps, i} + \gamma_i^2 \Delta^{-1}) w_i\right \rangle\right| + \eps^3 \delta^{-2p}\|{\boldsymbol{w}}\|_{\boldsymbol{X}_\eps}
    \end{equation}
    Therefore, it remains to estimate the Lamb-Oseen contribution. 
    Set
    $$
a_i(t)
:=
\frac{t\partial_t\rho_i(t)}
{\delta\eps \dhk(t)}
=
\bigl(a_{i,1}(t),a_{i,2}(t)\bigr),
$$
and note that, by Proposition \ref{prop:ControlZZhk}, it holds that
\begin{equation}
    |a_i(t)| \lesssim \eps(t) \delta^{-2p}.
    \label{eq:aitBound}
\end{equation}
Then
\begin{equation}
\left|
\left\langle
a_i(t)\cdot\nabla G,
\left(
A_{\eps, i}
+
\gamma_i^2\Delta^{-1}
\right)w_i
\right\rangle
\right| \leq
\sum_{\ell=1}^2
|a_{i,\ell}(t)|
\left|
\left\langle
\partial_\ell G,
\left(
A_{\eps, i}
+
\gamma_i^2\Delta^{-1}
\right)w_i
\right\rangle
\right|.
\label{eq:FirstCancel}
\end{equation}

We estimate each component separately.
Splitting the inner and outer region, we have
    \begin{align}
        \left|\left \langle \de_l G, (A_{\eps, i} + \gamma_i^2 \Delta^{-1}) w_i\right \rangle \right| & \lesssim \left| \left \langle \de_l G, \mathbbm{1}_{\Ii} (A_{\eps, i} + \gamma_i^2 \Delta^{-1}) w_i\right \rangle \right| +  \eps^{\infty}  \| \boldsymbol{w}\|_{\Xg},
    \end{align}
    since outside the inner region we get
    \begin{equation}
        \left|\left \langle \de_l G, \mathbbm{1}_{\Ii^c}A_{\eps, i}  w_i\right \rangle \right| = \left|\left \langle A_{\eps, i}^{1/2}\de_l G, \mathbbm{1}_{\Ii^c}A_{\eps, i}^{1/2}  w_i\right \rangle \right| \lesssim \eps^{\infty} \| \boldsymbol{w}\|_{\Xg},
    \end{equation}
    \begin{equation}
        \left|\left \langle  \de_l G, \mathbbm{1}_{\Ii^c}  \Delta^{-1} w_i\right \rangle\right| \lesssim \| \mathbbm{1}_{\Ii^c}A_{\eps, i}^{1/2}  \de_l G  \| \| A_{\eps, i}^{-1/2} \Delta^{-1} w_i\| \lesssim \eps^{\infty} \| \boldsymbol{w}\|_{\Xg}.
    \end{equation}
    In the inner region, the weight coincides with the functional relationship obtained previously in Section \ref{sec:FunctionalRelationship}. Moreover, using Proposition \ref{prop:weights} we have
    \begin{equation}
        |F'_i(\Omega_{{\rm app},i}^{(2),E},t) - \gamma_iF'_{} (G)| \lesssim \eps^{\eta_2}F'_{} (G) ,
    \end{equation}
    for any $\eta_2 <2$. Thus,
    \begin{align}
        \left|\left \langle \de_l G, \mathbbm{1}_{\Ii} (\gamma_i  F'_i(\Omega_{{\rm app},i}^{(2),E}) + \gamma_i^2 \Delta^{-1}) w_i\right \rangle \right| & \lesssim \left|\left \langle \de_l G, \mathbbm{1}_{\Ii} (\gamma_i^2 F_{}'(G) + \gamma_i^2 \Delta^{-1}) w_i\right \rangle \right| + \eps^{\eta_2} \| \boldsymbol{w}\|_{\Xg} \\ 
        & \leq  C \gamma_i^2 \left|\left \langle \de_l G, \mathbbm{1}_{\Ii} \mathsf{H} w_i\right \rangle \right| + \eps^{\eta_2} \| \boldsymbol{w}\|_{\Xg}.
    \end{align}
    By property \ref{it:kernel} of Proposition \ref{prop:Lambda} for the operator $\mathsf{H}$ defined in \cref{def:AH}, we have 
    \begin{equation}
        0 = \left \langle \de_l G, \mathsf{H} w_i\right \rangle \implies \left |\left \langle \de_l G, \mathbbm{1}_{\Ii} \mathsf{H} w_i\right \rangle \right| =  \left |\left \langle \de_l G, \mathbbm{1}_{\Ii^c} \mathsf{H} w_i\right \rangle \right|
    \end{equation}
    Therefore,
        \begin{align}
        \left|\left \langle \de_l G, \mathbbm{1}_{\Ii} (\gamma_i  F'_i(\Omega_{{\rm app},i}^{(2),E}) + \gamma_i^2 \Delta^{-1}) w_i\right \rangle \right| & \lesssim   \left|\left \langle \de_l G, \mathbbm{1}_{\Ii^c} \gamma_i ^2\mathsf{H} w_i\right \rangle \right| + \eps^{\eta_2} \| \boldsymbol{w}\|_{\Xg}
        \\
        & \lesssim \eps^{\eta_2} \| \boldsymbol{w}\|_{\Xg}.
    \end{align}
    Thus, combining the above estimates with \cref{eq:aitBound}, we get that
    \begin{equation}
        \left| \mathrm{F}_{2;i}^{\mathrm h}\right| \lesssim \eps^{\eta_2+1}  \delta^{-2p}\|{\boldsymbol{w}}\|_{\boldsymbol{X}_\eps}.
    \end{equation}
    By Young's inequality, taking $\eta_2 \in (1,2)$ and $p< 1/4$, we get 
    \begin{equation}
        \left| \mathrm{F}_{2;i}^{\mathrm h}\right| \leq \eta D_\eps (\boldsymbol{w}) + C_\eta \eps^2,
        \label{eq:F_1}
    \end{equation}
    with $\eta> 0$ sufficiently small.
     Now, let us consider $\mathrm{F}_{2;i,j}^{\mathrm h}$. Using Lemma \ref{lem:TeProperties}, we have that
    \begin{equation}
        \left( \cT_{\eps,\Z}^{i,j} - \cT_{\eps,\boldsymbol{z}}^{i,j}\right) \Delta^{-1} \Omega_{j,app} = C_\eps + \frac{1}{2 \pi} \eps \dhk \left[ \textnormal{Re} \left( \xi \frac{z_{ij} - Z_{ij}}{Z_{ij} z_{ij}}\right)\right] + \cO_{\cS_{*}}(\eps^2) \frac{|z_{ij} - Z_{ij}|}{\dhk}
    \end{equation}
where the constant $C_\eps$ cancels out due to the Poisson bracket. Thus, using \cref{eq:gTijf} and expanding the approximate solution, we get
\begin{equation}
    |\mathrm{F}_{2;i,j}^{\mathrm h}| \lesssim \left|\delta^{-2} \eps \left \langle \left \{  \dhk \left[ \textnormal{Re} \left( \xi \frac{z_{ij} - Z_{ij}}{Z_{ij} z_{ij}}\right)\right], G \right\},(A_{\eps, i} + \gamma_i^2 \Delta^{-1}) w_i\right \rangle \right| + \delta^{-2} \eps^2 \frac{|z_{ij}-Z_{ij}|}{\dhk} \| \boldsymbol{w}\|_{\Xg}.
\end{equation}
Note that 
\begin{equation}
    \left \{ \dhk \left[ \textnormal{Re} \left( \xi \frac{z_{ij} - Z_{ij}}{Z_{ij} z_{ij}}\right)\right], G \right\} =  f_1(t) \partial_1 G + f_2(t) \partial_2 G
\end{equation}
with
\begin{equation}
    |f_1| + |f_2| \lesssim \frac{|Z_{ij} - z_{ij}|}{\dhk}.
\end{equation}
Therefore, also this term has the same structure as \cref{eq:FirstCancel}. Thus, combining Proposition \ref{prop:ControlZZhk} with the argument adopted above, we get 
\begin{align}
    |\mathrm{F}_{2;i,j}^{\mathrm h}|  & \lesssim \left(\eps^{1+\eta_2} + \eps^{2} \right)  \delta^{-2p} \|{\boldsymbol{w}}\|_{\boldsymbol{X}_\eps} \\ &  \lesssim \eps^{2} \delta^{-2p} \|{\boldsymbol{w}}\|_{\boldsymbol{X}_\eps}
\end{align}
which, using Young's inequality and the fact that $p<1/4$ gives
\begin{equation}
    |\mathrm{F}_{2;i,j}^{\mathrm h}| \leq \eta D_\eps (\boldsymbol{w}) + C_\eta\eps^2.
    \label{eq:F_2}
\end{equation}
For the two easy terms $\mathrm{F}_{2;i,j}^e$ and $\mathrm{F}_{2;i,j,k}^e$ it is enough to combine Proposition \ref{prop:ControlZZhk} with Lemma \ref{lem:estimatesMepsKeps} to get the bound 
\begin{equation}
    |\mathrm{F}_{2;i,j}^e | + |\mathrm{F}_{2;i,j,k}^e| \lesssim \eps^3 \delta^{-2p} \|{\boldsymbol{w}}\|_{\boldsymbol{X}_\eps},
\end{equation}
and again a Young's inequality gives 
\begin{equation}
    |\mathrm{F}_{2;i,j}^e | + |\mathrm{F}_{2;i,j,k}^e| \leq \eta D_\eps (\boldsymbol{w}) + C_\eta \eps^2.
    \label{eq:F_3}
\end{equation}
Hence, combining \cref{eq:F_1},     \cref{eq:F_2} and \cref{eq:F_3}, taking $\eta$ small enough, we get
\begin{equation}
    |\mathrm{F}_2| \leq \frac{1}{20} D_\eps(\boldsymbol{w}) + C \eps^2,
\end{equation}
which is coherent with the desired estimate.
\end{proof}
\subsubsection{\textbf{Bound of the linear terms}}
We now bound the linear error term $\textnormal L$ defined in \cref{eq:linear}. In particular we prove that

\begin{lemma}
\label{lem:linear}
   Let $\varsigma=1$ and fix $p \in [0,1/4)$.
    Then,  for $\sigma_1 >0$ small enough and $\sigma_2>\frac32$, under the bootstrap hypothesis \cref{eq:BootstrapHP} there exists $\sigma_\star \in (0,1)$ and a constant $C>0$, independent of $\eps$ and $\delta$ such that 
    \begin{equation}
        |\mathrm L(t)| \leq C \eps(t)^{\sigma_\star}D_\eps(\boldsymbol{w})(t) \qquad \text{for all } t \in [0,T), \quad T \leq T_{\rm adv} \exp( \delta^{-p}).
    \end{equation}
\end{lemma}
\begin{proof}
We split the linear term as follows
\begin{equation}
    \mathrm{L} =: \mathrm{L}_1 + \mathrm{L}_2 + \mathrm{L}_3 + \mathrm{L}_4
\end{equation}
with
\begin{equation}
    \mathrm{L}_m :=\begin{cases}
        \displaystyle
        \sum_{i=1}^3 \langle \mathsf{L}^i_m(\boldsymbol{w}), \mathcal{M}_{\eps}(\boldsymbol{w})_i\rangle, \quad \text{for} \ m =1,2  \\[6pt]
        \displaystyle
        \sum_{i=1}^3 \langle \mathsf{L}^i_m(\boldsymbol{w},\boldsymbol{\rho}), \mathcal{M}_{\eps}(\boldsymbol{w})_i\rangle, \quad \text{for}  \ m = 3,4. 
    \end{cases}  
\end{equation}
where $\mathsf{L}^i_m$ are defined in \cref{eq:LijkDefinition}.
We now estimate each of these terms.
\smallskip
\paragraph{$\blackinwhitediamond$ \textit{Bound on $\mathrm{L}_1$}}
By \cref{eq:L1Definition}, we have that
\begin{equation}
    \mathrm{L}_1 = \sum_{i=1}^3 \left\langle \left\{
            w_i,\Phi^{(2),NS}_{\app,i}
            +\delta^{-1}\eps^3\widetilde\Phi_{\app,i}
        \right\}, \mathcal{M}_{\eps}(\boldsymbol{w})_i \right\rangle .
\end{equation}
Since the Navier--Stokes terms in the approximate solution enters at the level of the second order correction, we know that $\Phi_{{\rm app},i} ^{(2),NS} = \cO_{\cS_*}(\eps^2)$. 
By \cref{eq:ApproxProfilesBound}, the vorticity profiles
$\Omega_{n,i}^q$, $n=2,3,4$ and
$\mathsf{q}\in\{E,E_{\rm r},NS\}$, are uniformly bounded in
$\mathcal Z$, together with their spatial derivatives.
Moreover, $\delta^{-1} \eps^3 \lesssim \eps$ and $\widetilde{\Phi}_{{\rm app},i}= \cO_{\cS_*}(1)$. Therefore, using Lemma \ref{lem:SobolevEstimates}, 
we have 
\begin{equation}
    \left\|
        \nabla\Phi_{\app,i}^{(2),NS}
    \right\|_{L^\infty}
    +
    \delta^{-1}\eps^3
    \left\|
        \nabla\widetilde\Phi_{\app,i}
    \right\|_{L^\infty}
    \lesssim
    \eps^2+\delta^{-1}\eps^3
    \lesssim
    \eps,
    \label{eq:L1ProfileBound}
\end{equation}     
and, combining \cref{eq:L1ProfileBound} with Lemma \ref{lem:estimatesMepsKeps}, we obtain 
\begin{align}
    |\mathrm{L}_1|\lesssim
\eps
\left\|
\nabla\boldsymbol w
\right\|_{\boldsymbol{\mathcal X}_\eps}
\left\|
\boldsymbol w
\right\|_{\boldsymbol{\mathcal X}_\eps}
    \lesssim
    \eps \, D_\eps(\boldsymbol{w}).
    \label{eq:L1L2LinearBound}
\end{align}
which is coherent with the desired bound.
\paragraph{$\blackinwhitediamond$ \textit{Bound on $\mathrm{L}_2$}}
Recalling \cref{eq:L2Definition}, we have that
\begin{equation}
    \mathrm{L}_2 = \sum_{i=1}^3 \left\langle \Big\{\Omega_{\app,i}^{(2),NS}
            +\delta^{-1}\eps^3(\Omega_{3,i}+\eps\Omega_{4,i}),
            \gamma_i\Delta^{-1}w_i
            +\sum_{j\neq i}\gamma_j\cT_{\eps,\boldsymbol{z}}^{i,j}\Delta^{-1}w_j
        \Big\}, \mathcal{M}_{\eps}(\boldsymbol{w})_i \right\rangle .
\end{equation}
Note that, setting 
\begin{equation}
    \widetilde{\Psi}_i := \gamma_i\Delta^{-1}w_i
            +\sum_{j\neq i}\gamma_j\cT_{\eps,\boldsymbol{z}}^{i,j}\Delta^{-1}w_j,
\end{equation}
we have that
\begin{equation}
    \mathrm{L}_2 = \sum_{i=1}^3 \left\langle \Big\{\Omega_{\app,i}^{(2),NS}
            +\delta^{-1}\eps^3(\Omega_{3,i}+\eps\Omega_{4,i}),
            \widetilde{\Psi}_i
        \Big\}, A_{\eps, i} w_i + \gamma_i \widetilde{\Psi}_i\right\rangle.
\end{equation}
Since $\langle \{f,g\}, g\rangle = 0$, we get
\begin{equation}
    \mathrm{L}_2 = \sum_{i=1}^3 \left\langle \Big\{\Omega_{\app,i}^{(2),NS}
            +\delta^{-1}\eps^3(\Omega_{3,i}+\eps\Omega_{4,i}),
            \widetilde{\Psi}_i
        \Big\}, A_{\eps, i} w_i\right \rangle.
\end{equation}
Using the fact that $\eps^2 + \delta^{-1}\eps^3 \lesssim \eps$ as done for $\mathrm{L}_1$, splitting inner and outer region to exploit the decay properties of the approximate solution and appealing to Lemma \ref{lem:SobolevEstimates}, it is not hard to get
\begin{equation}
    | \mathrm{L}_2| \lesssim \eps D_{\eps}(\boldsymbol{w}).
\end{equation}
\smallskip
\paragraph{$\blackinwhitediamond$ \textit{Bound on $\mathrm{L}_3$}}
Recalling the expression of $\mathcal M_\eps(\boldsymbol w)_i$ and \cref{eq:L3Definition}, we split
\begin{equation}
    \mathrm L_3 = \sum_{i = 1}^3 \sum_{j \neq i}\left(\mathrm{L}_{3;i,j}^{\mathrm{h},1} + \mathrm{L}_{3;i,j}^{\mathrm{h},2} + \sum_{k \neq i} \mathrm{L}_{3;i,j,k}^\mathrm{e}\right)
\end{equation}
where
\begin{align}
    &\mathrm{L}_{3;i,j}^{\mathrm{h},1}:=\delta^{-1} \gamma_i^2 \gamma_j 
\left\langle 
        \left\{
\big(\cT_{\eps,\boldsymbol{z}}^{i,j}-\cT_{\eps,\boldsymbol{z}+\boldsymbol{\rho}}^{i,j}\big)\Delta^{-1}w_j,
            \Omega_{\app,i}
        \right\},
\Delta^{-1}w_i
\right\rangle, \\
&\mathrm{L}_{3;i,j}^{\mathrm{h},2}:=\delta^{-1} \gamma_j 
\left\langle
        \left\{
\big(\cT_{\eps,\boldsymbol{z}}^{i,j}-\cT_{\eps,\boldsymbol{z}+\boldsymbol{\rho}}^{i,j}\big)\Delta^{-1}w_j,
            \Omega_{\app,i}
        \right\},
A_{\eps, i}w_i
\right\rangle,\\
&\mathrm{L}_{3;i,j,k}^\mathrm{e} : = \delta^{-1} \gamma_i \gamma_j \gamma_k \left\langle
    \left\{
    \big(\cT_{\eps,\boldsymbol{z}}^{i,j}-\cT_{\eps,\boldsymbol{z}+\boldsymbol{\rho}}^{i,j}\big)\Delta^{-1}w_j,
            \Omega_{\app,i}
        \right\},
    \mathcal T_{\eps,\boldsymbol z}^{i,k}
    \Delta^{-1}w_k
    \right\rangle.
\end{align}
Let us start with $\mathrm{L}_{3;i,j}^{\mathrm{h},1}$. Since we do not exploit the smallness of $(\boldsymbol{z}+\boldsymbol{\rho}) - \boldsymbol{z} $, but the fact that $\boldsymbol{w}$ has vanishing moments, we can treat just one of the translating operators in the difference $\big(\cT_{\eps,\boldsymbol{z}}^{i,j}-\cT_{\eps,\boldsymbol{z}+\boldsymbol{\rho}}^{i,j}\big)$ since the other can be treated similarly. Therefore, we estimate
\begin{equation}
    \delta^{-1}
    \left\langle
        \nabla^\perp
        \mathcal{T}_{\eps,\boldsymbol{z}+\boldsymbol{\rho}}^{i,j}
        \Delta^{-1}w_j
        \cdot
        \nabla \Omega_{{\rm app},i}, \Delta^{-1}w_i
    \right\rangle .
    \label{eq:L3MainTerm}
\end{equation}
We split the domain into the
inner region \(\Ii\) and its complement \(\Ii^c\). In the inner region \(\Ii\), Hölder's inequality
with \(1/q_1+1/q_2=1\), \(q_2>2\), together with Proposition \ref{prop:PropertiesSpaces} and Lemma \ref{lem:SobolevEstimates}, gives
\begin{align}
    &
    \left|
    \delta^{-1}
    \left\langle
        \nabla^\perp
        \mathcal{T}_{\eps,\boldsymbol{z}+\boldsymbol{\rho}}^{i,j}
        \Delta^{-1}w_j
        \cdot
        \nabla \Omega_{{\rm app},i},
        \mathbbm{1}_{\Ii}\Delta^{-1}w_i
    \right\rangle
    \right| \\  &  \qquad \lesssim
    \delta^{-1}
    \left\| \left(1+|\xi|\right)
        \nabla\Omega_{{\rm app},i}
    \right\|_{L^{q_1}(\Ii)}
    \left\| \left(1+|\xi|\right)^{-1}
        \nabla^\perp
        \mathcal{T}_{\eps,\boldsymbol{z}+\boldsymbol{\rho}}^{i,j}
        \Delta^{-1}w_j
    \right\|_{L^{q_2}(\Ii)}
    \left\|
        \Delta^{-1}w_i
    \right\|_{L^\infty}
    \notag
    \\
    & \qquad
    \lesssim
    \delta^{-1}\eps^3
    \|\boldsymbol{w}\|_{\boldsymbol{\mathcal X}_\eps}^2
    \lesssim
    \eps
    \|\boldsymbol{w}\|_{\boldsymbol{\mathcal{X}}_\eps}^2 .
    \label{eq:L3DeltaInner}
\end{align}
On the complementary region \( \Ii^c \), the decay of the approximate profile gives,
for any \(N\geq 1\),
\begin{align}
    &
    \left|
    \delta^{-1}
    \left\langle
        \nabla^\perp
        \mathcal{T}_{\eps,\boldsymbol{z}+\boldsymbol{\rho}}^{i,j}
        \Delta^{-1}w_j
        \cdot
        \nabla \Omega_{{\rm app},i},\mathbbm{1}_{\Ii^c}
        \Delta^{-1}w_i
    \right\rangle
    \right|
    \notag
    \\
    &\qquad
    \lesssim
    \delta^{-1}
    \left\|
        A_{\eps, i}^{1/2}\nabla \Omega_{{\rm app},i}
    \right\|_{L^2(\Ii^c)}
    \left\|
        A_{\eps, i}^{-1/2}
        \nabla^\perp
        \mathcal{T}_{\eps,\boldsymbol{z}+\boldsymbol{\rho}}^{i,j}
        \Delta^{-1}w_j
    \right\|_{L^2(\Ii^c)}
    \left\|
        \Delta^{-1}w_i
    \right\|_{L^\infty}
    \notag
    \\
    &\qquad
    \lesssim
    \delta^{-1}\eps^N
    \|\boldsymbol{w}\|_{\boldsymbol{\mathcal{X}}_\eps}^2 .
    \label{eq:L3DeltaOuter}
\end{align}
Therefore, picking $N\geq 3$, we have
\begin{equation}
    | \mathrm{L}_{3;i,j}^{\mathrm{h},1}| \lesssim \eps D_{\eps}({\boldsymbol{w}}).
\end{equation}
We now treat $\mathrm{L}_{3;i,j}^{\mathrm{h},2}$ that contains \(A_{\eps, i}w_i\). 
As done for $\mathrm{L}_{3;i,j}^{\mathrm{h},1}$, we consider just one of the translation operator terms in $\big(\cT_{\eps,\boldsymbol{z}}^{i,j}-\cT_{\eps,\boldsymbol{z}+\boldsymbol{\rho}}^{i,j}\big)$.
In the inner region \( \Ii\),
by Cauchy--Schwarz and Hölder's inequality with
\(1/q_1+1/q_2=1/2\), \(q_2>2\), we obtain
\begin{align}
    &
    \left|
    \delta^{-1}
    \left\langle
        \nabla^\perp
        \mathcal{T}_{\eps,\boldsymbol{z}+\boldsymbol{\rho}}^{i,j}
        \Delta^{-1}w_j
        \cdot
        \nabla \Omega_{{\rm app},i},
        A_{\eps, i}w_i
    \right\rangle_{\Ii}
    \right|
    \notag
    \\
    &\qquad
    \lesssim
    \delta^{-1}
    \left\|
        (1+|\xi|)A_{\eps, i}^{1/2}
        \nabla \Omega_{{\rm app},i} \right \|_{L^{q_1}(\Ii)} \left \|
        (1+|\xi|)^{-1}\nabla^\perp
        \mathcal{T}_{\eps,\boldsymbol{z}+\boldsymbol{\rho}}^{i,j}
        \Delta^{-1}w_j
    \right\|_{L^{q_2}(\Ii)}
    \|\boldsymbol{w}\|_{\boldsymbol{X}_\eps}
    \notag
    \\
    &\qquad
    \lesssim
    \delta^{-1}\eps^3
    \|\boldsymbol{w}\|_{\boldsymbol{\mathcal X}_\eps}^2
    \lesssim
    \eps
    \|\boldsymbol{w}\|_{\boldsymbol{\mathcal X}_\eps}^2 .
    \label{eq:L3AInner}
\end{align}
In the outer region \(\Ii^c\), the decay of \(\nabla\Omega_{{\rm app},i}\), combined again with
Cauchy--Schwarz inequality, yields
\begin{align}
    \left|
    \delta^{-1}
    \left\langle
        \nabla^\perp
        \mathcal{T}_{\eps,\boldsymbol{z}+\boldsymbol{\rho}}^{i,j}
        \Delta^{-1}w_j
        \cdot
        \nabla \Omega_{{\rm app},i},
        A_{\eps, i}w_i
    \right\rangle_{\Ii^c}
    \right|
    \lesssim
    \delta^{-1}\eps^N
    \|\boldsymbol{w}\|_{\boldsymbol{\mathcal X}_\eps}^2 ,
    \label{eq:L3AOuter}
\end{align}
for any \(N\geq 1\). Thus,  picking $N \geq 3$,
\begin{equation}
    | \mathrm{L}_{3;i,j}^{\mathrm{h},2}| \lesssim \eps D_{\eps}({\boldsymbol{w}}).
\end{equation}

It remains to consider $\mathrm{L}^e_{3;i,j,k}$, which involve the additional translation term \(\mathcal{T}_{\eps,\boldsymbol{z}}^{i,k}\Delta^{-1}w_k\). It can actually be treated in the same way as before, with the advantage that in the inner region we have additional gain in terms of power of $\eps$ thanks to Corollary \ref{cor:EstTZapp}. The outer region is treated in the same way. Then, 
it is not hard to check that  
\begin{align}
    \left|\delta^{-1}\left\langle\nabla^\perp\mathcal{T}_{\eps,\boldsymbol{z}+ \boldsymbol{\rho}}^{i,j}
        \Delta^{-1}w_j
        \cdot
        \nabla \Omega_{{\rm app},i},\mathcal{T}_{\eps,\boldsymbol{z}}^{i,k}\Delta^{-1}w_k\right\rangle_{\Ii}\right|\lesssim\delta^{-1}\eps^5\|\boldsymbol{w}\|_{\boldsymbol{\mathcal X}_\eps}^2.\label{eq:L3OffDiagonalInner}
\end{align}
Thus, we conclude that
\begin{equation}
    \left|
        \mathrm{L}_3\right| 
    \lesssim
    \eps D_\eps(\boldsymbol{w}).
    \label{eq:L3FinalBound}
\end{equation}
which is consistent with the desired bound.
\smallskip
\paragraph{$\blackinwhitediamond$ \textit{Bound on $\mathrm{L}_4$}}
By \cref{eq:L4Definition}, we can decompose it as follows.
\begin{equation}
    \mathrm{L}_4 = \sum_{i = 1}^3 \sum_{j \neq i}\left(\mathrm{L}_{4;i,j}^{\mathrm{h},1} + \mathrm{L}_{4;i,j}^{\mathrm{h},2} + \sum_{k \neq i} \mathrm{L}_{4;i,j,k}^\mathrm{e}\right) 
\end{equation}
\begin{align}
    &\mathrm{L}_{4;i,j}^{\mathrm{h},1}:=\delta^{-1} \gamma_i^2 \gamma_j 
    \left\langle
            \left\{
\big(\cT_{\eps,\boldsymbol{z}}^{i,j}-\cT_{\eps,\boldsymbol{z}+\boldsymbol{\rho}}^{i,j}\big)\Delta^{-1}\Omega_{\app,j},
            w_i
        \right\},
    \Delta^{-1}w_i
    \right\rangle, \\
    &\mathrm{L}_{4;i,j}^{\mathrm{h},2}:=\delta^{-1} \gamma_j 
    \left\langle
            \left\{
\big(\cT_{\eps,\boldsymbol{z}}^{i,j}-\cT_{\eps,\boldsymbol{z}+\boldsymbol{\rho}}^{i,j}\big)\Delta^{-1}\Omega_{\app,j},
            w_i
        \right\},
    A_{\eps, i}w_i
    \right\rangle,\\
    &\mathrm{L}_{4;i,j,k}^\mathrm{e} : = \delta^{-1} \gamma_i \gamma_j \gamma_k \left\langle
        \left\{
\big(\cT_{\eps,\boldsymbol{z}}^{i,j}-\cT_{\eps,\boldsymbol{z}+\boldsymbol{\rho}}^{i,j}\big)\Delta^{-1}\Omega_{\app,j},
            w_i
        \right\},
        \mathcal T_{\eps,\boldsymbol z}^{i,k}
        \Delta^{-1}w_k
        \right\rangle.
\end{align}
Define 
\begin{equation}
    B_{ij}
    :=
    \left(
        \mathcal{T}_{\eps,\boldsymbol{z}}^{i,j}
        -
        \mathcal{T}_{\eps,\boldsymbol{z}+\boldsymbol{\rho}}^{i,j}
    \right)
    \Delta^{-1}\Omega_{j,app}.
\end{equation}
Using the Taylor
expansion of the translation operator in Lemma \ref{lem:TeProperties} and cutting it at order $2$, recalling the notation in \cref{notation_fij} we get that
\begin{equation}
    \left\|\nabla^\perp B_{ij}\right\|_{L^{\infty}(\Ii)} \lesssim \eps \dhk \frac{ |\rho_{ij}|}{|Z_{ij}||z_{ij}|} + \eps^2 \frac{|\rho_{ij}|}{\dhk} |\xi|.  
\end{equation}
Using the fact that in the inner region $\Ii$ we have $|\xi| \lesssim \eps ^{-\sigma_1}$, together with Proposition \ref{prop:ControlZZhk}, taking $\sigma_1$ small enough, we have
\begin{equation}
    \left\|\nabla^\perp B_{ij}\right\|_{L^{\infty}(\Ii)}
    \lesssim
    \eps\delta^2 
    \delta^{-2p}.
    \label{eq:L4BijInnerBound}
\end{equation}
We first estimate $\mathrm{L}_{4;i,j}^{\mathrm{h},1}$, which we rewrite as
\begin{equation}
    \mathrm{L}_{4;i,j}^{\mathrm{h},1} = \delta^{-1} \gamma_i \gamma_j 
    \left \langle\nabla^\perp B_{ij} \cdot \nabla
            w_i,
    \gamma_i\Delta^{-1}w_i
    \right\rangle
\end{equation}
By splitting the
domain into the inner region \(\Ii\) and its complement $\Ii^c$, we get
\begin{align}
    &
    \delta^{-1}
    \left|
    \left\langle
        \nabla^\perp B_{ij}\cdot \nabla w_i, 
        \Delta^{-1}w_i
    \right\rangle
    \right|
    \notag
    \\
    &\qquad
    \lesssim
    \delta^{-1}
    \left\|
        \nabla^\perp B_{ij}
    \right\|_{L^\infty(\Ii)}
    \|\nabla w_i\|_{L^2(\Ii)}
    \|\Delta^{-1}w_i\|_{L^2}
    +
    \delta^{-1}\eps^\infty D_\eps(\boldsymbol{w}),
    \label{eq:L4DeltaSplit}
\end{align}
where the second term comes from the decay properties of the weight $A_{\eps, i}^{-1/2}$ on \(\Ii^c\). 
Therefore, by Lemma \ref{lem:SobolevEstimates} and Proposition \ref{prop:EnDiss} we get
\begin{align}
    \left|
    \mathrm{L}_{4;i,j}^{\mathrm{h},1}
    \right|
    \lesssim
    \eps
    \delta^{1-2p}
    D_\eps(\boldsymbol{w}).
    \label{eq:L4DeltaBound}
\end{align}
that, for $p<1/4$ is consistent with the desired bound.

We now estimate $\mathrm{L}_{4;i,j}^{\mathrm{h},2}$. Being
\(\nabla^\perp B_{ij}\)  divergence-free, an integration by parts yields
\begin{align}
    \delta^{-1}
    \left\langle
        \nabla^\perp B_{ij}\cdot \nabla w_i,
        A_{\eps, i}w_i
    \right\rangle =
    -\frac{1}{2}\delta^{-1}
    \int
        w_i^2
        \nabla^\perp B_{ij}^{}
        \cdot
        \nabla A_{\eps, i}
    \,d\xi .
    \label{eq:L4IBP}
\end{align}
Since the weight is spatially constant in $\IIi$ only the inner region \(\Ii\) and the outer region
\(\IIIi\) contribute. In the inner region $\Ii$ by Proposition \ref{prop:PropertiesSpaces} we have 
\begin{align}
    |\nabla A_{\eps, i}| & \leq |\nabla A_{\eps, i} - \nabla A_{0,i}| + |\nabla A_{0,i}| \\ &
    \lesssim \eps^{\eta_2} A_{0,i} + \left( 1+ |\xi| \right) |A_{0,i}|  \\ & 
    \lesssim \eps^{- \sigma_1} |A_{\eps, i}|.\label{eq:nablaAiEst}
\end{align}
Combining \cref{eq:L4BijInnerBound} and \cref{eq:nablaAiEst}, we obtain
\begin{align}
    \delta^{-1}\left|
    \left\langle
        \nabla^\perp B_{ij}\cdot \nabla w_i, \mathbbm{1}_{\Ii}
        A_{\eps, i}w_i
    \right\rangle\right|
    \lesssim
    \eps^{1-\sigma_1}
    \delta^{1-2p}
    D_\eps(\boldsymbol{w}).
    \label{eq:L4InnerA}
\end{align}
In the outer region $\IIIi$, 
the weight is $A_{\eps, i} = \gamma_i^2 \exp\left( |\xi|^{2 {\widetilde{\sigma}}} /4\right)$ and, thus, we have that
\begin{equation}
    |\nabla A_{\eps, i}|
    \lesssim
    |\xi|^{2{\widetilde{\sigma}}-1} A_{\eps, i}.
\end{equation}
Moreover, 
the far field estimate for the approximate velocity field gives
\begin{equation}
    | \nabla B_{ij}(\xi)| \lesssim |\xi|^{-1}, \qquad \xi \in \IIIi.
\end{equation}
Then, since $|\xi| \gtrsim \eps^{-\sigma_2}$ 
and, by the definition of \(\mathrm{r}_\eps\) in \cref{def:WeightPho} and using \cref{prop:DepsCoercivity}, we obtain
\begin{align}
    \delta^{-1}\left|
    \left\langle
        \nabla^\perp B_{ij}\cdot \nabla w_i, \mathbbm{1}_{\IIIi}
        A_{\eps, i}w_i
    \right\rangle\right|& \lesssim
    \delta^{-1}  
    \left|
    \int_{\IIIi}
        w_i^2
        \nabla^\perp B_{ij}^{}
        \cdot
        \nabla A_{\eps, i}
    \,d\xi
    \right| \\
    &\lesssim
    \delta^{-1}\eps^{ 2 \sigma_2}
    \|\mathrm{r}_\eps \boldsymbol{w}\|_{\boldsymbol{\mathcal X}_\eps}^2
    \notag
    \\
    &\lesssim
    \eps^{2 \sigma_2 -2}
    D_\eps(\boldsymbol{w})
    \\
    &\lesssim
    \eps
    D_\eps(\boldsymbol{w}),
    \label{eq:L4OuterA}
\end{align}
provided \(\sigma_2 \geq \frac{3}{2}\).
Therefore, combining \cref{eq:L4InnerA} and \cref{eq:L4OuterA}, for $\sigma_2\geq \frac32$, $p<1/4$ and $\sigma_1$ small enough, we have that there exists $\sigma_\star \in (0,1)$ such that
\begin{equation}
    \left|
    \mathrm{L}_{4;i,j}^{\mathrm{h},2}\right| \lesssim \eps^{\sigma_\star} D_{\eps}(\boldsymbol w).
\end{equation}
The term $\left|\mathrm{L}_{4;i,j,k}^{\mathrm{e}}\right|$ is controlled similarly using the
bounds for the translated stream functions of $w_k$. Thus, the lemma is proved.
\end{proof}

\subsubsection{\textbf{Bound of the advection terms}}
We now estimate the advection term $\mathsf A$ defined in \cref{eq:advection}, for which we follow the strategy in \cite{dolce2026viscousvortexcrystals}. 
The main point is that the Arnold functional relation
cancels the leading part of the transport operator in the inner region,
whereas the weighted tails are negligible outside that region.
The goal is to show that this contribution is perturbative with respect to the dissipation functional. 
\begin{lemma}
\label{lem:advection}
Let $\varsigma=1$ and    fix $p $ such that $0<p<\sigma_1<\sigma$. For $\sigma_1 >0$ small enough and $\sigma_2>3/2$, there exists $\sigma_\star \in (0,1)$ and a constant $C$ independent of $\eps, \delta$ such that 
    \begin{equation}
        |\mathrm A(t)| \leq C\eps(t)^{\sigma_\star} D_\eps(\boldsymbol{w})(t) \qquad \textnormal{for all } t \in [0,T), \quad T \leq T_{\rm adv}\exp(\delta^{-p}).
    \end{equation}
\end{lemma}
\begin{proof}
    We split the operators $\cM_\eps$ and $\cJ_{\mathrm{app}}$ into their diagonal and nonlocal parts:
\begin{equation}
    \cM_\eps = \cM_d + \cM_{nl}, 
    \qquad 
    \cJ_{\mathrm{app}} = \cJ_d + \cJ_{nl}.
\end{equation}
The diagonal part of $\cM_\eps$ is given by
\begin{equation}
    \cM_d =
    \begin{pmatrix}
         A_{\eps,1} & 0 & 0\\
        0 &  A_{\eps,2} & 0\\
        0 & 0 &  A_{\eps,3}
    \end{pmatrix},
\end{equation}
whereas
\begin{equation}
    \cJ_d :=
    \begin{pmatrix}
        \left\{ \Phi_{{\rm app}, 1}^{(2),E}, \cdot \right\} & 0 & 0\\
        0 & \left\{ \Phi_{{\rm app}, 2}^{(2),E}, \cdot \right\} & 0\\
        0 & 0 & \left\{ \Phi_{{\rm app}, 3}^{(2),E}, \cdot \right\}
    \end{pmatrix}.
\end{equation}
The nonlocal part of $\cM_\eps$ is
\begin{equation}
    \cM_{nl} =
    \begin{pmatrix}
        \gamma_1^2 \Delta^{-1} &
        \gamma_1\gamma_2 \cT_{\eps,\boldsymbol{z}}^{1,2}\Delta^{-1} &
        \gamma_1\gamma_3 \cT_{\eps,\boldsymbol{z}}^{1,3}\Delta^{-1}
        \\
        \gamma_2\gamma_1 \cT_{\eps,\boldsymbol{z}}^{2,1}\Delta^{-1} &
        \gamma_2^2 \Delta^{-1} &
        \gamma_2\gamma_3 \cT_{\eps,\boldsymbol{z}}^{2,3}\Delta^{-1}
        \\
        \gamma_3\gamma_1 \cT_{\eps,\boldsymbol{z}}^{3,1}\Delta^{-1} &
        \gamma_3\gamma_2 \cT_{\eps,\boldsymbol{z}}^{3,2}\Delta^{-1} &
        \gamma_3^2 \Delta^{-1}
    \end{pmatrix}.
\end{equation}
Moreover,
\begin{equation}
    \cJ_{nl} :=
    \begin{pmatrix}
        \left\{\gamma_1 \Delta^{-1}(\cdot),\Omega_{{\rm app}, 1}^{(2),E}\right\}
        & \Upsilon_{\rm app}^{12}
        & \Upsilon_{\rm app}^{13}
        \\
        \Upsilon_{\rm app}^{21}
        & \left\{\gamma_2 \Delta^{-1}(\cdot),\Omega_{{\rm app}, 2}^{(2),E}\right\}
        & \Upsilon_{\rm app}^{23}
        \\
        \Upsilon_{\rm app}^{31}
        & \Upsilon_{\rm app}^{32}
        & \left\{\gamma_3 \Delta^{-1}(\cdot),\Omega_{{\rm app}, 3}^{(2),E}\right\}
    \end{pmatrix},
\end{equation}
where the off-diagonal operators $\Upsilon_{\rm app}^{i,j}$ are defined in \cref{eq:UpsilonIJApp}.
With this notation, the advection term in \cref{eq:advection} becomes
\begin{equation}
    \mathrm A = - (\mathrm A_{1} + \mathrm A_{2} + \mathrm A_{3})
\end{equation}
where 
\begin{align}
    & \mathrm A_{1} = \frac{1}{\delta}  \left\langle \cJ_d(\boldsymbol w),\cM_d(\boldsymbol w)\right\rangle,\\ 
    & \mathrm A_{2} = \frac{1}{\delta} \left\langle \cJ_{nl}(\boldsymbol w),\cM_{nl}(\boldsymbol w)\right\rangle, \\
    & \mathrm 
     A_3 = \frac{1}{\delta} \bigg(\left\langle \cJ_d(\boldsymbol w),\cM_{nl}(\boldsymbol w)\right\rangle
    + \left\langle \cJ_{nl}(\boldsymbol w),\cM_d(\boldsymbol w)\right\rangle
    \bigg).
\end{align}

\smallskip
\paragraph{$\blackinwhitediamond$ \textit{Bound on $\mathrm A_1$}}
Using the identity $\left\langle \{f,g\},hf\right\rangle
    =
    \frac{1}{2}
    \left\langle \{g,h\}f,f\right\rangle$,
with $f=w_i$, $g=\Phi_{{\rm app},i}^{(2),E}$ and $h=A_{\eps, i}$, we get
\begin{equation}
    \mathrm A_1
    = \frac{1}{\delta}
    \sum_{i=1}^3 \mathsf{A}_1^i, \qquad \mathsf{A}_1^i := 
    \frac{1}{2}
    \left\langle
        \left\{A_{\eps, i},\Phi_{{\rm app},i}^{(2),E}\right\}w_i,
        w_i
    \right\rangle.
\end{equation}
We recall that, in the inner region $\Ii$, the weight is chosen according to Arnold's functional relationship  $A_{\eps, i} = \gamma_i  F'_i\left(\Omega_{{\rm app},i}^{(2),E},t\right).$
Since the weight is constant in $\IIi$, by using Proposition \ref{prop:functionalRel}, it follows that
\begin{equation}
\begin{aligned}
    2 \mathsf A _{1}^i
    &= 
    \gamma_i \left\langle
        \mathbbm{1}_{\Ii}
        \left\{
             F'_i\left( \Omega_{{\rm app},i}^{(2),E},t\right),
            \Phi_{{\rm app},i}^{(2),E}
        \right\}w_i,
        w_i
    \right\rangle
    +
    \left\langle
        \mathbbm{1}_{\IIIi}
        \left\{
            A_{\eps, i},
            \Phi_{{\rm app},i}^{(2),E}
        \right\}w_i,
        w_i
    \right\rangle
    \\
    &= \gamma_i
    \left\langle
        \mathbbm{1}_{\Ii}
        \left\{F'_i\left(\Omega_{{\rm app},i}^{(2),E},t\right),
            \Theta_i
        \right\}w_i,
        w_i
    \right\rangle
    +
    \left\langle
        \mathbbm{1}_{\IIIi}
        \left\{
            A_{\eps, i},
            \Phi_{{\rm app},i}^{(2),E}
        \right\}w_i,
        w_i
    \right\rangle.
\end{aligned}
\end{equation}
The inner contribution is controlled by the approximate functional relationship. Indeed, using \cref{eq:NablaThetai}, for some $N \in \mathbb{N}$ we have 
\begin{equation}
    \left|
    \left\langle
        \mathbbm{1}_{\Ii}
        \left\{
            F_i'\left(\Omega_{{\rm app},i}^{(2),E}\right),
            \Theta_i
        \right\}w_i,
        w_i
    \right\rangle
    \right|
    \lesssim
    \eps^{3-N\sigma_1}
    \|\boldsymbol w\|_{\Xg}^2.
\end{equation}
In the outer region $\IIIi$, the weight is radial and satisfies $|\nabla A_{\eps, i}|
\lesssim|\xi|^{2{\widetilde{\sigma}}-1}A_{\eps, i}$.
Using the fact that $|\xi|\geq \eps^{-\sigma_2}$ in $\IIIi$, combined with the bound $\left|\nabla\Phi_{{\rm app},i}^{(2),E}\right| \lesssim \eps$ we get
\begin{equation}
    \left|
    \left\langle
        \mathbbm{1}_{\IIIi}
        \left\{
            A_{\eps, i},
            \Phi_{{\rm app},i}^{(2),E}
        \right\}w_i,
        w_i
    \right\rangle
    \right|
    \lesssim
    \eps^{1+\sigma_2}
    \|\mathrm{r}_\eps \boldsymbol w\|_{\Xg}^2.
\end{equation}
Hence,
\begin{equation}
    |\mathrm A_1|
    \lesssim \delta^{-1} \left(
    \eps^{3-N\sigma_1}
    \|\boldsymbol w\|_{\Xg}^2
    +
    \eps^{1+\sigma_2}
    \|\mathrm{r}_\eps \boldsymbol w\|_{\Xg}^2\right).
    \label{eq:AdvDiagTerm}
\end{equation}
Choosing $\sigma_1$ small enough such that $1 - \sigma_1 N > 0 $ and $\sigma_2>1$, we get
\begin{equation}
    |\mathrm A_1| \leq C \eps^s  D_\eps(\boldsymbol{w}),
\end{equation}
for $C>0$ and $s \in (0,1)$, which is consistent with the desired bound.
\smallskip
\paragraph{$\blackinwhitediamond$ \textit{Bound on $\mathrm A_2$}}
Setting 
\begin{equation}
    \mathcal{Q}_i(\boldsymbol w)
    :=
    \gamma_i \Delta^{-1}w_i
    + \sum_{j \neq i}
    \gamma_j \cT_{\eps,\boldsymbol{z}}^{i,j}\Delta^{-1}w_j,
    \label{eq:PiOmega}
\end{equation}
we can rewrite $\mathcal{J}_{nl}$ and $\mathcal{M}_{nl}$ as 
\begin{equation}
    \cJ_{nl}(\boldsymbol w)_i
    =
    \left\{\mathcal{Q}_i(\boldsymbol w),\Omega_{{\rm app},i}^{(2),E}\right\},
    \qquad
    \cM_{nl}(\boldsymbol w)_i
    =
    \gamma_i \mathcal{Q}_i(\boldsymbol w).
\end{equation}
Therefore, using the  identity $\left\langle \{f,g\},f\right\rangle = 0$,
we obtain
\begin{equation}
\begin{aligned}
    |\mathsf A_2|
    &= \frac{1}{\delta}
    \sum_{i=1}^3
    \gamma_i
    \left\langle
        \left\{\mathcal{Q}_i(\boldsymbol w),\Omega_{{\rm app},i}^{(2),E}\right\},
        \mathcal{Q}_i(\boldsymbol w)
    \right\rangle
    =0.
    \label{eq:AdvNLTerm}
\end{aligned}
\end{equation}

\smallskip
\paragraph{$\blackinwhitediamond$ \textit{Bound on $\mathrm A_3$}} Using $\left\langle \{f,g\},h\right\rangle =  \left\langle \{h,f\},g\right\rangle$, we have
\begin{equation}
\begin{aligned}
    \left\langle \cJ_d(\boldsymbol w),\cM_{nl}(\boldsymbol w)\right\rangle
    &=
    \sum_{i=1}^3
    \gamma_i
    \left\langle
        \left\{
            \mathcal{Q}_i(\boldsymbol w),
            \Phi_{{\rm app},i}^{(2),E}
        \right\},
        w_i
    \right\rangle.
\end{aligned}
\end{equation}
where $\mathcal{Q}_i(\boldsymbol w)$ is defined in \cref{eq:PiOmega}.
On the other hand, the second mixed term can be rewritten as follows.
\begin{equation}
\begin{aligned}
    \left\langle \cJ_{nl}(\boldsymbol w),\cM_d(\boldsymbol w)\right\rangle
    &=
    \sum_{i=1}^3
    \left\langle
        \left\{
            \mathcal{Q}_i(\boldsymbol w),
            \Omega_{{\rm app},i}^{(2),E}
        \right\},
        A_{\eps, i} w_i
    \right\rangle.
\end{aligned}
\end{equation}
Splitting the second term into  the inner region $\Ii$ (where $A_{\eps, i}= \gamma_i  F'_i\left(\Omega_{{\rm app},i}^{(2),E},t\right)$) and the complementary one, using $\nabla F_i\left( \Omega_{{\rm app},i}^{(2),E},t \right) =  F_{i}'\left(\Omega_{{\rm app},i}^{(2),E},t \right) \nabla  \Omega_{{\rm app},i}^{(2),E} $, we have
\begin{equation}
\begin{aligned}
    \left\langle \cJ_{nl}(\boldsymbol w),\cM_d(\boldsymbol w)\right\rangle
    &=
    \sum_{i=1}^3
    \gamma_i
    \left\langle
        \left\{
            \mathcal{Q}_i(\boldsymbol w),
            F_i\left(\Omega_{{\rm app},i}^{(2),E},t\right)
        \right\},
         \mathbbm{1}_{\Ii}w_i
    \right\rangle +
    \sum_{i=1}^3
    \left\langle
        \cJ_{nl}(\boldsymbol w)_i,
        \mathbbm{1}_{\Ii^c}\cM_d(\boldsymbol w)_i
    \right\rangle.
\end{aligned}
\end{equation}
Therefore, summing the two mixed terms we obtain
\begin{equation}
\begin{aligned}
    \left\langle \cJ_d(\boldsymbol w),\cM_{nl}(\boldsymbol w)\right\rangle
    +
    \left\langle \cJ_{nl}(\boldsymbol w),\cM_d(\boldsymbol w)\right\rangle & =
    \sum_{i=1}^3
    \gamma_i
    \left\langle
        \left\{
            \mathcal{Q}_i(\boldsymbol w),
            \Theta_i
        \right\},
         \mathbbm{1}_{\Ii}w_i
    \right\rangle 
    \\ & \quad + \sum_{i=1}^3
    \gamma_i
    \left\langle
        \left\{
            \mathcal{Q}_i(\boldsymbol w),
            \Phi_{{\rm app},i}^{(2),E}
        \right\}, \mathbbm{1}_{\Ii^c}
        w_i
    \right\rangle
    \\ & \quad +
    \sum_{i=1}^3
    \left\langle
        \cJ_{nl}(\boldsymbol w)_i,
        \mathbbm{1}_{\Ii^c}\cM_d(\boldsymbol w)_i
    \right\rangle.
\end{aligned}
\end{equation}
The first term is small because of the bound on the $\Theta_i$ in Proposition \ref{prop:functionalRel}, while the last two terms are supported outside the inner region and are controlled by the growth of the weight. 
Therefore,
\begin{equation}
\begin{aligned}
    &\left|
    \left\langle \cJ_d(\boldsymbol w),\cM_{nl}(\boldsymbol w)\right\rangle
    +
    \left\langle \cJ_{nl}(\boldsymbol w),\cM_d(\boldsymbol w)\right\rangle
    \right| \lesssim
    \eps^{3-N\sigma_1}
    \|\boldsymbol w\|_{\Xg}^2
    +
    \eps^{1+\sigma_2}
    \|\mathrm{r}_\eps\boldsymbol w\|_{\Xg}^2.
    \label{eq:AdvMixedTerms}
\end{aligned}
\end{equation}
Finally, combining \cref{eq:AdvNLTerm}, \cref{eq:AdvDiagTerm} and \cref{eq:AdvMixedTerms}, we have for the advection term $\mathsf{A}_3$ in \cref{eq:advection} that
\begin{equation}
    |\mathrm A_3| \lesssim \delta^{-1} \left( \eps^{3-\sigma_1N} \| \boldsymbol{w}\|_{\Xg}^2+ \eps^{1+\sigma_2} \| \mathrm{r}_\eps\boldsymbol{w}\|_{\Xg}^2 \right),
\end{equation}
and, choosing $\sigma_1$ small enough such that $1 - \sigma_1 N > 0 $ and $\sigma_2>1$  we get
\begin{equation}
    |\mathrm A_3| \leq C \eps^{\tilde{s}} D_\eps(\boldsymbol{w}),
\end{equation}
for $C>0$ and $\tilde{s} \in (0,1)$, which is consistent with the desired bound and therefore the lemma is proved upon defining $\sigma_\star$ properly.
\end{proof}

\subsubsection{\textbf{Bound of the nonlinear terms}}

We now estimate the terms that are nonlinear in the correction, as defined
in \cref{eq:NLterms}. We prove the following.
\begin{lemma}
\label{lem:nonlinear}
Let $\varsigma=1$ and    fix $p \in [0,1/4)$. Under the bootstrap hypothesis \cref{eq:BootstrapHP}, for $\sigma_1 >0$ small enough and $\sigma_2>1$ sufficiently large, there exists a constant $C>0$, independent of $\eps$ and $\delta$, such that 
    
    \begin{equation}
        |\mathrm N(t)| \leq \frac{1}{10} D_\eps(\boldsymbol{w})(t) + C \eps^3, \qquad \textit{for } t \in [0,T), \quad T \leq T_{\rm adv} \exp( \delta^{-p}).
        \label{eq:Nbound}
    \end{equation}
\end{lemma}
\begin{proof}
    Recalling \cref{eq:NLterms}, we write
    \begin{equation}
        \mathrm{N} =: \mathrm{N}_1 + \mathrm{N}_2 +  \mathrm{N}_3
    \end{equation}
    where
    \begin{align}
        &\mathrm{N}_1 :=\sum_{i=1}^3\langle {\mathsf{N}_1^i}(\boldsymbol{w},\boldsymbol{\rho}), \mathcal{M}_\eps(\boldsymbol{w})_i\rangle, \quad \mathrm{N}_2 :=\sum_{i=1}^3\langle {\mathsf{N}_2^i}(\boldsymbol{w},\boldsymbol{\rho}), \mathcal{M}_\eps(\boldsymbol{w})_i\rangle, \quad \mathrm{N_3} := \frac{1}{2}\langle \boldsymbol{w}, [t\partial_t,\cK_\eps](\boldsymbol{w}) \rangle.
    \end{align}
    with $\mathsf{N}_1^i$ and $\mathsf{N}_2^i$ defined respectively in \cref{eq:N1Definition} and \cref{eq:N2Definition}. We now bound each one of these terms.
    
    \smallskip
    \paragraph{$\blackinwhitediamond$ \textit{Bound on $\mathrm N_1$.}} 
    By \cref{eq:NijkDefinition}, we have 
    \begin{equation}
        \mathrm{N}_1 :=\sum_{i=1}^3\mathsf{N}_{1;i}
     \end{equation}
     where 
     \begin{equation}
          \mathsf{N}_{1;i}: =  \left\langle  \Big\{
            w_i,\gamma_i\varphi_i
            +\sum_{j\neq i}\gamma_j\cT_{\eps,\boldsymbol{z}+\boldsymbol{\rho}}^{i,j}\varphi_j
        \Big\}, \mathcal{M}_\eps(\boldsymbol{w})_i\right\rangle.
    \end{equation}
    By Lemma \ref{lem:estimatesMepsKeps}, 
    \begin{equation}
       \left | \mathsf{N}_{1;i}\right| \leq C  \Big(
        \|\nabla\varphi_i\|_{L^\infty}
        + \sum_{j\neq i}
        \|\nabla \cT_{\eps,\boldsymbol{z}+\boldsymbol{\rho}}^{i,j}\varphi_j\|_{L^\infty}
    \Big)
    \|\nabla \boldsymbol w\|_{\boldsymbol{\mathcal X}_\eps}
    \|\boldsymbol w\|_{\boldsymbol{ \mathcal{X}}_\eps}.
    \label{eq:N1FirstBound}
    \end{equation}

    Moreover, by Lemma \ref{lem:SobolevEstimates}, applied to
\(\varphi_i=\Delta^{-1}w_i\), together with the uniform bounds on the
translation operators, we have
\begin{align}
    \|\nabla\varphi_i\|_{L^\infty}
    +
    \sum_{j\neq i}
        \|\nabla \cT_{\eps,\boldsymbol{z}+\boldsymbol{\rho}}^{i,j}\varphi_j\|_{L^\infty}
    \leq
    C
    \left(
        \|\boldsymbol w\|_{\boldsymbol{ \mathcal{X}}_\eps}^{1/2}
        +
        \|\nabla \boldsymbol w\|_{\boldsymbol{ \mathcal{X}}_\eps}^{1/2}
    \right)
    \|\boldsymbol w\|_{\boldsymbol{ \mathcal{X}}_\eps}^{1/2}.
    \label{eq:VelocityN1Bound}
\end{align}
Combining \cref{eq:N1FirstBound} and \cref{eq:VelocityN1Bound}, we obtain
\begin{align}
    \left| \mathsf{N}_{1;i}
    \right|
    &\leq
    C
    \left(
        \|\boldsymbol w\|_{\boldsymbol{ \mathcal{X}}_\eps}^{1/2}
        +
        \|\nabla \boldsymbol w\|_{\boldsymbol{ \mathcal{X}}_\eps}^{1/2}
    \right)
    \|\boldsymbol w\|_{\boldsymbol{ \mathcal{X}}_\eps}^{3/2}
    \|\nabla\boldsymbol w\|_{\boldsymbol{ \mathcal{X}}_\eps}.
    \label{eq:N1PreYoung}
\end{align}
Applying Young's inequality and using the coercivity of the dissipation functional $D_\eps$ in Proposition \ref{prop:EnDiss}, we get
\begin{align}
    \left|\mathsf{N}_{1;i}
    \right|
    &\leq
    \frac{1}{100} D_\eps(\boldsymbol w)
    +
    C\left(
        \|\boldsymbol w\|_{\boldsymbol{ \mathcal{X}}_\eps}^{4}
        +
        \|\boldsymbol w\|_{\boldsymbol{ \mathcal{X}}_\eps}^{6}
    \right).
    \label{eq:EstimateN1}
\end{align}
    Therefore, 
    \begin{equation}
        |\mathrm{N}_1| \leq \frac{3}{100} D_{\eps} (\boldsymbol{w}) + C \left(
        \|\boldsymbol w\|_{\boldsymbol \cX_\eps}^{4}
        +
        \|\boldsymbol w\|_{\boldsymbol \cX_\eps}^{6}
    \right),
    \end{equation}
    which is consistent with the desired bound using the bootstrap hypothesis \cref{eq:BootstrapHP}.
    
    \smallskip
    \paragraph{$\blackinwhitediamond$ \textit{Bound on $\mathrm N_2$.}}
    By \cref{eq:NijkDefinition}, we can rewrite $\mathrm{N}_2$ as follows.
    \begin{equation}
        \mathrm{N}_2 :=\sum_{i=1}^3\mathsf{N}_{2;i}, \qquad \mathsf{N}_{2;i} := \left \langle\frac{ t}{\eps\dhk}
        \left\{
        \partial_t\rho_i\cdot\xi^\perp,
            w_i
        \right\}, \mathcal{M}_{\eps}(\boldsymbol{w})_i\right \rangle.
     \end{equation}
     By Lemma \ref{lem:estimatesMepsKeps} combined with Proposition \ref{prop:ControlZZhk}, we have
     \begin{equation}
         |\mathsf{N}_{2,i}| = \left | \left \langle \frac{ t}{\eps\dhk}
        \partial_t\rho_i\cdot
            \nabla w_i \,\mathcal{M}_{\eps}(\boldsymbol{w})_i\right\rangle\right| \lesssim \ \eps \, \delta^{1-2p} \| \nabla \boldsymbol{w} \|_{\boldsymbol{\mathcal{X}}_\eps} \| \boldsymbol{w}\|_{\boldsymbol{\mathcal{X}}_\eps}.
     \end{equation}
     By Young's inequality 
     \begin{equation}
         |\mathsf{N}_{2,i}| \leq\frac{1}{100} D_{\eps}(\boldsymbol{w}) + C \eps^2  \ \delta^{2-4p} \| \boldsymbol{w}\|_{\boldsymbol{\mathcal{X}}_\eps}^2,
     \end{equation}
     and, for $ p $ sufficiently small  we have that 
     \begin{equation}
         |\mathrm{N}_2| \leq \frac{3}{100} D_{\eps}(\boldsymbol{w}) + C \eps^3, 
     \end{equation}
     which is consistent with the desired bound under the bootstrap hypothesis \cref{eq:BootstrapHP}.
    \smallskip
    \paragraph{$\blackinwhitediamond$ \textit{Bound on $\mathrm N_3$.}}
    Finally, we control the commutator term $\mathrm N _3$. We rewrite it as 
    \begin{equation}
    \mathrm{N}_3 =\sum_{i = 1}^3 \mathsf \gamma_i N_{3; i}, \qquad \mathsf N_{3; i} := \frac{1}{2} \sum_{j \neq i}\gamma_j\left\langle w_i,
        [t\partial_t,\cT_{\eps,\boldsymbol{z}}^{i,j}\Delta^{-1}]( w_j)
    \right\rangle.
\end{equation}
For $ j \neq i$, we set 
\begin{equation}
    b_{ij}(t)
    :=
    \frac{z_i(t)-z_j(t)}{\eps(t)\dhk(t)}.
    \label{eq:CommutatorShiftDefinition}
\end{equation}
Then, 
\begin{equation}
    \cT_{\eps,\boldsymbol{z}}^{i,j}\Delta^{-1}w_j
    =
    \varphi_j\bigl(t,\xi+b_{ij}(t)\bigr),
\end{equation}
and, since $\Delta^{-1} (t \partial_t w_j) = t \partial_t \varphi_j$, we have
\begin{equation}
    [t\partial_t,
    \cT_{\eps,\boldsymbol{z}}^{i,j}\Delta^{-1}](w_j)
    =
    t\partial_t b_{ij}(t)\cdot
    \nabla\cT_{\eps,\boldsymbol{z}}^{i,j}\varphi_j,
    \label{eq:CommutatorExactFormula}
\end{equation}
where
\begin{equation}
    t\partial_t b_{ij}
    =
    \frac{
        t\partial_t z_{ij}\,\eps\dhk
        -
        z_{ij}\,t\partial_t\eps\,\dhk
        -
        z_{ij}\,\eps\,t\partial_t\dhk
    }{\eps^2\dhk^2}.
    \label{eq:ShiftDerivative}
\end{equation}
By the Helmholtz--Kirchhoff dynamics and the definitions of $\eps$ and
$\dhk$,
\begin{equation}
    |z_i-z_j|\lesssim\dhk,
    \qquad
    |t\partial_t(z_i-z_j)|\lesssim\frac{t\Gamma}{\dhk},
    \qquad
    |t\partial_t\eps|\lesssim\eps,
    \qquad
    |t\partial_t\dhk|\lesssim\dhk.
\end{equation}
Using also $t\Gamma/\dhk^2\lesssim1$, we deduce that
\begin{equation}
    |t\partial_t b_{ij}(t)|
    \lesssim
    \frac1\eps.
    \label{eq:CommutatorShiftBound}
\end{equation}
Then, by Lemma \ref{lem:estimatesMepsKeps}, we have
\begin{equation}
    | \mathrm{N}_3| \lesssim \eps^2 \| \boldsymbol{w} \|^2_{\boldsymbol{\mathcal{X}}_\eps},
\end{equation}
which agrees with the desired estimate thanks to the bootstrap hypothesis \cref{eq:BootstrapHP}.
\end{proof}
\subsubsection{\textbf{Improvement of the bootstrap}}
\label{sec:bootstrapimprovement}
We are finally ready to prove \cref{bd:goalw}. Combining the energy identity \cref{eq:tdtE} with Lemmas \ref{lem:forcing}, \ref{lem:linear}, \ref{lem:advection} and \ref{lem:nonlinear}, taking $\delta_0$ small enough, we get that 
\begin{align}
    t\de_tE_{\eps}(\boldsymbol{w})+\frac{1}{100}D_{\eps}(\boldsymbol{w})\leq C\eps^2,
\end{align}
where the constant $C$ does not depend on $\delta$ and $\widetilde{C}$ in \cref{eq:BootstrapHP}. 
Using Proposition \ref{prop:EnDiss} together with the
upper energy bound
\[
E_\eps(\boldsymbol w)
\leq C_E\|\boldsymbol w\|_{\Xg}^2
\]
from Lemma \ref{lem:estimatesMepsKeps}, we obtain
\[
D_\eps(\boldsymbol w)
\geq \frac{\kappa_D}{C_E}E_\eps(\boldsymbol w).
\]
Then, for 
$\kappa_\star$ small enough, we have
\begin{align}
    t\de_tE_{\eps}(\boldsymbol{w})+\kappa_\star E_{\eps}(\boldsymbol{w})\leq C\eps^2 \quad \Longrightarrow \quad E_{\eps}(\boldsymbol{w}(t))\leq\frac{C}{t^{\kappa_{\star}}}\int_{0}^t \frac{\eps^2(s)}{s^{1-\kappa_\star}}ds.
\end{align}
The integration from zero is justified by the short-time
bound, which gives 
\[r^{\kappa_\star}E_\eps(\boldsymbol w(r))
=\mathcal O(r^{\kappa_\star+1})\to0 \text{ as } r\to 0.
\]
Since $\eps(t)$ is an increasing function, we readily get that
\begin{align}
    E_{\eps}(\boldsymbol{w}(t))\leq \frac{C}{\kappa_\star} \eps(t)^2.
\end{align}
Therefore, taking 
$\widetilde C$ large enough leads to the desired improved bootstrap bound 
\cref{eq:BootstrapHP}.
This concludes the proof of Theorem \ref{thm:NonLinControl}
in the expanding case $\varsigma=1$.

\subsection{Adjustments for the collapsing case}
\label{sec:collapsing}
We finally consider the collapsing case $\varsigma=-1$, where we only highlight the main differences without repeating the whole proof.  We define 
\begin{equation}
    \label{def:relcollapsing}
    {\sf a}(t):=1-\frac{t}{T_{\rm adv}} \quad \Longrightarrow \quad \eps(t)^2
    =
    \frac{\delta}{|\alpha|}
    \frac{\frac{t}{T_{\rm adv}}}{{\sf a}(t)}, \qquad     t\partial_t\eps(t)
    =
\frac{1}{2{\sf a}(t)}\eps(t).
\end{equation}
and restrict the evolution to
\[
    0\leq t\leq(1-\delta^{p'})T_{\rm adv},
    \qquad
    0<p'<p'_0,
    \qquad
    p'_0<\frac1{10}.
\]
In particular,
\[
    {\sf a}(t)\geq\delta^{p'},
    \qquad
    \eps(t)^2\lesssim\delta^{1-p'}.
\]
With the notation in Theorem \ref{thm:NonLinControl}, we anticipate that we can fix
\[
    q(p')=5p',
    \qquad
    \tilde q(p')=2p'.
\]
This ensures that for $\delta$ sufficiently small we have
\[
    \frac{\delta^{-2}\eps^5}
         {\eps^{1-q}}
    \leq C
    \delta^{\frac{q(1-p')-4p'}{2}}
    =
    C\delta^{\frac{p'(1-5p')}{2}}\ll 1
\]
Moreover, by Proposition \ref{prop:ApproximateSol}, since $p'<1/5$ we have
$\eps^2{\sf a}^{-2}=\cO(\eps)$.  It follows that
\begin{equation}
\label{eq:collapsing-remainder}
    \delta^{-2}\cR_i
    =
    \cO_{\cZ}
    \left(
        \delta^{-2}\eps^5
        +
        \eps^2{\sf a}(t)^{-2}
    \right)
    =
    \cO_{\cZ}\bigl(\eps^{1-q}\bigr).
\end{equation}
Accordingly, the bootstrap assumption in the collapsing case becomes
\begin{equation}
\label{eq:collapsing-bootstrap}
    \|\boldsymbol w(t)\|_{\Xg}
    \leq
    4\widetilde C\eps(t)^{1-q}.
\end{equation}

We next reconsider the bounds for $\boldsymbol\rho$ in Proposition 
\ref{prop:ControlZZhk}. Let $\mathfrak B$ denote the autonomous matrix introduced in
\cref{sec:selfsimilar_HK}. Here the important observation is that, using  $\widetilde{\boldsymbol{\rho}}=(\overline{\boldsymbol{\rho}},\boldsymbol{\rho})$ we get that the 
linearized collapsing system of the analogue of \cref{eq:rhoExpandedSystem} takes the form
\[
    \partial_\tau g=  P \mathfrak B P g,
    \qquad
    \tau=-T_{\rm adv}\log{\sf a}(t), \qquad P = \begin{pmatrix}
        0 & I_3 \\ I_3 & 0
    \end{pmatrix}.
\]
Namely, we have simply considered negative times and changed the role of the first and second components in the enlarged system for the expanding case.
Since the spectrum of $\mathfrak B$ is symmetric, the spectral picture is unchanged.  Consequently,
\begin{equation}
\label{eq:collapsing-propagator}
    \left|\mathcal U_{\mathfrak L}(t,s)\right|
    \lesssim \left(1+\log \frac{\mathsf{a}(s)}{\mathsf{a}(t)} \right)
    \frac{\dhk(s)}{\dhk(t)},
    \qquad
    0\leq s\leq t<T_{\rm adv}.
\end{equation}
This is the reciprocal of the bounds in the expanding case. Then, writing
$\Omega_i=G+\delta\widetilde\Omega_i$, under
\cref{eq:collapsing-bootstrap} and exploiting \cref{eq:ApproxProfilesBound} we have
\[
    \|\widetilde{\boldsymbol\Omega}(t)\|_{\Xg}
    \lesssim
    \frac{\eps(t)^2}{\delta}
    +
    \frac{\eps(t)^2}{{\sf a}(t)}
    +
    \eps(t)^{1-q}.
\]
Notice also that
\[
    \delta
    \left(
        \frac{\eps^2}{\delta}
        +
        \frac{\eps^2}{{\sf a}}
        +
        \eps^{1-q}
    \right)
    \ll1
\]
on the time interval under consideration.  Hence the quadratic
contribution in $\widetilde{\boldsymbol\Omega}$ arising in the proof of Proposition \ref{prop:ControlZZhk} is perturbative. Repeating the bootstrap argument in that proof, now using
\cref{eq:collapsing-propagator}, gives
\begin{align}
    \frac{|\boldsymbol\rho(t)|}{\dhk(t)}
    \lesssim{}&
    \frac{1}{T_{\rm adv}{\sf a}(t)}
    \int_0^t \left(1+\log \frac{\mathsf{a}(s)}{\mathsf{a}(t)} \right)
    \left(
        \frac{|\boldsymbol\rho(s)|}{\dhk(s)}
    \right)^2
    \,\dd s
    \notag\\
    &+
    \frac{1}{{\sf a}(t)}
    \int_0^t \left(1+\log \frac{\mathsf{a}(s)}{\mathsf{a}(t)} \right)
    \frac{{\sf a}(s)\eps(s)^4}{s}
    \left(
        \frac{\eps(s)^2}{\delta}
        +
        \frac{\eps(s)^2}{{\sf a}(s)}
        +
        \eps(s)^{1-q}
    \right)
    \,\dd s
    +
    \cO(\eps^\infty).
\label{eq:collapsing-rho-integral}
\end{align}
Appealing to \cref{def:relcollapsing}, the leading order contribution arises from the $\delta^{-1}\eps^2$ term in parenthesis, which can be rewritten after a change of variable as
\begin{align}
    \frac{\delta^2}{{\sf a}(t)}
    \int_0^{t/T_{\rm adv}}&\left(1+\log \frac{1-x}{\mathsf{a}(t)} \right)
    \frac{x^2}{(1-x)^2}\,\dd x
    \\
    &=
    \frac{\delta^2}{{\sf a}(t)}
    \left(
        \frac{2}{{\sf a}(t)}
        +
        2\log{\sf a}(t)
        - 2 -(\log ({\sf a}(t)))^2
    \right)\lesssim
    \frac{\delta^2}{{\sf a}(t)^2}
    \left(\frac{t}{T_{\rm adv}}\right)^3.
\label{eq:collapsing-algebraic-loss}
\end{align}
This is the algebraic loss which replaces the logarithmic loss in the
expanding case.  The contributions of $\eps^2/{\sf a}$ and
$\eps^{1-q}$ in \cref{eq:collapsing-rho-integral}, as well as the
quadratic term in $\boldsymbol\rho$, are lower order when
${\sf a}(t)\geq\delta^{p'}$. From the bootstrap we therefore get
\begin{equation}
\label{eq:collapsing-rho}
    \frac{|\boldsymbol\rho(t)|}{\dhk(t)}
    \lesssim
    \delta^{1-2p'}\eps(t)^2, \qquad     \frac{|t\de_t\boldsymbol\rho(t)|}{\dhk(t)}
    \lesssim
    \delta^{1-2p'}
    \frac{\eps(t)^2}{{\sf a}(t)}.
\end{equation}
Consequently,
\begin{align}
    \frac{|t\de_t\boldsymbol\rho|}
         {\delta\eps\dhk}
    +
    \frac{\eps|\boldsymbol\rho|}
         {\delta^2\dhk}
    &\lesssim
    \delta^{-2p'}\frac{\eps}{{\sf a}}
    \lesssim
    \delta^{\frac{1-7p'}{2}}\ll 1
\label{eq:collapsing-translation-size}
\end{align}
These bounds are used to replace the estimates in \cref{eq:mainideas1non,eq:mainideas2non} to control the terms involving $\boldsymbol{\mathsf T}(\boldsymbol\rho)$ and those containing
$ \cT_{\eps,\boldsymbol z+\boldsymbol\rho}^{i,j}
    -
    \cT_{\eps,\boldsymbol z}^{i,j}.$ 
Since we can argue as in the expanding case to remove the leading contributions from $\de_1G$ and $\de_2G$  by  using
$\mathsf H(\de_\ell G)=0$, we can ensure that these terms are consistent with the bootstrap bound $\eps^{1-q}$ upon further reducing $p_0'$ if necessary. 

Regarding the weight, we note that we could use the standard one adopted in \cite{dolce2024long,dolce2026viscousvortexcrystals,zhang2025long}. Namely, we can define the regions as in Section~\ref{sec:weights}, 
but replace \cref{def:hepst} by the usual height
\[
    \mathsf h_\eps(t)
    =
    \frac14\eps(t)^{-2\sigma_1},
\]
without any logarithmic correction.  Indeed, in the intermediate
region,
\begin{equation}
    -\frac12
    \mathbbm{1}_{\IIi}
    t\de_tA_{\eps,i}
    =
    \frac{\sigma_1}{8{\sf a}(t)}
    \eps(t)^{-2\sigma_1}
    \mathbbm{1}_{\IIi}A_{\eps,i}.
\end{equation}
For $\delta$ sufficiently small, this provides more than the fixed
intermediate-region coercivity required in the proof of Proposition \ref{prop:EnDiss}.  In the inner region, we have the worst bounds in \cref{bd:tdetF2}, but these errors are still perturbative after reducing $\sigma_1,p_0'$ if necessary since
\[
    \frac{\eps^{2-N\sigma_1}}{{\sf a}}
    \lesssim
    \delta^{
        \frac{(2-N\sigma_1)(1-p')}{2}-p'
    }\ll 1.
\]
We can now repeat all the estimates of Section \ref{sec:NLbounds} with some attention to additional losses of factors $\sf a ^{-1}$ (that can be handled on the desired time interval properly setting the parameters) and conclude that 
there exists $\sigma_\star\in (0,1)$ such that
\begin{align}
    |\mathrm A(t)|+|\mathrm L(t)|
    &\leq
    C\delta^{\sigma_\star}
    D_\eps(\boldsymbol w)(t),
    \label{eq:collapsing-AL}\\
    |\mathrm F(t)|+|\mathrm N(t)|
    &\leq
    \frac1{50}D_\eps(\boldsymbol w)(t)
    +
    C\eps(t)^{2(1-q)}.
    \label{eq:collapsing-FN}
\end{align}
Combining these estimates with \cref{eq:tdtE}, and taking $\delta$
sufficiently small we can conclude the bootstrap as in Section \ref{sec:bootstrapimprovement}. \qed
\section{The vanishing viscosity vortex burst limit}
\label{sec:VanishingViscosity}
In this section, we prove Theorem \ref{theo:burst_intro} and give a quantitative
formulation of the vortex burst mechanism. Since the circulations of a
self-similar Gr\"obli configuration do not have the same sign, we cannot rely on the usual
Wasserstein distance between probability measures to characterize the limit. Following the point of view adopted in
\cite{CeciSeis24}, we therefore use its dual
Kantorovich--Rubinstein formulation. Let us briefly recall these notions before stating a more detailed version of Theorem \ref{theo:burst_intro}.
\subsubsection{Wasserstein distance and Kantorovich-Rubinstein formulation}
Let $\mathcal{P}_1(\RR^2)$ be the space of probability measures on $\RR^2$ with finite first order moments. 
For $\mu, \ \eta \in \mathcal{P}_1(\RR^2)$ their $1$-Wasserstein distance is defined by 
\begin{equation}
    W_1(\mu, \eta) := \inf_{\gamma \in \Pi(\mu, \eta)} \int_{\RR^2 \times \RR^2} |x-y| d \gamma(x,y),
\end{equation}
where $\Pi(\mu, \eta)$ is the set of couplings between $\mu$ and $\eta$.
\begin{remark}
    Note that for every $\mu\in\mathcal P_1(\mathbb R^2)$ and $y\in\mathbb R^2$,
    the only coupling of $\mu$ and $\delta_y$ is concentrated on
    $\mathbb R^2\times\{y\}$. Hence
    \begin{equation}
    \label{eq:W1DiracFormula}
        W_1(\mu,\delta_y)
        =
        \int_{\mathbb R^2}|x-y|\,d\mu(x).
    \end{equation}
\end{remark}
By Kantorovich-Rubinstein duality, we have the following equivalent representation
\begin{equation}
    W_1(\mu, \eta) = \sup_{ \substack{\varphi \in \textnormal{Lip}(\RR^2)\\ \textnormal{Lip}(\varphi) \leq 1}} \left | \int_{\RR^2} \varphi \ d(\mu - \eta) \right|.
\end{equation}
This dual formulation can actually be extended to finite signed Radon measure on $\RR^2$. Thus, let $\mu$ be a finite signed Radon measure on $\RR^2$ such that
\begin{equation}
    \mu(\RR^2) = 0,\qquad \int_{\RR^2} |x| d |\mu|(x) < +\infty.
\end{equation}
Then, we define the Kantorovich-Rubinstein norms as 
\begin{equation}
    \| \mu \|_{\rm KR} := \sup_{ \substack{\varphi \in \textnormal{Lip}(\RR^2)\\ \textnormal{Lip}(\varphi) \leq 1}}  \left| \int_{\RR^2} \varphi d\mu \right |.
\end{equation}
Accordingly, if $\eta_1$ and $\eta_2$ are finite signed Radon measures with the
same total mass and finite first order moments, we set
\begin{equation}
\label{def:KRDistance}
    d_{\rm KR}(\eta_1,\eta_2)
    :=
    \|\eta_1-\eta_2\|_{\rm KR}.
\end{equation}
If $\eta_1$ and $\eta_2$ are nonnegative measures with the same mass $m>0$, then
\begin{equation}
\label{eq:KRAndWasserstein}
    d_{\rm KR}(\eta_1,\eta_2)
    =
    m
    W_1\left(
        \frac{\eta_1}{m},
        \frac{\eta_2}{m}
    \right).
\end{equation}
Hence the Kantorovich--Rubinstein distance agrees with the usual
$1$-Wasserstein distance after normalization.

\subsection{Setting and vortex burst theorem}
In our case, each vortex has constant circulation 
\begin{equation}
    \int_{\RR^2} \omega_i(x,t) dx = \Gamma_i.
\end{equation}
 Thus,  setting
\begin{equation}
    \mu_i (t) : = \frac{\omega_i(t)}{\Gamma_i},
\end{equation}
we get a probability measure. It follows that
\begin{equation}
\label{eq:KRvsW1}
    d_{\rm KR} \left( \sum_{i=1}^3 \omega_i(t), \sum_{i=1}^3 \Gamma_i \delta_{Z_i(t)}\right) \leq \sum_{i=1}^3 | \Gamma_i| W_1(\mu_i(t), \delta_{Z_i(t)}).
\end{equation}

Throughout this section, we fix an expanding normalized Gr\"obli
configuration, so that $\varsigma=1$ (hence $\alpha>0$), and an exponent
$p\in(0,p_0)$ as in Theorem \ref{thm:main_intro}.
Fix a limiting phase $\beta_*\in\mathbb R$ and a reference time $t_\star>0$, and let $d_\nu>0$ be the initial distance  satisfying
\begin{equation}
\label{eq:BurstScaleFinal}
    d_\nu\xrightarrow[\nu\to0]{}0,
    \qquad
    \frac{d_\nu^2}{\Gamma}
    \exp\left(
        \left(\frac{\Gamma}{\nu}\right)^p
    \right)
    \xrightarrow[\nu\to0]{}+\infty.
\end{equation}
Then, we define 
\begin{equation}
\label{eq:BurstPhaseFinal}
    \beta_\nu
    :=
    \beta_*
    -
    \frac{\beta}{\alpha}\log\Big(\alpha\frac{\Gamma t_\star}{d^2_\nu}\Big)\qquad     \mod 2\pi
\end{equation}
and consider the initial vorticity
\begin{equation}
\label{eq:BurstInitialDataFinal}
    \omega_\nu^{\rm in}
    :=
    \sum_{i=1}^3
    \Gamma_i
    \delta_{d_\nu e^{\mathrm i\beta_\nu}{\sf n}_i}.
\end{equation}
We then introduce the signed measure
\begin{equation}
\label{def:NormalizedTriangleMeasure}
\mu_\star
    :=
    \sum_{i=1}^3
    \Gamma_i\delta_{{\sf n}_i}
    -
    \Gamma_{\rm tot}\delta_0, \qquad \Gamma_{\rm tot}
    :=
    \sum_{i=1}^3\Gamma_i.
\end{equation}
Since the normalized Gr\"obli triangle is nondegenerate,
$\mu_\star\neq0$. We define
\begin{equation}
\label{def:BurstGeometricConstant}
    c_{\Gamma,{\sf n}}
    :=
    \frac{1}{\Gamma}
    \|\mu_\star\|_{\rm KR}
    >0.
\end{equation}
For $t>0$, let
\begin{equation}
\label{def:BurstTrajectoryFinal}
    z_{0,i}(t)
    :=
    \sqrt{\alpha\Gamma t}\,
    \exp\left(
        \mathrm i\beta_*
        +
        \mathrm i\frac{\beta}{\alpha}\log \frac{t}{t_\star}
    \right)
    {\sf n}_i,
    \qquad i=1,2,3,
\end{equation}
and set $z_{0,i}(0):=0$. This is the expected Helmholtz-Kirchhoff position of the limiting branch, and therefore we define 
\begin{equation}
    \omega_{0}^{\rm b}(t) :=
    \sum_{i=1}^3
    \Gamma_i\delta_{z_{0,i}(t)},
    \qquad
    t\geq0.
    \label{def:limitingBranch}
\end{equation}
Note that 
$\omega_{0}^{\rm b}(0) = \Gamma_{\rm tot} \delta_0$. Finally, we fix a final target distance $d_{\rm fin} = d_{\rm fin}(\nu) > d_\nu$ reached at time $T_{\rm fin}$, i.e. by the Helmholtz-Kirchhoff formula,
\begin{equation}
    d_{\rm fin}^2 = d_\nu^2 \left( 1+ \frac{T_{\rm fin}}{T_{\rm adv, \nu}} \right), \qquad T_{\rm adv, \nu} := \frac{d_\nu^2}{ \alpha \Gamma}.
\end{equation}
Then, 
\begin{equation}
    T_{\rm fin} = T_{\rm adv, \nu} \left(\frac{d^2_{\rm fin} - d^2_\nu}{ d_\nu ^2} \right) = \frac{d^2_{\rm fin} - d^2_\nu}{\alpha \Gamma}.
    \label{eq:TfinDef}
\end{equation}
To have that $T_{\rm fin}$ lies in the time interval of validity of Theorem \ref{thm:main_intro}, we assume that
\begin{equation}
    \frac{d_{\rm fin}^2}{d_\nu^2} \exp(-\delta^{-p}) \xrightarrow[\nu\to0]{}0 \quad \Longrightarrow \quad T_{\rm fin} \ll T_{\rm adv, \nu} \exp (  \delta^{-p}).
    \label{eq:admissabilityT}
\end{equation}
For later use, we note that for \cref{eq:BurstInitialDataFinal} and \cref{def:limitingBranch}, for every $t\in[0,T_{\rm fin}]$,  we have that 
\begin{equation}
\label{eq:InitialKRExact}
    \frac{1}{\Gamma}
    d_{\rm KR}
    \left(
        \omega_\nu^{\rm in},
        \Gamma_{\rm tot}\delta_0
    \right)
    =
    c_{\Gamma,{\sf n}}d_\nu, \qquad \frac{1}{\Gamma}
    d_{\rm KR}
    \left(
        \omega_0^{\rm b}(t),
        \Gamma_{\rm tot}\delta_0
    \right)
    =
    c_{\Gamma,{\sf n}}
    \sqrt{\alpha\Gamma t}.
\end{equation}

\begin{figure}[tbp]
    \centering    \includegraphics[width=0.95\textwidth]{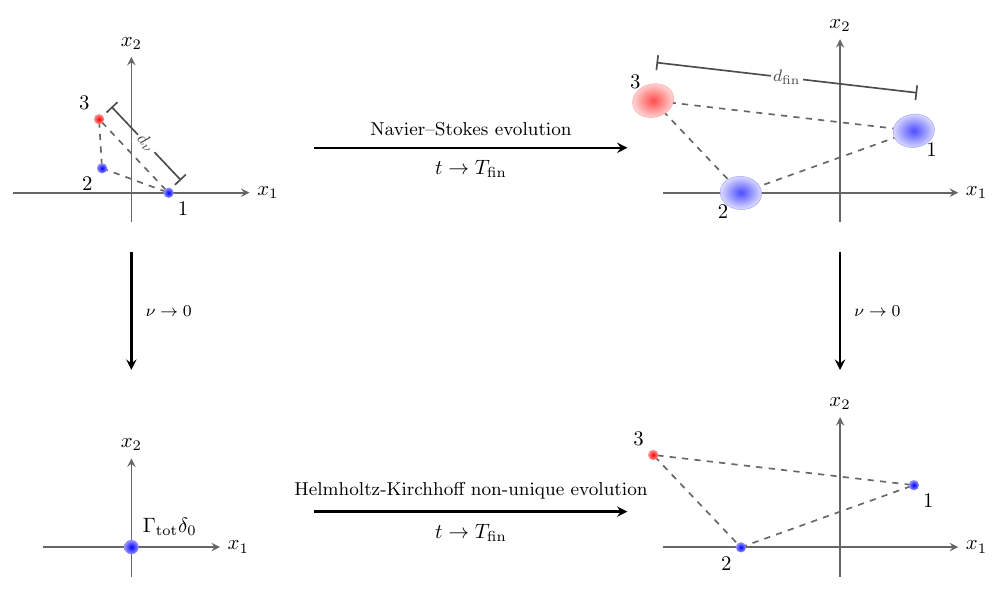} 
    \caption{A schematic picture of the vanishing viscosity vortex burst limit.}
    \label{fig:burst}
\end{figure}

Then, we have the following.
\begin{theorem}
    Let $\omega_\nu$ be the unique Navier--Stokes solution with initial
    datum \cref{eq:BurstInitialDataFinal}, where
    \cref{eq:BurstScaleFinal,eq:BurstPhaseFinal} hold.
    Let $d_{\rm fin}$ satisfying \cref{eq:admissabilityT} and $T_{\rm fin}$ be defined in \cref{eq:TfinDef}.
    Suppose that
    \begin{equation}
        \sqrt{\delta} d_{\rm fin} \xrightarrow[\nu \to 0]{} 0, \qquad \liminf_{\nu \to 0} d_{\rm fin} > 0.
     \end{equation}
    Then, for every
$T\leq  \lim \inf_{\nu \to 0} T_{\rm fin}$,  
    \begin{equation}
    \label{eq:UniformBurstKRConvergence}
    \sup_{0\leq t\leq T_{\rm fin}}
    \frac{1}{\Gamma}
    d_{\rm KR}
    \left(
    \omega_\nu(t),
    \omega_0^{\rm b}(t)
    \right)
    \xrightarrow[\nu\to0]{}0, \quad \Longrightarrow\quad \omega_\nu
    \xrightharpoonup[\nu\to0]{*}
    \omega_0^{\rm b}
    \qquad
    \text{in }
    C\left(
        [0,T];
        \mathcal M(\mathbb R^2)
    \right)
    \end{equation}
in the weak-$*$ topology.
Moreover,
\begin{equation}
\label{eq:BurstInitialCollapseLimit}
\frac{1}{\Gamma}
d_{\rm KR}
\left(
\omega_\nu^{\rm in},
\Gamma_{\rm tot}\delta_0
\right)
\xrightarrow[\nu\to0]{}0,
\end{equation}
whereas
\begin{equation}
\label{eq:PositiveBurstSeparation}
\liminf_{\nu\to0}
\frac{1}{\Gamma}
d_{\rm KR}
\left(
\omega_\nu(T_{\rm fin}),
\Gamma_{\rm tot}\delta_0
\right)
>0.
\end{equation}
\label{thm:ViscousBurst}
\end{theorem}
Therefore, although the initial vorticities converge to the collapsed
measure $\Gamma_{\rm tot}\delta_0$, their Navier--Stokes evolutions
converge to the expanding branch $\omega_0^{\rm b}$ and reach
configurations that remain separated from the collapsed state once the vanishing viscosity limit is considered.

\begin{remark}
    To recover Theorem \ref{theo:burst_intro}, fix $T>0$ and choose
    \begin{equation}
        d_{\rm fin}(\nu)
        :=
        \sqrt{d_\nu^2+ 2\alpha\Gamma T}.
    \end{equation}
    Then $T_{\rm fin}=2T$. Moreover, \cref{eq:BurstScaleFinal} implies
    \cref{eq:admissabilityT}, while
    $\sqrt{\delta}\,d_{\rm fin}\to0$ and
    $d_{\rm fin}\to\sqrt{2\alpha\Gamma T}>0$. Hence Theorem
    \ref{theo:burst_intro} follows from Theorem \ref{thm:ViscousBurst}.
\end{remark}
\begin{proof}
    Thanks to condition \cref{eq:BurstScaleFinal} and \cref{eq:admissabilityT} we have that
    \begin{equation}
        \frac{T_{\rm fin}}{T_{\adv,\nu} \exp( \delta^{-p})} = \left(
            \frac{d_{\rm fin}^2}{d_\nu^2}-1
         \right)
            \exp(-\delta^{-p})
        \xrightarrow[\nu\to0]{}0.
    \end{equation}
    Thus, for all sufficiently small $\nu$, the whole interval
    $[0,T_{\rm fin}]$ is contained in the validity interval of Theorem
\ref{thm:main_intro}. 
We claim that,
\begin{equation}
    \sup_{0\leq t\leq T_{\rm fin}}\frac{1}{\Gamma}
    d_{\rm KR}
    \left(
        \omega_\nu(t),
        \omega_0^{\rm b}(t)
    \right)
    \leq
    C\left(d_\nu + \sqrt{\delta} d_{\rm fin}\right),
    \label{eq:BoundDKR}
\end{equation}
\begin{equation}
\label{eq:ViscousBurstSeparation}
\left|
\frac{1}{\Gamma}
d_{\rm KR}
\left(
\omega_\nu(t),
\Gamma_{\rm tot}\delta_0
\right)
-
c_{\Gamma,{\sf n}}\sqrt{\alpha\Gamma t}
\right|
\leq
C
\left(
d_\nu+\sqrt{\delta}d_{\rm fin}
\right),
\end{equation}
for every $t \in [0,T_{\rm fin}]$ and, in particular, at the target time $T_{\rm fin}$,
\begin{equation}
\label{eq:TargetBurstSeparation}
\left|
\frac{1}{\Gamma}
d_{\rm KR}
\left(
\omega_\nu(T_{\rm fin}),
\Gamma_{\rm tot}\delta_0
\right)
-
c_{\Gamma,{\sf n}}
\sqrt{d_{\rm fin}^2-d_\nu^2}
\right|
\leq
C
\left(
d_\nu+\sqrt{\delta}d_{\rm fin}
\right).
\end{equation}
Before proving the claim, we provide the proof of Theorem \ref{thm:ViscousBurst} provided that the claim holds true. 
Since $d_\nu \to 0 $ and $\sqrt{\delta} d_{\rm fin} \to 0$ as $\nu \to 0$, \cref{eq:UniformBurstKRConvergence} follows from \cref{eq:BoundDKR}.

Fix a finite $T>0$ independently of $\nu$.
By \cref{eq:BurstScaleFinal}, the approximation horizon
$T_{\rm adv,\nu}\exp(\delta^{-p})$ tends to infinity,
and therefore contains $[0,T]$ for sufficiently small $\nu$.
Applying the three-term estimates below on this fixed
interval gives
\[
\sup_{0\leq t\leq T}
\frac{1}{\Gamma}
d_{\rm KR}\bigl(\omega_\nu(t),\omega_0^{\rm b}(t)\bigr)
\leq C\bigl(d_\nu+\sqrt{\nu T}\bigr)\longrightarrow0.
\]
This also covers the endpoint allowed in the theorem.
The convergence in the Kantorovich--Rubinstein distance, together
with the uniform bound on the total variations
\begin{equation}
    \|\omega_\nu(t)\|_{\mathcal M}
    \leq
    \sum_{i=1}^3|\Gamma_i|,
\end{equation}
implies weak-$*$ convergence of measures, uniformly on $[0,T]$.
Indeed, let
\begin{equation}
    \sigma_\nu(t)
    :=
    \omega_\nu(t)-\omega_0^{\rm b}(t).
\end{equation}
For any $\varphi\in C_0(\mathbb R^2)$ and $\eta>0$, choose
$\psi\in C_c^\infty(\mathbb R^2)$ such that
$\|\varphi-\psi\|_{L^\infty}\leq\eta$. Since
$\sigma_\nu(t)(\mathbb R^2)=0$, we have
\begin{equation}
    \left|
        \int_{\mathbb R^2}
        \psi\,d\sigma_\nu(t)
    \right|
    \leq
    {\rm Lip}(\psi)
    \|\sigma_\nu(t)\|_{\rm KR}.
\end{equation}
Moreover,
\begin{equation}
    \left|
        \int_{\mathbb R^2}
        (\varphi-\psi)\, d\sigma_\nu(t)
    \right|
    \leq
    \eta
    \left(
        \|\omega_\nu(t)\|_{\mathcal M}
        +
        \|\omega_0^{\rm b}(t)\|_{\mathcal M}
    \right)
    \leq
    2\eta\sum_{i=1}^3|\Gamma_i|.
\end{equation}
Therefore,
\begin{equation}
    \begin{aligned}
    \sup_{t\in[0,T]}
    \left|
        \int_{\mathbb R^2}
        \varphi\,
         d\sigma_\nu(t)
    \right|
    &\leq
    {\rm Lip}(\psi)
    \sup_{t\in[0,T]}
    \|\sigma_\nu(t)\|_{\rm KR}
    +
    2\eta\sum_{i=1}^3|\Gamma_i|.
\end{aligned}
\end{equation}
We first let $\nu\to0$ and then $\eta\to0$, which proves the
weak-$*$ convergence uniformly on $[0,T]$.

The initial collapse
\cref{eq:BurstInitialCollapseLimit} follows directly from
\cref{eq:InitialKRExact} and $d_\nu\to0$. On the other hand,
\cref{eq:TargetBurstSeparation} and the fact that $d_\nu \to 0$ and $\sqrt{\delta} d_{\rm fin} \to 0$ imply
\begin{equation}
\frac{1}{\Gamma}
d_{\rm KR}
\left(
\omega_\nu(T_{\rm fin}),
\Gamma_{\rm tot}\delta_0
\right)
=
c_{\Gamma,{\sf n}}
d_{\rm fin}
+
o(1),
\end{equation}
and, since 
\begin{equation}
\liminf_{\nu\to0}d_{\rm fin}>0,
\end{equation}
then
\cref{eq:PositiveBurstSeparation} follows.

We are now left to prove \cref{eq:BoundDKR}, \cref{eq:ViscousBurstSeparation} and \cref{eq:TargetBurstSeparation}.
Let $z_{i,\nu}(t)$ denote the expanding Helmholtz--Kirchhoff trajectory with
initial condition
\begin{equation}
    z_{i,\nu}(0)
    =
    d_\nu e^{\mathrm i\beta_\nu}{\sf n}_i.
\end{equation}
For every $t\in[0,T_{\sf fin}]$,
we have
\begin{align}
d_{\rm KR}
\left(
    \omega_\nu(t),
    \omega_0^{\rm b}(t)
\right) & \leq \mathcal{D}_1(t) + \mathcal{D}_2(t) + \mathcal{D}_3(t)
\end{align}
where 
\begin{align}
& \mathcal{D}_1 := d_{\rm KR} 
\left(
    \omega_\nu(t),
    \sum_{i=1}^3\Gamma_i\delta_{Z_i(t)}
\right),\\
&\mathcal{D}_2 := d_{\rm KR}
\left(
    \sum_{i=1}^3\Gamma_i\delta_{Z_i(t)},
    \sum_{i=1}^3\Gamma_i\delta_{z_{i,\nu}(t)}
\right),\\ & \mathcal{D}_3 :=
d_{\rm KR}
\left(
    \sum_{i=1}^3\Gamma_i\delta_{z_{i,\nu}(t)},
    \omega_0^{\rm b}(t)
\right).
\label{eq:BurstMainTriangleInequality}
\end{align}
We bound each one of the terms above separately.
    By the self-similar ansatz \cref{eq:vorticityAnsatz}, 
    \begin{equation}
        \omega_\nu = \sum_{i=1}^3 \omega_{i,\nu}(x,t), \qquad \omega_{i,\nu}(x,t):= \frac{\Gamma_i}{\nu t} \Omega_i \left( \frac{x - Z_i(t)}{\sqrt{ \nu t}}, t \right)
    \end{equation}
    Moreover, Lemma \ref{lem:PositionChoice} gives 
    \begin{equation}
    \int_{\mathbb R^2}\Omega_i(\xi,t)\,\dd\xi=1.
\end{equation}
The maximum principle applied to \cref{eq:OmegaIEq}, since $\Omega_i(\cdot,0) = G \geq 0 $, yields that $\Omega_i(\xi,t) \geq 0$. Therefore, 
\begin{equation}
    \mu_{i,\nu}(t) := \frac{\omega_{i,\nu}(t)}{\Gamma_i}
\end{equation}
is a probability measure. We claim that
\begin{equation}
    \sup_{0<t<T_{\rm fin}} \int_{\RR^2} |\xi| \Omega_{i}(\xi,t) d \xi \leq C. 
    \label{eq:MomentBound}
\end{equation}
Recall that, 
\begin{equation}
    \Omega_i =  G + \sum_{n=2}^4 \eps(t)^n\Omega_{n,i} + \delta w_i.
\end{equation}
By \cref{eq:ApproxProfilesBoundE,eq:ApproxProfilesBound} in the expanding case,  $\Omega_{n,i} \in \cZ$ uniformly and we have that
\begin{equation}
    \int_{\RR^2} |\xi| |\Omega_{n,i}(\xi,t) |d \xi \leq C, \qquad n = 2,3,4.
\end{equation}
Moreover, for the nonlinear correction, by Theorem \ref{thm:NonLinControl},  we have
\begin{equation}
    \int_{\mathbb R^2}
    |\xi||w_i(\xi,t)|
    \,\dd\xi
    \leq
    C
    \|\boldsymbol{w}(t)\|_{\boldsymbol{\mathcal X}_{\eps}} \leq C  \eps(t).
\end{equation}
proving the claim.

Since one of the measures in the Wasserstein distance is a Dirac delta, \cref{eq:W1DiracFormula} and the change of variables $x = Z_i(t) + \sqrt{\nu t} \ \xi$ give
\begin{align}
    W_1
    \left(
        \mu_{i,\nu}(t),
        \delta_{Z_i(t)}
    \right)
    &=
    \int_{\mathbb R^2}
    |x-Z_i(t)|
    \,\mathrm d\mu_{i,\nu}(t)(x)=
    \sqrt{\nu t}
    \int_{\mathbb R^2}
    |\xi|\Omega_i(\xi,t)
    \,\dd\xi.
\end{align}
Therefore, thanks to \cref{eq:MomentBound} we have
\begin{equation}
\label{eq:ComponentW1Final}
    W_1
    \left(
        \mu_{i,\nu}(t),
        \delta_{Z_i(t)}
    \right)
    \leq
    C\sqrt{\nu t}.
\end{equation}
Using \cref{eq:KRvsW1},
\cref{eq:ComponentW1Final} implies
\begin{equation}
\label{eq:ExactVorticityToExactCentersFinal}
    \mathcal{D}_1(t)
    \leq
    C\Gamma\sqrt{\nu t}.
\end{equation}
We now treat $\mathcal{D}_2$. 
By \cref{eq:main_bound_intro}, \(\epsilon\lesssim\sqrt\delta\), and \(d_{\rm HK}\epsilon=\sqrt{\nu t}\), \begin{equation}
\label{eq:ExactCentersToHKFinal}
\mathcal D_2(t)
\leq \sum_i|\Gamma_i||Z_i(t)-z_{i,\nu}(t)|
\leq C\Gamma d_{\rm HK}(t)\sqrt{\delta}\,\eps(t)^2
\leq C\Gamma\delta\sqrt{\nu t}.
\end{equation}
It remains to estimate $\mathcal{D}_3$.
Using \cref{def:dhk} and the phase choice
\cref{eq:BurstPhaseFinal}, we can rewrite
\begin{equation}
\label{eq:RenormalizedTrajectoryFinal}
    z_{i,\nu}(t)
    =
    \sqrt{d_\nu^2+\alpha\Gamma t}\,
    \exp\left(
        \mathrm i\beta_*
        +
        \mathrm i\frac{\beta}{\alpha}
        \log\left( \frac{
            t+\frac{d_\nu^2}{\alpha\Gamma}}{t_\star}
        \right)
    \right)
    {\sf n}_i.
\end{equation}
Set
\begin{equation}
    s_\nu
    :=
    \frac{d_\nu^2}{\alpha\Gamma}.
\end{equation}
The radial factors satisfy
\begin{equation}
    \left|
        \sqrt{d_\nu^2+\alpha\Gamma t}
        -
        \sqrt{\alpha\Gamma t}
    \right|
    \leq
    d_\nu.
\end{equation}
For the angular factors, using $
    |e^{\mathrm ia}-e^{\mathrm ib}|
    \leq
    \min\{2,|a-b|\}$, we get 
\begin{equation}
    \sqrt{t}
    \left|
        e^{\mathrm i\frac{\beta}{\alpha}\log(\frac{t+s_\nu}{t_\star})}
        -
        e^{\mathrm i\frac{\beta}{\alpha}\log \frac{t}{t_\star}}
    \right|
    \leq
    C\sqrt{s_\nu},
    \qquad t>0.
\end{equation}
Indeed, if $0<t\leq s_\nu$, then
\begin{equation}
    \sqrt t
    \left|
        e^{\mathrm i\frac{\beta}{\alpha}\log(\frac{t+s_\nu}{t_\star})}
        -
        e^{\mathrm i\frac{\beta}{\alpha}\log \frac{t}{t_\star}}
    \right|
    \leq
    2\sqrt{s_\nu},
\end{equation}
while, if $t\geq s_\nu$, we have
\begin{align}
    \sqrt t
    \left|
        e^{\mathrm i\frac{\beta}{\alpha}\log(\frac{t+s_\nu}{t_\star})}
        -
        e^{\mathrm i\frac{\beta}{\alpha}\log \frac{t}{t_\star}}
    \right|
    \leq
    C\sqrt t
    \log\left(1+\frac{s_\nu}{t}\right) \leq
    C\frac{s_\nu}{\sqrt t}
    \leq
    C\sqrt{s_\nu}.
\end{align}
Since $\sqrt{\alpha\Gamma s_\nu}= d_\nu $
we conclude that
\begin{equation}
\label{eq:HKToBurstTrajectoryFinal}
    \sup_{0\leq t\leq T_{\rm fin}}
    |z_{i,\nu}(t)-z_{0,i}(t)|
    \leq
    Cd_\nu.
\end{equation}
Therefore,
\begin{equation}
\label{eq:HKToBurstKRFinal}
    \sup_{0\leq t\leq T_{\rm fin}}
    \mathcal{D}_3
    \leq
    C\  \Gamma d_\nu.
\end{equation}
Combining
\cref{eq:ExactVorticityToExactCentersFinal}, \cref{eq:ExactCentersToHKFinal} and \cref{eq:HKToBurstKRFinal}
with the fact that for $t \leq T_{\rm fin}$ we have that $\sqrt{\nu t} \leq C \sqrt{\delta} d_{\rm fin}$, we get 
\begin{equation}
    \sup_{0\leq t\leq T_{\rm fin}}\frac{1}{\Gamma}
    d_{\rm KR}
    \left(
        \omega_\nu(t),
        \omega_0^{\rm b}(t)
    \right)
    \leq
    C\left(d_\nu + \sqrt{\delta} d_{\rm fin}\right).
\end{equation}
which proves
\cref{eq:BoundDKR}.
To prove the identities in
\cref{eq:InitialKRExact}, let us define
\begin{equation}
    S_{r,\vartheta}(x)
    :=
    r e^{\mathrm i\vartheta}x.
\end{equation}
Since
\begin{equation}
    \omega_\nu^{\rm in}
    -
    \Gamma_{\rm tot}\delta_0
    =
    (S_{d_\nu,\beta_\nu})_\#\mu_\star,
\end{equation}
the rotational invariance and the homogeneity of the
Kantorovich--Rubinstein norm give
\begin{equation}
    \frac{1}{\Gamma}
    d_{\rm KR}
    \left(
        \omega_\nu^{\rm in},
        \Gamma_{\rm tot}\delta_0
    \right)
    =
    \frac{d_\nu}{\Gamma}
    \|\mu_\star\|_{\rm KR}
    =
    c_{\Gamma,{\sf n}}d_\nu.
\end{equation}
Similarly, for $t>0$,
\begin{equation}
    \omega_0^{\rm b}(t)
    -
    \Gamma_{\rm tot}\delta_0
    =
    \left(
        S_{
            \sqrt{\alpha\Gamma t},
            \,
            \beta_*+\frac{\beta}{\alpha}\log \frac{t}{t_\star}
        }
    \right)_\#
\mu_\star.
\end{equation}
Thus,
\begin{equation}
    \frac{1}{\Gamma}
    d_{\rm KR}
    \left(
        \omega_0^{\rm b}(t),
        \Gamma_{\rm tot}\delta_0
    \right)
    =
    c_{\Gamma,{\sf n}}
    \sqrt{\alpha\Gamma t},
\end{equation}
and the same identity holds at $t=0$, which shows \cref{eq:InitialKRExact}.

Finally, the reverse triangle inequality gives
\begin{equation}
    \left|
    d_{\rm KR}
    \left(
    \omega_\nu(t),
    \Gamma_{\rm tot}\delta_0
    \right)
    -
    d_{\rm KR}
    \left(
    \omega_0^{\rm b}(t),
    \Gamma_{\rm tot}\delta_0
    \right)
    \right|
    \leq
    d_{\rm KR}
    \left(
    \omega_\nu(t),
    \omega_0^{\rm b}(t)
    \right).
\end{equation}
Combining it with \cref{eq:BoundDKR} and \cref{eq:InitialKRExact} we get \cref{eq:ViscousBurstSeparation}.
At $t=T_{\rm fin}$, \cref{eq:TfinDef} gives
\begin{equation}
\alpha\Gamma T_{\rm fin}
=
d_{\rm fin}^2-d_\nu^2,
\end{equation}
and \cref{eq:TargetBurstSeparation} follows.
\end{proof}

\appendix
\section{Exceptional configurations on Aref's circle}
\label{app:explicit}
Let us compute the coefficient $\kappa$ appearing in the proof of Lemma \ref{lem:SpectralPropertiesB}, which determines if $\mathfrak{B}$ is diagonalizable or not.  We work directly in the Aref
normalization introduced in \cref{eq:aref_circle_app}, and assume that vortices $1$ and
$2$ are at distance one.  Since in our coordinates the maximal distance is normalized to one, 
we write them as ${\rm n}_j$ instead of ${\sf n}_j$. We always renormalize the circulations so that
\begin{equation}
\label{eq:scaled_circulations_jordan}
    \upgamma_j:=\frac{\gamma_j}{\gamma_1+\gamma_2},
    \qquad j=1,2,3.
\end{equation}
Then
\begin{equation}
\label{eq:scaled_circulation_relations_jordan}
    \upgamma_1+\upgamma_2=1,
    \qquad
    \upgamma_3=-\upgamma_1\upgamma_2,
\end{equation}
and the positions on Aref's circle are
\begin{equation}
\label{eq:aref_positions_jordan}
    {\rm n}_1=\upgamma_2,
    \qquad
    {\rm n}_2=-\upgamma_1,
    \qquad
    {\rm n}_3=z=x_1+\mathrm{i}x_2,
    \qquad
    |z|^2=1+\upgamma_3.
\end{equation}
Let $\ell:=|z_2^{\rm in}-z_1^{\rm in}|$ and choose
$\varphi$ so that
\[
z_j^{\rm in}-z_k^{\rm in}
=\ell e^{\mathrm i\varphi}({\rm n}_j-{\rm n}_k).
\]
Restoring this scale and rotation in \cref{eq:def_M_app}
gives
\begin{equation}
\label{eq:M_scaled_jordan}
M=e^{-2\mathrm i\varphi}
\frac{\Gamma(\gamma_1+\gamma_2)}{2\pi\ell^2}\,
{\mathsf M},
\qquad
\overline M M
=\left(
\frac{\Gamma(\gamma_1+\gamma_2)}{2\pi\ell^2}
\right)^2{\mathsf A},
\qquad
{\mathsf A}:=\overline{\mathsf M}\,{\mathsf M},
\end{equation}
where
\begin{equation}
\label{eq:mathsf_M_jordan}
{\mathsf M}
=
\begin{pmatrix}
\displaystyle
\upgamma_2+\frac{\upgamma_3}{(z-\upgamma_2)^2}
&-\upgamma_2
&\displaystyle-\frac{\upgamma_3}{(z-\upgamma_2)^2}
\\[3mm]
-\upgamma_1
&\displaystyle
\upgamma_1+\frac{\upgamma_3}{(z+\upgamma_1)^2}
&\displaystyle-\frac{\upgamma_3}{(z+\upgamma_1)^2}
\\[3mm]
\displaystyle-\frac{\upgamma_1}{(z-\upgamma_2)^2}
&\displaystyle-\frac{\upgamma_2}{(z+\upgamma_1)^2}
&\displaystyle
\frac{\upgamma_1}{(z-\upgamma_2)^2}
+\frac{\upgamma_2}{(z+\upgamma_1)^2}
\end{pmatrix}.
\end{equation}
Then, let $\boldsymbol{1}=(1,1,1)^T$ and $e_3=(0,0,1)^T$.  The center of vorticity
of \cref{eq:aref_positions_jordan} is
\begin{equation}
    {\rm n}_{\rm cv}
    =\frac{\upgamma_3z}{1+\upgamma_3},
\end{equation}
and we finally set
\begin{equation}
\label{eq:centered_aref_vector_jordan}
    \boldsymbol{\mathrm n}
    :=
    \begin{pmatrix}
        {\rm n}_1-{\rm n}_{\rm cv}\\
        {\rm n}_2-{\rm n}_{\rm cv}\\
        {\rm n}_3-{\rm n}_{\rm cv}
    \end{pmatrix}.
\end{equation}
The identities used in the proof of Lemma \ref{lem:SpectralPropertiesB} give, for
some $\lambda_{\upgamma}>0$,
\begin{equation}
\label{eq:eigen_relations_mathsf_A_jordan}
    {\mathsf A}\boldsymbol{1}=0,
    \qquad
    {\mathsf A}\boldsymbol{\mathrm n}
    =\lambda_{\upgamma}\boldsymbol{\mathrm n},
    \qquad
    \operatorname{tr}({\mathsf A})=2\lambda_{\upgamma}.
\end{equation}
Moreover,
\begin{equation}
    \det(\boldsymbol{1},\boldsymbol{\mathrm n},e_3)
    ={\rm n}_2-{\rm n}_1=-1.
\end{equation}
Consequently, there are unique $\mu,\kappa_{\upgamma}\in\mathbb C$ such that
\begin{equation}
\label{eq:def_kappa_upgamma_jordan}
    ({\mathsf A}-\lambda_{\upgamma}I_3)e_3
    =\mu\boldsymbol{1}+\kappa_{\upgamma}\boldsymbol{\mathrm n}.
\end{equation}
Since the trace identity in \cref{eq:eigen_relations_mathsf_A_jordan} removes the
$e_3$ component, by  taking the difference of the second and first components in
\cref{eq:def_kappa_upgamma_jordan} gives
\begin{equation}
\label{eq:kappa_as_entry_difference_jordan}
    \kappa_{\upgamma}={\mathsf A}_{13}-{\mathsf A}_{23}.
\end{equation}
If $w=e_3+(\mu/\lambda_{\upgamma})\boldsymbol{1}$, then
\begin{equation}
    ({\mathsf A}-\lambda_{\upgamma}I_3)w
    =\kappa_{\upgamma}\boldsymbol{\mathrm n},
\end{equation}
and hence in the basis generated by $(\boldsymbol{1},\boldsymbol{\mathrm n},w)$ we have the desired Jordan block decomposition
\begin{equation}
\label{eq:jordan_form_mathsf_A_jordan}
    [{\mathsf A}]_{(\boldsymbol{1},\boldsymbol{\mathrm n},w)}
    =
    \begin{pmatrix}
        0&0&0\\
        0&\lambda_{\upgamma}&\kappa_{\upgamma}\\
        0&0&\lambda_{\upgamma}
    \end{pmatrix}.
\end{equation}
It remains to compute \cref{eq:kappa_as_entry_difference_jordan}.  From
\cref{eq:mathsf_M_jordan}, a direct calculation shows that
\begin{align}
\label{eq:kappa_direct_expansion_jordan}
\kappa_{\upgamma}
={}& -
\left(1+\frac{\upgamma_3}{(\overline z-\upgamma_2)^2}\right)
\frac{\upgamma_3}{(z-\upgamma_2)^2}
+
\left(1+\frac{\upgamma_3}{(\overline z+\upgamma_1)^2}\right)
\frac{\upgamma_3}{(z+\upgamma_1)^2}
\notag\\
&\quad
-\upgamma_3
\left(
\frac{1}{(\overline z-\upgamma_2)^2}
-\frac{1}{(\overline z+\upgamma_1)^2}
\right)
\left(
\frac{\upgamma_1}{(z-\upgamma_2)^2}
+\frac{\upgamma_2}{(z+\upgamma_1)^2}
\right).
\end{align}
Rearranging these terms and using $z=x_1+{\rm i}x_2$ with
$z\overline z=1+\upgamma_3$ yields
\begin{equation}
\label{eq:kappa_factorized_jordan}
    \kappa_{\upgamma}(z)
    =
    2\upgamma_3\,
    \frac{\overline z\,{\sf p}_{\upgamma}(x_1)}
    {|z+\upgamma_1|^4|z-\upgamma_2|^4},
\end{equation}
where
\begin{equation}
\label{eq:polynomial_p_upgamma_jordan}
\begin{split}
    {\sf p}_{\upgamma}(X)
    :=
    4(\upgamma_2-\upgamma_1)X^3
    -4(2+5\upgamma_3)X^2
    +2(\upgamma_2-\upgamma_1)(1+4\upgamma_3)X
    +2+\upgamma_3-4\upgamma_3^2.
\end{split}
\end{equation}
Therefore, with the normalization in \cref{eq:jordan_form_mathsf_A_jordan},
the coefficient for $\overline M M$ is
\begin{equation}
\label{eq:full_kappa_jordan}
    \kappa
    =
    \left(\frac{\Gamma(\gamma_1+\gamma_2)}{2\pi \ell^2}\right)^2
    2\upgamma_3\,
    \frac{\overline z\,{\sf p}_{\upgamma}(x_1)}
    {|z+\upgamma_1|^4|z-\upgamma_2|^4}.
\end{equation}
Therefore, on Aref's circle we get that
\begin{equation}
\label{eq:kappa_zero_condition_jordan}
    \kappa=0
    \quad\Longleftrightarrow\quad
    {\sf p}_{\upgamma}(x_1)=0.
\end{equation}
In particular, the condition depends only on $x_1$.  Thus every exceptional
configuration $x_1+\mathrm{i}x_2$ is paired with its reflection
$x_1-\mathrm{i}x_2$ across the horizontal axis and we deduce the following.
\begin{lemma}
\label{lem:four_exceptional_points_jordan}
There are exactly four configurations on Aref's circle for which $\kappa=0$,
one in the interior of each arc delimited by the four relative equilibria.
\end{lemma}
See Fig. \ref{fig:aref_circle_jordan_free} for a visualization of the special configurations.
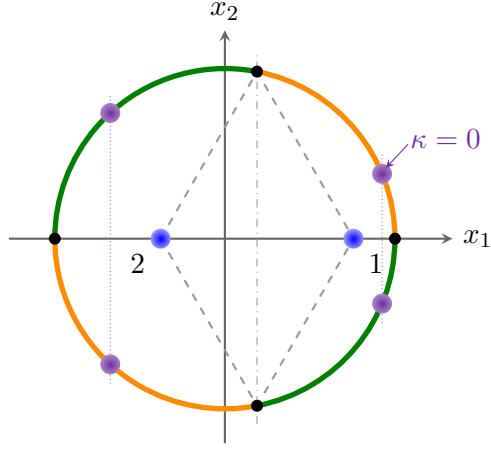
\begin{figure}[tbp]
\centering
\begin{tikzpicture}[scale=2.55,>=stealth]
\definecolor{collorange}{RGB}{255,140,0}
\definecolor{jordanviolet}{RGB}{112,48,160}
\tikzset{
    vortex pos/.style={circle, shading=radial, inner color=blue,
        outer color=blue!40, draw=none, minimum size=7.5pt, inner sep=0pt},
    equilibrium/.style={circle, fill=black, draw=none,
        minimum size=4.5pt, inner sep=0pt},
    jordan free/.style={circle, shading=radial, inner color=jordanviolet, outer color=jordanviolet!60,
        draw=none, minimum size=7.5pt, inner sep=0pt},
    arc coll/.style={line width=2pt, collorange, line cap=round},
    arc exp/.style={line width=2pt, green!50!black, line cap=round},
    gridline/.style={dashed, black!40, thick}
}

\pgfmathsetmacro{\R}{sqrt(7)/3}
\pgfmathsetmacro{\xone}{-1/3}
\pgfmathsetmacro{\xtwo}{2/3}
\pgfmathsetmacro{\xmid}{1/6}
\pgfmathsetmacro{\yeq}{sqrt(3)/2}
\pgfmathsetmacro{\thetaE}{atan2(\yeq,\xmid)}
\pgfmathsetmacro{\xJp}{0.8152347940351869}
\pgfmathsetmacro{\yJp}{0.3364075034421558}
\pgfmathsetmacro{\xJm}{-0.5943979148639263}
\pgfmathsetmacro{\yJm}{0.6515127754566249}

\coordinate (V1) at (\xone,0);
\coordinate (V2) at (\xtwo,0);
\coordinate (E1) at (\R,0);
\coordinate (E3) at (-\R,0);
\coordinate (E2) at (\xmid,\yeq);
\coordinate (E4) at (\xmid,-\yeq);

\draw[->,thick,black!60] (-1.12,0) -- (1.18,0)
    node[right,text=black] {$x_1$};
\draw[->,thick,black!60] (0,-1.06) -- (0,1.08)
    node[above,text=black] {$x_2$};
\draw[dash dot,black!40] (\xmid,-0.96) -- (\xmid,0.96);
\draw[gridline] (V1) -- (E2) -- (V2);
\draw[gridline] (V1) -- (E4) -- (V2);

\draw[arc coll] (0:\R) arc (0:\thetaE:\R);
\draw[arc exp]  (\thetaE:\R) arc (\thetaE:180:\R);
\draw[arc coll] (180:\R) arc (180:{360-\thetaE}:\R);
\draw[arc exp]  ({360-\thetaE}:\R) arc ({360-\thetaE}:360:\R);

\node[equilibrium] at (E1) {};
\node[equilibrium] at (E2) {};
\node[equilibrium] at (E3) {};
\node[equilibrium] at (E4) {};
\node[vortex pos] at (V1) {};
\node[below left=2pt] at (V1) {$2$};
\node[vortex pos] at (V2) {};
\node[below right=2pt] at (V2) {$1$};

\draw[densely dotted,jordanviolet!70]
    (\xJp,-\yJp-0.10) -- (\xJp,\yJp+0.10);
\draw[densely dotted,jordanviolet!70]
    (\xJm,-\yJm-0.10) -- (\xJm,\yJm+0.10);
\node[jordan free] at (\xJp,\yJp) {};
\node[jordan free] at (\xJp,-\yJp) {};
\node[jordan free] at (\xJm,\yJm) {};
\node[jordan free] at (\xJm,-\yJm) {};

\node[anchor=west,text=jordanviolet] at (0.91,0.52)
    { $\kappa=0$};
\draw[->,jordanviolet] (0.95,0.49) -- (\xJp+0.02,\yJp+0.03);
\end{tikzpicture}
\caption{Aref's circle with normalized circulations $\gamma_1=1, \gamma_2=2$.
The violet circles are the configurations for which the repeated eigenvalue is
semisimple, and hence no Jordan block occurs.}
\label{fig:aref_circle_jordan_free}
\end{figure}

\begin{proof}
Let $R_{\upgamma}:=\sqrt{1+\upgamma_3}$.  The equilateral configurations have
first coordinate $(\upgamma_2-\upgamma_1)/2$, and direct substitution gives
\begin{equation}
    {\sf p}_{\upgamma}
    \left(\frac{\upgamma_2-\upgamma_1}{2}\right)
    =\frac32,
    \qquad
    {\sf p}_{\upgamma}(\pm R_{\upgamma})<0.
\end{equation}
The endpoint inequalities follow by using
$\upgamma_3=-\upgamma_1\upgamma_2$ and
$\upgamma_1+\upgamma_2=1$ in \cref{eq:polynomial_p_upgamma_jordan}.
Hence ${\sf p}_{\upgamma}$ has one root in each of the intervals
\begin{equation}
    \left(-R_{\upgamma},\frac{\upgamma_2-\upgamma_1}{2}\right),
    \qquad
    \left(\frac{\upgamma_2-\upgamma_1}{2},R_{\upgamma}\right).
\end{equation}
If $\mathsf\gamma_1>\mathsf\gamma_2$, the third root of the
cubic lies to the left of $-R_\gamma$.
If $\mathsf\gamma_1<\mathsf\gamma_2$, it lies to the right
of $R_\gamma$.   
When $\upgamma_1=\upgamma_2$, the polynomial is quadratic
and has exactly the two roots above.  Each admissible first coordinate gives
two points related by $x_2\mapsto-x_2$, proving the claim.
\end{proof}
\section{Proof of the construction of the approximate solution}
\label{app:ApproxSolProof}

In this appendix we prove Proposition \ref{prop:ApproximateSol}. Throughout the proof, the
approximate vortex centers are fixed once and for all as the
Helmholtz--Kirchhoff trajectories solving \cref{eq:HK}. We use the expansion of the translation
operators in Lemma \ref{lem:TeProperties}.

We first expand the residual up to the order needed in the construction. Recall
that
\begin{equation}
    \Omega_{{\rm app},i}
    =
    G
    +
    \sum_{n=2}^4
    \eps^n\Omega_{n,i},
    \qquad
    \Psi_{{\rm app},i}
    =
    \Psi_G
    +
    \sum_{n=2}^4
    \eps^n\Psi_{n,i}.
\end{equation}
Since
\begin{equation}
    t\partial_t\eps^n
    =
    \frac{n}{2}
    \frac{1}{1+\varsigma \frac{t}{T_{\rm adv}}}
    \eps^n,
    \label{eq:tdtepsnAppendix}
\end{equation}
the dissipative term gives
\begin{equation}
\begin{aligned}
    \delta(t\partial_t-\cL)\Omega_{\app,i}
    =
    \sum_{n=2}^4
    \eps^n\delta
    \left(
        t\partial_t
        +
        \frac{n}{2}
         \frac{1}{1+\varsigma \frac{t}{T_{\rm adv}}}
        -
        \cL
    \right)
    \Omega_{n,i}.
\end{aligned}
\label{eq:DissExpansionAppendix}
\end{equation}

We now turn to the transport term. The external field generated by the
zeroth-order profiles contributes, after the Helmholtz--Kirchhoff cancellation,
only from order $\eps^2$ on. Indeed, we can write 
\begin{align}
    \Phi_{\eps,i}(\Psi_G,\boldsymbol{z})&-\gamma_i\Psi_G -
        \sum_{j\neq i}
        \frac{\gamma_j}{2\pi}
        \log\left(
            \frac{|z_i-z_j|}{\eps\dhk}\right)
    \\
    &=
        \sum_{j\neq i}
        \gamma_j
        \mathcal T_{\eps,\boldsymbol z}^{i,j}\Psi_G
        -
        \frac{\delta t}{\eps\dhk}
        \partial_t z_i\cdot \xi^\perp
        -
        \sum_{j\neq i}
        \frac{\gamma_j}{2\pi}
        \log\left(
            \frac{|z_i-z_j|}{\eps\dhk}
        \right)
     =: \eps^2 h_i + \cO_{\mathcal{S}_*}(\eps^3)
    \label{def:Di20Appendix}
\end{align}
where we separate the leading order contribution from the higher order one in $\epsilon$ after expanding the translation operator $\mathcal{T}_{\eps,z}^{i,j}$ and noting that the zero and first order terms cancel out.
Note that $h_i$ is exactly the same defined in \cref{eq:defhiFunctional}.
Using Lemma \ref{lem:R0}, the expansion of the translation operators, and the moment
constraints imposed on the profiles, we obtain
\begin{equation}
\begin{aligned}
    \cR_{i}
    =
    &\,
    \sum_{n=2}^4
    \eps^n
    \left[
        \cH_{n,i}
        +
        \gamma_i\Lambda\Omega_{n,i}
        +
        \delta
        \left(
            t\partial_t
            +
            \frac{n}{2}
             \frac{1}{1+\varsigma \frac{t}{T_{\rm adv}}}
            -
            \cL
        \right)
        \Omega_{n,i}
    \right]
    \\
    &+
    \eps^4 \ \mathcal{H}_{4,i}^{S}
    +
    \cO_{\cZ}(\eps^5),
\end{aligned}
\label{eq:R4ExpansionAppendix}
\end{equation}
where
\begin{equation}
    \mathcal{H}_{4,i}^{S} : = \gamma_i\{\Psi_{2,i},\Omega_{2,i}\}
        +
        \{h_i,\Omega_{2,i}\}
\end{equation}
is the additional term arising from the second order profiles interaction. 
Let us briefly explain why the translated profiles of the other vortices only
enter the remainder in \cref{eq:R4ExpansionAppendix}. Since the profiles have zero mass and zero first moments, by Lemma \ref{lem:TeProperties} we know that  the first nonconstant
contribution of their translated streamfunctions is of order $\eps^3$. Thus, since a Poisson bracket is involved we can ignore constants and, after multiplication by the amplitude
$\eps^n$, $n\geq2$, we get an $\cO_{\cZ}(\eps^5)$ remainder.
Let us focus now on the term $\mathcal{H}_{4,i}^S$. Using the definition of $\Omega_{2,i}^E$, Proposition \ref{prop:Lambda} and the identities in \eqref{eq:iffFunctional}, its Eulerian part can be rewritten as follows:
\begin{align}
    \gamma_i\{\Psi_{2,i}^E,\Omega_{2,i}^E\}
        +
        \{h_i,\Omega_{2,i}^E\} & = \gamma_i \{\Psi_{2,i}^E,\Omega_{2,i}^E\} - \gamma_i \{\mathsf{H} \Omega_{i,2}^E, \Omega_{i,2}^E\} \\ & = - \gamma_i\{A \Omega_{i,2}^E, \Omega_{i,2}^E\} \\ & = - \gamma_i\{F'(G) \Omega_{i,2}^E, \Omega_{i,2}^E\} \\ 
        & = -\frac{\gamma_i}{2} \{A,(\Omega_{i,2}^E)^2 \}
\end{align} 
The last expression can only have angular mode $4$. Then,
\begin{equation}
\mathcal{P}_0\left(\gamma_i\{\Psi_{2,i}^E,\Omega_{2,i}^E\}
        +
        \{h_i,\Omega_{2,i}^E\} \right) = 0.
\end{equation}
Therefore, the Eulerian contribution at order $\eps^4$ can be canceled by   inverting the operator $\Lambda$. Then, we set
\begin{equation}
\begin{aligned}
    \widetilde{\cH}_{4,i}^{E}
    :=
    &\,
    \cH_{4,i}
    +
    \left(
        \gamma_i\{\Psi_{2,i}^{E},\Omega_{2,i}^{E}\}
        +
        \{h_i,\Omega_{2,i}^{E}\}
    \right).
\end{aligned}
\label{def:TildeH4EAppendix}
\end{equation}
On the other hand, the terms at order $\delta\eps^4$ are collected in the Navier--Stokes term
\begin{equation}
\begin{aligned}
    \widetilde{\cH}_{4,i}^{NS}
    :=
    &\,
    \gamma_i
    \left(
        \{\Psi_{2,i}^{E},\Omega_{2,i}^{NS}\}
        +
        \{\Psi_{2,i}^{NS},\Omega_{2,i}^{E}\}
    \right)
    +
    \{h_i,\Omega_{2,i}^{NS}\}.
\end{aligned}
\label{def:TildeH4NSAppendix}
\end{equation}
Finally, the Navier-Stokes interaction term in $\mathcal{H}_{4,i}^S$ carries a $\delta^2$ and it is therefore absorbed in the error.

\begin{remark}
    In \cite{Gallay2011}, a second order radial correction is introduced as a solution of a parabolic equation in order to remove the second-order profile interaction term that arises at fourth order. In our setting, the Euler and Navier--Stokes decomposition allows us to determine that the leading Eulerian part has actually  mode $4$. Then, $\Lambda$ can be inverted on it. All the others possibly radial contributions comes from the viscous profiles.
\end{remark}
We now highlight how the definitions of the approximate solution accomplish the task of canceling the required orders.

\smallskip
\paragraph{$\blackinwhitediamond$ \textbf{Cancellation at order $\eps^2$}:} the non-radial Eulerian contribution $\Omega_{2,i}^{E}$ ensures that
\begin{equation}
    \cH_{2,i}
    +
    \gamma_i\Lambda\Omega_{2,i}^{E} = 0.
\end{equation}
The first dissipative contribution that it generates is removed by the
Navier--Stokes correction $\Omega_{2,i}^{NS}$
\begin{equation}
    \gamma_i\Lambda\Omega_{2,i}^{NS}
    +
    \left(
        t\partial_t
        +
        \frac{1}{1+\varsigma \frac{t}{T_{\rm adv}}}
        -
        \cL
    \right)
    \Omega_{2,i}^{E}
    =
    0.
    \label{eq:NS2CancellationAppendix}
\end{equation}
Note that the operator $\Lambda$ can be inverted since $\mathcal{H}_{2,i}$ satisfies the hypothesis in Proposition \ref{prop:Lambda}. Thanks to previous discussion, the second order Eulerian profile interactions generated at fourth order do not have a radial component; therefore, we do not need to introduce a second order Eulerian profile to remove them. 
\smallskip
\paragraph{$\blackinwhitediamond$ \textbf{Cancellation at order $\eps^3$:}}
following the same strategy as the one adopted at second order the non-radial Eulerian part and the Navier-Stokes corrections are defined in the exact same way upon changing $\mathcal{H}_{2,i}$ with $\mathcal{H}_{3,i}$ and the coefficients in the equations. However, since the interaction between the second order profiles does not generate radial terms of third order, we can set  $\Omega_{3,i}^{E_r} = 0$.

\smallskip
\paragraph{$\blackinwhitediamond$ \textbf{Cancellation at order $\eps^4$:}}
At order $\eps^4$, after the second-order cancellations have been performed,
the remaining purely Eulerian contribution is the one collected in
$\widetilde{\cH}_{4,i}^{E}$. Since it is non radial and satisfies the moments conditions, we can invert the operator $\Lambda$ on it and define $\Omega_{4,i}^E $ as the unique solution to
\begin{equation}
    \gamma_i\Lambda\Omega_{4,i}^{E}
    +
    \widetilde{\cH}_{4,i}^{E}
    =
    0.
\end{equation}
It remains to treat the terms of size $\delta\eps^4$. These contain two types
of contributions: the dissipative term generated by $\Omega_{4,i}^{E}$ and the
Navier--Stokes interactions collected in $\widetilde{\cH}_{4,i}^{NS}$. The
non-radial part is removed by the definition of $\Omega_{4,i}^{NS}$, which
ensures that
\begin{equation}
    \begin{aligned}
\gamma_i\Lambda\Omega_{4,i}^{NS}
+
(1-\cP_0)\widetilde{\cH}_{4,i}^{NS}
\
+
\left(
t\partial_t
+
2\frac{1}{1+\varsigma \frac{t}{T_{\rm adv}}}
-
\cL
\right)
\Omega_{4,i}^{E}
=
0.
\end{aligned}
\label{eq:NS4CancellationAppendix}
\end{equation}
The radial part of $\widetilde{\cH}_{4,i}^{NS}$ cannot be canceled through
$\Lambda^{-1}$. It is therefore absorbed by the radial profile
$\Omega_{4,i}^{E_{\rm r}}$, which solves
\begin{equation}
\left(
t\partial_t
+
2\frac{1}{1+\varsigma \frac{t}{T_{\rm adv}}}
-
\cL
\right)
\Omega_{4,i}^{E_{\rm r}}
=
-
\cP_0\widetilde{\cH}_{4,i}^{NS},
\qquad
\Omega_{4,i}^{E_{\rm r}}\big|_{t=0}=-(2-\mathcal{L})^{-1} \mathcal{P}_0 \widetilde{\mathcal{H}}_{4,i}^{NS}(\cdot, 0),
\label{eq:Radial4CancellationAppendix}
\end{equation}
that completes the cancellation of the fourth order terms.
\begin{remark}
    The initial datum $\Omega_{4,i}^{E_{\rm r}}(\cdot, 0)$ does not affect the initial profile $G$, since it is multiplied by $\epsilon(t)^4$.
\end{remark}

\smallskip
\paragraph{$\blackinwhitediamond$ \textbf{Moment constraints:}}
With the cancellations at hand, let us verify that indeed the desired moment constraints are satisfied.  By Lemma \ref{lem:R0} we know that
\begin{equation}
    M(\cH_{n,i})=0,
    \qquad
    m^1(\cH_{n,i})=(0,0),
    \qquad
    n=2,3,4.
\end{equation}
Hence, by Proposition \ref{prop:Lambda}, the Eulerian profiles at order $n = 2, 3$ can be defined and they have zero mass and zero first order moment. The same conclusion holds for the Navier--Stokes profiles since the operators $t\partial_t, \, \cL $ preserve the relevant compatibility conditions.
For $n = 4$, $\widetilde\cH_{4,i}^{E}$ contains only terms of angular mode $4$ as previously observed, while, applying $1- \cP_0$ to  $\widetilde\cH_{4,i}^{NS}$ leads to a forcing with angular modes different from $0$ and $1$. Then, it is not hard to verify that
\begin{equation}
    M(\Omega_{n,i}^{\mathsf{q}})=0,
    \qquad
    m^1(\Omega_{n,i}^{\mathsf{q}})=(0,0), 
    \qquad
    n=2,3,4, \quad \mathsf{q}\in \{E,NS\}.
\end{equation}
It remains to treat the fourth order radial profile. Its first order moment readily vanishes by radial
symmetry, that is
\begin{equation}
    m^1(\Omega_{4,i}^{E_{\rm r}})=(0,0).
\end{equation}
Taking the spatial integral in the defining equation gives
\begin{equation}
    t\partial_t M(\Omega_{4,i}^{E_{\rm r}})
    +
    2
    \frac{1}{1+\varsigma \frac{t}{T_{\rm adv}}}
    M(\Omega_{4,i}^{E_{\rm r}})
    =
    0.
\end{equation}
Here we used that $M(\cL f)=0$ and that the spatial integral of a Poisson
bracket vanishes to cancel out the terms in the right-hand side. Since the radial profile has zero mass at initial time,
its mass remains zero for all time where the solution is defined and therefore
\begin{equation}
    M(\Omega_{4,i}^{E_{\rm r}})=0.
\end{equation}
Recalling that for $n=2,3$ we have $\Omega_{n,i}^{E_{\rm r}}=0$, we finally get
\begin{equation}
    M(\Omega_{n,i}^{\mathsf{q}})=0,
    \qquad
    m^1(\Omega_{n,i}^{\mathsf{q}})=(0,0),
    \qquad
    n=2,3,4,
    \qquad
    \mathsf{q}\in\{E,E_{\rm r},NS\}.
    \label{eq:MomentConditionsAppendix}
\end{equation}
With this we have justified that we can indeed rigorously define the approximate solutions as in Definition \ref{def:ApproxProfilesOrder4}.

\smallskip
\paragraph{$\blackinwhitediamond$ \textbf{Bounds on the profiles:}}
By Lemma \ref{lem:R0},
\begin{equation}
    \cH_{n,i}=\cO_{\cZ}(1),
    \qquad
    n=2,3,4.
\end{equation}
The time-derivative bounds follow from the explicit
self-similar motion. Indeed, the coefficients
$\dhk^n(z_i-z_j)^{-n}$ depend on time only through
the rotation angle $\theta(t)$. Their repeated
$t\partial_t$ derivatives are uniformly bounded in
the expanding case, while each derivative costs at
most one factor $(1-t/T_{\rm adv})^{-1}$ in the
collapsing case. Since $\Lambda^{-1}$ and $\cL$
commute with rotations, the same bounds hold for
the nonradial Eulerian profiles, including their
second time derivatives. Differentiating the defining
formulas for the Navier--Stokes profiles then gives
the stated bounds. This argument also applies to
the quadratic fourth-order forcings.
For the radial profile, Lemma \ref{lem:lin}, applied
with $A_j(\tau)\lesssim(1-\tau)^{-j-1}$ in the
collapsing case and bounded $A_j$ in the expanding
case, yields the required estimates.
Then, one can directly check that the bounds for on the profiles for $n = 2, 3$ in Proposition \ref{prop:ApproximateSol} hold. From these combined with \cref{def:TildeH4EAppendix,def:TildeH4NSAppendix}, we get that for $k=0,1,2$ and $m=0,1$
\begin{align}
        (t\partial_t)^k\widetilde{\cH}_{4,i}^{E}
 & =
\begin{cases}
\cO_{\cZ}(1), & \text{if }\varsigma =1,\\
\cO_{\cZ}\bigl((1-t/T_{\rm adv})^{-k}\bigr),
& \text{if }\varsigma=-1,
\end{cases}
\\
    \qquad
(t\de_t)^m\widetilde{\cH}_{4,i}^{NS} &=\begin{cases}\cO_{\cZ}(1), &\quad \text{ if } \varsigma=1,\\
    \cO_{\cZ}\big((1-\frac{t}{T_{\rm adv}})^{-(1+m)}\big), &\quad \text{ if } \varsigma=-1.
    \end{cases}
\end{align}
The bounds for the fourth order profiles then follow.

\smallskip
\paragraph{$\blackinwhitediamond$ \textbf{Bounds on the remainder:}}
Thanks to the cancellations that we highlighted before, we have  removed all contributions of order $\delta^m\eps^n$ for $m=0,1$ and $n=2,3,4$. We are thus left with the remainder
\begin{equation}
    \cR_{i}
    =
    \sum_{n=2}^4
    \eps^n\delta^2
    \left(
        t\partial_t
        +
        \frac{n}{2}
        \frac{1}{1+\varsigma \frac{t}{T_{\rm adv}}}
        -
        \cL
    \right)
    \Omega_{n,i}^{NS}+
    \gamma_i\delta^2\eps^4
        \gamma_i\{\Psi_{2,i}^{NS},\Omega_{2,i}^{NS}\}
    +
    \cO_{\cZ}(\eps^5).
\label{eq:R4cancellations}
\end{equation}

By the bounds on the profiles, we deduce that
\begin{equation}
    \cR_{i}
    =
    \begin{cases}\cO_{\cZ}(\eps^5+\delta^2\eps^2), &\quad \text{ if } \varsigma=1,\\
    \cO_{\cZ}\big(\eps^5+\delta^2\eps^2(1-\frac{t}{T_{\rm adv}})^{-2}\big), &\quad \text{ if } \varsigma=-1.
\end{cases}
    \label{eq:R4FinalAppendix}
\end{equation}
To conclude, we note that by \cref{def:keytaueps} we get
\begin{align}
    \eps(t)^2 (1-\frac{t}{T_{\rm adv}})^{-2}=\frac{\delta}{|\alpha|}\frac{\frac{t}{T_{\rm adv}}}{(1-\frac{t}{T_{\rm adv}})^3} \ll 1.
\end{align}
Then the proof of the last item in Proposition \ref{prop:ApproximateSol} is a direct calculation.

\section{Proof of Lemma \ref{lem:lin}}
\label{app:LemmaLin}

Here, we provide a proof of Lemma \ref{lem:lin}, for which we use the
known explicit formula for the semigroup associated to the operator
\(\mathcal{L}\), see \cite[Appendix A]{GalWay2002} for instance.

\begin{proof}[Proof of Lemma \ref{lem:lin}]

Throughout the proof, the implicit constants in $\lesssim$
may depend on $q$, on the order of differentiation , and on the
constants and polynomial growth exponents in the assumptions
on $H$. 
They are independent of $\varsigma$ and of all time, space,
and integration variables.
We use $N$ to denote a nonnegative polynomial exponent,
depending only on the differentiation orders and the growth
exponents in the assumptions on $H$, which may be increased
from one occurrence to the next.

    We start by deriving the representation formula for a solution of
    \cref{eq:LinOm2}. As in \cite{GalWay2002}, let \(S(t)\) be the
    semigroup generated by the operator \(\mathcal{L}\). By variations
    of constants, we readily get that
    \begin{equation}
    \label{eq:f_pm_formula_r_tau0}
        f_{\varsigma,q}(\tau)
        =
         \left(
    \frac{\tau_0(1+\varsigma\tau)}
    {\tau(1+\varsigma\tau_0)}
    \right)^q
    S\left(\log\frac{\tau}{\tau_0}\right)
    f_{\varsigma,q}(\tau_0)+ \left(\frac{1+\varsigma\tau}{\tau}\right)^q
        \int_{\tau_0}^\tau
        \frac{r^{q-1}}{(1+\varsigma r)^{q}}
        S\left(\log\frac{\tau}{r}\right)H(r,\xi)
        \,dr.
    \end{equation}
    Since the solution is regular at zero, it is bounded in
$\mathcal Y$ near $\tau=0$. The first term therefore tends
to zero in $\mathcal Y$ as $\tau_0\to 0$, because
$q>0$ and $S(t)$ is contractive in $\mathcal Y$.
Consequently,
\begin{equation}
\label{eq:f_pm_formula_r}
    f_{\varsigma,q}(\tau)
    =
    \left(\frac{1+\varsigma\tau}{\tau}\right)^q
    \int_0^\tau
    \frac{r^{q-1}}{(1+\varsigma r)^q}
    S\left(\log\frac{\tau}{r}\right)H(r,\cdot)\,dr.
\end{equation}
The same argument applied to the difference of two regular
solutions proves uniqueness.\\
We recall the explicit formula for $u \in \cZ$
    \begin{equation}
        \left(S(t)u\right)(\xi)
        =
        \frac{1}{4\pi(1-\e^{-t})}
        \int_{\mathbb{R}^2}
        \e^{-\frac{|\xi-\eta\e^{-t/2}|^2}{4(1-\e^{-t})}}
        u(\eta)\,d\eta.
    \end{equation}
    Note that we can rewrite
    \begin{equation}
        \e^{|\xi|^2/4}\left(S(t)u\right)(\xi)
        =
        \frac{1}{4\pi(1-\e^{-t})}
        \int_{\mathbb{R}^2}
        \e^{-\frac{|\eta-\xi\e^{-t/2}|^2}{4(1-\e^{-t})}}
        \widetilde u(\eta)\,d\eta,
        \qquad
        \widetilde u(\eta):=\e^{|\eta|^2/4}u(\eta).
    \end{equation}
    Setting
    \[
        \widetilde f_{\varsigma,q}(\tau,\xi)
        :=
        \e^{|\xi|^2/4}f_{\varsigma,q}(\tau,\xi),
        \qquad 
        \widetilde H_j(r,\eta)
    :=
    e^{|\eta|^2/4}(r\partial_r)^jH(r,\eta),
    \]
    we then have
    \begin{align}
        \widetilde f_{\varsigma,q}(\tau,\xi)
        =
        \left(\frac{1+\varsigma\tau}{\tau}\right)^q
        \int_0^\tau
        \frac{r^{q-1}}{(1+\varsigma r)^{q}}
        \frac{1}{4\pi(1-r/\tau)}
        \int_{\mathbb{R}^2}
        \e^{-\frac{|\eta-\xi\sqrt{r/\tau}|^2}{4(1-r/\tau)}}
        \widetilde H_0(r,\eta)\,d\eta\,dr.
    \end{align}
    By the change of variables \(s=r/\tau\), we rewrite
    \begin{equation}
\label{eq:FWK}
    \widetilde f_{\varsigma,q}(\tau,\xi)
    =
    \int_0^1
    \mathcal W_{\varsigma,q}(\tau,s)
    \mathcal K_s
    \bigl(\widetilde H_0(\tau s,\cdot)\bigr)(\xi)\,ds,
\end{equation}
with
\begin{equation}
    \mathcal W_{\varsigma,q}(\tau,s)
    =
    s^{q-1}
    \left(
    \frac{1+\varsigma\tau}{1+\varsigma s\tau}
    \right)^q,
    \qquad
    \mathcal K_s(g)(\xi)
    :=
    \frac{1}{4\pi(1-s)}
    \int_{\mathbb R^2}
    e^{-\frac{|\eta-\xi\sqrt s|^2}{4(1-s)}}
    g(\eta)\,d\eta.
\end{equation}

The assumptions on $H$ imply continuity at zero in
$\mathcal Y$. Moreover,
$\mathcal W_{\varsigma,q}(\tau,s)\to s^{q-1}$
as $\tau\to 0$.
For $\tau$ sufficiently small, the integrand in the
semigroup representation is dominated in $\mathcal Y$
by an integrable multiple of $s^{q-1}$.
Hence
\begin{align}
    \lim_{\tau\to 0}f_{\varsigma,q}(\tau)
    &=
    \int_0^1
    s^{q-1}S\left(\log\frac1s\right)H(0,\cdot)\,ds
    \notag\\
    &=
    \int_0^\infty e^{-qv}S(v)H(0,\cdot)\,dv
    =
    (q-\mathcal L)^{-1}H(0,\cdot).
\end{align}
Thus the representation has exactly the prescribed
initial trace, which belongs to $\mathcal Z$ by Lemma \ref{lem:Linvert}.

We next establish estimates for $\mathcal K_s$ at a fixed
time $r$ and for a fixed $j\in\mathbb N_0$.
    By hypothesis, for every \(\alpha\in\mathbb{N}^2\),
    \begin{equation}
        |\partial_\eta^\alpha\widetilde H_j(r,\eta)|
        \lesssim
        (1+|\eta|)^{N} A_j(r).
    \end{equation}
    Moreover,
\begin{equation}
    \int_{\mathbb R^2}
    e^{-|\eta|^2/4}\widetilde H_j(r,\eta)\,d\eta
    =
    (r\partial_r)^jM(H(r,\cdot))
    =
    0.
\end{equation}
Differentiation under the integral is justified by
the weighted polynomial bounds and the local
boundedness of the functions $A_j$.
    In the sequel, we will often use the following standard Gaussian
    moment estimate: for every \(c>0\), \(a\in\mathbb{R}^2\), and every
    \(m\in\mathbb{N}\), we have
    \begin{equation}
    \label{eq:GaussaniMomentBound}
        \int_{\mathbb{R}^2}
        \e^{-c|\eta-a|^2}(1+|\eta|)^m\,d\eta
        \lesssim
        (1+|a|)^m.
    \end{equation}
    We now claim that
    \begin{equation}
    \label{eq:Ks_zero_order_sqrt_gain}
        |\mathcal{K}_s(\widetilde H_j)(r,\xi)|
        \lesssim
         A_j(r)\sqrt{s}(1+|\xi|)^{N},
        \qquad 0<s<1.
    \end{equation}
    Let us consider separately the cases \(s\in(0,1/2)\) and
    \(s\in(1/2,1)\). For the latter, the Gaussian moment estimates with the assumptions yields
\begin{align}
    &\left|
    \mathcal K_s
    \bigl(\widetilde H_j(r,\cdot)\bigr)(\xi)
    \right|\lesssim
     A_j(r)(1+|\xi|)^{N}.
\end{align}
    Since \(\sqrt{s}\) is bounded from below, this gives the desired
    bound for \(s\in(1/2,1)\).
    When \(s\in(0,1/2)\), we crucially use the fact that
    \begin{equation}
        0=M((r \de_r)^jH(r,\cdot))
        =
        \int_{\mathbb{R}^2}
        \e^{-|\eta|^2/4}\widetilde H_j(r,\eta)\,d\eta
    \end{equation}
    to rewrite
    \begin{align}
        4\pi \mathcal{K}_s(\widetilde H_j)(r,\xi)
        &=
        \int_{\mathbb{R}^2}
        \left[
        \frac{1}{(1-s)}
        \e^{-\frac{|\eta-\sqrt{s}\xi|^2}{4(1-s)}}
        -
        \e^{-|\eta|^2/4}
        \right]
        \widetilde H_j(r,\eta)\,d\eta
        \notag\\
        &=
        \int_{\mathbb{R}^2}
        \left[
        \mathcal{I}_1(\xi,\eta,s)
        +
        \mathcal{I}_2(\xi,\eta,s)
        +
        \mathcal{I}_3(\xi,\eta,s)
        \right]
        \widetilde H_j(r,\eta)\,d\eta,
    \end{align}
    where
    \begin{align}
        \mathcal{I}_1(\xi,\eta,s)
        &:=\frac{s}{1-s}
        \e^{-\frac{|\eta-\sqrt{s}\xi|^2}{4(1-s)}},
        \\
        \mathcal{I}_2(\xi,\eta,s)
        &:=
        \e^{-\frac{|\eta-\sqrt{s}\xi|^2}{4}}
        \left(
        \e^{-\frac{|\eta-\sqrt{s}\xi|^2}{4}\frac{s}{1-s}}
        -1
        \right),
        \\
        \mathcal{I}_3(\xi,\eta,s)
        &:=
        \e^{-\frac{|\eta-\sqrt{s}\xi|^2}{4}}
        -\e^{-|\eta|^2/4}.
    \end{align}
    For \(\mathcal{I}_1\), since \(s\in(0,1/2)\), we have
    \begin{equation}
        |\mathcal{I}_1(\xi,\eta,s)|
        \lesssim
        s\, 
        \e^{-\frac{|\eta-\sqrt{s}\xi|^2}{4(1-s)}}.
    \end{equation}
        For the term \(\mathcal{I}_2\), the bound
    \(1-\e^{-u}\leq u\) for \(u\geq0\) and the fact that $1-s\geq 1/2$ gives
    \begin{equation}
        |\mathcal{I}_2(\xi,\eta,s)|
        \lesssim
        s\,
        \e^{-\frac{|\eta-\sqrt{s}\xi|^2}{4}}
        |\eta-\sqrt{s}\xi|^2.
    \end{equation}
    Hence, using \cref{eq:GaussaniMomentBound} we arrive at
    \begin{equation}
    \label{eq:I1Estimate}
        \int_{\mathbb{R}^2}
        (|\mathcal{I}_1(\xi,\eta,s)|+|\mathcal{I}_2(\xi,\eta,s)|)
        |\widetilde H_j(r,\eta)|\,d\eta
        \lesssim
         A_j(r) s(1+|\xi|)^{N_{}}.
    \end{equation}    
    Finally, we estimate \(\mathcal{I}_3\), for which we note that
\begin{align}
    \mathcal{I}_3(\xi,\eta,s)
    &=
    \int_0^1
    \frac{d}{d\theta}
    \e^{-|\eta-\theta\sqrt{s}\xi|^2/4}
    \,d\theta=
    \frac{\sqrt{s}}{2}
    \int_0^1
    \xi\cdot(\eta-\theta\sqrt{s}\xi)
    \e^{-|\eta-\theta\sqrt{s}\xi|^2/4}
    \,d\theta.
\end{align}
Consequently,
\begin{equation}
    |\mathcal{I}_3(\xi,\eta,s)|
    \leq  C\sqrt{s}|\xi|
    \int_0^1
    \e^{-|\eta-\theta\sqrt{s}\xi|^2/8}
    \,d\theta.
\end{equation}
By the assumptions and
\cref{eq:GaussaniMomentBound}, we obtain
\begin{align}
    \int_{\mathbb{R}^2}
    |\mathcal{I}_3(\xi,\eta,s)|
    |\widetilde H_j(r,\eta)|\,d\eta
    \lesssim
    C A_j(r)\sqrt{s}|\xi|
    \int_0^1
    (1+\theta\sqrt{s}|\xi|)^N\,d\theta\lesssim
    C A_j(r)\sqrt{s}(1+|\xi|)^{N}.
    \label{eq:I3Estimate}
\end{align}
    Combining \cref{eq:I1Estimate,eq:I3Estimate}, we obtain
    \cref{eq:Ks_zero_order_sqrt_gain}.

Then, for every multi-index \(\alpha\), integration by parts gives
\begin{equation}
    \partial_\xi^\alpha
    \mathcal K_s
    \bigl(\widetilde H_j(r,\cdot)\bigr)
    =
    s^{|\alpha|/2}
    \mathcal K_s
    \bigl(\partial_\eta^\alpha
    \widetilde H_j(r,\cdot)\bigr).
\end{equation}
    When \(|\alpha|\geq1\), the Gaussian moment estimate and
    \(s^{|\alpha|/2}\leq\sqrt{s}\) give the desired gain.     Consequently,
    \begin{equation}
\label{eq:Ks_derivative_sqrt_gain}
    \left|
    \partial_\xi^\alpha
    \mathcal K_s
    \bigl(\widetilde H_j(r,\cdot)\bigr)(\xi)
    \right|
    \lesssim
    A_j(r)\sqrt s
    (1+|\xi|)^{N},
    \qquad 0<s<1.
\end{equation}

Going back to \cref{eq:FWK}, time differentiation acts
both on $\mathcal W_{\varsigma,q}$ and on the forcing.
For fixed $s$, we have
\[
    (\tau\partial_\tau)^j
    \widetilde H_0(\tau s,\eta)
    =
    \widetilde H_j(\tau s,\eta).
\]
Therefore, the Leibniz formula gives
\begin{align}
    &(\tau\partial_\tau)^\beta
    \partial_\xi^\alpha
    \widetilde f_{\varsigma,q}(\tau,\xi) =
    \sum_{j=0}^{\beta}\binom{\beta}{j}
    \int_0^1
    (\tau\partial_\tau)^{\beta-j}
    \mathcal W_{\varsigma,q}(\tau,s)
    \partial_\xi^\alpha
    \mathcal K_s
    \bigl(\widetilde H_j(\tau s,\cdot)\bigr)(\xi)\,ds.
\end{align}
Differentiation under the integral is justified by
\cref{eq:Ks_derivative_sqrt_gain}, the local boundedness
of $A_j$, and the integrability of the resulting
weights on compact time intervals.
We estimate this expression separately for
$\varsigma=1$ and $\varsigma=-1$.

    \smallskip
    \paragraph{\(\blackinwhitediamond\)
    \textbf{Case: \(\varsigma=1\).}}
    We have
    \begin{equation}
        \mathcal W_{1,q}(\tau,s)
        =
        s^{q-1}
        \left(\frac{1+\tau}{1+s\tau}\right)^q.
    \end{equation}

    Since
\begin{equation}
\tau\partial_\tau\log\mathcal W_{1,q}(\tau,s)
    =
    q\left(
    \frac{\tau}{1+\tau}
    -
    \frac{s\tau}{1+s\tau}
    \right),
\end{equation}
and all logarithmic time derivatives of the fractions
on the right-hand side are uniformly bounded, repeated
differentiation yields
\begin{equation}
\label{eq:Wplus_q_derivatives}
    \left|
    (\tau\partial_\tau)^k
    \mathcal W_{1,q}(\tau,s)
    \right|
    \lesssim\mathcal W_{1,q}(\tau,s),
    \qquad k\in\mathbb N_0.
\end{equation}
Moreover,
\[
    \mathcal W_{1,q}(\tau,s)\sqrt s
    =
    s^{-1/2}
    \left(
    \frac{s(1+\tau)}{1+s\tau}
    \right)^q
    \leq s^{-1/2},
\]
since
\[
    0\leq\frac{s(1+\tau)}{1+s\tau}\leq1.
\]
For $\tau \geq 0$, integration directly gives
\[\int_0^1\mathcal W_{1,q}(\tau,s)\sqrt s\,ds\le\int_0^1s^{-1/2}\,ds=2.\]
For each $j\leq\beta$,
\[
    A_j(\tau s)\leq
    \sup_{0\leq r\leq\tau}A_j(r).
\]
Combining the Leibniz formula with
\cref{eq:Ks_derivative_sqrt_gain} we obtain
\begin{equation}
    \left|
    (\tau\partial_\tau)^\beta
    \partial_\xi^\alpha
    \widetilde f_{1,q}(\tau,\xi)
    \right|
    \lesssim
    (1+|\xi|)^{N}
    \sum_{j=0}^{\beta}
    \sup_{0\leq r\leq\tau}A_j(r).
\end{equation}

    \smallskip
    \paragraph{\(\blackinwhitediamond\)
    \textbf{Case: \(\varsigma=-1\).}}
    We have
    \begin{equation}
        \mathcal W_{-1,q}(\tau,s)
        =
        s^{q-1}
        \left(\frac{1-\tau}{1-s\tau}\right)^q,
        \qquad 0\leq\tau<1.
    \end{equation}
As before, since $1-s\tau\geq1-\tau$, repeated logarithmic time
differentiation gives
\begin{equation}
\label{eq:Wminus_q_derivatives}
    \left|
    (\tau\partial_\tau)^k
    \mathcal W_{-1,q}(\tau,s)
    \right|
    \lesssim
    \frac{1}{(1-\tau)^k}
    \mathcal W_{-1,q}(\tau,s),
    \qquad k\in\mathbb N_0.
\end{equation}
We next claim that
\begin{equation}
\label{eq:Wminus_q_integral}
    \int_0^1
    \frac{\mathcal W_{-1,q}(\tau,s)}
    {1-s\tau}\,ds
    \leq\frac{2}{q},
    \qquad 0\leq\tau<1.
\end{equation}
Indeed, for $0\le\tau<1$, let $u=\frac{s(1-\tau)}{(1-s\tau)}$. Then,
\[\int_0^1\frac{\mathcal W_{-1,q}(\tau,s)}{1-s\tau}\,ds=\frac1q.\]
Since $\sqrt s\le1$ and $1-s\tau\ge1-\tau$, for every $j\ge0$,
\[\int_0^1\frac{\mathcal W_{-1,q}(\tau,s)\sqrt s}{(1-s\tau)^{j+1}}\,ds\le\frac1{(1-\tau)^j}\int_0^1\frac{\mathcal W_{-1,q}(\tau,s)}{1-s\tau}\,ds=\frac1{q(1-\tau)^j}.\]
Then, for every $0\leq j\leq\beta$,
\[
    A_j(\tau s)
    \leq
    \frac{
    \displaystyle
    \sup_{0\leq r\leq\tau}(1-r)^{j+1}A_j(r)
    }{
    (1-s\tau)^{j+1}
    }.
\]
Therefore, we have
\begin{align}
    \int_0^1
    \left|
    (\tau\partial_\tau)^{\beta-j}
    \right .& \left. \mathcal W_{-1,q}(\tau,s)
    \right|
    \left|
    \partial_\xi^\alpha
    \mathcal K_s
    \bigl(\widetilde H_j(\tau s,\cdot)\bigr)(\xi)
    \right|\,ds \\
    & \lesssim
    \frac{
    (1+|\xi|)^{N}
    }{
    (1-\tau)^{\beta-j}
    }
    \sup_{0\leq r\leq\tau}(1-r)^{j+1}A_j(r)
    \int_0^1
    \frac{\mathcal W_{-1,q}(\tau,s)\sqrt s}
    {(1-s\tau)^{j+1}}\,ds
    \\
    &\lesssim
    \frac{(1+|\xi|)^N}{(1-\tau)^\beta}
    \sup_{0\leq r\leq\tau}
    (1-r)^{j+1}A_j(r).
\end{align}
Summing over $0\leq j\leq\beta$, 
we obtain
\begin{equation}
    \left|
    (\tau\partial_\tau)^\beta
    \partial_\xi^\alpha
    \widetilde f_{-1,q}(\tau,\xi)
    \right|
    \lesssim
    \frac{(1+|\xi|)^{N}}{(1-\tau)^\beta}
    \sum_{j=0}^{\beta}
    \sup_{0\leq r\leq\tau}
    (1-r)^{j+1}A_j(r).
\end{equation}
This proves the estimate in the collapsing case.
    Finally, the assumptions near $\tau=0$ justify the
continuous extension of the logarithmic time derivatives.
Indeed, for $j\geq1$,
\[
    \widetilde H_j(r,\eta)\longrightarrow0
    \qquad\text{as }r\to 0,
\]
by smoothness of $H$. The local boundedness of $A_j$
and the weighted polynomial bounds provide the required
domination.
Moreover, for every fixed $0<s<1$,
\[
    (\tau\partial_\tau)^k
    \mathcal W_{\varsigma,q}(\tau,s)
    \longrightarrow0
    \qquad\text{as }\tau\to 0,
    \quad k\geq1.
\]
In each Leibniz term with $\beta\geq1$, either a positive
time derivative of the weight occurs, or $j\geq1$.
Dominated convergence therefore gives
\[
    (\tau\partial_\tau)^\beta
    \partial_\xi^\alpha
    \widetilde f_{\varsigma,q}(\tau,\xi)
    \longrightarrow0
    \qquad\text{as }\tau\to 0,
    \quad \beta\geq1.
\]
For $\beta=0$, smoothness of $H$ at zero and the
weighted polynomial bounds allow passage to the limit
in the spatially differentiated kernel formula
\cref{eq:FWK}. 
Indeed, for $0\leq\tau\leq1/2$, its
integrand is bounded by
$s^{q-1/2}(1+|\xi|)^N$, which is integrable in $s$.
Dominated convergence and the resolvent representation
therefore yield, for every multi-index $\alpha$,
\[
    \partial_\xi^\alpha
    \widetilde f_{\varsigma,q}(\tau,\xi)
    \longrightarrow
    \partial_\xi^\alpha
    \left[
        e^{|\xi|^2/4}(q-\mathcal L)^{-1}H(0,\cdot)
    \right](\xi)
    \qquad\text{as }\tau\to 0.
\]
Together with the initial trace proved above, this completes the proof.
\end{proof}

\newpage
\subsection*{AI usage statement}
This paper is the first research project of the second author, carried out during
the first year of his PhD. The key arguments and the outline of the manuscript were
developed through regular blackboard discussions and pen-and-paper computations.
Commercial LLMs were used to better understand the techniques employed and to
assist with the tasks specified below.
\begin{itemize}
\item During the editing of the manuscript, AI was used to review the manuscript in
depth, check for mathematical and typographical errors, generate all the figures,
refine some arguments, and assist with the verification of results in the
literature.
\item An important AI contribution concerns the spectral characterization of the
matrix associated with the linearized Helmholtz--Kirchhoff system, presented in
\cref{sec:selfsimilar_HK} and Lemma \ref{lem:SpectralPropertiesB}. The analysis in
\cite{ArefSelfSimilar} relies on algebraically involved identities that exploit in
a highly nontrivial way the conserved quantities of the system, such as those in
\cref{eq:groebli_condition_intro}. It yields the eigenvalues and claims, in
addition, that three eigenvectors are linearly independent, see Remark
\ref{rem:3arefvectors}. This last claim does not affect the main conclusions of
\cite{ArefSelfSimilar}, but from it we initially deduced that the matrix is
diagonalizable. In one of the review phases, AI flagged this
conclusion as an error and provided concrete examples in which the eigenvalues are
correct but Jordan blocks occur. We then checked all the computations in
\cite{ArefSelfSimilar} with a SageMath 10.7 script, fully generated by AI and
subsequently checked by the last author, which is available in the supplementary
material of the arXiv version of the paper. This confirmed that the key identities
in \cite{ArefSelfSimilar} are correct and that only the conclusion on
diagonalizability is not. The complete characterization of the configurations for
which the matrix is diagonalizable, presented in Appendix \ref{app:explicit}, was also
generated with AI and subsequently checked and refined by us.
\item The initial draft of the paper treated only the expanding case. We later
observed from the expression of $\epsilon$ that the collapsing case could also be
addressed on a suitable time interval, upon minor modifications of the arguments
for the expanding case. We identified the required modifications, and the
presentation in \cref{sec:collapsing} was written with the assistance of AI. After
\cite{nguyen2026inviscid} appeared, we used AI assistance also to check our
interpretation of its key points, as discussed in the introduction.
\item For \cref{sec:VanishingViscosity}, we sketched a proof outline based on the
techniques of \cite{CeciSeis24,albritton2025vanishingviscositynonuniquesolutions}.
AI assistance was used to fill in the details, which we then rewrote and
simplified.
\item We outlined the proof of Lemma \ref{lem:lin}. AI then supplied the full details, which we subsequently refined.
\item For all other arguments in the paper, we developed the key ideas and the
details ourselves. AI assistance was used to refine the exposition and to flag
typos and missing steps.
\end{itemize}

\smallskip
\textbf{Acknowledgments.} We are indebted to M.~Colombo for suggesting the
connection with vortex bursts and for discussions and suggestions throughout this work. We also thank Th.~Gallay for useful comments. Both authors were supported by the Swiss National Science
Foundation (SNSF Ambizione grant PZ00P2\_223294).
\section*{}
\bibliographystyle{siam}
\bibliography{bibclass}

\end{document}